\documentclass{amsart}
\usepackage{amssymb}
\usepackage{amsmath}
\usepackage{amsfonts}

\newtheorem{theorem}{Theorem}[section]
\theoremstyle{plain}

\newtheorem{corollary}[theorem]{Corollary}

\newtheorem{definition}[theorem]{Definition}

\newtheorem{lemma}[theorem]{Lemma}

\newtheorem{proposition}[theorem]{Proposition}
\newtheorem{remark}[theorem]{Remark}

\numberwithin{equation}{section}
\input{tcilatex}

\begin{document}
\title[On Schatten-class properties of pseudo-differential operators]{On
Schatten-von Neumann class properties of pseudo-differential operators.
Boulkhemair's method.}
\author{Gruia Arsu}
\address{Institute of Mathematics of The Romanian Academy\\
P.O. Box 1-174\\
RO-70700\\
Bucharest \\
Romania}
\email{Gruia.Arsu@imar.ro, agmilro@yahoo.com, agmilro@hotmail.com}
\maketitle

\begin{abstract}
We investigate the Schatten-von Neumann properties of pseudo-differential
operators using the method proposed by A. Boulkhemair in \cite{Boulkhemair 2}%
. The symbols are elements of the ideals $S_{w}^{p}$ of the Sj\"{o}strand
algebra $S_{w}$.
\end{abstract}

\tableofcontents

\section{Introduction}

In \cite{Arsu 1} and \cite{Arsu 2} we developed a method that allows us to
deal with Schatten-von Neumann class properties of pseudo-differential
operators. This method is inspired by the classical method of Kato and
Cordes used in the study of $L^{2}$ boundedness of pseudo-differential
operators.

In this paper we extend the results from \cite{Arsu 1} and \cite{Arsu 2} to
more general classses of symbols. These classes of symbols are particular
cases of modulation spaces and are ideals in the Sj\"{o}strand algebra $%
S_{w} $. The method we propose here is based on a spectral characterization
of these classes of symbols and on some results from \cite{Arsu 1} and \cite%
{Arsu 2}. This spectral characterization was first given in the case of the
Sj\"{o}strand algebra by A. Boulkhemair in \cite{Boulkhemair 2} and it was
used in the study of $L^{2}$ boundedness of pseudo-differential operators
and of global Fourier integral operators.

While in \cite{Arsu 1} and \cite{Arsu 2} the symbol spaces are defined by
Sobolev-type conditions, the classes of symbols used in this paper are
particular cases of modulation spaces defined by the decrease of the
short-time Fourier transform of the elements of these spaces. That the
results of this paper are more general than the results in \cite{Arsu 1} and 
\cite{Arsu 2} is a consequence of an embedding theorem which states that the
spaces of symbols used in the papers cited above are contained in symbol
spaces considered in this paper. This theorem allows to recover and to
extend all the results of the papers mentioned above.

Here is a brief description of the paper. The ideals $S_{w}^{p}$, $1\leq
p\leq \infty $, of the Sj\"{o}strand algebra $S_{w}$, are introduced in
section \ref{st31} . We give two definitions, a continuous one and a
discrete one. The equivalence of these definitions and the properties
derived from them are proven in this section. In section \ref{st32} we prove
the spectral characterization theorem of the ideals $S_{w}^{p}$. Then we
show how this theorem can be used to obtain some properties of the ideals $%
S_{w}^{p}$. The Schatten-von Neumann class properties of pseudo-differential
operators and links with the results in \cite{Arsu 1} and \cite{Arsu 2} are
presented in the sections that follow. We have two appendices where we prove
some auxiliary results that we need.

\section{The Ideals $S_{w}^{p}$ of the Sj\"{o}strand algebra $S_{w}\label%
{st31}$}

Modulation spaces were introduced by H. Feichtinger \cite{Feichtinger} and
may be regarded as a concrete means of measuring the time-frequency content
of functions. Since from the beginning the modulation spaces were treated by
analogy with the theory of Besov spaces, dyadic decomposition of the unity
being replaced with uniform decomposition of the unity. One of the reasons
to do this was to characterize the smoothness of the functions defined on
locally compact abelian groups which do not have dilation. Characterization
of smoothness is possible by replacing dilation by modulation operators.
Modulation spaces have found numerous applications in harmonic analysis,
time-frequency analysis, Gabor analysis, .... They also occur naturally in
the theory of pseudo-differential operators.

In this section we shall study a family of increasing spaces of symbols $%
S_{w}^{p}$, $1\leq p\leq \infty $ where $S_{w}^{\infty }=S_{w}$ is the Sj%
\"{o}strand algebra and $S_{w}^{1}$ is the Feichtinger algebra. These spaces
are special cases of modulation spaces. They were used by authors such as Sj%
\"{o}strand, Boulkhemair, Toft in the analysis of pseudo-differential
operators defined by symbols more general than usual. The study of these
spaces of symbols will be done by extending the methods used in the papers 
\cite{Sjostrand 1} and \cite{Sjostrand 2}.

We begin by proving some results that will be useful later. Their proof is
based on the Fubini theorem for distributions.

Let $\varphi ,\psi \in \mathcal{C}_{0}^{\infty }\left( \mathbb{R}^{n}\right) 
$ (or $\varphi ,\psi \in \mathcal{S}\left( \mathbb{R}^{n}\right) $). Then
the maps 
\begin{eqnarray*}
\mathbb{R}^{n}\times \mathbb{R}^{n} &\ni &\left( x,y\right) \overset{f}{%
\longrightarrow }\varphi \left( x\right) \psi \left( x-y\right) =\left(
\varphi \tau _{y}\psi \right) \left( x\right) \in 
\mathbb{C}
, \\
\mathbb{R}^{n}\times \mathbb{R}^{n} &\ni &\left( x,y\right) \overset{g}{%
\longrightarrow }\varphi \left( y\right) \psi \left( x-y\right) =\varphi
\left( y\right) \left( \tau _{y}\psi \right) \left( x\right) \in 
\mathbb{C}
,
\end{eqnarray*}%
are in $\mathcal{C}_{0}^{\infty }\left( \mathbb{R}^{n}\times \mathbb{R}%
^{n}\right) $ (respectively in $\mathcal{S}\left( \mathbb{R}^{n}\times 
\mathbb{R}^{n}\right) $). To see this we note that%
\begin{equation*}
f=\left( \varphi \otimes \psi \right) \circ T,\quad g=\left( \varphi \otimes
\psi \right) \circ S
\end{equation*}%
where 
\begin{eqnarray*}
T &:&\mathbb{R}^{n}\times \mathbb{R}^{n}\rightarrow \mathbb{R}^{n}\times 
\mathbb{R}^{n},\quad T\left( x,y\right) =\left( x,x-y\right) ,\quad T\equiv
\left( 
\begin{array}{cc}
\mathtt{I} & 0 \\ 
\mathtt{I} & -\mathtt{I}%
\end{array}%
\right) , \\
S &:&\mathbb{R}^{n}\times \mathbb{R}^{n}\rightarrow \mathbb{R}^{n}\times 
\mathbb{R}^{n},\quad S\left( x,y\right) =\left( y,x-y\right) ,\quad S\equiv
\left( 
\begin{array}{cc}
0 & \mathtt{I} \\ 
\mathtt{I} & -\mathtt{I}%
\end{array}%
\right) .
\end{eqnarray*}

Let $u\in \mathcal{D}^{\prime }\left( \mathbb{R}^{n}\right) $ (or $u\in 
\mathcal{S}^{\prime }\left( \mathbb{R}^{n}\right) $). Then 
\begin{eqnarray*}
\left\langle u\otimes 1,f\right\rangle &=&\left\langle \left( u\otimes
1\right) \left( x,y\right) ,\varphi \left( x\right) \psi \left( x-y\right)
\right\rangle \\
&=&\left\langle u\left( x\right) ,\left\langle 1\left( y\right) ,\varphi
\left( x\right) \psi \left( x-y\right) \right\rangle \right\rangle \\
&=&\left\langle u\left( x\right) ,\varphi \left( x\right) \left\langle
1\left( y\right) ,\psi \left( x-y\right) \right\rangle \right\rangle \\
&=&\left\langle u\left( x\right) ,\varphi \left( x\right) \dint \psi \left(
x-y\right) \mathtt{d}y\right\rangle \\
&=&\left( \dint \psi \right) \left\langle u,\varphi \right\rangle
\end{eqnarray*}%
and 
\begin{eqnarray*}
\left\langle u\otimes 1,f\right\rangle &=&\left\langle 1\left( y\right)
,\left\langle u\left( x\right) ,\varphi \left( x\right) \psi \left(
x-y\right) \right\rangle \right\rangle \\
&=&\dint \left\langle u,\varphi \tau _{y}\psi \right\rangle \mathtt{d}y.
\end{eqnarray*}%
It follows that

\begin{equation*}
\left( \dint \psi \right) \left\langle u,\varphi \right\rangle =\dint
\left\langle u,\varphi \tau _{y}\psi \right\rangle \mathtt{d}y
\end{equation*}%
valid for

\begin{itemize}
\item[\texttt{(i)}] $u\in \mathcal{D}^{\prime }\left( \mathbb{R}^{n}\right) $%
, $\varphi ,\psi \in \mathcal{C}_{0}^{\infty }\left( \mathbb{R}^{n}\right) $;

\item[\texttt{(ii)}] $u\in \mathcal{S}^{\prime }\left( \mathbb{R}^{n}\right)
,$ $\varphi ,\psi \in \mathcal{S}\left( \mathbb{R}^{n}\right) $.
\end{itemize}

We also have%
\begin{eqnarray*}
\left\langle u\otimes 1,g\right\rangle &=&\left\langle \left( u\otimes
1\right) \left( x,y\right) ,\varphi \left( y\right) \psi \left( x-y\right)
\right\rangle \\
&=&\left\langle u\left( x\right) ,\left\langle 1\left( y\right) ,\varphi
\left( y\right) \psi \left( x-y\right) \right\rangle \right\rangle \\
&=&\left\langle u\left( x\right) ,\left( \varphi \ast \psi \right) \left(
x\right) \right\rangle \\
&=&\left\langle u,\varphi \ast \psi \right\rangle
\end{eqnarray*}%
and 
\begin{eqnarray*}
\left\langle u\otimes 1,g\right\rangle &=&\left\langle 1\left( y\right)
,\left\langle u\left( x\right) ,\varphi \left( y\right) \psi \left(
x-y\right) \right\rangle \right\rangle \\
&=&\dint \varphi \left( y\right) \left\langle u,\tau _{y}\psi \right\rangle 
\mathtt{d}y.
\end{eqnarray*}%
Hence%
\begin{equation*}
\left\langle u,\varphi \ast \psi \right\rangle =\dint \varphi \left(
y\right) \left\langle u,\tau _{y}\psi \right\rangle \mathtt{d}y
\end{equation*}%
true for

\begin{itemize}
\item[\texttt{(i)}] $u\in \mathcal{D}^{\prime }\left( \mathbb{R}^{n}\right) $%
, $\varphi ,\psi \in \mathcal{C}_{0}^{\infty }\left( \mathbb{R}^{n}\right) $;

\item[\texttt{(ii)}] $u\in \mathcal{S}^{\prime }\left( \mathbb{R}^{n}\right)
,$ $\varphi ,\psi \in \mathcal{S}\left( \mathbb{R}^{n}\right) $.
\end{itemize}

\begin{lemma}
Let $\varphi ,\psi \in \mathcal{C}_{0}^{\infty }\left( \mathbb{R}^{n}\right) 
$ $($or $\varphi ,\psi \in \mathcal{S}\left( \mathbb{R}^{n}\right) )$ and $%
u\in \mathcal{D}^{\prime }\left( \mathbb{R}^{n}\right) $ $($or $u\in 
\mathcal{S}^{\prime }\left( \mathbb{R}^{n}\right) )$. Then 
\begin{equation}
\left( \dint \psi \right) \left\langle u,\varphi \right\rangle =\dint
\left\langle u,\varphi \tau _{y}\psi \right\rangle \mathtt{d}y  \label{sjo1}
\end{equation}

\begin{equation}
\left\langle u,\varphi \ast \psi \right\rangle =\dint \varphi \left(
y\right) \left\langle u,\tau _{y}\psi \right\rangle \mathtt{d}y  \label{sjo2}
\end{equation}
\end{lemma}

\begin{corollary}
\label{st9}Let $\varphi ,\psi \in \mathcal{S}\left( \mathbb{R}^{n}\right) $
and $u\in \mathcal{S}^{\prime }\left( \mathbb{R}^{n}\right) $. Then 
\begin{equation}
\varphi \ast \psi \left( D\right) u=\dint \varphi \left( \eta \right) \psi
\left( D-\eta \right) u\mathtt{d}\eta \qquad in\ \mathcal{S}^{\prime }\left( 
\mathbb{R}^{n}\right) .  \label{sjo3}
\end{equation}
\end{corollary}

\begin{proof}
Let $f\in \mathcal{S}\left( \mathbb{R}^{n}\right) $. Then 
\begin{eqnarray*}
\left\langle \varphi \ast \psi \left( D\right) u,f\right\rangle
&=&\left\langle \mathcal{F}^{-1}\left( \varphi \ast \psi \right) \widehat{u}%
,f\right\rangle \\
&=&\left\langle \left( \varphi \ast \psi \right) \widehat{u},\mathcal{F}%
^{-1}f\right\rangle \\
&=&\left\langle \left( \mathcal{F}^{-1}f\right) \widehat{u},\varphi \ast
\psi \right\rangle \\
&=&\dint \varphi \left( \eta \right) \left\langle \left( \mathcal{F}%
^{-1}f\right) \widehat{u},\tau _{\eta }\psi \right\rangle \mathtt{d}\eta \\
&=&\dint \varphi \left( \eta \right) \left\langle \psi \left( \cdot -\eta
\right) \widehat{u},\mathcal{F}^{-1}f\right\rangle \mathtt{d}\eta \\
&=&\dint \varphi \left( \eta \right) \left\langle \mathcal{F}^{-1}\psi
\left( \cdot -\eta \right) \widehat{u},f\right\rangle \mathtt{d}\eta \\
&=&\dint \varphi \left( \eta \right) \left\langle \psi \left( D-\eta \right)
u,f\right\rangle \mathtt{d}\eta .
\end{eqnarray*}
\end{proof}

\begin{lemma}
\label{st7}$(\mathtt{a})$ Let $\chi \in \mathcal{S}\left( \mathbb{R}%
^{n}\right) $ and $u\in \mathcal{S}^{\prime }\left( \mathbb{R}^{n}\right) $.
Then $\widehat{\chi u}\in \mathcal{S}^{\prime }\left( \mathbb{R}^{n}\right)
\cap \mathcal{C}_{pol}^{\infty }\left( \mathbb{R}^{n}\right) $. In fact we
have 
\begin{equation*}
\widehat{\chi u}\left( \xi \right) =\left\langle \mathtt{e}^{-\mathtt{i}%
\left\langle \cdot ,\xi \right\rangle }u,\chi \right\rangle =\left\langle u,%
\mathtt{e}^{-\mathtt{i}\left\langle \cdot ,\xi \right\rangle }\chi
\right\rangle ,\quad \xi \in \mathbb{R}^{n}.
\end{equation*}

$(\mathtt{b})$ Let $u\in \mathcal{D}^{\prime }\left( \mathbb{R}^{n}\right) $ 
$($or $u\in \mathcal{S}^{\prime }\left( \mathbb{R}^{n}\right) )$ and $\chi
\in \mathcal{C}_{0}^{\infty }\left( \mathbb{R}^{n}\right) $ $($or $\chi \in 
\mathcal{S}\left( \mathbb{R}^{n}\right) )$. Then 
\begin{equation*}
\mathbb{R}^{n}\times \mathbb{R}^{n}\ni \left( y,\xi \right) \rightarrow 
\widehat{u\tau _{y}\chi }\left( \xi \right) =\left\langle u,\mathtt{e}^{-%
\mathtt{i}\left\langle \cdot ,\xi \right\rangle }\chi \left( \cdot -y\right)
\right\rangle \in 
\mathbb{C}%
\end{equation*}%
is a $\mathcal{C}^{\infty }$-function.
\end{lemma}

\begin{proof}
Let $q:\mathbb{R}_{x}^{n}\times \mathbb{R}_{\xi }^{n}\rightarrow \mathbb{R}$%
, $q\left( x,\xi \right) =\left\langle x,\xi \right\rangle $. Then $\mathtt{e%
}^{-\mathtt{i}q}\left( u\otimes 1\right) \in \mathcal{S}^{\prime }\left( 
\mathbb{R}_{x}^{n}\times \mathbb{R}_{\xi }^{n}\right) $. If $\varphi \in 
\mathcal{S}\left( \mathbb{R}_{\xi }^{n}\right) $, then we have%
\begin{eqnarray*}
\left\langle \mathtt{e}^{-\mathtt{i}q}\left( u\otimes 1\right) ,\chi \otimes
\varphi \right\rangle &=&\left\langle u\otimes 1,\mathtt{e}^{-\mathtt{i}%
q}\left( \chi \otimes \varphi \right) \right\rangle \\
&=&\left\langle u\left( x\right) ,\left\langle 1\left( \xi \right) ,\mathtt{e%
}^{-\mathtt{i}q\left( x,\xi \right) }\chi \left( x\right) \varphi \left( \xi
\right) \right\rangle \right\rangle \\
&=&\left\langle u\left( x\right) ,\chi \left( x\right) \left\langle 1\left(
\xi \right) ,\mathtt{e}^{-\mathtt{i}\left\langle x,\xi \right\rangle
}\varphi \left( \xi \right) \right\rangle \right\rangle \\
&=&\left\langle u,\chi \widehat{\varphi }\right\rangle =\left\langle 
\widehat{\chi u},\varphi \right\rangle
\end{eqnarray*}%
and 
\begin{eqnarray*}
\left\langle \widehat{\chi u},\varphi \right\rangle &=&\left\langle \mathtt{e%
}^{-\mathtt{i}q}\left( u\otimes 1\right) ,\chi \otimes \varphi \right\rangle
\\
&=&\left\langle 1\left( \xi \right) ,\left\langle u\left( x\right) ,\mathtt{e%
}^{-\mathtt{i}\left\langle x,\xi \right\rangle }\chi \left( x\right) \varphi
\left( \xi \right) \right\rangle \right\rangle \\
&=&\left\langle 1\left( \xi \right) ,\varphi \left( \xi \right) \left\langle
u,\mathtt{e}^{-\mathtt{i}\left\langle \cdot ,\xi \right\rangle }\chi
\right\rangle \right\rangle \\
&=&\left\langle 1\left( \xi \right) ,\varphi \left( \xi \right) \left\langle 
\mathtt{e}^{-\mathtt{i}\left\langle \cdot ,\xi \right\rangle }u,\chi
\right\rangle \right\rangle \\
&=&\dint \varphi \left( \xi \right) \left\langle \mathtt{e}^{-\mathtt{i}%
\left\langle \cdot ,\xi \right\rangle }u,\chi \right\rangle \mathtt{d}\xi
\end{eqnarray*}%
This proves that%
\begin{equation*}
\widehat{\chi u}\left( \xi \right) =\left\langle \mathtt{e}^{-\mathtt{i}%
\left\langle \cdot ,\xi \right\rangle }u,\chi \right\rangle ,\quad \xi \in 
\mathbb{R}^{n}.
\end{equation*}
\end{proof}

Let $u\in \mathcal{D}^{\prime }\left( \mathbb{R}^{n}\right) $ $($or $u\in 
\mathcal{S}^{\prime }\left( \mathbb{R}^{n}\right) )$ and $\chi \in \mathcal{C%
}_{0}^{\infty }\left( \mathbb{R}^{n}\right) $ $($or $\chi \in \mathcal{S}%
\left( \mathbb{R}^{n}\right) )$. For $1\leq p\leq \infty $ we define%
\begin{gather*}
U_{\chi ,p}:\mathbb{R}^{n}\rightarrow \left[ 0,+\infty \right) , \\
U_{\chi ,p}\left( \xi \right) =\left\{ 
\begin{array}{ccc}
\sup_{y\in \mathbb{R}^{n}}\left\vert \widehat{u\tau _{y}\chi }\left( \xi
\right) \right\vert & \text{\textit{if}} & p=\infty \\ 
\left( \int \left\vert \widehat{u\tau _{y}\chi }\left( \xi \right)
\right\vert ^{p}\mathtt{d}y\right) ^{1/p} & \text{\textit{if}} & 1\leq
p<\infty%
\end{array}%
\right. , \\
\widehat{u\tau _{y}\chi }\left( \xi \right) =\left\langle u,\mathtt{e}^{-%
\mathtt{i}\left\langle \cdot ,\xi \right\rangle }\chi \left( \cdot -y\right)
\right\rangle .
\end{gather*}

$U_{\chi ,\infty }$ is a lower semicontinuous function. If $1\leq p<\infty $%
, then we proceed as follows. The function 
\begin{equation*}
\left( y,\xi \right) \rightarrow \left\vert \widehat{u\tau _{y}\chi }\left(
\xi \right) \right\vert ^{p}=\left\vert \left\langle u,\mathtt{e}^{-\mathtt{i%
}\left\langle \cdot ,\xi \right\rangle }\chi \left( \cdot -y\right)
\right\rangle \right\vert ^{p}
\end{equation*}%
is continuous. Then we use theorem 8.8 $\left( \mathtt{a}\right) $ of \cite%
{Rudin1} to obtain that the function%
\begin{equation*}
\xi \rightarrow \int \left\vert \widehat{u\tau _{y}\chi }\left( \xi \right)
\right\vert ^{p}\mathtt{d}y
\end{equation*}%
is measurable. Hence $U_{\chi ,p}$ is a measurable function for $1\leq p\leq
\infty $.

Suppose that there is $\chi \in \mathcal{C}_{0}^{\infty }\left( \mathbb{R}%
^{n}\right) \smallsetminus 0$ $($or $\chi \in \mathcal{S}\left( \mathbb{R}%
^{n}\right) \smallsetminus 0)$ such that the function $U_{\chi ,p}$ belongs
to $L^{1}\left( \mathbb{R}^{n}\right) $.

Let $\widetilde{\chi }\in \mathcal{C}_{0}^{\infty }\left( \mathbb{R}%
^{n}\right) $ $($or $\widetilde{\chi }\in \mathcal{S}\left( \mathbb{R}%
^{n}\right) )$. By using $\left( \ref{sjo1}\right) $ we get%
\begin{eqnarray*}
\widehat{u\tau _{z}\widetilde{\chi }}\left( \xi \right) &=&\left\langle u,%
\mathtt{e}^{-\mathtt{i}\left\langle \cdot ,\xi \right\rangle }\tau _{z}%
\widetilde{\chi }\right\rangle \\
&=&\frac{1}{\left\Vert \chi \right\Vert _{L^{2}}^{2}}\dint \left\langle u,%
\mathtt{e}^{-\mathtt{i}\left\langle \cdot ,\xi \right\rangle }\left( \tau
_{z}\widetilde{\chi }\right) \left( \tau _{y}\overline{\chi }\right) \left(
\tau _{y}\chi \right) \right\rangle \mathtt{d}y \\
&=&\frac{1}{\left\Vert \chi \right\Vert _{L^{2}}^{2}}\dint \left\langle 
\widehat{u\tau _{y}\chi },\mathcal{F}^{-1}\left( \mathtt{e}^{-\mathtt{i}%
\left\langle \cdot ,\xi \right\rangle }\left( \tau _{z}\widetilde{\chi }%
\right) \left( \tau _{y}\overline{\chi }\right) \right) \right\rangle 
\mathtt{d}y.
\end{eqnarray*}

Next we evaluate 
\begin{equation*}
\mathcal{F}^{-1}\left( \mathtt{e}^{-\mathtt{i}\left\langle \cdot ,\xi
\right\rangle }\left( \tau _{z}\widetilde{\chi }\right) \left( \tau _{y}%
\overline{\chi }\right) \right) \left( \eta \right) =\left( 2\pi \right)
^{-n}\dint \mathtt{e}^{\mathtt{i}\left\langle x,\eta -\xi \right\rangle
}\left( \tau _{z}\widetilde{\chi }\right) \left( x\right) \left( \tau _{y}%
\overline{\chi }\right) \left( x\right) \mathtt{d}x.
\end{equation*}%
Using the identity 
\begin{equation*}
\left\langle \eta -\xi \right\rangle ^{2k}\mathtt{e}^{\mathtt{i}\left\langle
x,\eta -\xi \right\rangle }=\left( \mathtt{1}\mathbf{-\triangle }\right)
^{k}\left( \mathtt{e}^{\mathtt{i}\left\langle \cdot ,\eta -\xi \right\rangle
}\right) \left( x\right)
\end{equation*}%
and integrating by parts we obtain the representation 
\begin{multline*}
\left\langle \eta -\xi \right\rangle ^{2k}\mathcal{F}^{-1}\left( \mathtt{e}%
^{-\mathtt{i}\left\langle \cdot ,\xi \right\rangle }\left( \tau _{z}%
\widetilde{\chi }\right) \left( \tau _{y}\overline{\chi }\right) \right)
\left( \eta \right) \\
=\left( 2\pi \right) ^{-n}\dint \mathtt{e}^{\mathtt{i}\left\langle x,\eta
-\xi \right\rangle }\left( \mathtt{1}\mathbf{-\triangle }\right) ^{k}\left(
\left( \tau _{z}\widetilde{\chi }\right) \left( \tau _{y}\overline{\chi }%
\right) \right) \left( x\right) \mathtt{d}x.
\end{multline*}%
There is a continuous seminorm $p_{k}$ on $\mathcal{S}\left( \mathbb{R}%
^{n}\right) $ so that 
\begin{multline*}
\left( 2\pi \right) ^{-n}\left\vert \left( \mathtt{1}\mathbf{-\triangle }%
\right) ^{k}\left( \left( \tau _{z}\widetilde{\chi }\right) \left( \tau _{y}%
\overline{\chi }\right) \right) \left( x\right) \right\vert \leq p_{k}\left( 
\widetilde{\chi }\right) p_{k}\left( \chi \right) \left\langle
x-z\right\rangle ^{-4k}\left\langle x-y\right\rangle ^{-4k} \\
\leq 2^{2k}p_{k}\left( \widetilde{\chi }\right) p_{k}\left( \chi \right)
\left\langle 2x-z-y\right\rangle ^{-2k}\left\langle y-z\right\rangle ^{-2k}.
\end{multline*}%
Here we used the inequality%
\begin{equation*}
\left\langle X\right\rangle ^{-2N}\left\langle Y\right\rangle ^{-2N}\leq
2^{N}\left\langle X+Y\right\rangle ^{-N}\left\langle X-Y\right\rangle
^{-N},\quad X,Y\in \mathbb{R}^{m}
\end{equation*}%
which is a consequence of Peetre's inequality:%
\begin{equation*}
\left. 
\begin{array}{c}
\left\langle X+Y\right\rangle ^{N}\leq 2^{\frac{N}{2}}\left\langle
X\right\rangle ^{N}\left\langle Y\right\rangle ^{N} \\ 
\\ 
\left\langle X-Y\right\rangle ^{N}\leq 2^{\frac{N}{2}}\left\langle
X\right\rangle ^{N}\left\langle Y\right\rangle ^{N}%
\end{array}%
\right\} \Rightarrow \left\langle X+Y\right\rangle ^{N}\left\langle
X-Y\right\rangle ^{N}\leq 2^{N}\left\langle X\right\rangle ^{2N}\left\langle
Y\right\rangle ^{2N}
\end{equation*}

We choose $k=\left[ \frac{n}{2}\right] +1$ such that $n+1\leq 2k\leq n+2$.
Returning to the integral we obtain%
\begin{multline*}
\left\vert \mathcal{F}^{-1}\left( \mathtt{e}^{-\mathtt{i}\left\langle \cdot
,\xi \right\rangle }\left( \tau _{z}\widetilde{\chi }\right) \left( \tau _{y}%
\overline{\chi }\right) \right) \left( \eta \right) \right\vert \\
\leq 2^{2k}p_{k}\left( \widetilde{\chi }\right) p_{k}\left( \chi \right)
\left\langle \eta -\xi \right\rangle ^{-2k}\left\langle y-z\right\rangle
^{-2k}\dint \left\langle 2x-z-y\right\rangle ^{-2k}\mathtt{d}x \\
\leq 2^{2k-n}p_{k}\left( \widetilde{\chi }\right) p_{k}\left( \chi \right)
\left( \dint \left\langle X\right\rangle ^{-2k}\mathtt{d}X\right)
\left\langle \eta -\xi \right\rangle ^{-2k}\left\langle y-z\right\rangle
^{-2k} \\
\leq C\left\langle \eta -\xi \right\rangle ^{-2k}\left\langle
y-z\right\rangle ^{-2k}.
\end{multline*}%
where $C=C_{\chi ,\widetilde{\chi },n}=2^{2}p_{k}\left( \widetilde{\chi }%
\right) p_{k}\left( \chi \right) \left\Vert \left\langle \cdot \right\rangle
^{-2k}\right\Vert _{L^{1}}$. Hence 
\begin{equation*}
\left\vert \mathcal{F}^{-1}\left( \mathtt{e}^{-\mathtt{i}\left\langle \cdot
,\xi \right\rangle }\left( \tau _{z}\widetilde{\chi }\right) \left( \tau _{y}%
\overline{\chi }\right) \right) \left( \eta \right) \right\vert \leq
C\left\langle \eta -\xi \right\rangle ^{-2k}\left\langle y-z\right\rangle
^{-2k}.
\end{equation*}%
Further%
\begin{multline*}
\left\Vert \chi \right\Vert _{L^{2}}^{2}\left\vert \widehat{u\tau _{z}%
\widetilde{\chi }}\left( \xi \right) \right\vert \leq \iint \left\vert 
\widehat{u\tau _{y}\chi }\left( \eta \right) \right\vert \left\vert \mathcal{%
F}^{-1}\left( \mathtt{e}^{-\mathtt{i}\left\langle \cdot ,\xi \right\rangle
}\left( \tau _{z}\widetilde{\chi }\right) \left( \tau _{y}\overline{\chi }%
\right) \right) \left( \eta \right) \right\vert \mathtt{d}y\mathtt{d}\eta \\
\leq C\iint \left\vert \widehat{u\tau _{y}\chi }\left( \eta \right)
\right\vert \left\langle y-z\right\rangle ^{-2k}\left\langle \eta -\xi
\right\rangle ^{-2k}\mathtt{d}y\mathtt{d}\eta \\
\leq C\int \left( \dint \left\vert \widehat{u\tau _{y}\chi }\left( \eta
\right) \right\vert ^{p}\left\langle y-z\right\rangle ^{-2k}\mathtt{d}%
y\right) ^{1/p}\left( \dint \left\langle y-z\right\rangle ^{-2k}\mathtt{d}%
y\right) ^{1/q}\left\langle \eta -\xi \right\rangle ^{-2k}\mathtt{d}\eta \\
=C\left\Vert \left\langle \cdot \right\rangle ^{-2k}\right\Vert
_{L^{1}}^{1/q}\int \left( \dint \left\vert \widehat{u\tau _{y}\chi }\left(
\eta \right) \right\vert ^{p}\left\langle y-z\right\rangle ^{-2k}\mathtt{d}%
y\right) ^{1/p}\left\langle \eta -\xi \right\rangle ^{-2k}\mathtt{d}\eta \\
=C\left\Vert \left\langle \cdot \right\rangle ^{-2k}\right\Vert
_{L^{1}}^{1/q}\int f\left( z,\eta \right) \left\langle \eta -\xi
\right\rangle ^{-2k}\mathtt{d}\eta ,
\end{multline*}%
where%
\begin{equation*}
f\left( z,\eta \right) ^{p}=\dint \left\vert \widehat{u\tau _{y}\chi }\left(
\eta \right) \right\vert ^{p}\left\langle y-z\right\rangle ^{-2k}\mathtt{d}y
\end{equation*}%
and 
\begin{eqnarray*}
\left\Vert f\left( \cdot ,\eta \right) \right\Vert _{L^{p}} &=&\left\Vert
\left\langle \cdot \right\rangle ^{-2k}\right\Vert _{L^{1}}^{1/p}\left( \int
\left\vert \widehat{u\tau _{y}\chi }\left( \eta \right) \right\vert ^{p}%
\mathtt{d}y\right) ^{1/p} \\
&=&\left\Vert \left\langle \cdot \right\rangle ^{-2k}\right\Vert
_{L^{1}}^{1/p}U_{\chi ,p}\left( \eta \right) .
\end{eqnarray*}%
Next we apply the integral version of Minkowski's inequality. We obtain that%
\begin{eqnarray*}
U_{\widetilde{\chi },p}\left( \xi \right) &=&\left( \int \left\vert \widehat{%
u\tau _{z}\widetilde{\chi }}\left( \xi \right) \right\vert ^{p}\mathtt{d}%
z\right) ^{1/p} \\
&\leq &\frac{C\left\Vert \left\langle \cdot \right\rangle ^{-2k}\right\Vert
_{L^{1}}^{1/q}}{\left\Vert \chi \right\Vert _{L^{2}}^{2}}\int \left\Vert
f\left( \cdot ,\eta \right) \right\Vert _{L^{p}}\left\langle \eta -\xi
\right\rangle ^{-2k}\mathtt{d}\eta \\
&=&\frac{C\left\Vert \left\langle \cdot \right\rangle ^{-2k}\right\Vert
_{L^{1}}^{1/q+1/p}}{\left\Vert \chi \right\Vert _{L^{2}}^{2}}\int U_{\chi
,p}\left( \eta \right) \left\langle \eta -\xi \right\rangle ^{-2k}\mathtt{d}%
\eta \\
&=&\frac{C\left\Vert \left\langle \cdot \right\rangle ^{-2k}\right\Vert
_{L^{1}}}{\left\Vert \chi \right\Vert _{L^{2}}^{2}}\left( U_{\chi ,p}\ast
\left\langle \cdot \right\rangle ^{-2k}\right) \left( \xi \right) .
\end{eqnarray*}%
Hence 
\begin{equation*}
U_{\widetilde{\chi },p}\leq \left( 2\cdot \frac{\left\Vert \left\langle
\cdot \right\rangle ^{-2k}\right\Vert _{L^{1}}}{\left\Vert \chi \right\Vert
_{L^{2}}}\right) ^{2}p_{k}\left( \widetilde{\chi }\right) p_{k}\left( \chi
\right) U_{\chi ,p}\ast \left\langle \cdot \right\rangle ^{-2k}
\end{equation*}%
with $\left\langle \cdot \right\rangle ^{-2k}\ast U_{\chi ,p}\in L^{1}\left( 
\mathbb{R}^{n}\right) $ since $k=\left[ \frac{n}{2}\right] +1$ implies that $%
\left\langle \cdot \right\rangle ^{-2k}\in L^{1}\left( \mathbb{R}^{n}\right) 
$.

\begin{definition}
Let $1\leq p\leq \infty $. We say that a distribution $u\in \mathcal{D}%
^{\prime }\left( \mathbb{R}^{n}\right) $ belongs to $S_{w}^{p}\left( \mathbb{%
R}^{n}\right) $ if there is $\chi \in \mathcal{C}_{0}^{\infty }\left( 
\mathbb{R}^{n}\right) \smallsetminus 0$ such that the measurable function%
\begin{gather*}
U_{\chi ,p}:\mathbb{R}^{n}\rightarrow \left[ 0,+\infty \right) , \\
U_{\chi ,p}\left( \xi \right) =\left\{ 
\begin{array}{ccc}
\sup_{y\in \mathbb{R}^{n}}\left\vert \widehat{u\tau _{y}\chi }\left( \xi
\right) \right\vert & \text{\textit{if}} & p=\infty \\ 
\left( \int \left\vert \widehat{u\tau _{y}\chi }\left( \xi \right)
\right\vert ^{p}\mathtt{d}y\right) ^{1/p} & \text{\textit{if}} & 1\leq
p<\infty%
\end{array}%
\right. , \\
\widehat{u\tau _{y}\chi }\left( \xi \right) =\left\langle u,\mathtt{e}^{-%
\mathtt{i}\left\langle \cdot ,\xi \right\rangle }\chi \left( \cdot -y\right)
\right\rangle .
\end{gather*}%
belongs to $L^{1}\left( \mathbb{R}^{n}\right) $.
\end{definition}

Until now we have shown the following:

\begin{proposition}
\label{st6}$(\mathtt{a})$ Let $u\in S_{w}^{p}\left( \mathbb{R}^{n}\right) $
and let $\chi \in \mathcal{C}_{0}^{\infty }\left( \mathbb{R}^{n}\right) $.
Then the measurable function%
\begin{gather*}
U_{\chi ,p}:\mathbb{R}^{n}\rightarrow \left[ 0,+\infty \right) , \\
U_{\chi ,p}\left( \xi \right) =\left\{ 
\begin{array}{ccc}
\sup_{y\in \mathbb{R}^{n}}\left\vert \widehat{u\tau _{y}\chi }\left( \xi
\right) \right\vert & \text{\textit{if}} & p=\infty \\ 
\left( \int \left\vert \widehat{u\tau _{y}\chi }\left( \xi \right)
\right\vert ^{p}\mathtt{d}y\right) ^{1/p} & \text{\textit{if}} & 1\leq
p<\infty%
\end{array}%
\right. , \\
\widehat{u\tau _{y}\chi }\left( \xi \right) =\left\langle u,\mathtt{e}^{-%
\mathtt{i}\left\langle \cdot ,\xi \right\rangle }\chi \left( \cdot -y\right)
\right\rangle .
\end{gather*}
belongs to $L^{1}\left( \mathbb{R}^{n}\right) $.

$(\mathtt{b})$ If we fix $\chi \in \mathcal{C}_{0}^{\infty }\left( \mathbb{R}%
^{n}\right) \smallsetminus 0$ and if we put%
\begin{equation*}
\left\Vert u\right\Vert _{S_{w}^{p},\chi }=\int U_{\chi ,p}\left( \xi
\right) \mathtt{d}\xi =\left\Vert U_{\chi ,p}\right\Vert _{L^{1}},\quad u\in
S_{w}\left( \mathbb{R}^{n}\right) ,
\end{equation*}%
then $\left\Vert \cdot \right\Vert _{S_{w}^{p},\chi }$ is a norm on $%
S_{w}^{p}\left( \mathbb{R}^{n}\right) $ and the topology that defines does
not depend on the choice of the function $\chi \in \mathcal{C}_{0}^{\infty
}\left( \mathbb{R}^{n}\right) \smallsetminus 0$.
\end{proposition}

Let $k$ be an integer $\geq 0$ or $k=\infty $. We define the following space 
\begin{gather*}
\mathcal{BC}^{k}\left( \mathbb{R}^{n}\right) =\left\{ f\in \mathcal{C}%
^{k}\left( \mathbb{R}^{n}\right) :f\text{ \textit{and its derivatives of
order }}\leq k\text{ \textit{are bounded}}\right\} , \\
\left\Vert f\right\Vert _{\mathcal{BC}^{l}}=\max_{j\leq l}\sup_{x\in
X}\left\Vert f^{\left( j\right) }\left( x\right) \right\Vert <\infty ,\quad
l<k+1.
\end{gather*}

To obtain an equivalent definition of the space $S_{w}^{p}\left( \mathbb{R}%
^{n}\right) $ we need some preparation. If $\varepsilon _{1},...,\varepsilon
_{n}$ is a basis in $\mathbb{R}^{n}$, we say that $\Gamma =\oplus _{j=1}^{n}%
\mathbb{Z}
\varepsilon _{j}$ is a lattice. We shall use the following notations: 
\begin{eqnarray*}
\mathtt{C} &=&\mathtt{C}_{\Gamma }=\left\{ \sum_{j=1}^{n}t_{j}\varepsilon
_{j}:\left( t_{1},...,t_{n}\right) \in \left[ 0,1\right) ^{n}\right\} , \\
\mathtt{C}_{\gamma } &=&\mathtt{C}_{\Gamma ,\gamma }=\gamma +\mathtt{C}%
_{\Gamma },\quad \gamma \in \Gamma .
\end{eqnarray*}

Let $s>n$. We consider the function%
\begin{equation*}
f_{s}\left( x\right) =\sum_{\gamma \in \Gamma }\left\langle x+\gamma
\right\rangle ^{-s}\leq 2^{s/2}\left\langle x\right\rangle ^{s}\sum_{\gamma
\in \Gamma }\left\langle \gamma \right\rangle ^{-s}<\infty .
\end{equation*}%
Then $f_{s}\in \mathcal{BC}^{\infty }\left( \mathbb{R}^{n}\right) $ and is $%
\Gamma $-periodic.

Let $x,y\in \overline{\mathtt{C}}$. Then using Peetre's inequality we obtain:%
\begin{eqnarray*}
f_{s}\left( x\right) &=&\sum_{\gamma \in \Gamma }\left\langle x-y+y+\gamma
\right\rangle ^{-s}\leq 2^{s/2}\left\langle x-y\right\rangle
^{s}\sum_{\gamma \in \Gamma }\left\langle y+\gamma \right\rangle ^{-s} \\
&\leq &2^{s}\left\langle x\right\rangle ^{s}\left\langle y\right\rangle
^{s}f_{s}\left( y\right) \\
&\leq &2^{s}\left( \sup_{z\in \overline{\mathtt{C}}}\left\langle
z\right\rangle ^{s}\right) ^{2}f_{s}\left( y\right) ,
\end{eqnarray*}%
i.e.%
\begin{equation}
f_{s}\left( x\right) \leq 2^{s}\left( \sup_{z\in \overline{\mathtt{C}}%
}\left\langle z\right\rangle ^{s}\right) ^{2}f_{s}\left( y\right) .
\label{st1}
\end{equation}

Since $f_{s}\in \mathcal{BC}^{\infty }\left( \mathbb{R}^{n}\right) $ and is $%
\Gamma $-periodic, it follows that there are $x_{0},y_{0}\in \overline{%
\mathtt{C}}$ such that 
\begin{equation*}
f_{s}\left( x_{0}\right) =\sup_{z\in \overline{\mathtt{C}}}f_{s}\left(
z\right) ,\quad f_{s}\left( y_{0}\right) =\inf_{z\in \overline{\mathtt{C}}%
}f_{s}\left( z\right) .
\end{equation*}%
Then 
\begin{equation*}
f_{s}\left( y_{0}\right) \leq \frac{1}{\mathtt{vol}\left( \mathtt{C}\right) }%
\int_{\mathtt{C}}f_{s}\left( y\right) \mathtt{d}y\leq f_{s}\left(
x_{0}\right) .
\end{equation*}%
If we note that%
\begin{equation*}
\int_{\mathtt{C}}f_{s}\left( y\right) \mathtt{d}y=\sum_{\gamma \in \Gamma
}\int_{\mathtt{C}}\left\langle y+\gamma \right\rangle ^{-s}\mathtt{d}%
y=\sum_{\gamma \in \Gamma }\int_{\gamma +\mathtt{C}}\left\langle
y\right\rangle ^{-s}\mathtt{d}y=\int \left\langle y\right\rangle ^{-s}%
\mathtt{d}y,
\end{equation*}%
then we get%
\begin{equation*}
f_{s}\left( y_{0}\right) \leq \frac{1}{\mathtt{vol}\left( \mathtt{C}\right) }%
\int \left\langle y\right\rangle ^{-s}\mathtt{d}y\leq f_{s}\left(
x_{0}\right) .
\end{equation*}%
Using $\left( \text{\ref{st1}}\right) $ we obtain that%
\begin{eqnarray*}
2^{-s}\left( \sup_{z\in \overline{\mathtt{C}}}\left\langle z\right\rangle
^{s}\right) ^{-2}f_{s}\left( x_{0}\right) &\leq &f_{s}\left( y_{0}\right)
\leq \frac{1}{\mathtt{vol}\left( \mathtt{C}\right) }\int \left\langle
y\right\rangle ^{-s}\mathtt{d}y \\
&\leq &f_{s}\left( x_{0}\right) \leq 2^{s}\left( \sup_{z\in \overline{%
\mathtt{C}}}\left\langle z\right\rangle ^{s}\right) ^{2}f_{s}\left(
y_{0}\right) ,
\end{eqnarray*}%
\begin{equation*}
2^{-s}\left( \sup_{z\in \overline{\mathtt{C}}}\left\langle z\right\rangle
^{s}\right) ^{-2}f_{s}\left( x_{0}\right) \leq \frac{1}{\mathtt{vol}\left( 
\mathtt{C}\right) }\int \left\langle y\right\rangle ^{-s}\mathtt{d}y\leq
2^{s}\left( \sup_{z\in \overline{\mathtt{C}}}\left\langle z\right\rangle
^{s}\right) ^{2}f_{s}\left( y_{0}\right) .
\end{equation*}%
For $x\in \overline{\mathtt{C}}$ we have 
\begin{eqnarray*}
f_{s}\left( x\right) &\geq &f_{s}\left( y_{0}\right) \geq 2^{-s}\left(
\sup_{z\in \overline{\mathtt{C}}}\left\langle z\right\rangle ^{s}\right)
^{-2}\frac{1}{\mathtt{vol}\left( \mathtt{C}\right) }\int \left\langle
y\right\rangle ^{-s}\mathtt{d}y, \\
f_{s}\left( x\right) &\leq &f_{s}\left( x_{0}\right) \leq 2^{s}\left(
\sup_{z\in \overline{\mathtt{C}}}\left\langle z\right\rangle ^{s}\right) ^{2}%
\frac{1}{\mathtt{vol}\left( \mathtt{C}\right) }\int \left\langle
y\right\rangle ^{-s}\mathtt{d}y.
\end{eqnarray*}%
Since $f_{s}$ is $\Gamma $-periodic, it follows that 
\begin{equation*}
2^{-s}\left( \sup_{z\in \overline{\mathtt{C}}}\left\langle z\right\rangle
^{s}\right) ^{-2}\frac{1}{\mathtt{vol}\left( \mathtt{C}\right) }\int
\left\langle y\right\rangle ^{-s}\mathtt{d}y\leq f_{s}\left( x\right) \leq
2^{s}\left( \sup_{z\in \overline{\mathtt{C}}}\left\langle z\right\rangle
^{s}\right) ^{2}\frac{1}{\mathtt{vol}\left( \mathtt{C}\right) }\int
\left\langle y\right\rangle ^{-s}\mathtt{d}y
\end{equation*}%
for any $x\in \mathbb{R}^{n}$.

Thus we proved the following result.

\begin{lemma}
\label{st2}Let $s>n$. Then for any $x\in \mathbb{R}^{n}$ we have 
\begin{equation*}
\frac{2^{-s}\left( \sup_{z\in \overline{\mathtt{C}}}\left\langle
z\right\rangle ^{s}\right) ^{-2}}{\mathtt{vol}\left( \mathtt{C}\right) }\int
\left\langle y\right\rangle ^{-s}\mathtt{d}y\leq \sum_{\gamma \in \Gamma
}\left\langle x+\gamma \right\rangle ^{-s}\leq \frac{2^{s}\left( \sup_{z\in 
\overline{\mathtt{C}}}\left\langle z\right\rangle ^{s}\right) ^{2}}{\mathtt{%
vol}\left( \mathtt{C}\right) }\int \left\langle y\right\rangle ^{-s}\mathtt{d%
}y.
\end{equation*}
\end{lemma}

Let $\Gamma \subset \mathbb{R}^{n}$ be a lattice. Let $\psi \in \mathcal{S}%
\left( \mathbb{R}^{n}\right) $. Then $\sum_{\gamma \in \Gamma }\tau _{\gamma
}\psi =\sum_{\gamma \in \Gamma }\psi \left( \cdot -\gamma \right) $ is
uniformly convergent on compact subsets of $\mathbb{R}^{n}$. Since $\partial
^{\alpha }\psi \in \mathcal{S}\left( \mathbb{R}^{n}\right) $, it follows
that there is $\Psi \in \mathcal{C}^{\infty }\left( \mathbb{R}^{n}\right) $
such that%
\begin{equation*}
\Psi =\sum_{\gamma \in \Gamma }\tau _{\gamma }\psi =\sum_{\gamma \in \Gamma
}\psi \left( \cdot -\gamma \right) \text{\quad \textit{in }}\mathcal{C}%
^{\infty }\left( \mathbb{R}^{n}\right) .
\end{equation*}%
In addition we have $\tau _{\gamma }\Psi =\Psi \left( \cdot -\gamma \right)
=\Psi $ for any $\gamma \in \Gamma $. From here we obtain that $\Psi \in 
\mathcal{BC}^{\infty }\left( \mathbb{R}^{n}\right) $. If $\Psi \left(
y\right) \neq 0$ for any $y\in 
\mathbb{R}
^{n}$, then $\frac{1}{\Psi }\in \mathcal{BC}^{\infty }\left( \mathbb{R}%
^{n}\right) $.

Let $\varphi \in \mathcal{S}\left( \mathbb{R}^{n}\right) $. Then 
\begin{equation*}
\varphi \Psi =\sum_{\gamma \in \Gamma }\varphi \left( \tau _{\gamma }\psi
\right)
\end{equation*}%
with the series convergent in $\mathcal{S}\left( \mathbb{R}^{n}\right) $.

Indeed we have 
\begin{multline*}
\sum_{\gamma \in \Gamma }\left\langle x\right\rangle ^{k}\left\vert \partial
^{\alpha }\varphi \left( x\right) \partial ^{\beta }\psi \left( x-\gamma
\right) \right\vert \\
\leq \sup_{y}\left\langle y\right\rangle ^{n+1}\left\vert \partial ^{\beta
}\psi \left( y\right) \right\vert \sum_{\gamma \in \Gamma }\left\langle
x\right\rangle ^{k}\left\vert \partial ^{\alpha }\varphi \left( x\right)
\left\langle x-\gamma \right\rangle ^{-n-1}\right\vert \\
\leq 2^{\frac{n+1}{2}}\sup_{y}\left\langle y\right\rangle ^{n+1}\left\vert
\partial ^{\beta }\psi \left( y\right) \right\vert \sup_{z}\left\langle
z\right\rangle ^{k+n+1}\left\vert \partial ^{\alpha }\varphi \left( z\right)
\right\vert \sum_{\gamma \in \Gamma }\left\langle \gamma \right\rangle
^{-n-1}.
\end{multline*}%
This estimate proves the convergence of the series in $\mathcal{S}\left( 
\mathbb{R}^{n}\right) $. Let $\chi $ be the sum of the series $\sum_{\gamma
\in \Gamma }\varphi \left( \tau _{\gamma }\psi \right) $ in $\mathcal{S}%
\left( \mathbb{R}^{n}\right) $. Then for any $y\in 
\mathbb{R}
^{n}$ we have%
\begin{eqnarray*}
\chi \left( y\right) &=&\left\langle \delta _{y},\chi \right\rangle
=\left\langle \delta _{y},\sum_{\gamma \in \Gamma }\varphi \left( \tau
_{\gamma }\psi \right) \right\rangle \\
&=&\sum_{\gamma \in \Gamma }\left\langle \delta _{y},\varphi \left( \tau
_{\gamma }\psi \right) \right\rangle =\sum_{\gamma \in \Gamma }\varphi
\left( y\right) \psi \left( y-\gamma \right) \\
&=&\varphi \left( y\right) \Psi \left( y\right) .
\end{eqnarray*}%
So $\varphi \Psi =\sum_{\gamma \in \Gamma }\varphi \left( \tau _{\gamma
}\psi \right) $ in $\mathcal{S}\left( \mathbb{R}^{n}\right) $.

If $\psi ,\varphi \in \mathcal{C}_{0}^{\infty }\left( \mathbb{R}^{n}\right) $
and $\mathcal{S}\left( \mathbb{R}^{n}\right) $ is replaced by $\mathcal{C}%
_{0}^{\infty }\left( \mathbb{R}^{n}\right) $, then the previous observations
are trivial.

\begin{corollary}
\label{st3}Let $u\in \mathcal{D}^{\prime }\left( \mathbb{R}^{n}\right) $ $($%
or $u\in \mathcal{S}^{\prime }\left( \mathbb{R}^{n}\right) )$ and $\psi
,\varphi \in \mathcal{C}_{0}^{\infty }\left( \mathbb{R}^{n}\right) $ $($or $%
\psi ,\varphi \in \mathcal{S}\left( \mathbb{R}^{n}\right) )$. Then $\Psi
=\sum_{\gamma \in \Gamma }\tau _{\gamma }\psi \in \mathcal{BC}^{\infty
}\left( \mathbb{R}^{n}\right) $ is $\Gamma $-periodic and%
\begin{equation*}
\left\langle u,\varphi \Psi \right\rangle =\sum_{\gamma \in \Gamma
}\left\langle u,\varphi \left( \tau _{\gamma }\psi \right) \right\rangle .
\end{equation*}
\end{corollary}

Let $\Gamma \subset \mathbb{R}^{n}$ be a lattice. Let $u\in \mathcal{D}%
^{\prime }\left( \mathbb{R}^{n}\right) $ $($or $u\in \mathcal{S}^{\prime
}\left( \mathbb{R}^{n}\right) )$ and $\chi \in \mathcal{C}_{0}^{\infty
}\left( \mathbb{R}^{n}\right) $ $($or $\chi \in \mathcal{S}\left( \mathbb{R}%
^{n}\right) )$. For $1\leq p\leq \infty $ we define%
\begin{gather*}
U_{\Gamma ,\chi ,p,\mathtt{disc}}:\mathbb{R}^{n}\rightarrow \left[ 0,+\infty
\right) , \\
U_{\Gamma ,\chi ,p,\mathtt{disc}}\left( \xi \right) =\left\{ 
\begin{array}{ccc}
\sup_{\gamma \in \Gamma }\left\vert \widehat{u\tau _{\gamma }\chi }\left(
\xi \right) \right\vert & \text{\textit{if}} & p=\infty \\ 
\left( \sum_{\gamma \in \Gamma }\left\vert \widehat{u\tau _{\gamma }\chi }%
\left( \xi \right) \right\vert ^{p}\right) ^{1/p} & \text{\textit{if}} & 
1\leq p<\infty%
\end{array}%
\right. , \\
\widehat{u\tau _{\gamma }\chi }\left( \xi \right) =\left\langle u,\mathtt{e}%
^{-\mathtt{i}\left\langle \cdot ,\xi \right\rangle }\chi \left( \cdot
-\gamma \right) \right\rangle .
\end{gather*}

Suppose that there is $\chi \in \mathcal{C}_{0}^{\infty }\left( \mathbb{R}%
^{n}\right) $ $($or $\chi \in \mathcal{S}\left( \mathbb{R}^{n}\right) )$
such that 
\begin{equation*}
\Phi =\Phi _{\Gamma ,\chi }=\sum_{\gamma \in \Gamma }\left\vert \tau
_{\gamma }\chi \right\vert >0\Leftrightarrow \Psi =\Psi _{\Gamma ,\chi
}=\sum_{\gamma \in \Gamma }\left\vert \tau _{\gamma }\chi \right\vert ^{2}>0
\end{equation*}%
and $U_{\Gamma ,\chi ,p,\mathtt{disc}}\in L^{1}\left( \mathbb{R}^{n}\right) $%
. Then $\Psi ,\frac{1}{\Psi }\in \mathcal{BC}^{\infty }\left( \mathbb{R}%
^{n}\right) $ and both are $\Gamma $-periodic. Using corollary \ref{st3},
actually the representation in $\mathcal{C}_{0}^{\infty }\left( \mathbb{R}%
^{n}\right) $ $($respectively in $\mathcal{S}\left( \mathbb{R}^{n}\right) )$ 
\begin{equation*}
\varphi =\left( \varphi \frac{1}{\Psi }\right) \Psi =\sum_{\gamma \in \Gamma
}\frac{1}{\Psi }\varphi \left( \tau _{\gamma }\chi \right) \left( \tau
_{\gamma }\overline{\chi }\right)
\end{equation*}%
we obtain that%
\begin{eqnarray*}
\left\langle u,\varphi \right\rangle &=&\sum_{\gamma \in \Gamma
}\left\langle u\tau _{\gamma }\chi ,\frac{1}{\Psi }\varphi \tau _{\gamma }%
\overline{\chi }\right\rangle \\
&=&\sum_{\gamma \in \Gamma }\left\langle \widehat{u\tau _{\gamma }\chi },%
\mathcal{F}^{-1}\left( \frac{1}{\Psi }\varphi \tau _{\gamma }\overline{\chi }%
\right) \right\rangle .
\end{eqnarray*}

Let $\widetilde{\chi }\in \mathcal{C}_{0}^{\infty }\left( \mathbb{R}%
^{n}\right) $ $($or $\widetilde{\chi }\in \mathcal{S}\left( \mathbb{R}%
^{n}\right) )$ and let $\widetilde{\Gamma }\subset \mathbb{R}^{n}$ be
another lattice. If put $\varphi =\mathtt{e}^{-\mathtt{i}\left\langle \cdot
,\xi \right\rangle }\tau _{\widetilde{\gamma }}\widetilde{\chi }$, $%
\widetilde{\gamma }\in \widetilde{\Gamma }$, in the above representation,
then we obtain%
\begin{eqnarray*}
\widehat{u\tau _{\widetilde{\gamma }}\widetilde{\chi }}\left( \xi \right)
&=&\left\langle u,\mathtt{e}^{-\mathtt{i}\left\langle \cdot ,\xi
\right\rangle }\tau _{\widetilde{\gamma }}\widetilde{\chi }\right\rangle \\
&=&\sum_{\gamma \in \Gamma }\left\langle \widehat{u\tau _{\gamma }\chi },%
\mathcal{F}^{-1}\left( \mathtt{e}^{-\mathtt{i}\left\langle \cdot ,\xi
\right\rangle }\frac{1}{\Psi }\left( \tau _{\widetilde{\gamma }}\widetilde{%
\chi }\right) \left( \tau _{\gamma }\overline{\chi }\right) \right)
\right\rangle .
\end{eqnarray*}%
Now we continue as in the continuous case. We evaluate 
\begin{equation*}
\mathcal{F}^{-1}\left( \mathtt{e}^{-\mathtt{i}\left\langle \cdot ,\xi
\right\rangle }\frac{1}{\Psi }\left( \tau _{\widetilde{\gamma }}\widetilde{%
\chi }\right) \left( \tau _{\gamma }\overline{\chi }\right) \right) \left(
\eta \right) =\left( 2\pi \right) ^{-n}\dint \mathtt{e}^{\mathtt{i}%
\left\langle x,\eta -\xi \right\rangle }\frac{1}{\Psi }\left( \tau _{%
\widetilde{\gamma }}\widetilde{\chi }\right) \left( \tau _{\gamma }\overline{%
\chi }\right) \left( x\right) \mathtt{d}x.
\end{equation*}%
Using the identity 
\begin{equation*}
\left\langle \eta -\xi \right\rangle ^{2k}\mathtt{e}^{\mathtt{i}\left\langle
x,\eta -\xi \right\rangle }=\left( \mathtt{1}\mathbf{-\triangle }\right)
^{k}\left( \mathtt{e}^{\mathtt{i}\left\langle \cdot ,\eta -\xi \right\rangle
}\right) \left( x\right)
\end{equation*}%
and integrating by parts we obtain the representation 
\begin{multline*}
\left\langle \eta -\xi \right\rangle ^{2k}\mathcal{F}^{-1}\left( \mathtt{e}%
^{-\mathtt{i}\left\langle \cdot ,\xi \right\rangle }\frac{1}{\Psi }\left(
\tau _{\widetilde{\gamma }}\widetilde{\chi }\right) \left( \tau _{\gamma }%
\overline{\chi }\right) \right) \left( \eta \right) \\
=\left( 2\pi \right) ^{-n}\dint \mathtt{e}^{\mathtt{i}\left\langle x,\eta
-\xi \right\rangle }\left( \mathtt{1}\mathbf{-\triangle }\right) ^{k}\left( 
\frac{1}{\Psi }\left( \tau _{\widetilde{\gamma }}\widetilde{\chi }\right)
\left( \tau _{\gamma }\overline{\chi }\right) \right) \left( x\right) 
\mathtt{d}x.
\end{multline*}%
There is a continuous seminorm $p_{k}$ on $\mathcal{S}\left( \mathbb{R}%
^{n}\right) $ and $C=C_{\Psi ,n,k,\Gamma ,\widetilde{\Gamma }}>0$ so that 
\begin{multline*}
\left( 2\pi \right) ^{-n}\left\vert \left( \mathtt{1}\mathbf{-\triangle }%
\right) ^{k}\left( \frac{1}{\Psi }\left( \tau _{\widetilde{\gamma }}%
\widetilde{\chi }\right) \left( \tau _{\gamma }\overline{\chi }\right)
\right) \left( x\right) \right\vert \leq Cp_{k}\left( \widetilde{\chi }%
\right) p_{k}\left( \chi \right) \left\langle x-\widetilde{\gamma }%
\right\rangle ^{-4k}\left\langle x-\gamma \right\rangle ^{-4k} \\
\leq 2^{2k}Cp_{k}\left( \widetilde{\chi }\right) p_{k}\left( \chi \right)
\left\langle 2x-\widetilde{\gamma }-\gamma \right\rangle ^{-2k}\left\langle
\gamma -\widetilde{\gamma }\right\rangle ^{-2k}.
\end{multline*}%
Here we used again the inequality%
\begin{equation*}
\left\langle X\right\rangle ^{-2N}\left\langle Y\right\rangle ^{-2N}\leq
2^{N}\left\langle X+Y\right\rangle ^{-N}\left\langle X-Y\right\rangle
^{-N},\quad X,Y\in \mathbb{R}^{m}
\end{equation*}%
which is a consequence of Peetre's inequality.

Next we choose $k=\left[ \frac{n}{2}\right] +1$ such that $n+1\leq 2k\leq
n+2 $. Returning to the integral we obtain 
\begin{multline*}
\left\vert \mathcal{F}^{-1}\left( \mathtt{e}^{-\mathtt{i}\left\langle \cdot
,\xi \right\rangle }\frac{1}{\Psi }\left( \tau _{\widetilde{\gamma }}%
\widetilde{\chi }\right) \left( \tau _{\gamma }\overline{\chi }\right)
\right) \left( \eta \right) \right\vert \\
\leq 2^{2k}Cp_{k}\left( \widetilde{\chi }\right) p_{k}\left( \chi \right)
\left\langle \eta -\xi \right\rangle ^{-2k}\left\langle \gamma -\widetilde{%
\gamma }\right\rangle ^{-2k}\dint \left\langle 2x-z-y\right\rangle ^{-2k}%
\mathtt{d}x \\
\leq 2^{2k-n}Cp_{k}\left( \widetilde{\chi }\right) p_{k}\left( \chi \right)
\left\Vert \left\langle \cdot \right\rangle ^{-2k}\right\Vert
_{L^{1}}\left\langle \eta -\xi \right\rangle ^{-2k}\left\langle \gamma -%
\widetilde{\gamma }\right\rangle ^{-2k} \\
\leq C_{\chi ,\widetilde{\chi },n}\left\langle \eta -\xi \right\rangle
^{-2k}\left\langle y-z\right\rangle ^{-2k}.
\end{multline*}%
where $C_{\chi ,\widetilde{\chi },n}=2^{2}Cp_{k}\left( \widetilde{\chi }%
\right) p_{k}\left( \chi \right) \left\Vert \left\langle \cdot \right\rangle
^{-2k}\right\Vert _{L^{1}}$. Hence%
\begin{equation*}
\left\vert \mathcal{F}^{-1}\left( \mathtt{e}^{-\mathtt{i}\left\langle \cdot
,\xi \right\rangle }\frac{1}{\Psi }\left( \tau _{\widetilde{\gamma }}%
\widetilde{\chi }\right) \left( \tau _{\gamma }\overline{\chi }\right)
\right) \left( \eta \right) \right\vert \leq C_{\chi ,\widetilde{\chi }%
,n}\left\langle \eta -\xi \right\rangle ^{-2k}\left\langle \gamma -%
\widetilde{\gamma }\right\rangle ^{-2k}.
\end{equation*}%
Further, by using Holder's inequality and lemma \ref{st2} we get 
\begin{multline*}
\left\vert \widehat{u\tau _{\widetilde{\gamma }}\widetilde{\chi }}\left( \xi
\right) \right\vert \leq \sum_{\gamma \in \Gamma }\int \left\vert \widehat{%
u\tau _{\gamma }\chi }\left( \eta \right) \right\vert \left\vert \mathcal{F}%
^{-1}\left( \mathtt{e}^{-\mathtt{i}\left\langle \cdot ,\xi \right\rangle }%
\frac{1}{\Psi }\left( \tau _{\widetilde{\gamma }}\widetilde{\chi }\right)
\left( \tau _{\gamma }\overline{\chi }\right) \right) \left( \eta \right)
\right\vert \mathtt{d}\eta \\
\leq C_{\chi ,\widetilde{\chi },n}\int \sum_{\gamma \in \Gamma }\left\vert 
\widehat{u\tau _{\gamma }\chi }\left( \eta \right) \right\vert \left\langle
\gamma -\widetilde{\gamma }\right\rangle ^{-2k}\left\langle \eta -\xi
\right\rangle ^{-2k}\mathtt{d}\eta \\
\leq C_{\chi ,\widetilde{\chi },n}\int \left( \sum_{\gamma \in \Gamma
}\left\vert \widehat{u\tau _{\gamma }\chi }\left( \eta \right) \right\vert
^{p}\left\langle \gamma -\widetilde{\gamma }\right\rangle ^{-2k}\right)
^{1/p}\left( \sum_{\gamma \in \Gamma }\left\langle \gamma -\widetilde{\gamma 
}\right\rangle ^{-2k}\right) ^{1/q}\left\langle \eta -\xi \right\rangle
^{-2k}\mathtt{d}\eta \\
\leq C_{\chi ,\widetilde{\chi },n}^{\prime }\left\Vert \left\langle \cdot
\right\rangle ^{-2k}\right\Vert _{L^{1}}^{1/q}\int \left( \sum_{\gamma \in
\Gamma }\left\vert \widehat{u\tau _{\gamma }\chi }\left( \eta \right)
\right\vert ^{p}\left\langle \gamma -\widetilde{\gamma }\right\rangle
^{-2k}\right) ^{1/p}\left\langle \eta -\xi \right\rangle ^{-2k}\mathtt{d}\eta
\\
=C_{\chi ,\widetilde{\chi },n}^{\prime }\left\Vert \left\langle \cdot
\right\rangle ^{-2k}\right\Vert _{L^{1}}^{1/q}\int f\left( \widetilde{\gamma 
},\eta \right) \left\langle \eta -\xi \right\rangle ^{-2k}\mathtt{d}\eta
\end{multline*}%
i.e.%
\begin{equation*}
\left\vert \widehat{u\tau _{\widetilde{\gamma }}\widetilde{\chi }}\left( \xi
\right) \right\vert \leq C_{\chi ,\widetilde{\chi },n}^{\prime }\left\Vert
\left\langle \cdot \right\rangle ^{-2k}\right\Vert _{L^{1}}^{1/q}\int
f\left( \widetilde{\gamma },\eta \right) \left\langle \eta -\xi
\right\rangle ^{-2k}\mathtt{d}\eta
\end{equation*}%
where%
\begin{equation*}
f\left( \widetilde{\gamma },\eta \right) ^{p}=\sum_{\gamma \in \Gamma
}\left\vert \widehat{u\tau _{\gamma }\chi }\left( \eta \right) \right\vert
^{p}\left\langle \gamma -\widetilde{\gamma }\right\rangle ^{-2k},\quad
C_{\chi ,\widetilde{\chi },n}^{\prime }=C_{\chi ,\widetilde{\chi },n}\frac{%
2^{2k}\left( \sup_{z\in \overline{\mathtt{C}}}\left\langle z\right\rangle
^{2k}\right) ^{2}}{\mathtt{vol}\left( \mathtt{C}\right) }.
\end{equation*}%
Next we apply the integral version of Minkowski's inequality. We obtain that%
\begin{multline*}
U_{\widetilde{\Gamma },\widetilde{\chi },p,\mathtt{disc}}\left( \xi \right)
=\left( \sum_{\widetilde{\gamma }\in \widetilde{\Gamma }}\left\vert \widehat{%
u\tau _{\widetilde{\gamma }}\widetilde{\chi }}\left( \xi \right) \right\vert
^{p}\right) ^{1/p} \\
\leq C_{\chi ,\widetilde{\chi },n}^{\prime }\left\Vert \left\langle \cdot
\right\rangle ^{-2k}\right\Vert _{L^{1}}^{1/q}\int \left( \sum_{\widetilde{%
\gamma }\in \widetilde{\Gamma }}f\left( \widetilde{\gamma },\eta \right)
^{p}\right) ^{1/p}\left\langle \eta -\xi \right\rangle ^{-2k}\mathtt{d}\eta
\\
\leq C_{\chi ,\widetilde{\chi },n}^{\prime }\left\Vert \left\langle \cdot
\right\rangle ^{-2k}\right\Vert _{L^{1}}^{1/q}\int \left( \sum_{\widetilde{%
\gamma }\in \widetilde{\Gamma }}\sum_{\gamma \in \Gamma }\left\vert \widehat{%
u\tau _{\gamma }\chi }\left( \eta \right) \right\vert ^{p}\left\langle
\gamma -\widetilde{\gamma }\right\rangle ^{-2k}\right) ^{1/p}\left\langle
\eta -\xi \right\rangle ^{-2k}\mathtt{d}\eta \\
\leq C_{\chi ,\widetilde{\chi },n}^{\prime \prime }\left\Vert \left\langle
\cdot \right\rangle ^{-2k}\right\Vert _{L^{1}}^{1/q}\left\Vert \left\langle
\cdot \right\rangle ^{-2k}\right\Vert _{L^{1}}^{1/p}\int \left( \sum_{\gamma
\in \Gamma }\left\vert \widehat{u\tau _{\gamma }\chi }\left( \eta \right)
\right\vert ^{p}\right) ^{1/p}\left\langle \eta -\xi \right\rangle ^{-2k}%
\mathtt{d}\eta \\
\leq C_{\chi ,\widetilde{\chi },n}^{\prime \prime }\left\Vert \left\langle
\cdot \right\rangle ^{-2k}\right\Vert _{L^{1}}\int U_{\Gamma ,\chi ,p,%
\mathtt{disc}}\left( \eta \right) \left\langle \eta -\xi \right\rangle ^{-2k}%
\mathtt{d}\eta
\end{multline*}%
where%
\begin{equation*}
C_{\chi ,\widetilde{\chi },n}^{\prime \prime }=C_{\chi ,\widetilde{\chi }%
,n}^{\prime }\frac{2^{2k}\left( \sup_{z\in \overline{\widetilde{\mathtt{C}}}%
}\left\langle z\right\rangle ^{2k}\right) ^{2}}{\mathtt{vol}\left( 
\widetilde{\mathtt{C}}\right) }.
\end{equation*}%
Hence%
\begin{equation*}
U_{\widetilde{\Gamma },\widetilde{\chi },p,\mathtt{disc}}\leq C_{\chi ,%
\widetilde{\chi },n}^{\prime \prime }\left\Vert \left\langle \cdot
\right\rangle ^{-2k}\right\Vert _{L^{1}}U_{\Gamma ,\chi ,p,\mathtt{disc}%
}\ast \left\langle \cdot \right\rangle ^{-2k}
\end{equation*}%
with $U_{\Gamma ,\chi ,p,\mathtt{disc}}\ast \left\langle \cdot \right\rangle
^{-2k}\in L^{1}\left( \mathbb{R}^{n}\right) $ since $k=\left[ \frac{n}{2}%
\right] +1$ implies that $\left\langle \cdot \right\rangle ^{-2k}\in
L^{1}\left( \mathbb{R}^{n}\right) $.

\begin{definition}
Let $1\leq p\leq \infty $. We say that a distribution $u\in \mathcal{D}%
^{\prime }\left( \mathbb{R}^{n}\right) $ belongs to $S_{w,\mathtt{disc}%
}^{p}\left( \mathbb{R}^{n}\right) $ if there are a lattice $\Gamma \subset 
\mathbb{R}^{n}$ and a function $\chi \in \mathcal{C}_{0}^{\infty }\left( 
\mathbb{R}^{n}\right) $ with the property that 
\begin{equation*}
\Phi =\Phi _{\Gamma ,\chi }=\sum_{\gamma \in \Gamma }\left\vert \tau
_{\gamma }\chi \right\vert >0\Leftrightarrow \Psi =\Psi _{\Gamma ,\chi
}=\sum_{\gamma \in \Gamma }\left\vert \tau _{\gamma }\chi \right\vert ^{2}>0
\end{equation*}%
such that the measurable function%
\begin{gather*}
U_{\Gamma ,\chi ,p,\mathtt{disc}}:\mathbb{R}^{n}\rightarrow \left[ 0,+\infty
\right) , \\
U_{\Gamma ,\chi ,p,\mathtt{disc}}\left( \xi \right) =\left\{ 
\begin{array}{ccc}
\sup_{\gamma \in \Gamma }\left\vert \widehat{u\tau _{\gamma }\chi }\left(
\xi \right) \right\vert & \text{\textit{dac\u{a}}} & p=\infty \\ 
\left( \sum_{\gamma \in \Gamma }\left\vert \widehat{u\tau _{\gamma }\chi }%
\left( \xi \right) \right\vert ^{p}\right) ^{1/p} & \text{\textit{dac\u{a}}}
& 1\leq p<\infty%
\end{array}%
\right. , \\
\widehat{u\tau _{\gamma }\chi }\left( \xi \right) =\left\langle u,\mathtt{e}%
^{-\mathtt{i}\left\langle \cdot ,\xi \right\rangle }\chi \left( \cdot
-\gamma \right) \right\rangle .
\end{gather*}%
belongs to $L^{1}\left( \mathbb{R}^{n}\right) $.
\end{definition}

We showed that the definition is correct, i.e. it does not depend on the
choice of the lattice $\Gamma $ and the function $\chi \in \mathcal{C}%
_{0}^{\infty }\left( \mathbb{R}^{n}\right) $ with the property that%
\begin{equation*}
\Phi =\Phi _{\Gamma ,\chi }=\sum_{\gamma \in \Gamma }\left\vert \tau
_{\gamma }\chi \right\vert >0\Leftrightarrow \Psi =\Psi _{\Gamma ,\chi
}=\sum_{\gamma \in \Gamma }\left\vert \tau _{\gamma }\chi \right\vert ^{2}>0.
\end{equation*}%
More explicitly, we have the following result:

\begin{proposition}
$(\mathtt{a})$ Let $u\in S_{w,\mathtt{disc}}^{p}\left( \mathbb{R}^{n}\right) 
$. Let $\Gamma \subset \mathbb{R}^{n}$ be a lattice and $\chi \in \mathcal{C}%
_{0}^{\infty }\left( \mathbb{R}^{n}\right) $. Then the measurable function%
\begin{gather*}
U_{\Gamma ,\chi ,p,\mathtt{disc}}:\mathbb{R}^{n}\rightarrow \left[ 0,+\infty
\right) , \\
U_{\Gamma ,\chi ,p,\mathtt{disc}}\left( \xi \right) =\left\{ 
\begin{array}{ccc}
\sup_{\gamma \in \Gamma }\left\vert \widehat{u\tau _{\gamma }\chi }\left(
\xi \right) \right\vert & \text{\textit{dac\u{a}}} & p=\infty \\ 
\left( \sum_{\gamma \in \Gamma }\left\vert \widehat{u\tau _{\gamma }\chi }%
\left( \xi \right) \right\vert ^{p}\right) ^{1/p} & \text{\textit{dac\u{a}}}
& 1\leq p<\infty%
\end{array}%
\right. , \\
\widehat{u\tau _{\gamma }\chi }\left( \xi \right) =\left\langle u,\mathtt{e}%
^{-\mathtt{i}\left\langle \cdot ,\xi \right\rangle }\chi \left( \cdot
-\gamma \right) \right\rangle .
\end{gather*}%
belongs to $L^{1}\left( \mathbb{R}^{n}\right) $.

$(\mathtt{b})$ If we fix $\Gamma $ and $\chi \in \mathcal{C}_{0}^{\infty
}\left( \mathbb{R}^{n}\right) $ with the property that 
\begin{equation*}
\Phi =\Phi _{\Gamma ,\chi }=\sum_{\gamma \in \Gamma }\left\vert \tau
_{\gamma }\chi \right\vert >0\Leftrightarrow \Psi =\Psi _{\Gamma ,\chi
}=\sum_{\gamma \in \Gamma }\left\vert \tau _{\gamma }\chi \right\vert ^{2}>0
\end{equation*}%
and if we define%
\begin{equation*}
\left\Vert u\right\Vert _{S_{w}^{p},\Gamma ,\chi }=\int U_{\Gamma ,\chi ,p,%
\mathtt{disc}}\left( \xi \right) \mathtt{d}\xi =\left\Vert U_{\Gamma ,\chi
,p,\mathtt{disc}}\right\Vert _{L^{1}},\quad u\in S_{w,\mathtt{disc}%
}^{p}\left( \mathbb{R}^{n}\right) ,
\end{equation*}%
then $\left\Vert \cdot \right\Vert _{S_{w}^{p},\Gamma ,\chi }$ is a norm on $%
S_{w,\mathtt{disc}}^{p}\left( \mathbb{R}^{n}\right) $ and the topology that
defines does not depend on the choice of $\Gamma $ and $\chi \in \mathcal{C}%
_{0}^{\infty }\left( \mathbb{R}^{n}\right) $.
\end{proposition}

\begin{remark}
If $\chi \in \mathcal{C}_{0}^{\infty }\left( \mathbb{R}^{n}\right)
\smallsetminus 0$ , then there is a a lattice $\Gamma \subset \mathbb{R}^{n}$
such that 
\begin{equation*}
\Phi =\Phi _{\Gamma ,\chi }=\sum_{\gamma \in \Gamma }\left\vert \tau
_{\gamma }\chi \right\vert >0\Leftrightarrow \Psi =\Psi _{\Gamma ,\chi
}=\sum_{\gamma \in \Gamma }\left\vert \tau _{\gamma }\chi \right\vert ^{2}>0.
\end{equation*}
\end{remark}

Let $\delta >0,x\in 
\mathbb{R}
$. We have%
\begin{equation*}
x+\left[ -\delta ,\delta \right] =\left[ x-\delta ,x+\delta \right] \subset
\left( \left[ x\right] -\left[ \delta \right] -1,\left[ x\right] +\left[
\delta \right] +2\right)
\end{equation*}%
which implies%
\begin{eqnarray*}
\mathtt{card}\left( 
\mathbb{Z}
\cap \left( x+\left[ -\delta ,\delta \right] \right) \right) &=&\mathtt{card}%
\left( 
\mathbb{Z}
\cap \left[ x-\delta ,x+\delta \right] \right) \\
&\leq &\mathtt{card}\left( \left( \left[ x\right] -\left[ \delta \right] -1,%
\left[ x\right] +\left[ \delta \right] +2\right) \cap 
\mathbb{Z}
\right) \\
&=&2\left( \left[ \delta \right] +1\right)
\end{eqnarray*}%
So for $x\in 
\mathbb{R}
^{n},\delta >0$%
\begin{equation*}
\mathtt{card}\left( 
\mathbb{Z}
^{n}\cap \left( x+\left[ -\delta ,\delta \right] ^{n}\right) \right) \leq
\left( 2\left( \left[ \delta \right] +1\right) \right) ^{n}=N\left( \delta
\right) .
\end{equation*}

Let $\chi \in \mathcal{C}_{0}^{\infty }\left( 
\mathbb{R}
^{n}\right) $ which verifies 
\begin{equation*}
\sum_{j\in 
\mathbb{Z}
^{n}}\chi _{j}=1,\quad \chi _{j}=\tau _{j}\chi =\chi \left( \cdot -j\right)
,\quad j\in 
\mathbb{Z}
^{n}.
\end{equation*}%
Let $\delta >0$ be such that $\mathtt{supp}\chi \subset \left[ -\delta
,\delta \right] ^{n}$. For $y\in \mathbb{R}^{n}$ we put 
\begin{equation*}
\mathtt{J}_{y}=\left\{ j\in 
\mathbb{Z}
^{n}:\left( \tau _{y}\chi \right) \left( \tau _{j}\chi \right) \neq 0\right\}
\end{equation*}%
Then $\mathtt{J}_{y}\subset 
\mathbb{Z}
^{n}\cap \left( y+\left[ -2\delta ,2\delta \right] ^{n}\right) $. Indeed, if 
$j\in \mathtt{J}_{y}$, then there is $x\in \mathbb{R}^{n}$ so that $\chi
\left( x-y\right) \chi \left( x-j\right) \neq 0$. It follows that 
\begin{equation*}
x-y,x-j\in \mathtt{supp}\chi \Rightarrow j-y\in \mathtt{supp}\chi -\mathtt{%
supp}\chi \subset \left[ -2\delta ,2\delta \right] ^{n}
\end{equation*}%
so 
\begin{equation*}
j\in 
\mathbb{Z}
^{n}\cap \left( y+\left[ -2\delta ,2\delta \right] ^{n}\right) .
\end{equation*}%
We get the estimate%
\begin{equation*}
\mathtt{card}\left( \mathtt{J}_{y}\right) \leq \left( 2\left( \left[ 2\delta %
\right] +1\right) \right) ^{n}=N\left( 2\delta \right)
\end{equation*}%
with $N\left( 2\delta \right) $ independent of $y$. Let us note that if $%
j\in \mathtt{J}_{y}$, then $y-j\in \mathtt{supp}\chi -\mathtt{supp}\chi =K$
which is a compact subset of $\mathbb{R}^{n}$.

Let $u\in \mathcal{D}^{\prime }\left( \mathbb{R}^{n}\right) $. Since 
\begin{equation*}
\tau _{y}\chi =\sum_{j\in 
\mathbb{Z}
^{n}}\left( \tau _{j}\chi \right) \left( \tau _{y}\chi \right) =\sum_{j\in 
\mathtt{J}_{y}}\left( \tau _{j}\chi \right) \left( \tau _{y}\chi \right)
\end{equation*}%
we obtain%
\begin{eqnarray*}
\widehat{u\tau _{y}\chi } &=&\sum_{j\in \mathtt{J}_{y}}\widehat{\left( u\tau
_{j}\chi \right) \left( \tau _{y}\chi \right) }=\left( 2\pi \right)
^{-n}\sum_{j\in \mathtt{J}_{y}}\widehat{u\tau _{j}\chi }\ast \widehat{\tau
_{y}\chi } \\
&=&\left( 2\pi \right) ^{-n}\sum_{j\in \mathtt{J}_{y}}\widehat{u\tau
_{j}\chi }\ast \left( \mathtt{e}^{-\mathtt{i}\left\langle y,\cdot
\right\rangle }\widehat{\chi }\right)
\end{eqnarray*}%
which implies the estimate%
\begin{equation*}
\left\vert \widehat{u\tau _{y}\chi }\right\vert \leq \left( 2\pi \right)
^{-n}\sum_{j\in \mathtt{J}_{y}}\chi _{K}\left( y-j\right) \left\vert 
\widehat{u\tau _{j}\chi }\right\vert \ast \left\vert \widehat{\chi }%
\right\vert .
\end{equation*}%
Using Holder's inequality and the estimate $\mathtt{card}\left( \mathtt{J}%
_{y}\right) \leq \left( 2\left( \left[ 2\delta \right] +1\right) \right)
^{n}=N\left( 2\delta \right) $ we continue as follows%
\begin{multline*}
\left\vert \widehat{u\tau _{y}\chi }\left( \xi \right) \right\vert \leq
\left( 2\pi \right) ^{-n}\int \left( \sum_{j\in \mathtt{J}_{y}}\chi
_{K}\left( y-j\right) \left\vert \widehat{u\tau _{j}\chi }\left( \eta
\right) \right\vert \right) \left\vert \widehat{\chi }\left( \xi -\eta
\right) \right\vert \mathtt{d}\eta \\
\leq \left( 2\pi \right) ^{-n}\int \left( \int \sum_{j\in \mathtt{J}%
_{y}}\chi _{K}\left( y-j\right) \left\vert \widehat{u\tau _{j}\chi }\left(
\eta \right) \right\vert ^{p}\right) ^{1/p}\left( \sum_{j\in \mathtt{J}%
_{y}}1\right) ^{1/q}\left\vert \widehat{\chi }\left( \xi -\eta \right)
\right\vert \mathtt{d}\eta \\
\leq \left( 2\pi \right) ^{-n}N\left( 2\delta \right) ^{1/q}\int \left(
\sum_{j\in 
\mathbb{Z}
^{n}}\chi _{K}\left( y-j\right) \left\vert \widehat{u\tau _{j}\chi }\left(
\eta \right) \right\vert ^{p}\right) ^{1/p}\left\vert \widehat{\chi }\left(
\xi -\eta \right) \right\vert \mathtt{d}\eta
\end{multline*}%
i.e 
\begin{equation*}
\left\vert \widehat{u\tau _{y}\chi }\left( \xi \right) \right\vert \leq
\left( 2\pi \right) ^{-n}N\left( 2\delta \right) ^{1/q}\int f\left( y,\eta
\right) \left\vert \widehat{\chi }\left( \xi -\eta \right) \right\vert 
\mathtt{d}\eta
\end{equation*}%
where%
\begin{eqnarray*}
f\left( y,\eta \right) ^{p} &=&\sum_{j\in 
\mathbb{Z}
^{n}}\chi _{K}\left( y-j\right) \left\vert \widehat{u\tau _{j}\chi }\left(
\eta \right) \right\vert ^{p}, \\
\left\Vert f\left( \cdot ,\eta \right) \right\Vert _{L^{p}\left( \mathtt{d}%
y\right) } &=&\left\vert K\right\vert ^{1/p}\int \left( \sum_{j\in 
\mathbb{Z}
^{n}}\left\vert \widehat{u\tau _{j}\chi }\left( \eta \right) \right\vert
^{p}\right) ^{1/p} \\
&=&\left\vert K\right\vert ^{1/p}U_{%
\mathbb{Z}
^{n},\chi ,p,\mathtt{disc}}\left( \eta \right) .
\end{eqnarray*}%
Once more we apply the integral version of Minkowski's inequality. We obtain
that%
\begin{multline*}
U_{\chi ,p}\left( \xi \right) =\left( \int \left\vert \widehat{u\tau
_{y}\chi }\left( \xi \right) \right\vert ^{p}\mathtt{d}y\right) ^{1/p} \\
\leq \left( 2\pi \right) ^{-n}N\left( 2\delta \right) ^{1/q}\int \left\Vert
f\left( \cdot ,\eta \right) \right\Vert _{L^{p}\left( \mathtt{d}y\right)
}\left\vert \widehat{\chi }\left( \xi -\eta \right) \right\vert \mathtt{d}%
\eta \\
=\left( 2\pi \right) ^{-n}N\left( 2\delta \right) ^{1/q}\left\vert
K\right\vert ^{1/p}\int U_{%
\mathbb{Z}
^{n},\chi ,p,\mathtt{disc}}\left( \eta \right) \left\vert \widehat{\chi }%
\left( \xi -\eta \right) \right\vert \mathtt{d}\eta \\
=\left( 2\pi \right) ^{-n}N\left( 2\delta \right) ^{1/q}\left\vert
K\right\vert ^{1/p}\left( U_{%
\mathbb{Z}
^{n},\chi ,p,\mathtt{disc}}\ast \left\vert \widehat{\chi }\right\vert
\right) \left( \xi \right)
\end{multline*}%
i.e.%
\begin{equation*}
U_{\chi ,p}\leq \left( 2\pi \right) ^{-n}N\left( 2\delta \right)
^{1/q}\left\vert K\right\vert ^{1/p}U_{%
\mathbb{Z}
^{n},\chi ,p,\mathtt{disc}}\ast \left\vert \widehat{\chi }\right\vert .
\end{equation*}

Let 
\begin{eqnarray*}
\mathtt{C} &=&\left[ -\frac{1}{2},\frac{1}{2}\right] ^{n}=\overline{%
B_{\infty }\left( 0;\frac{1}{2}\right) }, \\
\mathtt{C}_{j} &=&j+\mathtt{C}=\overline{B_{\infty }\left( j;\frac{1}{2}%
\right) },\quad j\in 
\mathbb{Z}
^{n}.
\end{eqnarray*}%
For $\chi \in \mathcal{C}_{0}^{\infty }\left( 
\mathbb{R}
^{n}\right) $ let $\psi \in \mathcal{C}_{0}^{\infty }\left( 
\mathbb{R}
^{n}\right) $ be such that $\psi _{|\mathtt{supp}\chi -\mathtt{C}}=1$. Then
for any $y\in \mathtt{C}_{j}$, $\tau _{y}\psi =1$ on $\mathtt{supp}\tau
_{j}\chi $. 
\begin{equation*}
\left. 
\begin{array}{c}
x\in \mathtt{supp}\tau _{j}\chi \Rightarrow x-j\in \mathtt{supp}\chi \\ 
y\in \mathtt{C}_{j}=j+\mathtt{C}\Rightarrow y-j\in \mathtt{C}%
\end{array}%
\right\} \Rightarrow x-y\in \mathtt{supp}\chi -\mathtt{C}\Rightarrow \psi
\left( x-y\right) =1.
\end{equation*}

Let $u\in \mathcal{D}^{\prime }\left( \mathbb{R}^{n}\right) $. Since 
\begin{equation*}
u\tau _{j}\chi =\left( u\tau _{y}\psi \right) \left( \tau _{j}\chi \right)
,\quad y\in \mathtt{C}_{j}=j+\mathtt{C},
\end{equation*}%
it follows that 
\begin{equation*}
\widehat{u\tau _{j}\chi }=\left( 2\pi \right) ^{-n}\widehat{u\tau _{y}\psi }%
\ast \widehat{\tau _{j}\chi }=\left( 2\pi \right) ^{-n}\widehat{u\tau
_{y}\psi }\ast \left( \mathtt{e}^{-\mathtt{i}\left\langle j,\cdot
\right\rangle }\widehat{\chi }\right) ,\quad y\in \mathtt{C}_{j}=j+\mathtt{C}%
,
\end{equation*}%
which gives the estimate%
\begin{equation*}
\left\vert \widehat{u\tau _{j}\chi }\right\vert \leq \left( 2\pi \right)
^{-n}\left\vert \widehat{u\tau _{y}\psi }\right\vert \ast \left\vert 
\widehat{\chi }\right\vert ,\quad y\in \mathtt{C}_{j}=j+\mathtt{C}.
\end{equation*}%
By integrating with respect to $y$ on $\mathtt{C}_{j}$ and using Holder's
inequality we continue as follows%
\begin{multline*}
\left\vert \widehat{u\tau _{j}\chi }\left( \xi \right) \right\vert \leq
\left( 2\pi \right) ^{-n}\iint \chi _{\mathtt{C}}\left( y-j\right)
\left\vert \widehat{u\tau _{y}\psi }\left( \eta \right) \right\vert
\left\vert \widehat{\chi }\left( \xi -\eta \right) \right\vert \mathtt{d}y%
\mathtt{d}\eta \\
\leq \left( 2\pi \right) ^{-n}\int \left( \int \chi _{\mathtt{C}}\left(
y-j\right) \left\vert \widehat{u\tau _{y}\psi }\left( \eta \right)
\right\vert ^{p}\mathtt{d}y\right) ^{1/p}\left( \int \chi _{\mathtt{C}%
}\left( y-j\right) \mathtt{d}y\right) ^{1/q}\left\vert \widehat{\chi }\left(
\xi -\eta \right) \right\vert \mathtt{d}\eta \\
\leq \left( 2\pi \right) ^{-n}\int \left( \int \chi _{\mathtt{C}}\left(
y-j\right) \left\vert \widehat{u\tau _{y}\psi }\left( \eta \right)
\right\vert ^{p}\mathtt{d}y\right) ^{1/p}\left\vert \widehat{\chi }\left(
\xi -\eta \right) \right\vert \mathtt{d}\eta
\end{multline*}%
i.e. 
\begin{equation*}
\left\vert \widehat{u\tau _{j}\chi }\left( \xi \right) \right\vert \leq
\left( 2\pi \right) ^{-n}\int f_{j}\left( \eta \right) \left\vert \widehat{%
\chi }\left( \xi -\eta \right) \right\vert \mathtt{d}\eta
\end{equation*}%
where%
\begin{eqnarray*}
f_{j}\left( \eta \right) ^{p} &=&\int \chi _{\mathtt{C}}\left( y-j\right)
\left\vert \widehat{u\tau _{y}\psi }\left( \eta \right) \right\vert ^{p}%
\mathtt{d}y, \\
\left\Vert \left( f_{j}\left( \eta \right) \right) _{j}\right\Vert _{l^{p}}
&=&\left( \sum_{j\in 
\mathbb{Z}
^{n}}\int \chi _{\mathtt{C}}\left( y-j\right) \left\vert \widehat{u\tau
_{y}\psi }\left( \eta \right) \right\vert ^{p}\mathtt{d}y\right) ^{1/p} \\
&=&\left( \sum_{j\in 
\mathbb{Z}
^{n}}\int_{\mathtt{C}_{j}}\left\vert \widehat{u\tau _{y}\psi }\left( \eta
\right) \right\vert ^{p}\mathtt{d}y\right) ^{1/p} \\
&=&\left( \int \left\vert \widehat{u\tau _{y}\psi }\left( \eta \right)
\right\vert ^{p}\mathtt{d}y\right) ^{1/p} \\
&=&U_{\psi ,p}\left( \eta \right) .
\end{eqnarray*}%
We use the integral version of Minkowski's inequality in $l^{p}$. We obtain:%
\begin{eqnarray*}
U_{%
\mathbb{Z}
^{n},\chi ,p,\mathtt{disc}}\left( \xi \right) &=&\left\Vert \left( \widehat{%
u\tau _{j}\chi }\left( \xi \right) \right) _{j}\right\Vert _{l^{p}} \\
&\leq &\left( 2\pi \right) ^{-n}\int \left\Vert \left( f_{j}\left( \eta
\right) \right) _{j}\right\Vert _{l^{p}}\left\vert \widehat{\chi }\left( \xi
-\eta \right) \right\vert \mathtt{d}\eta \\
&=&\left( 2\pi \right) ^{-n}\int U_{\psi ,p}\left( \eta \right) \left\vert 
\widehat{\chi }\left( \xi -\eta \right) \right\vert \mathtt{d}\eta
\end{eqnarray*}%
i.e.%
\begin{equation*}
U_{%
\mathbb{Z}
^{n},\chi ,p,\mathtt{disc}}\leq \left( 2\pi \right) ^{-n}U_{\psi ,p}\ast
\left\vert \widehat{\chi }\right\vert .
\end{equation*}

We proved the following result.

\begin{lemma}
We have the estimates%
\begin{gather*}
U_{\chi ,p}\leq \left( 2\pi \right) ^{-n}N\left( 2\delta \right)
^{1/q}\left\vert K\right\vert ^{1/p}U_{%
\mathbb{Z}
^{n},\chi ,p,\mathtt{disc}}\ast \left\vert \widehat{\chi }\right\vert , \\
U_{%
\mathbb{Z}
^{n},\chi ,p,\mathtt{disc}}\leq \left( 2\pi \right) ^{-n}U_{\psi ,p}\ast
\left\vert \widehat{\chi }\right\vert .
\end{gather*}
\end{lemma}

\begin{corollary}
$S_{w}^{p}\left( \mathbb{R}^{n}\right) =S_{w,\mathtt{disc}}^{p}\left( 
\mathbb{R}^{n}\right) $.
\end{corollary}

\begin{corollary}
Let $1\leq p\leq q\leq \infty $. Then%
\begin{equation*}
S_{w}^{1}\left( \mathbb{R}^{n}\right) \subset S_{w}^{p}\left( \mathbb{R}%
^{n}\right) \subset S_{w}^{q}\left( \mathbb{R}^{n}\right) \subset
S_{w}^{\infty }\left( \mathbb{R}^{n}\right) =S_{w}\left( \mathbb{R}%
^{n}\right) \subset \mathcal{BC}\left( \mathbb{R}^{n}\right) \subset 
\mathcal{S}^{\prime }\left( \mathbb{R}^{n}\right) .
\end{equation*}
\end{corollary}

\begin{proof}
The inclusions%
\begin{equation*}
S_{w}^{1}\left( \mathbb{R}^{n}\right) \subset S_{w}^{p}\left( \mathbb{R}%
^{n}\right) \subset S_{w}^{q}\left( \mathbb{R}^{n}\right) \subset
S_{w}^{\infty }\left( \mathbb{R}^{n}\right)
\end{equation*}%
are consequences of the well known inclusions $l^{1}\subset l^{p}\subset
l^{q}\subset l^{\infty }$.

Let $u\in S_{w}\left( \mathbb{R}^{n}\right) $. Let $\chi \in \mathcal{C}%
_{0}^{\infty }\left( \mathbb{R}^{n}\right) $, $\chi \left( 0\right) =1$ be
such that the lower semicontinuous function%
\begin{gather*}
U_{\chi }:\mathbb{R}^{n}\rightarrow \left[ 0,+\infty \right) , \\
U_{\chi }\left( \xi \right) =\sup_{y\in \mathbb{R}^{n}}\left\vert \widehat{%
u\tau _{y}\chi }\left( \xi \right) \right\vert =\sup_{y\in \mathbb{R}%
^{n}}\left\vert \left\langle u,\mathtt{e}^{-\mathtt{i}\left\langle \cdot
,\xi \right\rangle }\chi \left( \cdot -y\right) \right\rangle \right\vert
\end{gather*}%
belongs to $L^{1}\left( \mathbb{R}^{n}\right) $.

For $y\in \mathbb{R}^{n}$ we have 
\begin{equation*}
\widehat{u\tau _{y}\chi }\in L^{1}\left( \mathbb{R}^{n}\right) ,\quad
\left\Vert \widehat{u\tau _{y}\chi }\right\Vert _{L^{1}}\leq \left\Vert
U_{\chi }\right\Vert _{L^{1}}.
\end{equation*}%
Since 
\begin{equation*}
u\tau _{y}\chi =\mathcal{F}^{-1}\left( \widehat{u\tau _{y}\chi }\right)
=\left( 2\pi \right) ^{-n}\dint \mathtt{e}^{\mathtt{i}\left\langle \cdot
,\xi \right\rangle }\widehat{u\tau _{y}\chi }\left( \xi \right) \mathtt{d}%
\xi ,
\end{equation*}%
it follows that $u\tau _{y}\chi \in \mathcal{BC}\left( \mathbb{R}^{n}\right) 
$ for any $y\in \mathbb{R}^{n}$ and 
\begin{equation*}
\left\Vert u\tau _{y}\chi \right\Vert _{L^{\infty }}\leq \left( 2\pi \right)
^{-n}\left\Vert \widehat{u\tau _{y}\chi }\right\Vert _{L^{1}}\leq \left(
2\pi \right) ^{-n}\left\Vert U_{\chi }\right\Vert _{L^{1}}.
\end{equation*}%
Hence $u\in \mathcal{BC}\left( \mathbb{R}^{n}\right) $ and 
\begin{equation*}
\left\vert u\left( y\right) \right\vert =\left\vert \left( u\tau _{y}\chi
\right) \left( y\right) \right\vert \leq \left\Vert u\tau _{y}\chi
\right\Vert _{L^{\infty }}\leq \left( 2\pi \right) ^{-n}\left\Vert U_{\chi
}\right\Vert _{L^{1}},\quad y\in \mathbb{R}^{n}.
\end{equation*}
\end{proof}

\begin{lemma}
\label{st4}Let $f,g\in \mathcal{S}\left( 
\mathbb{R}
^{n}\right) $ and $\nu \geq 1$. Then for any $s>0$ 
\begin{equation*}
0\leq F_{\nu }\left( \xi \right) =\int \nu ^{n}\left\vert f\left( \nu \eta
\right) \right\vert \cdot \left\vert g\left( \xi -\eta \right) \right\vert 
\mathtt{d}\eta \leq 2^{s/2}\left\Vert \left\langle \cdot \right\rangle
^{s}f\right\Vert _{L^{1}}\left\Vert \left\langle \cdot \right\rangle
^{s}g\right\Vert _{L^{\infty }}\left\langle \xi \right\rangle ^{-s}
\end{equation*}
\end{lemma}

\begin{proof}
We use Peetre's inequality and $\left\langle \eta \right\rangle \leq
\left\langle \nu \eta \right\rangle $. Then 
\begin{eqnarray*}
\left\langle \xi \right\rangle ^{s}F_{\nu }\left( \xi \right) &\leq
&2^{s/2}\int \nu ^{n}\left\vert f\left( \nu \eta \right) \right\vert
\left\langle \eta \right\rangle ^{s}\cdot \left\vert g\left( \xi -\eta
\right) \right\vert \left\langle \xi -\eta \right\rangle ^{s}\mathtt{d}\eta
\\
&\leq &2^{s/2}\int \nu ^{n}\left\vert f\left( \nu \eta \right) \right\vert
\left\langle \nu \eta \right\rangle ^{s}\cdot \left\vert g\left( \xi -\eta
\right) \right\vert \left\langle \xi -\eta \right\rangle ^{s}\mathtt{d}\eta
\\
&\leq &2^{s/2}\left\Vert \left\langle \cdot \right\rangle ^{s}g\right\Vert
_{L^{\infty }}\int \nu ^{n}\left\vert f\left( \nu \eta \right) \right\vert
\left\langle \nu \eta \right\rangle ^{s}\mathtt{d}\eta \\
&=&2^{s/2}\left\Vert \left\langle \cdot \right\rangle ^{s}g\right\Vert
_{L^{\infty }}\left\Vert \left\langle \cdot \right\rangle ^{s}f\right\Vert
_{L^{1}}.
\end{eqnarray*}
\end{proof}

\begin{definition}
Let $\left( u_{\nu }\right) _{\nu \in 
\mathbb{N}
}\subset S_{w}^{p}\left( \mathbb{R}^{n}\right) $ and $u\in \mathcal{S}%
^{\prime }\left( \mathbb{R}^{n}\right) $. We say that $u_{\nu }\rightarrow u$
narrowly, if $u_{\nu }\rightarrow u$ weakly in $\mathcal{S}^{\prime }\left( 
\mathbb{R}^{n}\right) $ and if there are $\chi \in \mathcal{S}\left( \mathbb{%
R}^{n}\right) \smallsetminus 0$ and $U\in L^{1}\left( \mathbb{R}^{n}\right) $
such that 
\begin{equation*}
\begin{array}{ccc}
\sup_{y\in \mathbb{R}^{n}}\left\vert \widehat{u_{\nu }\tau _{y}\chi }\left(
\xi \right) \right\vert \leq U\left( \xi \right) & \text{\textit{if}} & 
p=\infty , \\ 
\left( \int \left\vert \widehat{u_{\nu }\tau _{y}\chi }\left( \xi \right)
\right\vert ^{p}\mathtt{d}y\right) ^{1/p}\leq U\left( \xi \right) & \text{%
\textit{if}} & 1\leq p<\infty ,%
\end{array}%
\end{equation*}
$\nu \in 
\mathbb{N}
$.
\end{definition}

\begin{remark}
Let us note that the definition is equivalent to:

Let $\left( u_{\nu }\right) _{\nu \in 
\mathbb{N}
}\subset S_{w}^{p}\left( \mathbb{R}^{n}\right) $ and $u\in \mathcal{S}%
^{\prime }\left( \mathbb{R}^{n}\right) $. We say that $u_{\nu }\rightarrow u$
narrowly, if $u_{\nu }\rightarrow u$ weakly in $\mathcal{S}^{\prime }\left( 
\mathbb{R}^{n}\right) $ and if for any $\chi \in \mathcal{S}\left( \mathbb{R}%
^{n}\right) $ there is $U=U_{\chi }\in L^{1}\left( \mathbb{R}^{n}\right) $
such that 
\begin{equation*}
\begin{array}{ccc}
\sup_{y\in \mathbb{R}^{n}}\left\vert \widehat{u_{\nu }\tau _{y}\chi }\left(
\xi \right) \right\vert \leq U\left( \xi \right) & \text{\textit{if}} & 
p=\infty , \\ 
\left( \int \left\vert \widehat{u_{\nu }\tau _{y}\chi }\left( \xi \right)
\right\vert ^{p}\mathtt{d}y\right) ^{1/p}\leq U\left( \xi \right) & \text{%
\textit{if}} & 1\leq p<\infty ,%
\end{array}%
\end{equation*}%
$\nu \in 
\mathbb{N}
$.
\end{remark}

\begin{lemma}
$\left( \mathtt{a}\right) $ If $u_{\nu }\rightarrow u$ narrowly, then $u\in
S_{w}^{p}\left( \mathbb{R}^{n}\right) $.

$\left( \mathtt{b}\right) $ $\mathcal{C}_{0}^{\infty }\left( \mathbb{R}%
^{n}\right) $ is dense in $S_{w}^{p}\left( \mathbb{R}^{n}\right) $ for the
narrow convergence.
\end{lemma}

\begin{proof}
$\left( \mathtt{a}\right) $ Let $\chi \in \mathcal{S}\left( \mathbb{R}%
^{n}\right) \smallsetminus 0$ and $U\in L^{1}\left( \mathbb{R}^{n}\right) $
such that 
\begin{equation*}
\begin{array}{ccc}
\sup_{y\in \mathbb{R}^{n}}\left\vert \widehat{u_{\nu }\tau _{y}\chi }\left(
\xi \right) \right\vert \leq U\left( \xi \right) & \text{\textit{if}} & 
p=\infty , \\ 
\left( \int \left\vert \widehat{u_{\nu }\tau _{y}\chi }\left( \xi \right)
\right\vert ^{p}\mathtt{d}y\right) ^{1/p}\leq U\left( \xi \right) & \text{%
\textit{if}} & 1\leq p<\infty ,%
\end{array}%
\end{equation*}%
$\nu \in 
\mathbb{N}
$. Since%
\begin{equation*}
\widehat{u_{\nu }\tau _{y}\chi }\left( \xi \right) =\left\langle u_{\nu },%
\mathtt{e}^{-\mathtt{i}\left\langle \cdot ,\xi \right\rangle }\tau _{y}\chi
\right\rangle \rightarrow \left\langle u,\mathtt{e}^{-\mathtt{i}\left\langle
\cdot ,\xi \right\rangle }\tau _{y}\chi \right\rangle =\widehat{u\tau
_{y}\chi }\left( \xi \right) ,
\end{equation*}%
then using Fatou's lemma when $1\leq p<\infty $ we obtain 
\begin{equation*}
\begin{array}{ccc}
\sup_{y\in \mathbb{R}^{n}}\left\vert \widehat{u\tau _{y}\chi }\left( \xi
\right) \right\vert \leq U\left( \xi \right) & \text{\textit{if}} & p=\infty
, \\ 
\begin{array}{c}
\int \left\vert \widehat{u\tau _{y}\chi }\left( \xi \right) \right\vert ^{p}%
\mathtt{d}y=\int \lim \left\vert \widehat{u_{\nu }\tau _{y}\chi }\left( \xi
\right) \right\vert ^{p}\mathtt{d}y \\ 
\leq \lim \inf \int \left\vert \widehat{u_{\nu }\tau _{y}\chi }\left( \xi
\right) \right\vert ^{p}\mathtt{d}y\leq U\left( \xi \right) ^{p}%
\end{array}
& \text{\textit{if}} & 1\leq p<\infty .%
\end{array}%
\end{equation*}

$\left( \mathtt{b}\right) $ Let $u\in S_{w}^{p}\left( \mathbb{R}^{n}\right) $%
. Choose $\chi \in \mathcal{C}_{0}^{\infty }\left( \mathbb{R}^{n}\right) $
such that $\int \chi =1$. Let $\varphi \in \mathcal{C}_{0}^{\infty }\left( 
\mathbb{R}^{n}\right) $. Then applying the equality (\ref{sjo1}) we obtain 
\begin{equation*}
\varphi \ast u=\left\langle u,\tau _{\cdot }\check{\varphi}\right\rangle
=\dint \left\langle u,\left( \tau _{\cdot }\check{\varphi}\right) \tau
_{y}\chi \right\rangle \mathtt{d}y=\dint \left\langle u\tau _{y}\chi ,\tau
_{\cdot }\check{\varphi}\right\rangle \mathtt{d}y=\dint \varphi \ast \left(
u\tau _{y}\chi \right) \mathtt{d}y
\end{equation*}%
and 
\begin{equation*}
\left( \varphi \ast u\right) \tau _{z}\chi =\dint \left( \varphi \ast \left(
u\tau _{y}\chi \right) \right) \tau _{z}\chi \mathtt{d}y.
\end{equation*}%
The support of the integrand is contained in $\left( z+\mathtt{supp}\chi
\right) \cap \left( y+\mathtt{supp}\varphi +\mathtt{supp}\chi \right) $.
This set is a nonempty set if and only if $y-z\in \mathtt{supp}\chi -\mathtt{%
supp}\chi -\mathtt{supp}\varphi $. Hence the integral is calculated on the
compact set $z+\mathtt{supp}\chi -\mathtt{supp}\chi -\mathtt{supp}\varphi $.
It follows that 
\begin{eqnarray*}
\widehat{\left( \varphi \ast u\right) \tau _{z}\chi } &=&\dint_{\mathtt{supp}%
\chi -\mathtt{supp}\chi -\mathtt{supp}\varphi }\widehat{\left( \varphi \ast
\left( u\tau _{z+y}\chi \right) \right) \tau _{z}\chi }\mathtt{d}y \\
&=&\left( 2\pi \right) ^{-n}\dint_{\mathtt{supp}\chi -\mathtt{supp}\chi -%
\mathtt{supp}\varphi }\left( \widehat{\varphi }\cdot \widehat{u\tau
_{z+y}\chi }\right) \ast \widehat{\tau _{z}\chi }\mathtt{d}y
\end{eqnarray*}%
which implies 
\begin{multline*}
\left\vert \widehat{\left( \varphi \ast u\right) \tau _{z}\chi }\left( \xi
\right) \right\vert \\
\leq \left( 2\pi \right) ^{-n}\left\Vert \widehat{\varphi }\right\Vert
_{\infty }\iint \chi _{\mathtt{supp}\chi -\mathtt{supp}\chi -\mathtt{supp}%
\varphi }\left( y\right) \left\vert \widehat{u\tau _{z+y}\chi }\left( \xi
-\eta \right) \right\vert \left\vert \widehat{\chi }\left( \eta \right)
\right\vert \mathtt{d}\eta \mathtt{d}y \\
\leq \left( 2\pi \right) ^{-n}\left\Vert \varphi \right\Vert _{1}\iint \chi
_{\mathtt{supp}\chi -\mathtt{supp}\chi -\mathtt{supp}\varphi }\left(
y\right) \left\vert \widehat{u\tau _{z+y}\chi }\left( \xi -\eta \right)
\right\vert \left\vert \widehat{\chi }\left( \eta \right) \right\vert 
\mathtt{d}\eta \mathtt{d}y
\end{multline*}%
Next we apply the integral version of Minkowski's inequality. We obtain:%
\begin{multline*}
\left( \int \left\vert \widehat{\left( \varphi \ast u\right) \tau _{z}\chi }%
\left( \xi \right) \right\vert ^{p}\mathtt{d}z\right) ^{1/p} \\
\leq \left( 2\pi \right) ^{-n}\left\Vert \varphi \right\Vert _{1}\iint \chi
_{\mathtt{supp}\chi -\mathtt{supp}\chi -\mathtt{supp}\varphi }\left(
y\right) \left\vert \widehat{\chi }\left( \eta \right) \right\vert \left(
\int \left\vert \widehat{u\tau _{z+y}\chi }\left( \xi -\eta \right)
\right\vert ^{p}\mathtt{d}z\right) ^{1/p}\mathtt{d}\eta \mathtt{d}y \\
=\left( 2\pi \right) ^{-n}\left\Vert \varphi \right\Vert _{1}\iint \chi _{%
\mathtt{supp}\chi -\mathtt{supp}\chi -\mathtt{supp}\varphi }\left( y\right)
\left\vert \widehat{\chi }\left( \eta \right) \right\vert \left( \int
\left\vert \widehat{u\tau _{z}\chi }\left( \xi -\eta \right) \right\vert ^{p}%
\mathtt{d}z\right) ^{1/p}\mathtt{d}\eta \mathtt{d}y \\
\leq \left( 2\pi \right) ^{-n}\left\Vert \varphi \right\Vert _{1}\mathtt{vol}%
\left( \mathtt{supp}\chi -\mathtt{supp}\chi -\mathtt{supp}\varphi \right)
\int U_{\chi ,p}\left( \xi -\eta \right) \left\vert \widehat{\chi }\left(
\eta \right) \right\vert \mathtt{d}\eta \\
\leq \left( 2\pi \right) ^{-n}\left\Vert \varphi \right\Vert _{1}\mathtt{vol}%
\left( \mathtt{supp}\chi -\mathtt{supp}\chi -\mathtt{supp}\varphi \right)
U_{\chi ,p}\ast \left\vert \widehat{\chi }\right\vert \left( \xi \right)
\end{multline*}%
Let $\psi \in \mathcal{C}_{0}^{\infty }\left( \mathbb{R}^{n}\right) $. Then 
\begin{equation*}
\widehat{\psi \left( \varphi \ast u\right) \tau _{z}\chi }=\left( 2\pi
\right) ^{-n}\widehat{\psi }\ast \widehat{\left( \varphi \ast u\right) \tau
_{z}\chi }.
\end{equation*}%
and another use of Minkowski's inequality implies 
\begin{multline*}
\left( \int \left\vert \widehat{\psi \left( \varphi \ast u\right) \tau
_{z}\chi }\left( \xi \right) \right\vert ^{p}\mathtt{d}z\right) ^{1/p}\leq
\left( 2\pi \right) ^{-n}\left\vert \widehat{\psi }\right\vert \ast \left(
\int \left\vert \widehat{\left( \varphi \ast u\right) \tau _{z}\chi }\left(
\cdot \right) \right\vert ^{p}\mathtt{d}z\right) ^{1/p}\left( \xi \right) \\
\leq \left( 2\pi \right) ^{-2n}\left\Vert \varphi \right\Vert _{1}\mathtt{vol%
}\left( \mathtt{supp}\chi -\mathtt{supp}\chi -\mathtt{supp}\varphi \right)
\left\vert \widehat{\psi }\right\vert \ast U_{\chi ,p}\ast \left\vert 
\widehat{\chi }\right\vert \left( \xi \right) .
\end{multline*}%
For $\upsilon \geq 1$, $\upsilon \in 
\mathbb{N}
$ we set $u_{\nu }=\psi \left( \frac{\cdot }{\upsilon }\right) \left(
\varphi _{\frac{1}{\nu }}\ast u\right) $ where $\varphi ,\psi \in \mathcal{C}%
_{0}^{\infty }\left( \mathbb{R}^{n}\right) $ satisfy $\psi \left( 0\right)
=1 $, $\varphi \geq 0$, $\int \varphi =1$, $\varphi _{\frac{1}{\nu }%
}=\upsilon ^{n}\varphi \left( \upsilon \cdot \right) $, $\mathtt{supp}\chi
\cup \mathtt{supp}\varphi \subset B\left( 0;1\right) $. Then by using lemma %
\ref{st4} we get%
\begin{multline*}
\left( \int \left\vert \widehat{\psi \left( \frac{\cdot }{\upsilon }\right)
\left( \varphi _{\frac{1}{\nu }}\ast u\right) \tau _{z}\chi }\left( \xi
\right) \right\vert ^{p}\mathtt{d}z\right) ^{1/p} \\
\leq \left( 2\pi \right) ^{-2n}3^{n}\mathtt{vol}\left( B\left( 0;1\right)
\right) \left\vert \upsilon ^{n}\widehat{\psi }\left( \upsilon \cdot \right)
\right\vert \ast \left\vert \widehat{\chi }\right\vert \ast U_{\chi
,p}\left( \xi \right) \\
\leq \left( 2\pi \right) ^{-2n}3^{n}\mathtt{vol}\left( B\left( 0;1\right)
\right) C_{n+1,n,\widehat{\psi },\widehat{\chi }}\left\langle \cdot
\right\rangle ^{-n-1}\ast U_{\chi ,p}\left( \xi \right)
\end{multline*}%
where $C_{n+1,n,\widehat{\psi },\widehat{\chi }}=2^{\left( n+1\right)
/2}\left\Vert \left\langle \cdot \right\rangle ^{n+1}\widehat{\psi }%
\right\Vert _{L^{1}}\left\Vert \left\langle \cdot \right\rangle ^{n+1}%
\widehat{\chi }\right\Vert _{L^{\infty }}$. Since $u_{\nu }\in \mathcal{C}%
_{0}^{\infty }\left( \mathbb{R}^{n}\right) $, $u_{\nu }\rightarrow u$ weakly
in $\mathcal{S}^{\prime }\left( \mathbb{R}^{n}\right) $ and $\left\langle
\cdot \right\rangle ^{-n-1}\ast U_{\chi ,p}\in L^{1}\left( \mathbb{R}%
^{n}\right) $ we obtain that $\mathcal{C}_{0}^{\infty }\left( \mathbb{R}%
^{n}\right) $ is dense in $S_{w}^{p}\left( \mathbb{R}^{n}\right) $ for the
narrow convergence.
\end{proof}

\begin{remark}
In proposition \ref{st22} we shall prove that $\mathcal{S}\left( \mathbb{R}%
^{n}\right) $ is dense in $S_{w}^{p}\left( \mathbb{R}^{n}\right) $ when $%
1\leq p<\infty $ $($see also proposition $11.3.4$ in \cite{Grochenig}$)$.
\end{remark}

\begin{lemma}
\label{st5}The bilinear map $S_{w}^{p}\left( \mathbb{R}^{n}\right) \times
S_{w}\left( \mathbb{R}^{n}\right) \ni \left( u,v\right) \rightarrow uv\in
S_{w}^{p}\left( \mathbb{R}^{n}\right) $ is well defined and continuous. In
particular, $S_{w}^{p}\left( \mathbb{R}^{n}\right) $ is an ideal in $%
S_{w}\left( \mathbb{R}^{n}\right) $ with respect to the usual product.
\end{lemma}

\begin{proof}
We fix $\chi \in \mathcal{C}_{0}^{\infty }\left( \mathbb{R}^{n}\right)
\smallsetminus 0$\textit{.} Let $\left( u,v\right) \in S_{w}^{p}\left( 
\mathbb{R}^{n}\right) \times S_{w}\left( \mathbb{R}^{n}\right) $ and let $%
U_{\chi ,p},$ $V_{\chi }\in L^{1}\left( \mathbb{R}^{n}\right) $ be the
measurable functions defined by 
\begin{gather*}
U_{\chi ,p},U_{\chi },V_{\chi }:\mathbb{R}^{n}\rightarrow \left[ 0,+\infty
\right) , \\
U_{\chi ,p}\left( \xi \right) =\left( \int \left\vert \widehat{u\tau
_{y}\chi }\left( \xi \right) \right\vert ^{p}\mathtt{d}y\right)
^{1/p}=\left( \int \left\vert \left\vert \left\langle u,\mathtt{e}^{-\mathtt{%
i}\left\langle \cdot ,\xi \right\rangle }\chi \left( \cdot -y\right)
\right\rangle \right\vert \right\vert ^{p}\mathtt{d}y\right) ^{1/p}, \\
V_{\chi }\left( \xi \right) =\sup_{y\in \mathbb{R}^{n}}\left\vert \widehat{%
v\tau _{y}\chi }\left( \xi \right) \right\vert =\sup_{y\in \mathbb{R}%
^{n}}\left\vert \left\langle v,\mathtt{e}^{-\mathtt{i}\left\langle \cdot
,\xi \right\rangle }\chi \left( \cdot -y\right) \right\rangle \right\vert ,
\\
U_{\chi }\left( \xi \right) =\sup_{y\in \mathbb{R}^{n}}\left\vert \widehat{%
u\tau _{y}\chi }\left( \xi \right) \right\vert =\sup_{y\in \mathbb{R}%
^{n}}\left\vert \left\langle u,\mathtt{e}^{-\mathtt{i}\left\langle \cdot
,\xi \right\rangle }\chi \left( \cdot -y\right) \right\rangle \right\vert .
\end{gather*}%
Since $\chi ^{2}\in \mathcal{C}_{0}^{\infty }\left( \mathbb{R}^{n}\right)
\smallsetminus 0$ and 
\begin{equation*}
uv\left( \tau _{y}\chi ^{2}\right) =\left( u\tau _{y}\chi \right) \left(
v\tau _{y}\chi \right) \Rightarrow \widehat{uv\left( \tau _{y}\chi
^{2}\right) }=\left( 2\pi \right) ^{-n}\widehat{u\tau _{y}\chi }\ast 
\widehat{v\tau _{y}\chi },\quad y\in \mathbb{R}^{n},
\end{equation*}%
it follows that 
\begin{equation*}
\left\vert \widehat{uv\left( \tau _{y}\chi ^{2}\right) }\right\vert \leq
\left( 2\pi \right) ^{-n}\left\vert \widehat{u\tau _{y}\chi }\right\vert
\ast \left\vert \widehat{v\tau _{y}\chi }\right\vert \leq \left( 2\pi
\right) ^{-n}\left\vert \widehat{u\tau _{y}\chi }\right\vert \ast V_{\chi
},\quad y\in \mathbb{R}^{n}.
\end{equation*}%
If $1\leq p<\infty $, then the integral version of Minkowski's inequality
implies%
\begin{eqnarray*}
\left( \int \left\vert \widehat{uv\left( \tau _{y}\chi ^{2}\right) }\left(
\xi \right) \right\vert ^{p}\mathtt{d}y\right) ^{1/p} &\leq &\left( 2\pi
\right) ^{-n}\left( \int \left\vert \widehat{u\tau _{y}\chi }\left( \cdot
\right) \right\vert ^{p}\mathtt{d}y\right) ^{1/p}\ast V_{\chi }\left( \xi
\right) \\
&\leq &\left( 2\pi \right) ^{-n}U_{\chi ,p}\ast V_{\chi }\left( \xi \right)
\end{eqnarray*}%
with $U_{\chi ,p}\ast V_{\chi }\in L^{1}\left( \mathbb{R}^{n}\right) $.
Hence $uv\in S_{w}^{p}\left( \mathbb{R}^{n}\right) $ and\textit{\ }%
\begin{equation*}
\left\Vert uv\right\Vert _{S_{w}^{p},\chi ^{2}}\leq \left( 2\pi \right)
^{-n}\left\Vert u\right\Vert _{S_{w}^{p},\chi }\left\Vert v\right\Vert
_{S_{w},\chi }.
\end{equation*}

If $p=\infty $ then 
\begin{equation*}
\left\vert \widehat{uv\left( \tau _{y}\chi ^{2}\right) }\right\vert \leq
\left( 2\pi \right) ^{-n}\left\vert \widehat{u\tau _{y}\chi }\right\vert
\ast \left\vert \widehat{v\tau _{y}\chi }\right\vert \leq \left( 2\pi
\right) ^{-n}U_{\chi }\ast V_{\chi },\quad y\in \mathbb{R}^{n},
\end{equation*}%
implies 
\begin{equation*}
\left\Vert uv\right\Vert _{S_{w},\chi ^{2}}\leq \left( 2\pi \right)
^{-n}\left\Vert u\right\Vert _{S_{w},\chi }\left\Vert v\right\Vert
_{S_{w},\chi }.
\end{equation*}
\end{proof}

\begin{theorem}
$S_{w}^{p}\left( \mathbb{R}^{n}\right) $ with the usual product is a Banach
algebra.
\end{theorem}

\begin{proof}
Assume first that $1\leq p<\infty $. Let $\left\{ u_{\upsilon }\right\}
_{\upsilon \in 
\mathbb{N}
}$ be a Cauchy sequence in $S_{w}^{p}\left( \mathbb{R}^{n}\right) $. Since
we have the topological inclusion $S_{w}\left( \mathbb{R}^{n}\right) \subset 
\mathcal{BC}\left( \mathbb{R}^{n}\right) $ and $\mathcal{BC}\left( \mathbb{R}%
^{n}\right) $ is complete, it follows that there is $u\in \mathcal{BC}\left( 
\mathbb{R}^{n}\right) $ so that $u_{\upsilon }\rightarrow u$ in $\mathcal{BC}%
\left( \mathbb{R}^{n}\right) $. Recall that for $\varphi \in \mathcal{S}%
\left( \mathbb{R}^{n}\right) $ and $u\in \mathcal{S}^{\prime }\left( \mathbb{%
R}^{n}\right) $ we have $\widehat{\varphi u}\in \mathcal{S}^{\prime }\left( 
\mathbb{R}^{n}\right) \cap \mathcal{C}_{pol}^{\infty }\left( \mathbb{R}%
^{n}\right) $ and 
\begin{equation*}
\widehat{\varphi u}\left( \xi \right) =\left\langle \mathtt{e}^{-\mathtt{i}%
\left\langle \cdot ,\xi \right\rangle }u,\varphi \right\rangle =\left\langle
u,\mathtt{e}^{-\mathtt{i}\left\langle \cdot ,\xi \right\rangle }\varphi
\right\rangle ,\quad \xi \in \mathbb{R}^{n}.
\end{equation*}%
Since we have 
\begin{eqnarray*}
\sup_{y,\xi }\left\vert \widehat{\left( u-u_{\upsilon }\right) \tau _{y}\chi 
}\left( \xi \right) \right\vert &=&\sup_{y,\xi }\left\vert \left\langle
u-u_{\upsilon },\mathtt{e}^{-\mathtt{i}\left\langle \cdot ,\xi \right\rangle
}\tau _{y}\chi \right\rangle \right\vert \\
&\leq &\sup_{y,\xi }\left\Vert u-u_{\upsilon }\right\Vert _{\mathcal{BC}%
}\left\Vert \mathtt{e}^{-\mathtt{i}\left\langle \cdot ,\xi \right\rangle
}\tau _{y}\chi \right\Vert _{L^{1}} \\
&=&\left\Vert u-u_{\upsilon }\right\Vert _{\mathcal{BC}}\left\Vert \chi
\right\Vert _{L^{1}}.
\end{eqnarray*}%
it follows that 
\begin{equation*}
\lim_{\nu \rightarrow \infty }\sup_{y,\xi }\left\vert \widehat{\left(
u-u_{\upsilon }\right) \tau _{y}\chi }\left( \xi \right) \right\vert =0,
\end{equation*}

Let $K\subset \mathbb{R}^{n}$ be a compact subset. Since 
\begin{multline*}
\left( \int_{K}\left\vert \widehat{\left( u-u_{\kappa }\right) \tau _{y}\chi 
}\left( \xi \right) \right\vert ^{p}\mathtt{d}y\right) ^{1/p} \\
\leq \left( \int_{K}\left\vert \widehat{\left( u-u_{\nu }\right) \tau
_{y}\chi }\left( \xi \right) \right\vert ^{p}\mathtt{d}y\right)
^{1/p}+\left( \int_{K}\left\vert \widehat{\left( u_{\nu }-u_{\kappa }\right)
\tau _{y}\chi }\left( \xi \right) \right\vert ^{p}\mathtt{d}y\right) ^{1/p}
\end{multline*}%
we obtain that%
\begin{equation*}
\left( \int_{K}\left\vert \widehat{\left( u-u_{\kappa }\right) \tau _{y}\chi 
}\left( \xi \right) \right\vert ^{p}\mathtt{d}y\right) ^{1/p}\leq
\liminf_{\nu \rightarrow \infty }\left( \int_{K}\left\vert \widehat{\left(
u_{\nu }-u_{\kappa }\right) \tau _{y}\chi }\left( \xi \right) \right\vert
^{p}\mathtt{d}y\right) ^{1/p}.
\end{equation*}%
Next from%
\begin{multline*}
\left( \int_{K}\left\vert \widehat{\left( u_{\nu }-u_{\kappa }\right) \tau
_{y}\chi }\left( \xi \right) \right\vert ^{p}\mathtt{d}y\right) ^{1/p} \\
\leq \left( \int_{K}\left\vert \widehat{\left( u-u_{\nu }\right) \tau
_{y}\chi }\left( \xi \right) \right\vert ^{p}\mathtt{d}y\right)
^{1/p}+\left( \int_{K}\left\vert \widehat{\left( u-u_{\kappa }\right) \tau
_{y}\chi }\left( \xi \right) \right\vert ^{p}\mathtt{d}y\right) ^{1/p}
\end{multline*}%
it follows that%
\begin{equation*}
\limsup_{\nu \rightarrow \infty }\left( \int_{K}\left\vert \widehat{\left(
u_{\nu }-u_{\kappa }\right) \tau _{y}\chi }\left( \xi \right) \right\vert
^{p}\mathtt{d}y\right) ^{1/p}\leq \left( \int_{K}\left\vert \widehat{\left(
u-u_{\kappa }\right) \tau _{y}\chi }\left( \xi \right) \right\vert ^{p}%
\mathtt{d}y\right) ^{1/p}.
\end{equation*}%
Hence 
\begin{equation*}
\left( \int_{K}\left\vert \widehat{\left( u-u_{\kappa }\right) \tau _{y}\chi 
}\left( \xi \right) \right\vert ^{p}\mathtt{d}y\right) ^{1/p}=\lim_{\nu
\rightarrow \infty }\left( \int_{K}\left\vert \widehat{\left( u_{\nu
}-u_{\kappa }\right) \tau _{y}\chi }\left( \xi \right) \right\vert ^{p}%
\mathtt{d}y\right) ^{1/p}.
\end{equation*}

Let $\varepsilon >0$. Then there is $N\in 
\mathbb{N}
$ such that $\nu ,\kappa \geq N$ implies%
\begin{equation*}
\int \left( \int \left\vert \widehat{\left( u_{\nu }-u_{\kappa }\right) \tau
_{y}\chi }\left( \xi \right) \right\vert ^{p}\mathtt{d}y\right)
^{1/p}<\varepsilon .
\end{equation*}%
Then using Fatou's lemma ( $\int \liminf \leq \liminf \int $ ), we obtain
that%
\begin{multline*}
\int \left( \int_{K}\left\vert \widehat{\left( u-u_{\kappa }\right) \tau
_{y}\chi }\left( \xi \right) \right\vert ^{p}\mathtt{d}y\right) ^{1/p}%
\mathtt{d}\xi \\
\leq \liminf_{\nu \rightarrow \infty }\int \left( \int_{K}\left\vert 
\widehat{\left( u_{\nu }-u_{\kappa }\right) \tau _{y}\chi }\left( \xi
\right) \right\vert ^{p}\mathtt{d}y\right) ^{1/p} \\
\leq \liminf_{\nu \rightarrow \infty }\int \left( \int \left\vert \widehat{%
\left( u_{\nu }-u_{\kappa }\right) \tau _{y}\chi }\left( \xi \right)
\right\vert ^{p}\mathtt{d}y\right) ^{1/p}<\varepsilon .
\end{multline*}%
Hence%
\begin{equation*}
\int \left( \int_{K}\left\vert \widehat{\left( u-u_{\kappa }\right) \tau
_{y}\chi }\left( \xi \right) \right\vert ^{p}\mathtt{d}y\right) ^{1/p}%
\mathtt{d}\xi <\varepsilon
\end{equation*}%
for any compact subset $K\subset \mathbb{R}^{n}$. Using Lebesgue's monotone
convergence theorem we obtain 
\begin{equation*}
\int \left( \int \left\vert \widehat{\left( u-u_{\kappa }\right) \tau
_{y}\chi }\left( \xi \right) \right\vert ^{p}\mathtt{d}y\right) ^{1/p}%
\mathtt{d}\xi \leq \varepsilon .
\end{equation*}%
Hence $u\in S_{w}^{p}\left( \mathbb{R}^{n}\right) $, $u=u-u_{\kappa
}+u_{\kappa }$, and $u_{\kappa }\rightarrow u$ in $S_{w}^{p}\left( \mathbb{R}%
^{n}\right) $.

If $p=\infty $ we proceed similarly substituting $\int_{K}$ with $\sup_{y\in 
\mathbb{R}^{n}}$.

According to the previous lemma we have the continuous maps%
\begin{eqnarray*}
S_{w}^{p}\left( \mathbb{R}^{n}\right) \times S_{w}^{p}\left( \mathbb{R}%
^{n}\right) &\hookrightarrow &S_{w}^{p}\left( \mathbb{R}^{n}\right) \times
S_{w}\left( \mathbb{R}^{n}\right) \rightarrow S_{w}^{p}\left( \mathbb{R}%
^{n}\right) , \\
\left( u,v\right) &\hookrightarrow &\left( u,v\right) \rightarrow uv.
\end{eqnarray*}%
This implies that $S_{w}^{p}\left( \mathbb{R}^{n}\right) $ with the usual
product is a Banach algebra.
\end{proof}

\begin{lemma}
\label{st12}$(\mathtt{a})$ Let $\lambda :\mathbb{R}^{n}\rightarrow \mathbb{R}%
^{n}$ be a linear isomorphism. Then $u\circ \lambda \in S_{w}^{p}\left( 
\mathbb{R}^{n}\right) $ for any $u\in S_{w}^{p}\left( \mathbb{R}^{n}\right) $
and the map 
\begin{equation*}
S_{w}^{p}\left( \mathbb{R}^{n}\right) \ni u\rightarrow u\circ \lambda \in
S_{w}^{p}\left( \mathbb{R}^{n}\right)
\end{equation*}%
is continuous.

$(\mathtt{b})$ The map 
\begin{equation*}
S_{w}^{p}\left( \mathbb{R}^{n^{\prime }}\right) \times S_{w}^{p}\left( 
\mathbb{R}^{n^{\prime \prime }}\right) \ni \left( u^{\prime },u^{\prime
\prime }\right) \rightarrow u^{\prime }\otimes u^{\prime \prime }\in
S_{w}^{p}\left( \mathbb{R}^{n^{\prime }}\times \mathbb{R}^{n^{\prime \prime
}}\right)
\end{equation*}%
is well defined and continuous.

$(\mathtt{c})$ Let $L\subset \mathbb{R}^{n}$ be a linear subspace. Then the
map%
\begin{equation*}
S_{w}^{p}\left( \mathbb{R}^{n}\right) \ni u\rightarrow u_{|L}\in
S_{w}^{p}\left( L\right)
\end{equation*}%
is well defined and continuous.
\end{lemma}

\begin{proof}
$(\mathtt{a})$ Let $\chi \in \mathcal{S}\left( \mathbb{R}^{n}\right)
\smallsetminus 0$. Using the equalities%
\begin{eqnarray*}
\widehat{v\circ \lambda } &=&\left\vert \det \lambda \right\vert ^{-1}%
\widehat{v}\circ {}^{t}\lambda ^{-1},\quad v\in \mathcal{S}^{\prime }\left( 
\mathbb{R}^{n}\right) , \\
\left( u\circ \lambda \right) \tau _{y}\chi &=&\left( u\cdot \tau _{\lambda
y}\left( \chi \circ \lambda ^{-1}\right) \right) \circ \lambda ,
\end{eqnarray*}%
we obtain 
\begin{equation*}
\widehat{\left( u\circ \lambda \right) \tau _{y}\chi }=\left\vert \det
\lambda \right\vert ^{-1}\widehat{u\cdot \tau _{\lambda y}\left( \chi \circ
\lambda ^{-1}\right) }\circ {}^{t}\lambda ^{-1}.
\end{equation*}%
It follows that%
\begin{equation*}
\sup\limits_{y\in \mathbb{R}^{n}}\left\vert \widehat{\left( u\circ \lambda
\right) \tau _{y}\chi }\left( \xi \right) \right\vert =\left\vert \det
\lambda \right\vert ^{-1}\sup\limits_{y\in \mathbb{R}^{n}}\left\vert 
\widehat{u\cdot \tau _{y}\left( \chi \circ \lambda ^{-1}\right) }\left(
{}^{t}\lambda ^{-1}\xi \right) \right\vert
\end{equation*}%
if $p=\infty $ and 
\begin{equation*}
\left( \int \left\vert \widehat{\left( u\circ \lambda \right) \tau _{y}\chi }%
\left( \xi \right) \right\vert ^{p}\mathtt{d}y\right) ^{1/p}=\left\vert \det
\lambda \right\vert ^{-1-1/p}\left( \int \left\vert \widehat{u\tau
_{y}\left( \chi \circ \lambda ^{-1}\right) }\left( {}^{t}\lambda ^{-1}\xi
\right) \right\vert ^{p}\mathtt{d}y\right) ^{1/p}
\end{equation*}%
if $1\leq p<\infty $. Finally, for $1\leq p\leq \infty $ we have%
\begin{equation*}
\left\Vert u\circ \lambda \right\Vert _{S_{w}^{p},\chi }=\left\vert \det
\lambda \right\vert ^{-1/p}\left\Vert u\right\Vert _{S_{w}^{p},\chi \circ
\lambda ^{-1}},\quad u\in S_{w}^{p}\left( \mathbb{R}^{n}\right) .
\end{equation*}

$(\mathtt{b})$ This part is trivial. Let $\chi ^{\prime }\in \mathcal{S}%
\left( \mathbb{R}^{n^{\prime }}\right) \smallsetminus 0$ and $\chi ^{\prime
\prime }\in \mathcal{S}\left( \mathbb{R}^{n^{\prime \prime }}\right)
\smallsetminus 0$. Then $\chi =\chi ^{\prime }\otimes \chi ^{\prime \prime
}\in \mathcal{S}\left( \mathbb{R}^{n^{\prime }}\times \mathbb{R}^{n^{\prime
\prime }}\right) \smallsetminus 0$ and 
\begin{eqnarray*}
\widehat{\left( u^{\prime }\otimes u^{\prime \prime }\right) \cdot \tau
_{\left( y^{\prime },y^{\prime \prime }\right) }\chi } &=&\widehat{u^{\prime
}\tau _{y^{\prime }}\chi ^{\prime }}\otimes \widehat{u^{\prime \prime }\tau
_{y^{\prime \prime }}\chi ^{\prime \prime }}, \\
\sup_{y^{\prime },y^{\prime \prime }}\left\vert \widehat{\left( u^{\prime
}\otimes u^{\prime \prime }\right) \cdot \tau _{\left( y^{\prime },y^{\prime
\prime }\right) }\chi }\left( \xi ^{\prime },\xi ^{\prime \prime }\right)
\right\vert &=&\sup_{y^{\prime }}\left\vert \widehat{u^{\prime }\tau
_{y^{\prime }}\chi ^{\prime }}\left( \xi ^{\prime }\right) \right\vert
\sup_{y^{\prime \prime }}\left\vert \widehat{u^{\prime \prime }\tau
_{y^{\prime \prime }}\chi ^{\prime \prime }}\left( \xi ^{\prime \prime
}\right) \right\vert
\end{eqnarray*}%
\begin{multline*}
\left( \int \left\vert \widehat{\left( u^{\prime }\otimes u^{\prime \prime
}\right) \tau _{y}\chi }\left( \xi \right) \right\vert ^{p}\mathtt{d}%
y\right) ^{1/p}=\left( \int \left\vert \widehat{u^{\prime }\tau _{y^{\prime
}}\chi ^{\prime }}\left( \xi ^{\prime }\right) \right\vert ^{p}\mathtt{d}%
y^{\prime }\right) ^{1/p} \\
\cdot \left( \int \left\vert \widehat{u^{\prime \prime }\tau _{y^{\prime
\prime }}\chi ^{\prime \prime }}\left( \xi "\right) \right\vert ^{p}\mathtt{d%
}y^{\prime \prime }\right) ^{1/p}
\end{multline*}%
which implies%
\begin{equation*}
\left\Vert u^{\prime }\otimes u^{\prime \prime }\right\Vert _{S_{w}^{p},\chi
}=\left\Vert u^{\prime }\right\Vert _{S_{w}^{p},\chi ^{\prime }}\left\Vert
u^{\prime \prime }\right\Vert _{S_{w}^{p},\chi ^{\prime \prime }}
\end{equation*}

$(\mathtt{c})$ Let $\mathbb{R}^{n}=\mathbb{R}^{n^{\prime }}\times \mathbb{R}%
^{n^{\prime \prime }}$, $\psi \in \mathcal{S}\left( \mathbb{R}^{n}\right) ,$ 
$u\in \mathcal{S}^{\prime }\left( \mathbb{R}^{n}\right) $ such that $%
\widehat{\psi u}\in L^{1}\left( \mathbb{R}^{n}\right) .$ Then 
\begin{eqnarray*}
\psi u\left( x^{\prime },x^{\prime \prime }\right) &=&\left( 2\pi \right)
^{-n}\iint \mathtt{e}^{\mathtt{i}\left\langle x^{\prime },\xi ^{\prime
}\right\rangle +\mathtt{i}\left\langle x^{\prime \prime },\xi ^{\prime
\prime }\right\rangle }\widehat{\psi u}\left( \xi ^{\prime },\xi ^{\prime
\prime }\right) \mathtt{d}\xi ^{\prime }\mathtt{d}\xi ^{\prime \prime }, \\
\psi u\left( x^{\prime },0\right) &=&\left( 2\pi \right) ^{-n^{\prime }}\int 
\mathtt{e}^{\mathtt{i}\left\langle x^{\prime },\xi ^{\prime }\right\rangle
}u_{\psi }^{\prime }\left( \xi ^{\prime }\right) \mathtt{d}\xi ^{\prime },
\end{eqnarray*}%
where $u_{\psi }^{\prime }\in L^{1}\left( \mathbb{R}^{n^{\prime }}\right) $
is given by 
\begin{equation*}
u_{\psi }^{\prime }\left( \xi ^{\prime }\right) =\left( 2\pi \right)
^{-n^{\prime \prime }}\int \widehat{\psi u}\left( \xi ^{\prime },\xi
^{\prime \prime }\right) \mathtt{d}\xi ^{\prime \prime }.
\end{equation*}%
It follows that 
\begin{equation*}
\mathcal{F}_{\mathbb{R}^{n^{\prime }}}\left( \left( \psi u\right) _{|\mathbb{%
R}^{n^{\prime }}}\right) \left( \xi ^{\prime }\right) =u_{\psi }^{\prime
}\left( \xi ^{\prime }\right) =\left( 2\pi \right) ^{-n^{\prime \prime
}}\int \widehat{\psi u}\left( \xi ^{\prime },\xi ^{\prime \prime }\right) 
\mathtt{d}\xi ^{\prime \prime }.
\end{equation*}

The first part shows that we can assume $L=\mathbb{R}^{n^{\prime }}\times
\left\{ 0\right\} $ where $\mathbb{R}^{n}=\mathbb{R}^{n^{\prime }}\times 
\mathbb{R}^{n^{\prime \prime }}$.

\underline{\textit{The case }$p=\infty $.} Let $\chi ^{\prime }\in \mathcal{S%
}\left( \mathbb{R}^{n^{\prime }}\right) \smallsetminus 0$, $\chi ^{\prime
\prime }\in \mathcal{S}\left( \mathbb{R}^{n^{\prime \prime }}\right) $, $%
\chi ^{\prime \prime }\left( 0\right) =1$. Then $\chi =\chi ^{\prime
}\otimes \chi ^{\prime \prime }\in \mathcal{S}\left( \mathbb{R}^{n^{\prime
}}\times \mathbb{R}^{n^{\prime \prime }}\right) \smallsetminus 0$. Let $u\in
S_{w}\left( \mathbb{R}^{n}\right) $. Then $U_{\chi }\left( \xi \right)
=\sup_{y\in \mathbb{R}^{n}}\left\vert \widehat{u\tau _{y}\chi }\left( \xi
\right) \right\vert $ is a lower semicontinuous function, $U_{\chi }\geq 0$
and $U_{\chi }\in L^{1}\left( \mathbb{R}^{n}\right) $. Define%
\begin{equation*}
U_{\chi }^{\prime }\left( \xi ^{\prime }\right) =\left( 2\pi \right)
^{-n^{\prime \prime }}\int U_{\chi }\left( \xi ^{\prime },\xi ^{\prime
\prime }\right) \mathtt{d}\xi ^{\prime \prime }.
\end{equation*}%
Then $U_{\chi }^{\prime }\geq 0,$ $U_{\chi }^{\prime }\in L^{1}\left( 
\mathbb{R}^{n^{\prime }}\right) $.

Since 
\begin{eqnarray*}
u_{|\mathbb{R}^{n^{\prime }}}\cdot \tau _{y^{\prime }}\chi ^{\prime }
&=&\left( u\tau _{\left( y^{\prime },0\right) }\chi \right) _{|\mathbb{R}%
^{n^{\prime }}} \\
\mathcal{F}_{\mathbb{R}^{n^{\prime }}}\left( u_{|\mathbb{R}^{n^{\prime
}}}\cdot \tau _{y^{\prime }}\chi ^{\prime }\right) \left( \xi ^{\prime
}\right) &=&\mathcal{F}_{\mathbb{R}^{n^{\prime }}}\left( \left( u\tau
_{\left( y^{\prime },0\right) }\chi \right) _{|\mathbb{R}^{n^{\prime
}}}\right) \left( \xi ^{\prime }\right) \\
&=&\left( 2\pi \right) ^{-n^{\prime \prime }}\int \widehat{u\tau _{\left(
y^{\prime },0\right) }\chi }\left( \xi ^{\prime },\xi ^{\prime \prime
}\right) \mathtt{d}\xi ^{\prime \prime }
\end{eqnarray*}%
it follows that%
\begin{equation*}
\sup_{y^{\prime }}\left\vert \mathcal{F}_{\mathbb{R}^{n^{\prime }}}\left(
u_{|\mathbb{R}^{n^{\prime }}}\cdot \tau _{y^{\prime }}\chi ^{\prime }\right)
\left( \xi ^{\prime }\right) \right\vert \leq \left( 2\pi \right)
^{-n^{\prime \prime }}\int U_{\chi }\left( \xi ^{\prime },\xi ^{\prime
\prime }\right) \mathtt{d}\xi ^{\prime \prime }=U_{\chi }^{\prime }\left(
\xi ^{\prime }\right) .
\end{equation*}%
Hence $u_{|\mathbb{R}^{n^{\prime }}}\in S_{w}\left( \mathbb{R}^{n^{\prime
}}\right) $ and 
\begin{equation*}
\left\Vert u_{|\mathbb{R}^{n^{\prime }}}\right\Vert _{S_{w},\chi ^{\prime
}}\leq \left( 2\pi \right) ^{-n^{\prime \prime }}\left\Vert u\right\Vert
_{S_{w},\chi }.
\end{equation*}

\underline{\textit{The case }$1\leq p<\infty $.} Since $\mathcal{C}%
_{0}^{\infty }\left( \mathbb{R}^{n}\right) $ is dense in $S_{w}^{p}\left( 
\mathbb{R}^{n}\right) $ for the narrow convergence, we can assume that $u\in 
\mathcal{C}_{0}^{\infty }\left( \mathbb{R}^{n}\right) $. Let $\chi ^{\prime
}\in \mathcal{C}_{0}^{\infty }\left( \mathbb{R}^{n^{\prime }}\right)
\smallsetminus 0$, $\chi ^{\prime \prime }$, $\varphi \in \mathcal{C}%
_{0}^{\infty }\left( \mathbb{R}^{n^{\prime \prime }}\right) $, $\int \chi
^{\prime \prime }\left( x^{\prime \prime }\right) \mathtt{d}x^{\prime \prime
}=1$, $\varphi \left( 0\right) =1$. Define 
\begin{eqnarray*}
\chi &=&\chi ^{\prime }\otimes \chi ^{\prime \prime }\in \mathcal{C}%
_{0}^{\infty }\left( \mathbb{R}^{n^{\prime }}\times \mathbb{R}^{n^{\prime
\prime }}\right) \smallsetminus 0 \\
v &=&u_{|\mathbb{R}^{n^{\prime }}}=w_{|\mathbb{R}^{n^{\prime }}},\quad
v\left( x^{\prime }\right) =u\left( x^{\prime },0\right) =w\left( x^{\prime
},0\right) \\
w &=&u\left( 1\otimes \varphi \right) ,\quad w\left( x\right) =u\left(
x\right) \varphi \left( x^{\prime \prime }\right) .
\end{eqnarray*}%
We shall calculate $\mathcal{F}_{\mathbb{R}^{n^{\prime }}}\left( u_{|\mathbb{%
R}^{n^{\prime }}}\cdot \tau _{y^{\prime }}\chi ^{\prime }\right) \left( \xi
^{\prime }\right) $. Let $N=n+1$. Then%
\begin{multline*}
\widehat{v\tau _{y^{\prime }}\chi ^{\prime }}\left( \xi ^{\prime }\right)
=\left( 2\pi \right) ^{-n^{\prime \prime }}\int \left( \int u\left( x\right)
\varphi \left( x^{\prime \prime }\right) \tau _{y^{\prime }}\chi ^{\prime
}\left( x^{\prime }\right) \mathtt{e}^{-\mathtt{i}\left\langle x,\xi
\right\rangle }\mathtt{d}x\right) \mathtt{d}\xi ^{\prime \prime } \\
=\left( 2\pi \right) ^{-n^{\prime \prime }}\int \left( \int \left( \int
u\left( x\right) \varphi \left( x^{\prime \prime }\right) \tau _{y^{\prime
}}\chi ^{\prime }\left( x^{\prime }\right) \tau _{y^{\prime \prime }}\chi
^{\prime \prime }\left( x^{\prime \prime }\right) \mathtt{e}^{-\mathtt{i}%
\left\langle x,\xi \right\rangle }\mathtt{d}y^{\prime \prime }\right) 
\mathtt{d}x\right) \mathtt{d}\xi ^{\prime \prime } \\
=\left( 2\pi \right) ^{-n^{\prime \prime }}\int \left( \int \left( \int
u\left( x\right) \varphi \left( x^{\prime \prime }\right) \tau _{y^{\prime
}}\chi ^{\prime }\left( x^{\prime }\right) \tau _{y^{\prime \prime }}\chi
^{\prime \prime }\left( x^{\prime \prime }\right) \mathtt{e}^{-\mathtt{i}%
\left\langle x,\xi \right\rangle }\mathtt{d}x\right) \mathtt{d}y^{\prime
\prime }\right) \mathtt{d}\xi ^{\prime \prime } \\
=\left( 2\pi \right) ^{-n^{\prime \prime }}\int \left( \int \left( \int
u\left( x\right) \varphi \left( x^{\prime \prime }\right) \tau _{y}\chi
\left( x\right) \mathtt{e}^{-\mathtt{i}\left\langle x,\xi \right\rangle
}\left\langle y^{\prime \prime }\right\rangle ^{2N}\mathtt{d}x\right)
\left\langle y^{\prime \prime }\right\rangle ^{-2N}\mathtt{d}y^{\prime
\prime }\right) \mathtt{d}\xi ^{\prime \prime }
\end{multline*}%
We have%
\begin{equation*}
\left\langle y^{\prime \prime }\right\rangle ^{2N}=\sum_{\left\vert \alpha
+\beta \right\vert \leq 2N}c_{\alpha \beta }x^{\prime \prime \alpha }\left(
x-y\right) ^{\prime \prime \beta }
\end{equation*}%
with $c_{\alpha \beta }$ constants depending only on $n$, $\alpha $ and $%
\beta $. We get 
\begin{equation*}
\left\langle y^{\prime \prime }\right\rangle ^{2N}\varphi \left( x^{\prime
\prime }\right) \tau _{y}\chi \left( x\right) =\sum_{\left\vert \alpha
+\beta \right\vert \leq 2N}c_{\alpha \beta }\varphi _{\alpha }\left(
x^{\prime \prime }\right) \tau _{y}\chi _{\beta }\left( x\right)
\end{equation*}%
with $\varphi _{\alpha }\left( x^{\prime \prime }\right) =x^{\prime \prime
\alpha }\varphi \left( x^{\prime \prime }\right) $ and $\chi _{\beta }\left(
x\right) =x^{\prime \prime \beta }\chi \left( x\right) $. Hence 
\begin{eqnarray*}
\widehat{v\tau _{y^{\prime }}\chi ^{\prime }}\left( \xi ^{\prime }\right)
&=&\left( 2\pi \right) ^{-n^{\prime \prime }}\sum_{\left\vert \alpha +\beta
\right\vert \leq 2N}c_{\alpha \beta }J_{\alpha \beta }\left( y^{\prime },\xi
^{\prime }\right) , \\
J_{\alpha \beta }\left( y^{\prime },\xi ^{\prime }\right) &=&\int \left(
\int \left( \int u\left( x\right) \varphi _{\alpha }\left( x^{\prime \prime
}\right) \tau _{y}\chi _{\beta }\left( x\right) \mathtt{e}^{-\mathtt{i}%
\left\langle x,\xi \right\rangle }\mathtt{d}x\right) \left\langle y^{\prime
\prime }\right\rangle ^{-2N}\mathtt{d}y^{\prime \prime }\right) \mathtt{d}%
\xi ^{\prime \prime } \\
&=&\int \left( \int \widehat{u_{\alpha }\left( x\right) \tau _{y}\chi
_{\beta }}\left( \xi \right) \left\langle y^{\prime \prime }\right\rangle
^{-2N}\mathtt{d}y^{\prime \prime }\right) \mathtt{d}\xi ^{\prime \prime }
\end{eqnarray*}%
where $u_{\alpha }=u\left( 1\otimes \varphi _{\alpha }\right) $. Using
Holder's inequality we obtain the estimate%
\begin{eqnarray*}
\left\vert J_{\alpha \beta }\left( y^{\prime },\xi ^{\prime }\right)
\right\vert &\leq &\int \left( \int \left\vert \widehat{u_{\alpha }\left(
x\right) \tau _{y}\chi _{\beta }}\left( \xi \right) \right\vert \left\langle
y^{\prime \prime }\right\rangle ^{-2N}\mathtt{d}y^{\prime \prime }\right) 
\mathtt{d}\xi ^{\prime \prime } \\
&\leq &\int \left( \int \left\vert \widehat{u_{\alpha }\left( x\right) \tau
_{y}\chi _{\beta }}\left( \xi \right) \right\vert ^{p}\mathtt{d}y^{\prime
\prime }\right) ^{1/p}\left( \int \left\langle y^{\prime \prime
}\right\rangle ^{-2qN}\mathtt{d}y^{\prime \prime }\right) ^{1/q}\mathtt{d}%
\xi ^{\prime \prime } \\
&=&C_{q,n}\int \left( \int \left\vert \widehat{u_{\alpha }\left( x\right)
\tau _{y}\chi _{\beta }}\left( \xi \right) \right\vert ^{p}\mathtt{d}%
y^{\prime \prime }\right) ^{1/p}\mathtt{d}\xi ^{\prime \prime }.
\end{eqnarray*}%
with 
\begin{equation*}
C_{q,n}=\left( \int \left\langle y^{\prime \prime }\right\rangle ^{-2qN}%
\mathtt{d}y^{\prime \prime }\right) ^{1/q}
\end{equation*}%
Now the integral version of Minkowski's inequality implies%
\begin{eqnarray*}
\left\Vert J_{\alpha \beta }\left( \cdot ,\xi ^{\prime }\right) \right\Vert
_{p} &\leq &C_{q,n}\int \left( \int \left\vert \widehat{u_{\alpha }\left(
x\right) \tau _{y}\chi _{\beta }}\left( \xi \right) \right\vert ^{p}\mathtt{d%
}y\right) ^{1/p}\mathtt{d}\xi ^{\prime \prime } \\
&=&C_{q,n}\int U_{\alpha ,\chi _{\beta },p}\left( \xi \right) \mathtt{d}\xi
^{\prime \prime }.
\end{eqnarray*}%
If we combine these inequalities, we obtain%
\begin{eqnarray*}
\left\Vert u_{|\mathbb{R}^{n^{\prime }}}\right\Vert _{S_{w}^{p},\chi
^{\prime }} &\leq &\left( 2\pi \right) ^{-n^{\prime \prime
}}C_{q,n}\sum_{\left\vert \alpha +\beta \right\vert \leq 2N}\left\vert
c_{\alpha \beta }\right\vert \int U_{\alpha ,\chi _{\beta },p}\left( \xi
\right) \mathtt{d}\xi \\
&=&\left( 2\pi \right) ^{-n^{\prime \prime }}C_{q,n}\sum_{\left\vert \alpha
+\beta \right\vert \leq 2N}\left\vert c_{\alpha \beta }\right\vert
\left\Vert u_{\alpha }\right\Vert _{S_{w}^{p},\chi _{\beta }}.
\end{eqnarray*}%
Since $1\otimes \varphi _{\alpha }\in S_{w}\left( \mathbb{R}^{n}\right) $
then from lemma \ref{st5} and proposition \ref{st6} we obtain that that
there are constants $A_{\alpha \beta }>0$ such that 
\begin{equation*}
\left\Vert u_{\alpha }\right\Vert _{S_{w}^{p},\chi _{\beta }}\leq A_{\alpha
\beta }\left\Vert u\right\Vert _{S_{w}^{p},\chi }.
\end{equation*}%
Hence 
\begin{equation*}
\left\Vert u_{|\mathbb{R}^{n^{\prime }}}\right\Vert _{S_{w}^{p},\chi
^{\prime }}\leq \left( 2\pi \right) ^{-n^{\prime \prime }}C_{q,n}\left(
\sum_{\left\vert \alpha +\beta \right\vert \leq 2N}\left\vert c_{\alpha
\beta }\right\vert A_{\alpha \beta }\right) \left\Vert u\right\Vert
_{S_{w}^{p},\chi }.
\end{equation*}
\end{proof}

For $\lambda \in \mathtt{End}_{%
\mathbb{R}
}\left( 
\mathbb{R}
^{n}\right) $ and $u\in \mathcal{BC}\left( \mathbb{R}^{n}\right) $ put $%
u_{\lambda }=u\circ \lambda $.

\begin{proposition}
\label{sjo7}If $\lambda \in \mathtt{End}_{%
\mathbb{R}
}\left( 
\mathbb{R}
^{n}\right) $ is invertible and $u\in S_{w}^{p}\left( \mathbb{R}^{n}\right) $%
, then $u_{\lambda }\in S_{w}^{p}\left( \mathbb{R}^{n}\right) $ and there is 
$C\in \left( 0,+\infty \right) $ independent of $u$ and $\lambda $ such that%
\begin{equation*}
\left\Vert u_{\lambda }\right\Vert _{S_{w}^{p}}\leq C\left\vert \det \lambda
\right\vert ^{-n/p}\left( 1+\left\Vert \lambda \right\Vert \right)
^{n}\left\Vert u\right\Vert _{S_{w}^{p}}.
\end{equation*}
\end{proposition}

\begin{proof}
Let $\chi \in \mathcal{C}_{0}^{\infty }\left( \mathbb{R}^{n}\right) $ be
such that $\dint \chi \left( x\right) \mathtt{d}x=1$. We shall use the
notation $\left\Vert \cdot \right\Vert _{S_{w}^{p}}$ for $\left\Vert \cdot
\right\Vert _{S_{w}^{p},\chi }$. Let $r>0$ be such that $\mathtt{supp}\chi
\subset \left\{ x:\left\vert x\right\vert \leq r\right\} $. We denote by $%
\chi _{1}$ the characteristic function of the unit ball in $\mathbb{R}^{n}$.
We evaluate%
\begin{eqnarray*}
\widehat{u_{\lambda }\tau _{y}\chi }\left( \xi \right) &=&\dint \mathtt{e}^{-%
\mathtt{i}\left\langle x,\xi \right\rangle }u\left( \lambda x\right) \chi
\left( x-y\right) \mathtt{d}x \\
&=&\dint \int \mathtt{e}^{-\mathtt{i}\left\langle x,\xi \right\rangle
}u\left( \lambda x\right) \chi \left( x-y\right) \chi \left( \lambda
x-z\right) \mathtt{d}x\mathtt{d}z.
\end{eqnarray*}%
Since $\chi \left( x-y\right) \chi \left( \lambda x-z\right) \neq 0$ implies 
$\left\vert x-y\right\vert \leq r$ and $\left\vert \lambda x-z\right\vert
\leq r$, we get that $\left\vert z-\lambda y\right\vert \leq r\left(
1+\left\Vert \lambda \right\Vert \right) $ on the support of the integrand.
So%
\begin{equation*}
\chi \left( x-y\right) \chi \left( \lambda x-z\right) =\chi \left(
x-y\right) \chi \left( \lambda x-z\right) \chi _{1}\left( \frac{z-\lambda y}{%
r\left( 1+\left\Vert \lambda \right\Vert \right) }\right) .
\end{equation*}%
Therefore%
\begin{multline*}
\widehat{u_{\lambda }\tau _{y}\chi }\left( \xi \right) =\dint \int \mathtt{e}%
^{-\mathtt{i}\left\langle x,\xi \right\rangle }\left( u\tau _{z}\chi \right)
\left( \lambda x\right) \chi \left( x-y\right) \chi _{1}\left( \frac{%
z-\lambda y}{r\left( 1+\left\Vert \lambda \right\Vert \right) }\right) 
\mathtt{d}x\mathtt{d}z \\
=\left( 2\pi \right) ^{-n}\dint \chi _{1}\left( \frac{z-\lambda y}{r\left(
1+\left\Vert \lambda \right\Vert \right) }\right) \left( \diint \mathtt{e}^{-%
\mathtt{i}\left\langle x,\xi \right\rangle }\mathtt{e}^{\mathtt{i}%
\left\langle x-y,\eta \right\rangle }\left( u\tau _{z}\chi \right) \left(
\lambda x\right) \widehat{\chi }\left( \eta \right) \mathtt{d}x\mathtt{d}%
\eta \right) \mathtt{d}z
\end{multline*}%
and%
\begin{multline*}
\diint \mathtt{e}^{-\mathtt{i}\left\langle x,\xi \right\rangle }\mathtt{e}^{%
\mathtt{i}\left\langle x-y,\eta \right\rangle }\left( u\tau _{z}\chi \right)
\left( \lambda x\right) \widehat{\chi }\left( \eta \right) \mathtt{d}x%
\mathtt{d}\eta \\
=\diint \mathtt{e}^{-\mathtt{i}\left\langle \lambda ^{-1}x,\xi \right\rangle
+\mathtt{i}\left\langle x-\lambda y,\eta \right\rangle }\left( u\tau
_{z}\chi \right) \left( x\right) \widehat{\chi }\left( ^{t}\lambda \eta
\right) \mathtt{d}x\mathtt{d}\eta \\
=\diint \mathtt{e}^{-\mathtt{i}\left\langle x,^{t}\lambda ^{-1}\xi -\eta
\right\rangle -\mathtt{i}\left\langle y,^{t}\lambda \eta \right\rangle
}\left( u\tau _{z}\chi \right) \left( x\right) \widehat{\chi }\left(
^{t}\lambda \eta \right) \mathtt{d}x\mathtt{d}\eta \\
=\dint \mathtt{e}^{-\mathtt{i}\left\langle y,^{t}\lambda \eta \right\rangle }%
\widehat{u\tau _{z}\chi }\left( ^{t}\lambda ^{-1}\xi -\eta \right) \widehat{%
\chi }\left( ^{t}\lambda \eta \right) \mathtt{d}\eta \\
=\dint \mathtt{e}^{-\mathtt{i}\left\langle y,\xi -^{t}\lambda \eta
\right\rangle }\widehat{u\tau _{z}\chi }\left( \eta \right) \widehat{\chi }%
\left( \xi -^{t}\lambda \eta \right) \mathtt{d}\eta
\end{multline*}%
so%
\begin{multline*}
\widehat{u_{\lambda }\tau _{y}\chi }\left( \xi \right) =\left( 2\pi \right)
^{-n}\dint \dint \mathtt{e}^{-\mathtt{i}\left\langle y,\xi -^{t}\lambda \eta
\right\rangle }\chi _{1}\left( \frac{z-\lambda y}{r\left( 1+\left\Vert
\lambda \right\Vert \right) }\right) \widehat{u\tau _{z}\chi }\left( \eta
\right) \widehat{\chi }\left( \xi -^{t}\lambda \eta \right) \mathtt{d}\eta 
\mathtt{d}z \\
=\left( 2\pi \right) ^{-n}\dint \dint \mathtt{e}^{-\mathtt{i}\left\langle
y,\xi -^{t}\lambda \eta \right\rangle }\chi _{1}\left( \frac{x}{r\left(
1+\left\Vert \lambda \right\Vert \right) }\right) \widehat{u\tau _{x+\lambda
y}\chi }\left( \eta \right) \widehat{\chi }\left( \xi -^{t}\lambda \eta
\right) \mathtt{d}\eta \mathtt{d}x.
\end{multline*}%
If $y\in 
\mathbb{R}
^{n}$ then%
\begin{equation*}
\left\vert \widehat{u_{_{\lambda }}\tau _{y}\chi }\left( \xi \right)
\right\vert \leq \left( 2\pi \right) ^{-n}\dint \dint \chi _{1}\left( \frac{x%
}{r\left( 1+\left\Vert \lambda \right\Vert \right) }\right) \left\vert 
\widehat{u\tau _{x+\lambda y}\chi }\left( \eta \right) \right\vert
\left\vert \widehat{\chi }\left( \xi -^{t}\lambda \eta \right) \right\vert 
\mathtt{d}\eta \mathtt{d}x
\end{equation*}%
From this we obtain that%
\begin{multline*}
\left( \int \left\vert \widehat{u_{_{\lambda }}\tau _{y}\chi }\left( \xi
\right) \right\vert ^{p}\mathtt{d}y\right) ^{1/p} \\
\leq \left( 2\pi \right) ^{-n}\dint \dint \chi _{1}\left( \frac{x}{r\left(
1+\left\Vert \lambda \right\Vert \right) }\right) \left\vert \widehat{\chi }%
\left( \xi -^{t}\lambda \eta \right) \right\vert \left( \int \left\vert 
\widehat{u\tau _{x+\lambda y}\chi }\left( \eta \right) \right\vert ^{p}%
\mathtt{d}y\right) ^{1/p}\mathtt{d}\eta \mathtt{d}x \\
=\left( 2\pi \right) ^{-n}\left\vert \det \lambda \right\vert ^{-1/p}\dint
\dint \chi _{1}\left( \frac{x}{r\left( 1+\left\Vert \lambda \right\Vert
\right) }\right) \left\vert \widehat{\chi }\left( \xi -^{t}\lambda \eta
\right) \right\vert U_{\chi ,p}\left( \eta \right) \mathtt{d}\eta \mathtt{d}x
\\
=\left( 2\pi \right) ^{-n}r^{n}\left\vert \det \lambda \right\vert
^{-1/p}\left( 1+\left\Vert \lambda \right\Vert \right) ^{n}\limfunc{Vol}%
\left( \left\{ \left\vert x\right\vert \leq 1\right\} \right) \dint U_{\chi
,p}\left( \eta \right) \left\vert \widehat{\chi }\left( \xi -^{t}\lambda
\eta \right) \right\vert \mathtt{d}\eta
\end{multline*}%
which by integration with respect to $\xi $ gives us 
\begin{equation*}
\left\Vert u_{_{\lambda }}\right\Vert _{S_{w}^{p}}\leq \left( 2\pi \right)
^{-n}r^{n}\left\vert \det \lambda \right\vert ^{-1/p}\left( 1+\left\Vert
\lambda \right\Vert \right) ^{n}\limfunc{Vol}\left( \left\{ \left\vert
x\right\vert \leq 1\right\} \right) \left\Vert \widehat{\chi }\right\Vert
_{_{L^{1}}}\left\Vert u\right\Vert _{S_{w}^{p}}.
\end{equation*}
\end{proof}

We shall give now an example which will be developed later in section \ref%
{st35}.

Let $\Gamma \subset \mathbb{R}^{n}$ be a lattice. Let $\chi \in \mathcal{S}%
\left( 
\mathbb{R}
^{n}\right) $. Then $\Phi _{\Gamma ,\chi }=\sum_{\gamma \in \Gamma
}\left\vert \tau _{\gamma }\chi \right\vert \geq 0$, $\Phi _{\Gamma ,\chi
}\in \mathcal{BC}\left( \mathbb{R}^{n}\right) $ and we have 
\begin{eqnarray*}
\Phi _{\Gamma ,\chi } &\leq &p_{n+1}\left( \chi \right) \sum_{\gamma \in
\Gamma }\tau _{\gamma }\left\langle \cdot \right\rangle
^{-n-1}=p_{n+1}\left( \chi \right) f_{n+1} \\
&\leq &p_{n+1}\left( \chi \right) \frac{2^{n+1}\left( \sup_{z\in \overline{%
\mathtt{C}}}\left\langle z\right\rangle ^{n+1}\right) ^{2}}{\mathtt{vol}%
\left( \mathtt{C}\right) }\int \left\langle y\right\rangle ^{-n-1}\mathtt{d}%
y=p_{n+1}\left( \chi \right) C_{\Gamma ,n}
\end{eqnarray*}%
where%
\begin{equation*}
p_{n+1}\left( \chi \right) =\sup\nolimits_{x\in 
\mathbb{R}
^{n}}\left\langle x\right\rangle ^{n+1}\left\vert \chi \left( x\right)
\right\vert .
\end{equation*}%
We introduce the Banach space 
\begin{gather*}
H_{1}^{n+1}=\left\{ u\in \mathcal{S}^{\prime }\left( 
\mathbb{R}
^{n}\right) :\partial ^{\alpha }u\in L^{1}\left( \mathbb{R}^{n}\right)
,\forall \alpha ,\left\vert \alpha \right\vert \leq n+1\right\} , \\
\left\Vert u\right\Vert _{H_{1}^{n+1}}=\sum_{\left\vert \alpha \right\vert
\leq n+1}\left\Vert \partial ^{\alpha }u\right\Vert _{L^{1}},\quad u\in
H_{1}^{n+1}.
\end{gather*}

\begin{lemma}
\label{st26}$H_{1}^{n+1}\hookrightarrow S_{w}^{1}\left( \mathbb{R}%
^{n}\right) .$
\end{lemma}

\begin{proof}
For $\gamma \in \Gamma $ we have%
\begin{eqnarray*}
\int \left\vert \widehat{u\tau _{\gamma }\chi }\left( \xi \right)
\right\vert \mathtt{d}\xi &\leq &\int \left\langle \xi \right\rangle
^{-n-1}\left\langle \xi \right\rangle ^{n+1}\left\vert \widehat{u\tau
_{\gamma }\chi }\left( \xi \right) \right\vert \mathtt{d}\xi \\
&\leq &C\left\Vert \left\langle \cdot \right\rangle ^{-n-1}\right\Vert
_{L^{1}}\sum_{\left\vert \alpha +\beta \right\vert \leq n+1}\left\Vert 
\widehat{\partial ^{\alpha }u\tau _{\gamma }\partial ^{\beta }\chi }%
\right\Vert _{L^{\infty }} \\
&\leq &C\left\Vert \left\langle \cdot \right\rangle ^{-n-1}\right\Vert
_{L^{1}}\sum_{\left\vert \alpha +\beta \right\vert \leq n+1}\left\Vert
\partial ^{\alpha }u\tau _{\gamma }\left( \partial ^{\beta }\chi \right)
\right\Vert _{L^{1}}.
\end{eqnarray*}%
By summing with respect to $\gamma \in \Gamma $ we get 
\begin{eqnarray*}
\left\Vert u\right\Vert _{S_{w}^{1},\Gamma ,\chi } &=&\int \sum_{\gamma \in
\Gamma }\left\vert \widehat{u\tau _{\gamma }\chi }\left( \xi \right)
\right\vert \mathtt{d}\xi =\sum_{\gamma \in \Gamma }\int \left\vert \widehat{%
u\tau _{\gamma }\chi }\left( \xi \right) \right\vert \mathtt{d}\xi \\
&\leq &C\left\Vert \left\langle \cdot \right\rangle ^{-n-1}\right\Vert
_{L^{1}}\sum_{\left\vert \alpha +\beta \right\vert \leq n+1}\sum_{\gamma \in
\Gamma }\int \left\vert \partial ^{\alpha }u\tau _{\gamma }\left( \partial
^{\beta }\chi \right) \right\vert \mathtt{d}x \\
&=&C\left\Vert \left\langle \cdot \right\rangle ^{-n-1}\right\Vert
_{L^{1}}\sum_{\left\vert \alpha +\beta \right\vert \leq n+1}\int \left\vert
\partial ^{\alpha }u\right\vert \sum_{\gamma \in \Gamma }\left\vert \tau
_{\gamma }\left( \partial ^{\beta }\chi \right) \right\vert \mathtt{d}x \\
&=&C\left\Vert \left\langle \cdot \right\rangle ^{-n-1}\right\Vert
_{L^{1}}\sum_{\left\vert \alpha +\beta \right\vert \leq n+1}\int \left\vert
\partial ^{\alpha }u\right\vert \Phi _{\Gamma ,\partial ^{\beta }\chi }%
\mathtt{d}x \\
&\leq &CC_{\Gamma ,n}\left\Vert \left\langle \cdot \right\rangle
^{-n-1}\right\Vert _{L^{1}}\sum_{\left\vert \alpha +\beta \right\vert \leq
n+1}p_{n+1}\left( \partial ^{\beta }\chi \right) \left\Vert \partial
^{\alpha }u\right\Vert _{L^{1}} \\
&\leq &CC_{\Gamma ,n}\left\Vert \left\langle \cdot \right\rangle
^{-n-1}\right\Vert _{L^{1}}\sum_{\left\vert \beta \right\vert \leq
n+1}p_{n+1}\left( \partial ^{\beta }\chi \right) \sum_{\left\vert \alpha
\right\vert \leq n+1}\left\Vert \partial ^{\alpha }u\right\Vert _{L^{1}} \\
&\leq &CC_{\Gamma ,n}\left\Vert \left\langle \cdot \right\rangle
^{-n-1}\right\Vert _{L^{1}}\sum_{\left\vert \beta \right\vert \leq
n+1}p_{n+1}\left( \partial ^{\beta }\chi \right) \left\Vert u\right\Vert
_{H_{1}^{n+1}}
\end{eqnarray*}%
Hence 
\begin{equation*}
\left\Vert u\right\Vert _{S_{w}^{1},\chi }\leq CC_{\Gamma ,n}\left\Vert
\left\langle \cdot \right\rangle ^{-n-1}\right\Vert _{L^{1}}\sum_{\left\vert
\beta \right\vert \leq n+1}p_{n+1}\left( \partial ^{\beta }\chi \right)
\left\Vert u\right\Vert _{H_{1}^{n+1}}.
\end{equation*}
\end{proof}

Let 
\begin{equation*}
S_{0}^{m}\left( \mathbb{R}^{n}\right) =\left\{ u\in \mathcal{C}^{\infty
}\left( \mathbb{R}^{n}\right) :\left\vert \partial ^{\alpha }u\right\vert
\leq C_{\alpha }\left\langle \cdot \right\rangle ^{m},\forall \alpha
\right\} .
\end{equation*}

\begin{corollary}
If $m>n$, then $S_{0}^{-m}\left( \mathbb{R}^{n}\right) \hookrightarrow
H_{1}^{n+1}\hookrightarrow S_{w}^{1}\left( \mathbb{R}^{n}\right) .$
\end{corollary}

\section{A spectral characterization of the ideals $S_{w}^{p}$\label{st32}}

\begin{lemma}
\label{st37}$(\mathtt{a})$ Let $\chi \in \mathcal{S}\left( \mathbb{R}%
^{n}\right) $ and $u\in \mathcal{S}^{\prime }\left( \mathbb{R}^{n}\right) $.
Then $\chi \left( D\right) u\in \mathcal{S}^{\prime }\left( \mathbb{R}%
^{n}\right) \cap \mathcal{C}_{pol}^{\infty }\left( \mathbb{R}^{n}\right) $.
In fact we have 
\begin{equation*}
\chi \left( D\right) u\left( x\right) =\left( 2\pi \right) ^{-n}\left\langle 
\widehat{u},\mathtt{e}^{\mathtt{i}\left\langle x,\cdot \right\rangle }\chi
\right\rangle ,\quad x\in \mathbb{R}^{n}.
\end{equation*}

$(\mathtt{b})$ Let $u\in \mathcal{D}^{\prime }\left( \mathbb{R}^{n}\right) $ 
$($or $u\in \mathcal{S}^{\prime }\left( \mathbb{R}^{n}\right) )$ and $\chi
\in \mathcal{C}_{0}^{\infty }\left( \mathbb{R}^{n}\right) $ $($or $\chi \in 
\mathcal{S}\left( \mathbb{R}^{n}\right) )$. Then 
\begin{equation*}
\mathbb{R}^{n}\times \mathbb{R}^{n}\ni \left( x,\xi \right) \rightarrow \chi
\left( D-\xi \right) u\left( x\right) =\left( 2\pi \right) ^{-n}\left\langle 
\widehat{u},\mathtt{e}^{\mathtt{i}\left\langle x,\cdot \right\rangle }\chi
\left( \cdot -\xi \right) \right\rangle \in 
\mathbb{C}%
\end{equation*}%
is a $\mathcal{C}^{\infty }$-function.
\end{lemma}

\begin{proof}
Apply lemma \ref{st7}. Let $\varphi \in \mathcal{S}\left( \mathbb{R}%
_{x}^{n}\right) $. We have 
\begin{eqnarray*}
\left\langle \chi \left( D\right) u,\varphi \right\rangle &=&\left\langle 
\mathcal{F}^{-1}\left( \chi \widehat{u}\right) ,\varphi \right\rangle
=\left\langle \chi \widehat{u},\mathcal{F}^{-1}\varphi \right\rangle =\left(
2\pi \right) ^{-n}\left\langle \chi \widehat{u},\mathcal{F}\check{\varphi}%
\right\rangle \\
&=&\left( 2\pi \right) ^{-n}\left\langle \mathcal{F}\left( \chi \widehat{u}%
\right) ,\check{\varphi}\right\rangle =\left( 2\pi \right) ^{-n}\int
\left\langle \widehat{u},\mathtt{e}^{-\mathtt{i}\left\langle x,\cdot
\right\rangle }\chi \right\rangle \check{\varphi}\left( x\right) \mathtt{d}x
\\
&=&\left( 2\pi \right) ^{-n}\int \left\langle \widehat{u},\mathtt{e}^{%
\mathtt{i}\left\langle x,\cdot \right\rangle }\chi \right\rangle \varphi
\left( x\right) \mathtt{d}x
\end{eqnarray*}%
Hence 
\begin{equation*}
\chi \left( D\right) u\left( x\right) =\left( 2\pi \right) ^{-n}\left\langle 
\widehat{u},\mathtt{e}^{\mathtt{i}\left\langle x,\cdot \right\rangle }\chi
\right\rangle ,\quad x\in \mathbb{R}^{n}.
\end{equation*}
\end{proof}

\begin{lemma}
Let $\chi \in \mathcal{S}\left( \mathbb{R}^{n}\right) $ and $u\in \mathcal{S}%
^{\prime }\left( \mathbb{R}^{n}\right) $. Then%
\begin{equation*}
\chi \left( D-\xi \right) u\left( x\right) =\left( 2\pi \right) ^{-n}\mathtt{%
e}^{\mathtt{i}\left\langle x,\xi \right\rangle }\widehat{u\tau _{x}\widehat{%
\chi }}\left( \xi \right) ,\quad x,\xi \in \mathbb{R}^{n}.
\end{equation*}
\end{lemma}

\begin{proof}
Suppose that $u,\chi \in \mathcal{S}\left( \mathbb{R}^{n}\right) $. Then
using 
\begin{eqnarray*}
\mathcal{F}\left( \mathtt{e}^{\mathtt{i}\left\langle x,\cdot \right\rangle
}\chi \left( \cdot -\xi \right) \right) \left( y\right) &=&\int \mathtt{e}^{-%
\mathtt{i}\left\langle y,\eta \right\rangle }\mathtt{e}^{\mathtt{i}%
\left\langle x,\eta \right\rangle }\chi \left( \eta -\xi \right) \mathtt{d}%
\eta \\
&=&\int \mathtt{e}^{\mathtt{i}\left\langle x-y,\zeta +\xi \right\rangle
}\chi \left( \zeta \right) \mathtt{d}\zeta \\
&=&\mathtt{e}^{\mathtt{i}\left\langle x-y,\xi \right\rangle }\widehat{\chi }%
\left( y-x\right)
\end{eqnarray*}%
we obtain%
\begin{eqnarray*}
\chi \left( D-\xi \right) u\left( x\right) &=&\left( 2\pi \right) ^{-n}\int 
\mathtt{e}^{\mathtt{i}\left\langle x,\eta \right\rangle }\chi \left( \eta
-\xi \right) \widehat{u}\left( \eta \right) \mathtt{d}\eta \\
&=&\left( 2\pi \right) ^{-n}\int \mathcal{F}\left( \mathtt{e}^{\mathtt{i}%
\left\langle x,\cdot \right\rangle }\chi \left( \cdot -\xi \right) \right)
\left( y\right) u\left( y\right) \mathtt{d}y \\
&=&\left( 2\pi \right) ^{-n}\int \mathtt{e}^{\mathtt{i}\left\langle x-y,\xi
\right\rangle }\widehat{\chi }\left( y-x\right) u\left( y\right) \mathtt{d}y
\\
&=&\left( 2\pi \right) ^{-n}\mathtt{e}^{\mathtt{i}\left\langle x,\xi
\right\rangle }\int \mathtt{e}^{-\mathtt{i}\left\langle y,\xi \right\rangle }%
\widehat{\chi }\left( y-x\right) u\left( y\right) \mathtt{d}y \\
&=&\left( 2\pi \right) ^{-n}\mathtt{e}^{\mathtt{i}\left\langle x,\xi
\right\rangle }\widehat{u\tau _{x}\widehat{\chi }}\left( \xi \right)
\end{eqnarray*}%
The general case is obtained from the density of $\mathcal{S}\left( \mathbb{R%
}^{n}\right) $ in $\mathcal{S}^{\prime }\left( \mathbb{R}^{n}\right) $
noticing that both 
\begin{equation*}
\chi \left( D-\xi \right) u\left( x\right) =\left( 2\pi \right)
^{-n}\left\langle \widehat{u},\mathtt{e}^{\mathtt{i}\left\langle x,\cdot
\right\rangle }\chi \left( \cdot -\xi \right) \right\rangle
\end{equation*}%
and%
\begin{equation*}
u\tau _{x}\widehat{\chi }\left( \xi \right) =\left\langle u,\mathtt{e}^{-%
\mathtt{i}\left\langle \cdot ,\xi \right\rangle }\widehat{\chi }\left( \cdot
-x\right) \right\rangle
\end{equation*}%
depend continuously on $u$.
\end{proof}

\begin{corollary}
Let $\Gamma \subset \mathbb{R}^{n}$ be a lattice, $1\leq p\leq \infty $, $%
\chi \in \mathcal{S}\left( \mathbb{R}^{n}\right) $ and $u\in \mathcal{S}%
^{\prime }\left( \mathbb{R}^{n}\right) $. Then 
\begin{equation*}
\left\Vert \chi \left( D-\xi \right) u\right\Vert _{L^{p}}=\left( 2\pi
\right) ^{-n}U_{\widehat{\chi },p}\left( \xi \right) ,\quad \xi \in \mathbb{R%
}^{n}
\end{equation*}%
and%
\begin{equation*}
\left\Vert \left( \chi \left( D-\xi \right) u\left( \gamma \right) \right)
_{\gamma \in \Gamma }\right\Vert _{l^{p}\left( \Gamma \right) }=\left( 2\pi
\right) ^{-n}U_{\Gamma ,\widehat{\chi },p,\mathtt{disc}}\left( \xi \right)
,\quad \xi \in \mathbb{R}^{n}
\end{equation*}
\end{corollary}

\begin{corollary}
\label{st8}Let $1\leq p\leq \infty $ and $u\in \mathcal{S}^{\prime }\left( 
\mathbb{R}^{n}\right) $. Then the following statements are equivalent:

$\left( \mathtt{a}\right) $ $u\in S_{w}^{p}\left( \mathbb{R}^{n}\right) $;

$\left( \mathtt{b}\right) $ There is $\chi \in \mathcal{S}\left( \mathbb{R}%
^{n}\right) \smallsetminus 0$ so that $\xi \rightarrow \left\Vert \chi
\left( D-\xi \right) u\right\Vert _{L^{p}}$ is in $L^{1}\left( \mathbb{R}%
^{n}\right) $;

$\left( \mathtt{c}\right) $ There are $\chi \in \mathcal{S}\left( \mathbb{R}%
^{n}\right) \smallsetminus 0$ and $\Gamma \subset \mathbb{R}^{n}$ a lattice
such that%
\begin{equation*}
\Phi _{\Gamma ,\widehat{\chi }}=\sum_{\gamma \in \Gamma }\left\vert \tau
_{\gamma }\widehat{\chi }\right\vert >0\left( \Leftrightarrow \Psi _{\Gamma ,%
\widehat{\chi }}=\sum_{\gamma \in \Gamma }\left\vert \tau _{\gamma }\widehat{%
\chi }\right\vert ^{2}>0\right)
\end{equation*}%
and $\xi \rightarrow \left\Vert \left( \chi \left( D-\xi \right) u\left(
\gamma \right) \right) _{\gamma \in \Gamma }\right\Vert _{l^{p}\left( \Gamma
\right) }$ is in $L^{1}\left( \mathbb{R}^{n}\right) $.
\end{corollary}

\begin{corollary}
Let $1\leq p\leq \infty $ and $\chi \in \mathcal{S}\left( \mathbb{R}%
^{n}\right) \smallsetminus 0$. Then 
\begin{equation*}
S_{w}^{p}\left( \mathbb{R}^{n}\right) \ni u\rightarrow \int \left\Vert \chi
\left( D-\xi \right) u\right\Vert _{L^{p}}\mathtt{d}\xi \equiv \left\vert
\left\vert \left\vert u\right\vert \right\vert \right\vert _{S_{w}^{p},\chi }
\end{equation*}%
is a norm on $S_{w}^{p}$. The topology defined by this norm coincides with
the topology of $S_{w}^{p}$.
\end{corollary}

\begin{lemma}
\label{st24}Let $1\leq p\leq \infty $, $\chi \in \mathcal{S}\left( \mathbb{R}%
^{n}\right) $ and $v\in L^{p}\left( \mathbb{R}^{n}\right) $. Then $\chi
\left( D\right) v\in L^{p}\left( \mathbb{R}^{n}\right) $ and 
\begin{equation*}
\left\Vert \chi \left( D\right) v\right\Vert _{L^{p}}\leq \left\Vert 
\mathcal{F}^{-1}\chi \right\Vert _{L^{1}}\left\Vert v\right\Vert _{L^{p}}.
\end{equation*}
\end{lemma}

\begin{proof}
We have 
\begin{eqnarray*}
\chi \left( D\right) v &=&\mathcal{F}^{-1}\left( \chi \cdot \widehat{v}%
\right) =\left( 2\pi \right) ^{-n}\mathcal{F}\left( \check{\chi}\cdot 
\widehat{\check{v}}\right) \\
&=&\left( 2\pi \right) ^{-n}\left( 2\pi \right) ^{-n}\left( \mathcal{F}%
\check{\chi}\right) \ast \left( \mathcal{F}\widehat{\check{v}}\right) \\
&=&\mathcal{F}^{-1}\chi \ast v.
\end{eqnarray*}%
Now Schur's lemma implies the result.
\end{proof}

The corollary \ref{st8} enables us to give a spectral characterization of
the ideals $S_{w}^{p}$ which is similar to the well known characterizations
of functional spaces by means of dyadic techniques, the annuli being
replaced here by cubes.

Let $\Gamma \subset \mathbb{R}^{n}$ be a lattice. Let $\chi \in \mathcal{S}%
\left( 
\mathbb{R}
^{n}\right) $. Then $\Phi _{\Gamma ,\chi }=\sum_{\gamma \in \Gamma
}\left\vert \tau _{\gamma }\chi \right\vert \geq 0$, $\Phi _{\Gamma ,\chi
}\in \mathcal{BC}\left( \mathbb{R}^{n}\right) $ and we have 
\begin{eqnarray*}
\Phi _{\Gamma ,\chi } &\leq &p_{n+1}\left( \chi \right) \sum_{\gamma \in
\Gamma }\tau _{\gamma }\left\langle \cdot \right\rangle
^{-n-1}=p_{n+1}\left( \chi \right) f_{n+1} \\
&\leq &p_{n+1}\left( \chi \right) \frac{2^{n+1}\left( \sup_{z\in \overline{%
\mathtt{C}}}\left\langle z\right\rangle ^{n+1}\right) ^{2}}{\mathtt{vol}%
\left( \mathtt{C}\right) }\int \left\langle y\right\rangle ^{-n-1}\mathtt{d}%
y=p_{n+1}\left( \chi \right) C_{\Gamma ,n}
\end{eqnarray*}%
where%
\begin{equation*}
p_{n+1}\left( \chi \right) =\sup\nolimits_{x\in 
\mathbb{R}
^{n}}\left\langle x\right\rangle ^{n+1}\left\vert \chi \left( x\right)
\right\vert .
\end{equation*}%
We now state the spectral characterization theorem which is the core of this
paper.

\begin{theorem}
\label{st27}Let $\Gamma \subset \mathbb{R}^{n}$ be a lattice, $1\leq p\leq
\infty $ and $u\in \mathcal{S}^{\prime }\left( \mathbb{R}^{n}\right) $. Then
the following statements are equivalent:

$\left( \mathtt{a}\right) $ $u\in S_{w}^{p}\left( \mathbb{R}^{n}\right) $;

$\left( \mathtt{b}\right) $ There are $Q=Q_{\Gamma }\subset \mathbb{R}^{n}$
a compact subset $($which can be chosen the same for all elements of $%
S_{w}^{p}\left( \mathbb{R}^{n}\right) )$ and $\left( u_{\gamma }\right)
_{\gamma \in \Gamma }\subset L^{p}\left( \mathbb{R}^{n}\right) $ such that%
\begin{gather*}
\mathtt{supp}\left( \widehat{u}_{\gamma }\right) \subset \gamma +Q,\quad
\gamma \in \Gamma , \\
\sum_{\gamma \in \Gamma }\left\Vert u_{\gamma }\right\Vert _{L^{p}}<\infty ,
\\
u=\sum_{\gamma \in \Gamma }u_{\gamma }.
\end{gather*}%
Moreover%
\begin{equation*}
\left\Vert u\right\Vert _{S_{w}^{p}}\approx \sum_{\gamma \in \Gamma
}\left\Vert u_{\gamma }\right\Vert _{L^{p}}
\end{equation*}
\end{theorem}

\begin{proof}
$\left( \mathtt{a}\right) \Rightarrow \left( \mathtt{b}\right) $ Let $\chi
_{\Gamma }\in \mathcal{C}_{0}^{\infty }\left( \mathbb{R}^{n}\right) $ and $%
R_{\Gamma }>0$ be such that $\mathtt{supp}\chi _{\Gamma }\subset B\left(
0;R_{\Gamma }\right) $ and $\sum_{\gamma \in \Gamma }\tau _{\gamma }\chi
_{\Gamma }=1$. For $\gamma \in \Gamma $ we set $u_{\gamma }=\chi _{\Gamma
}\left( D-\gamma \right) u$.

Let $\psi ,$ $\chi \in \mathcal{C}_{0}^{\infty }\left( \mathbb{R}^{n}\right) 
$ be such that $\mathtt{supp}\chi \subset B\left( 0;1\right) $, $\int \chi
\left( \xi \right) \mathtt{d}\xi =1$ and $\psi =1$ on $B\left( 0;R_{\Gamma
}+1\right) $. Then $\psi \ast \chi =1$ on $\mathtt{supp}\chi _{\Gamma }$
i.e. $\chi _{\Gamma }\cdot \left( \psi \ast \chi \right) =\chi _{\Gamma }$.
It follows that 
\begin{equation*}
\left( \tau _{\gamma }\chi _{\Gamma }\right) \left( \left( \tau _{\gamma
}\psi \right) \ast \chi \right) =\left( \tau _{\gamma }\chi _{\Gamma
}\right) \left( \tau _{\gamma }\left( \psi \ast \chi _{1}\right) \right)
=\tau _{\gamma }\left( \chi _{\Gamma }\cdot \left( \psi \ast \chi \right)
\right) =\tau _{\gamma }\chi _{\Gamma }.
\end{equation*}%
Let $\gamma \in \Gamma $. Using corollary \ref{st9} the identity (\ref{sjo3}%
) we get%
\begin{eqnarray*}
\chi _{\Gamma }\left( D-\gamma \right) u &=&\chi _{\Gamma }\left( D-\gamma
\right) \left( \left( \tau _{\gamma }\psi \right) \ast \chi \right) \left(
D\right) u \\
&=&\chi _{\Gamma }\left( D-\gamma \right) \dint \psi \left( \xi -\gamma
\right) \chi \left( D-\xi \right) u\mathtt{d}\xi \\
&=&\dint \psi \left( \xi -\gamma \right) \chi _{\Gamma }\left( D-\gamma
\right) \chi \left( D-\xi \right) u\mathtt{d}\xi
\end{eqnarray*}%
Now we apply the integral version of the Minkowski's inequality and the
previous lemma. We obtain%
\begin{eqnarray*}
\left\Vert u_{\gamma }\right\Vert _{L^{p}} &=&\left\Vert \chi _{\Gamma
}\left( D-\gamma \right) u\right\Vert _{L^{p}} \\
&\leq &\dint \left\vert \psi \left( \xi -\gamma \right) \right\vert
\left\Vert \chi _{\Gamma }\left( D-\gamma \right) \chi \left( D-\xi \right)
u\right\Vert _{L^{p}}\mathtt{d}\xi \\
&\leq &\left\Vert \mathcal{F}^{-1}\chi _{\Gamma }\right\Vert _{L^{1}}\dint
\left\vert \psi \left( \xi -\gamma \right) \right\vert \left\Vert \chi
\left( D-\xi \right) u\right\Vert _{L^{p}}\mathtt{d}\xi .
\end{eqnarray*}%
By summing with respect to $\gamma \in \Gamma $ we get%
\begin{eqnarray*}
\sum_{\gamma \in \Gamma }\left\Vert u_{\gamma }\right\Vert _{L^{p}} &\leq
&\left\Vert \mathcal{F}^{-1}\chi _{\Gamma }\right\Vert _{L^{1}}\dint \left(
\sum_{\gamma \in \Gamma }\left\vert \psi \left( \xi -\gamma \right)
\right\vert \right) \left\Vert \chi \left( D-\xi \right) u\right\Vert
_{L^{p}}\mathtt{d}\xi \\
&\leq &\left\Vert \mathcal{F}^{-1}\chi _{\Gamma }\right\Vert
_{L^{1}}\sup_{\xi }\left( \sum_{\gamma \in \Gamma }\left\vert \psi \left(
\xi -\gamma \right) \right\vert \right) \dint \left\Vert \chi \left( D-\xi
\right) u\right\Vert _{L^{p}}\mathtt{d}\xi \\
&\leq &\left\Vert \mathcal{F}^{-1}\chi _{\Gamma }\right\Vert
_{L^{1}}\sup_{\xi }\Phi _{\Gamma ,\psi }\left\vert \left\vert \left\vert
u\right\vert \right\vert \right\vert _{S_{w}^{p},\chi } \\
&\leq &C_{\Gamma ,n}\left\Vert \mathcal{F}^{-1}\chi _{\Gamma }\right\Vert
_{L^{1}}p_{n+1}\left( \psi \right) \left\vert \left\vert \left\vert
u\right\vert \right\vert \right\vert _{S_{w}^{p},\chi },
\end{eqnarray*}%
where 
\begin{gather*}
C_{\Gamma ,n}=\frac{2^{n+1}\left( \sup_{z\in \overline{\mathtt{C}}%
}\left\langle z\right\rangle ^{n+1}\right) ^{2}}{\mathtt{vol}\left( \mathtt{C%
}\right) }\int \left\langle y\right\rangle ^{-n-1}\mathtt{d}y, \\
p_{n+1}\left( \psi \right) =\sup_{\xi }\left\langle \xi \right\rangle
^{n+1}\left\vert \psi \left( \xi \right) \right\vert .
\end{gather*}

$\left( \mathtt{b}\right) \Rightarrow \left( \mathtt{a}\right) $ Let $\chi
\in \mathcal{C}_{0}^{\infty }\left( \mathbb{R}^{n}\right) $. Let $V$, $%
Q\subset V=\overset{\ \circ }{V}\Subset \mathbb{R}^{n}$ and let $\varphi \in 
\mathcal{C}_{0}^{\infty }\left( \mathbb{R}^{n}\right) $ be such that $%
\varphi =1$ on $\mathtt{supp}\chi -\overline{V}$. Then 
\begin{eqnarray*}
\varphi \left( \gamma -\xi \right) \chi \left( \eta -\xi \right) \chi
_{V}\left( \eta -\gamma \right) &=&\chi \left( \eta -\xi \right) \chi
_{V}\left( \eta -\gamma \right) \\
\varphi \left( \gamma -\xi \right) \chi \left( \eta -\xi \right) \chi
_{\gamma +V}\left( \eta \right) &=&\chi \left( \eta -\xi \right) \chi
_{\gamma +V}\left( \eta \right) \\
\varphi \left( \gamma -\xi \right) \chi \left( \eta -\xi \right) \widehat{u}%
_{\gamma }\left( \eta \right) &=&\chi \left( \eta -\xi \right) \widehat{u}%
_{\gamma }\left( \eta \right) \\
\varphi \left( \gamma -\xi \right) \chi \left( D-\xi \right) u_{\gamma }
&=&\chi \left( D-\xi \right) u_{\gamma }
\end{eqnarray*}%
Hence 
\begin{equation*}
\chi \left( D-\xi \right) u=\sum_{\gamma \in \Gamma }\chi \left( D-\xi
\right) u_{\gamma }=\sum_{\gamma \in \Gamma }\varphi \left( \gamma -\xi
\right) \chi \left( D-\xi \right) u_{\gamma }.
\end{equation*}%
By using the previous lemma we get:%
\begin{equation*}
\left\Vert \chi \left( D-\xi \right) u\right\Vert _{L^{p}}\leq \sum_{\gamma
\in \Gamma }\left\vert \varphi \left( \gamma -\xi \right) \right\vert
\left\Vert \mathcal{F}^{-1}\chi \right\Vert _{L^{1}}\left\Vert u_{\gamma
}\right\Vert _{L^{p}}
\end{equation*}%
which by integration gives us%
\begin{eqnarray*}
\left\vert \left\vert \left\vert u\right\vert \right\vert \right\vert
_{S_{w}^{p},\chi } &=&\dint \left\Vert \chi \left( D-\xi \right)
u\right\Vert _{L^{p}}\mathtt{d}\xi \\
&\leq &\left\Vert \mathcal{F}^{-1}\chi \right\Vert _{L^{1}}\dint
\sum_{\gamma \in \Gamma }\left\vert \varphi \left( \gamma -\xi \right)
\right\vert \left\Vert u_{\gamma }\right\Vert _{L^{p}}\mathtt{d}\xi \\
&=&\left\Vert \mathcal{F}^{-1}\chi \right\Vert _{L^{1}}\sum_{\gamma \in
\Gamma }\left( \dint \left\vert \varphi \left( \gamma -\xi \right)
\right\vert \mathtt{d}\xi \right) \left\Vert u_{\gamma }\right\Vert _{L^{p}}
\\
&=&\left\Vert \mathcal{F}^{-1}\chi \right\Vert _{L^{1}}\left\Vert \varphi
\right\Vert _{L^{1}}\sum_{\gamma \in \Gamma }\left\Vert u_{\gamma
}\right\Vert _{L^{p}}.
\end{eqnarray*}

For $\varphi ,$ $\psi ,$ $\chi ,$ $\chi _{\Gamma }\in \mathcal{C}%
_{0}^{\infty }\left( \mathbb{R}^{n}\right) $ as above we have 
\begin{equation*}
\frac{1}{\left\Vert \mathcal{F}^{-1}\chi \right\Vert _{L^{1}}\left\Vert
\varphi \right\Vert _{L^{1}}}\left\vert \left\vert \left\vert u\right\vert
\right\vert \right\vert _{S_{w}^{p},\chi }\leq \sum_{\gamma \in \Gamma
}\left\Vert u_{\gamma }\right\Vert _{L^{p}}\leq C_{\Gamma ,n}\left\Vert 
\mathcal{F}^{-1}\chi _{\Gamma }\right\Vert _{L^{1}}p_{n+1}\left( \psi
\right) \left\vert \left\vert \left\vert u\right\vert \right\vert
\right\vert _{S_{w}^{p},\chi },
\end{equation*}%
with%
\begin{gather*}
C_{\Gamma ,n}=\frac{2^{n+1}\left( \sup_{z\in \overline{\mathtt{C}}%
}\left\langle z\right\rangle ^{n+1}\right) ^{2}}{\mathtt{vol}\left( \mathtt{C%
}\right) }\int \left\langle y\right\rangle ^{-n-1}\mathtt{d}y, \\
p_{n+1}\left( \psi \right) =\sup_{\xi }\left\langle \xi \right\rangle
^{n+1}\left\vert \psi \left( \xi \right) \right\vert .
\end{gather*}
\end{proof}

\begin{corollary}
\label{st14}Let $\Gamma \subset \mathbb{R}^{n}$ be a lattice, $1\leq p\leq
\infty $ and $u\in S_{w}^{p}\left( \mathbb{R}^{n}\right) $.

$\left( \mathtt{a}\right) $ Then there are $Q=Q_{\Gamma }\subset \mathbb{R}%
^{n}$ a compact subset $($which can be chosen the same for all elements of $%
S_{w}^{p}\left( \mathbb{R}^{n}\right) )$ and $\left( u_{\gamma }\right)
_{\gamma \in \Gamma }\subset \bigcap_{q\in \left[ p,\infty \right]
}L^{q}\left( \mathbb{R}^{n}\right) $ such that%
\begin{equation*}
u=\sum_{\gamma \in \Gamma }u_{\gamma },\quad \mathtt{supp}\left( \widehat{u}%
_{\gamma }\right) \subset \gamma +Q,\quad \gamma \in \Gamma
\end{equation*}%
and for any $q\in \left[ p,\infty \right] $ we have%
\begin{equation*}
\sum_{\gamma \in \Gamma }\left\Vert u_{\gamma }\right\Vert _{L^{q}}<\infty
,\quad \left\Vert u\right\Vert _{S_{w}^{q}}\approx \sum_{\gamma \in \Gamma
}\left\Vert u_{\gamma }\right\Vert _{L^{q}}
\end{equation*}

$\left( \mathtt{b}\right) $ $\sum_{\gamma \in \Gamma }u_{\gamma }$ converges
to $u$ in $L^{q}\left( \mathbb{R}^{n}\right) $ for any $q\in \left[ p,\infty %
\right] $.

$\left( \mathtt{c}\right) $ $\sum_{\gamma \in \Gamma }u_{\gamma }$ converges
to $u$ in $\mathcal{BC}\left( \mathbb{R}^{n}\right) $.
\end{corollary}

\begin{proof}
These statements are obvious if we take into account the inclusions%
\begin{equation*}
1\leq p\leq q\leq \infty \Rightarrow S_{w}^{1}\left( \mathbb{R}^{n}\right)
\subset S_{w}^{p}\left( \mathbb{R}^{n}\right) \subset S_{w}^{q}\left( 
\mathbb{R}^{n}\right) \subset S_{w}^{\infty }\left( \mathbb{R}^{n}\right)
\subset \mathcal{BC}\left( \mathbb{R}^{n}\right)
\end{equation*}%
and the definition of the family $\left( u_{\gamma }\right) _{\gamma \in
\Gamma }$, $u_{\gamma }=\chi _{\Gamma }\left( D-\gamma \right) u$, $\gamma
\in \Gamma $.
\end{proof}

\begin{corollary}
$S_{w}^{p}\left( \mathbb{R}^{n}\right) \subset \bigcap_{q\in \left[ p,\infty %
\right] }L^{q}\left( \mathbb{R}^{n}\right) \cap \mathcal{BC}\left( \mathbb{R}%
^{n}\right) $.
\end{corollary}

\begin{proposition}
\label{st22}If $1\leq p<\infty $, then $\mathcal{S}\left( \mathbb{R}%
^{n}\right) $ is dense in $S_{w}^{p}\left( \mathbb{R}^{n}\right) $.
\end{proposition}

\begin{proof}
We showed that there is a compact subset $Q\subset \mathbb{R}^{n}$,
depending on the lattice $%
\mathbb{Z}
^{n}$, with the property that for any $u\in S_{w}^{p}\left( \mathbb{R}%
^{n}\right) $ there is $\left( u_{k}\right) _{k\in 
\mathbb{Z}
^{n}}\subset L^{p}\left( \mathbb{R}^{n}\right) $ such that%
\begin{gather*}
\mathtt{supp}\left( \widehat{u}_{k}\right) \subset k+Q,\quad k\in 
\mathbb{Z}
^{n}, \\
\sum_{k\in 
\mathbb{Z}
^{n}}\left\Vert u_{k}\right\Vert _{L^{p}}<\infty ,u=\sum_{k\in 
\mathbb{Z}
^{n}}u_{k}, \\
\left\Vert u\right\Vert _{S_{w}^{p}}\approx \sum_{k\in 
\mathbb{Z}
^{n}}\left\Vert u_{k}\right\Vert _{L^{p}}.
\end{gather*}

Let $k_{0}\in 
\mathbb{Z}
^{n}$. Then 
\begin{equation*}
u_{k_{0}}=\sum_{k\in 
\mathbb{Z}
^{n}}\delta _{k_{0}k}u_{k}
\end{equation*}%
so $u_{k_{0}}\in S_{w}^{p}\left( \mathbb{R}^{n}\right) $ and$\left\Vert
u_{k_{0}}\right\Vert _{S_{w}^{p}}\approx \left\Vert u_{k_{0}}\right\Vert
_{L^{p}}$. Therefore, the series $\sum_{k\in 
\mathbb{Z}
^{n}}u_{k}$ converges to $u$ in $S_{w}^{p}\left( \mathbb{R}^{n}\right) $. In
view of this convergence, it is sufficient to approximate $u_{k}$ by
elements of $\mathcal{S}\left( \mathbb{R}^{n}\right) $ for any $k\in 
\mathbb{Z}
^{n}$.

Let$\varphi \in \mathcal{C}_{0}^{\infty }\left( \mathbb{R}^{n}\right) $ be
such that $\mathtt{supp}\varphi \subset B\left( 0;1\right) $, $\int \varphi
\left( \xi \right) \mathtt{d}\xi =\left( 2\pi \right) ^{n}$. For $%
\varepsilon \in \left( 0,1\right] $, set $\varphi _{\varepsilon
}=\varepsilon ^{-n}\varphi \left( \cdot /\varepsilon \right) $, $\psi =%
\mathcal{F}^{-1}\varphi $, $\psi ^{\varepsilon }=\mathcal{F}^{-1}\varphi
_{\varepsilon }=\psi \left( \varepsilon \cdot \right) $. Since 
\begin{eqnarray*}
\widehat{\psi ^{\varepsilon }u_{k}} &=&\left( 2\pi \right) ^{-n}\widehat{%
\psi ^{\varepsilon }}\ast \widehat{u}_{k}=\left( 2\pi \right) ^{-n}\varphi
_{\varepsilon }\ast \widehat{u}_{k}\in \mathcal{C}_{0}^{\infty }\left( 
\mathbb{R}^{n}\right) , \\
\mathtt{supp}\left( \widehat{\psi ^{\varepsilon }u_{k}}\right) &\subset &%
\mathtt{supp}\varphi _{\varepsilon }+\mathtt{supp}\left( \widehat{u}%
_{k}\right) \subset \varepsilon \overline{B\left( 0;1\right) }+Q\subset 
\overline{B\left( 0;1\right) }+Q.
\end{eqnarray*}%
it follows that $\psi ^{\varepsilon }u_{k}\in \mathcal{S}\left( \mathbb{R}%
^{n}\right) $. Then using the spectral characterization theorem, Lebesgue's
dominated convergence theorem and $\psi \left( 0\right) =1$ we obtain that%
\begin{equation*}
\left\Vert u_{k}-\psi ^{\varepsilon }u_{k}\right\Vert _{S_{w}^{p}}\leq
Cst\left\Vert u_{k}-\psi ^{\varepsilon }u_{k}\right\Vert _{L^{p}}\rightarrow
0,\quad \text{\textit{as} }\varepsilon \rightarrow 0.
\end{equation*}
\end{proof}

The last result is not true if $p=\infty $. Suppose that $\mathcal{S}\left( 
\mathbb{R}^{n}\right) $ is dense in $S_{w}^{\infty }\left( \mathbb{R}%
^{n}\right) $. Then there is $\psi \in \mathcal{S}\left( \mathbb{R}%
^{n}\right) $ so that 
\begin{equation*}
\left\Vert 1-\psi \right\Vert _{\infty }\leq C\left\Vert 1-\psi \right\Vert
_{S_{w}^{\infty }\left( \mathbb{R}^{n}\right) }<\frac{1}{2}.
\end{equation*}%
If $\varphi \in \mathcal{C}_{0}^{\infty }\left( \mathbb{R}^{n}\right) $
satisfies 
\begin{equation*}
\left\Vert \psi -\varphi \psi \right\Vert _{\infty }<\frac{1}{2},
\end{equation*}%
then 
\begin{equation*}
1\leq \left\Vert 1-\varphi \psi \right\Vert _{\infty }\leq \left\Vert 1-\psi
\right\Vert _{\infty }+\left\Vert \psi -\varphi \psi \right\Vert _{\infty }<%
\frac{1}{2}+\frac{1}{2}=1
\end{equation*}%
contradiction!

In the case of Sj\"{o}strand algebra we shall prove that $\mathcal{BC}%
^{\infty }\left( \mathbb{R}^{n}\right) $ is dense in $S_{w}\left( \mathbb{R}%
^{n}\right) $. The next result completes lemma \ref{st24}.

\begin{lemma}
If $\chi \in \mathcal{S}\left( \mathbb{R}^{n}\right) $ and $u\in L^{\infty
}\left( \mathbb{R}^{n}\right) $, then $\chi \left( D\right) u\in \mathcal{BC}%
^{\infty }\left( \mathbb{R}^{n}\right) $ and 
\begin{equation*}
\left\Vert D^{\alpha }\chi \left( D\right) u\right\Vert _{L^{\infty }}\leq
\left\Vert D^{\alpha }\mathcal{F}^{-1}\chi \right\Vert _{L^{1}}\left\Vert
u\right\Vert _{L^{\infty }}.
\end{equation*}
\end{lemma}

\begin{proof}
Since $D^{\alpha }\chi \left( D\right) u=\chi _{\alpha }\left( D\right) u$
with $\chi _{\alpha }=\xi ^{\alpha }\chi \in \mathcal{S}\left( \mathbb{R}%
^{n}\right) $, it is sufficient to show that $\chi \left( D\right) u\in 
\mathcal{BC}\left( \mathbb{R}^{n}\right) $ and 
\begin{equation*}
\left\Vert \chi \left( D\right) u\right\Vert _{L^{\infty }}\leq \left\Vert 
\mathcal{F}^{-1}\chi \right\Vert _{L^{1}}\left\Vert u\right\Vert _{L^{\infty
}}.
\end{equation*}%
In this case we have 
\begin{equation*}
\chi \left( D\right) u=\mathcal{F}^{-1}\chi \ast u.
\end{equation*}%
This representation implies 
\begin{eqnarray*}
\left\vert \chi \left( D\right) u\left( x\right) -\chi \left( D\right)
u\left( y\right) \right\vert &\leq &\left\Vert \mathcal{F}^{-1}\chi \left(
x-\cdot \right) -\mathcal{F}^{-1}\chi \left( y-\cdot \right) \right\Vert
_{L^{1}}\left\Vert u\right\Vert _{L^{\infty }}, \\
\left\Vert \chi \left( D\right) u\right\Vert _{L^{\infty }} &\leq
&\left\Vert \mathcal{F}^{-1}\chi \right\Vert _{L^{1}}\left\Vert u\right\Vert
_{L^{\infty }}.
\end{eqnarray*}%
Hence $\chi \left( D\right) u\in \mathcal{BC}^{\infty }\left( \mathbb{R}%
^{n}\right) $ and 
\begin{equation*}
\left\Vert D^{\alpha }\chi \left( D\right) u\right\Vert _{L^{\infty }}\leq
\left\Vert D^{\alpha }\mathcal{F}^{-1}\chi \right\Vert _{L^{1}}\left\Vert
u\right\Vert _{L^{\infty }}.
\end{equation*}
\end{proof}

\begin{lemma}
$(\mathtt{a})$ $\mathcal{BC}^{n+1}\left( \mathbb{R}^{n}\right) \subset
S_{w}\left( \mathbb{R}^{n}\right) .$

$(\mathtt{b})$ $\mathcal{BC}^{\infty }\left( \mathbb{R}^{n}\right) $ is
dense in $S_{w}\left( \mathbb{R}^{n}\right) $.
\end{lemma}

\begin{proof}
$(\mathtt{a})$ If $n$ is odd, then $n+1$ is even and $\left\langle \xi
\right\rangle ^{n+1}$ is a polynomial in $\xi $ of degree $n+1$. Then 
\begin{equation*}
\widehat{u\tau _{y}\chi }\left( \xi \right) =\left\langle \xi \right\rangle
^{-n-1}\mathcal{F}\left( \left\langle D\right\rangle ^{n+1}\left( u\tau
_{y}\chi \right) \right) \left( \xi \right) ,\quad \xi \in \mathbb{R}^{n}.
\end{equation*}%
It follows that there is a continuous seminorm $p$ on $\mathcal{S}\left( 
\mathbb{R}^{n}\right) $ so that%
\begin{eqnarray*}
\sup_{y\in \mathbb{R}^{n}}\left\vert \widehat{u\tau _{y}\chi }\left( \xi
\right) \right\vert &\leq &p\left( \chi \right) \left\Vert u\right\Vert _{%
\mathcal{BC}^{n+1}}\left\langle \xi \right\rangle ^{-n-1},\quad \xi \in 
\mathbb{R}^{n} \\
\text{\textit{and}\quad }\xi &\rightarrow &\left\langle \xi \right\rangle
^{-n-1}\quad \text{\textit{belongs to} }L^{1}\left( \mathbb{R}^{n}\right) .
\end{eqnarray*}

If $n$ is even, then $\left\langle \xi \right\rangle ^{n}$ is polynomial in $%
\xi $ of degree $n$. For $1\leq j\leq n$ we denote by $\Gamma _{j}$ the
closed cone $\left\{ \xi :\sqrt{n}\left\vert \xi _{j}\right\vert \geq
\left\vert \xi \right\vert \right\} $. Now we make use of a simple identity%
\begin{equation*}
\widehat{u\tau _{y}\chi }\left( \xi \right) =\left\langle \xi \right\rangle
^{-n}\left( \xi _{j}+\mathtt{i}\right) ^{-1}\mathcal{F}\left( \left( D_{j}+%
\mathtt{i}\right) \left\langle D\right\rangle ^{n}\left( u\tau _{y}\chi
\right) \right) \left( \xi \right) .
\end{equation*}%
It follows that there is a continuous seminorm $p_{j}$ on $\mathcal{S}\left( 
\mathbb{R}^{n}\right) $ so that%
\begin{equation*}
\sup_{y\in \mathbb{R}^{n}}\left\vert \widehat{u\tau _{y}\chi }\left( \xi
\right) \right\vert \leq p_{j}\left( \chi \right) \left\Vert u\right\Vert _{%
\mathcal{BC}^{n+1}}\left\langle \xi \right\rangle ^{-n-1},\quad \xi \in
\Gamma _{j}.
\end{equation*}%
Since $\mathbb{R}^{n}=\bigcup_{j=1}^{n}\Gamma _{j}$ it follows that there is
a continuous seminorm $p$ on $\mathcal{S}\left( \mathbb{R}^{n}\right) $ such
that%
\begin{equation*}
\sup_{y\in \mathbb{R}^{n}}\left\vert \widehat{u\tau _{y}\chi }\left( \xi
\right) \right\vert \leq p\left( \chi \right) \left\Vert u\right\Vert _{%
\mathcal{BC}^{n+1}}\left\langle \xi \right\rangle ^{-n-1},\quad \xi \in 
\mathbb{R}^{n}
\end{equation*}%
with $\xi \rightarrow \left\langle \xi \right\rangle ^{-n-1}$ in $%
L^{1}\left( \mathbb{R}^{n}\right) .$

$(\mathtt{b})$ We proceed as in the proof of proposition \ref{st22}. The
spectral characterization theorem implies that there is a compact subset $%
Q\subset \mathbb{R}^{n}$, depending on the lattice $%
\mathbb{Z}
^{n}$, with the property that for any $u\in S_{w}^{p}\left( \mathbb{R}%
^{n}\right) $ there is $\left( u_{k}\right) _{k\in 
\mathbb{Z}
^{n}}\subset L^{p}\left( \mathbb{R}^{n}\right) $ such that%
\begin{gather*}
\mathtt{supp}\left( \widehat{u}_{k}\right) \subset k+Q,\quad k\in 
\mathbb{Z}
^{n}, \\
\sum_{k\in 
\mathbb{Z}
^{n}}\left\Vert u_{k}\right\Vert _{L^{\infty }}<\infty ,u=\sum_{k\in 
\mathbb{Z}
^{n}}u_{k}, \\
\left\Vert u\right\Vert _{S_{w}}\approx \sum_{k\in 
\mathbb{Z}
^{n}}\left\Vert u_{k}\right\Vert _{L^{\infty }}.
\end{gather*}

Let $k_{0}\in 
\mathbb{Z}
^{n}$. Then 
\begin{equation*}
u_{k_{0}}=\sum_{k\in 
\mathbb{Z}
^{n}}\delta _{k_{0}k}u_{k}
\end{equation*}%
so $u_{k_{0}}\in S_{w}\left( \mathbb{R}^{n}\right) $ and $\left\Vert
u_{k_{0}}\right\Vert _{S_{w}}\approx \left\Vert u_{k_{0}}\right\Vert
_{L^{\infty }}$. It follows that $\sum_{k\in 
\mathbb{Z}
^{n}}u_{k}$ converges to $u$ in $S_{w}\left( \mathbb{R}^{n}\right) $. Let $%
\chi \in \mathcal{C}_{0}^{\infty }\left( \mathbb{R}^{n}\right) $ be such
that $\chi =1$ on $Q$. Set $\chi _{k}=\tau _{k}\chi $. Then previous lemma
shows that $u_{k}=\chi _{k}\left( D\right) u_{k}\in \mathcal{BC}^{\infty
}\left( \mathbb{R}^{n}\right) $, so for any finite subset $F\subset 
\mathbb{Z}
^{n}$, $\sum_{k\in F}u_{k}\in \mathcal{BC}^{\infty }\left( \mathbb{R}%
^{n}\right) $.
\end{proof}

\begin{corollary}
\label{st10}Let $1\leq p\leq \infty $ and $u\in \mathcal{S}^{\prime }\left( 
\mathbb{R}^{n}\right) $. Then the following statements are equivalent:

$\left( \mathtt{a}\right) $ $u\in S_{w}^{p}\left( \mathbb{R}^{n}\right) $;

$\left( \mathtt{b}\right) $ There are $Q\subset \mathbb{R}^{n}$ a compact
subset $($which can be chosen the same for all elements of $S_{w}^{p}\left( 
\mathbb{R}^{n}\right) )$ and $\left( u_{k}\right) _{k\in 
\mathbb{Z}
^{n}}\subset L^{p}\left( \mathbb{R}^{n}\right) $ such that%
\begin{gather*}
\mathtt{supp}\left( \widehat{u}_{k}\right) \subset Q, \\
\sum_{k\in 
\mathbb{Z}
^{n}}\left\Vert u_{k}\right\Vert _{L^{p}}<\infty , \\
u=\sum_{k\in 
\mathbb{Z}
^{n}}\mathtt{e}^{\mathtt{i}\left\langle \cdot ,k\right\rangle }u_{k}.
\end{gather*}%
Moreover%
\begin{equation*}
\left\Vert u\right\Vert _{S_{w}^{p}}\approx \sum_{k\in 
\mathbb{Z}
^{n}}\left\Vert u_{k}\right\Vert _{L^{p}}
\end{equation*}
\end{corollary}

\begin{proof}
$\left( \mathtt{a}\right) \Rightarrow \left( \mathtt{b}\right) $ There are $%
Q\subset \mathbb{R}^{n}$ a compact subset $($which can be chosen the same
for all elements of $S_{w}^{p}\left( \mathbb{R}^{n}\right) )$ and $\left(
v_{k}\right) _{k\in 
\mathbb{Z}
^{n}}\subset L^{p}\left( \mathbb{R}^{n}\right) $ such that%
\begin{gather*}
\mathtt{supp}\left( \widehat{v}_{k}\right) \subset k+Q,\quad k\in 
\mathbb{Z}
^{n}, \\
\sum_{k\in 
\mathbb{Z}
^{n}}\left\Vert v_{k}\right\Vert _{L^{p}}<\infty , \\
u=\sum_{k\in 
\mathbb{Z}
^{n}}v_{k}.
\end{gather*}%
We set $u_{k}=\mathcal{F}^{-1}\left( \tau _{-k}\widehat{v}_{k}\right) $.
Then 
\begin{equation*}
\mathtt{supp}\left( \widehat{u}_{k}\right) \subset Q,\quad \widehat{v}%
_{k}=\tau _{k}\widehat{u}_{k}
\end{equation*}%
and 
\begin{equation*}
v_{k}\left( x\right) =\mathtt{e}^{\mathtt{i}\left\langle x,k\right\rangle
}u_{k}\left( x\right) ,\quad \left\Vert v_{k}\right\Vert _{L^{p}}=\left\Vert
u_{k}\right\Vert _{L^{p}}.
\end{equation*}

$\left( \mathtt{b}\right) \Rightarrow \left( \mathtt{a}\right) $ There are $%
Q\subset \mathbb{R}^{n}$ a compact subset $($which can be chosen the same
for all elements of $S_{w}^{p}\left( \mathbb{R}^{n}\right) )$ and $\left(
u_{k}\right) _{k\in 
\mathbb{Z}
^{n}}\subset L^{p}\left( \mathbb{R}^{n}\right) $ such that%
\begin{gather*}
\mathtt{supp}\left( \widehat{u}_{k}\right) \subset Q, \\
\sum_{k\in 
\mathbb{Z}
^{n}}\left\Vert u_{k}\right\Vert _{L^{p}}<\infty , \\
u=\sum_{k\in 
\mathbb{Z}
^{n}}\mathtt{e}^{\mathtt{i}\left\langle \cdot ,k\right\rangle }u_{k}.
\end{gather*}%
We set $v_{k}=\mathtt{e}^{\mathtt{i}\left\langle \cdot ,k\right\rangle
}u_{k} $, $k\in 
\mathbb{Z}
^{n}$. Then 
\begin{equation*}
\widehat{v}_{k}=\tau _{k}\widehat{u}_{k},\ \mathtt{supp}\left( \widehat{v}%
_{k}\right) \subset k+Q,\ \left\Vert v_{k}\right\Vert _{L^{p}}=\left\Vert
u_{k}\right\Vert _{L^{p}},\quad k\in 
\mathbb{Z}
^{n}
\end{equation*}
and $u=\sum_{k\in 
\mathbb{Z}
^{n}}v_{k}$.
\end{proof}

\begin{lemma}
\label{st11}Let $1\leq p\leq \infty $, $u\in S_{w}^{p}\left( \mathbb{R}%
^{n}\right) $ and $\left( u_{k}\right) _{k\in 
\mathbb{Z}
^{n}}\subset L^{p}\left( \mathbb{R}^{n}\right) $ be such that 
\begin{equation*}
\mathtt{supp}\left( \widehat{u}_{k}\right) \subset Q,\sum_{k\in 
\mathbb{Z}
^{n}}\left\Vert u_{k}\right\Vert _{L^{p}}<\infty ,u=\sum_{k\in 
\mathbb{Z}
^{n}}\mathtt{e}^{\mathtt{i}\left\langle \cdot ,k\right\rangle }u_{k}.
\end{equation*}%
Then $\left( u_{k}\right) _{k\in 
\mathbb{Z}
^{n}}\subset \mathcal{C}^{\infty }\left( \mathbb{R}^{n}\right) $ and for any 
$\alpha \in 
\mathbb{N}
^{n}$ there is $C_{\alpha }$ such that 
\begin{equation*}
\left\Vert \partial ^{\alpha }u_{k}\right\Vert _{L^{p}}\leq C_{\alpha
}\left\Vert u_{k}\right\Vert _{L^{p}},\quad k\in 
\mathbb{Z}
^{n}.
\end{equation*}
\end{lemma}

\begin{proof}
Let $\chi \in \mathcal{C}_{0}^{\infty }\left( \mathbb{R}^{n}\right) $ be
such that $\chi =1$ on a neighborhood of $Q$. Then $\widehat{u}_{k}=\chi 
\widehat{u}_{k}$, $u_{k}=\left( \mathcal{F}^{-1}\chi \right) \ast u_{k}$ for
any $k\in 
\mathbb{Z}
^{n}$. It follows 
\begin{eqnarray*}
\partial ^{\alpha }u_{k} &=&\left( \partial ^{\alpha }\left( \mathcal{F}%
^{-1}\chi \right) \right) \ast u_{k}, \\
\left\Vert \partial ^{\alpha }u_{k}\right\Vert _{L^{p}} &\leq &\left\Vert 
\mathcal{F}^{-1}\left( \xi ^{\alpha }\chi \right) \right\Vert
_{L^{1}}\left\Vert u_{k}\right\Vert _{L^{p}},\quad k\in 
\mathbb{Z}
^{n}.
\end{eqnarray*}
\end{proof}

The sequence $\left( u_{k}\right) _{k\in 
\mathbb{Z}
^{n}}$ introduced in the corollary \ref{st10} is defined as follows. We
choose $\chi _{%
\mathbb{Z}
^{n}}\in \mathcal{C}_{0}^{\infty }\left( \mathbb{R}^{n}\right) $ so that $%
\sum_{k\in 
\mathbb{Z}
^{n}}\chi _{%
\mathbb{Z}
^{n}}\left( \cdot -k\right) =1$ and for $k\in 
\mathbb{Z}
^{n}$ we put%
\begin{equation*}
u_{k}=\mathcal{F}^{-1}\left( \tau _{-k}\widehat{v}_{k}\right) ,\quad v_{k}=%
\mathcal{F}^{-1}\left( \chi _{%
\mathbb{Z}
^{n}}\left( \cdot -k\right) \widehat{u}\right) .
\end{equation*}%
Hence for $k\in 
\mathbb{Z}
^{n}$ we have 
\begin{eqnarray*}
u_{k} &=&\mathcal{F}^{-1}\left( \tau _{-k}\left( \left( \tau _{k}\chi _{%
\mathbb{Z}
^{n}}\right) \widehat{u}\right) \right) =\mathcal{F}^{-1}\left( \chi _{%
\mathbb{Z}
^{n}}\cdot \tau _{k}\widehat{u}\right) \\
&=&\mathcal{F}^{-1}\chi _{%
\mathbb{Z}
^{n}}\cdot \mathcal{F}\left( \mathtt{e}^{-\mathtt{i}\left\langle \cdot
,k\right\rangle }u\right) =\chi _{%
\mathbb{Z}
^{n}}\left( D\right) \left( \mathtt{e}^{-\mathtt{i}\left\langle \cdot
,k\right\rangle }u\right) \\
&=&S_{k}u
\end{eqnarray*}%
where $S_{k}=\chi _{%
\mathbb{Z}
^{n}}\left( D\right) \circ M_{\mathtt{e}^{-\mathtt{i}\left\langle \cdot
,k\right\rangle }}$, $k\in 
\mathbb{Z}
^{n}$.

It follows corollary \ref{st10} that the linear operator 
\begin{equation*}
S:S_{w}^{p}\left( \mathbb{R}^{n}\right) \rightarrow l^{1}\left( 
\mathbb{Z}
^{n},L^{p}\left( \mathbb{R}^{n}\right) \right) ,\quad Su=\left(
S_{k}u\right) _{k\in 
\mathbb{Z}
^{n}}
\end{equation*}%
is well defined and continuous. In addition we have that%
\begin{equation*}
u=\sum_{k\in 
\mathbb{Z}
^{n}}\mathtt{e}^{\mathtt{i}\left\langle \cdot ,k\right\rangle }S_{k}u.
\end{equation*}%
On the other hand, the same corollary implies that for any $\chi \in 
\mathcal{C}_{0}^{\infty }\left( \mathbb{R}^{n}\right) $ the operator 
\begin{gather*}
R_{\chi }:l^{1}\left( 
\mathbb{Z}
^{n},L^{p}\left( \mathbb{R}^{n}\right) \right) \rightarrow S_{w}^{p}\left( 
\mathbb{R}^{n}\right) , \\
R_{\chi }\left( \left( u_{k}\right) _{k\in 
\mathbb{Z}
^{n}}\right) =\sum_{k\in 
\mathbb{Z}
^{n}}\mathtt{e}^{\mathtt{i}\left\langle \cdot ,k\right\rangle }\chi \left(
D\right) u_{k}
\end{gather*}%
is well defined and continuous.

If $\chi =1$ on a neighborhood of $\mathtt{supp}\chi _{%
\mathbb{Z}
^{n}}$, then $\chi \chi _{%
\mathbb{Z}
^{n}}=\chi _{%
\mathbb{Z}
^{n}}$ and as a consequence $R_{\chi }S=\mathtt{Id}_{S_{w}^{p}\left( \mathbb{%
R}^{n}\right) }$:%
\begin{eqnarray*}
R_{\chi }Su &=&\sum_{k\in 
\mathbb{Z}
^{n}}\mathtt{e}^{\mathtt{i}\left\langle \cdot ,k\right\rangle }\chi \left(
D\right) S_{k}u=\sum_{k\in 
\mathbb{Z}
^{n}}\mathtt{e}^{\mathtt{i}\left\langle \cdot ,k\right\rangle }\chi \left(
D\right) \chi _{%
\mathbb{Z}
^{n}}\left( D\right) \left( \mathtt{e}^{-\mathtt{i}\left\langle \cdot
,k\right\rangle }u\right) \\
&=&\sum_{k\in 
\mathbb{Z}
^{n}}\mathtt{e}^{\mathtt{i}\left\langle \cdot ,k\right\rangle }\chi _{%
\mathbb{Z}
^{n}}\left( D\right) \left( \mathtt{e}^{-\mathtt{i}\left\langle \cdot
,k\right\rangle }u\right) =\sum_{k\in 
\mathbb{Z}
^{n}}\mathtt{e}^{\mathtt{i}\left\langle \cdot ,k\right\rangle }S_{k}u \\
&=&u.
\end{eqnarray*}%
Thus we proved the following result.

\begin{proposition}
Under the above conditions, the operator$R_{\chi }:l^{1}\left( 
\mathbb{Z}
^{n},L^{p}\left( \mathbb{R}^{n}\right) \right) \rightarrow S_{w}^{p}\left( 
\mathbb{R}^{n}\right) $ is a retract.
\end{proposition}

Using the results of \cite{Triebel} section 1.18 we obtain the following
corollary.

\begin{corollary}
For $0<\theta <1$ 
\begin{equation*}
S_{w}^{\frac{1}{1-\theta }}\left( \mathbb{R}^{n}\right) =\left[
S_{w}^{1}\left( \mathbb{R}^{n}\right) ,S_{w}^{\infty }\left( \mathbb{R}%
^{n}\right) \right] _{\theta }
\end{equation*}
\end{corollary}

We shall now use the spectral characterization theorem to prove some
properties of the ideals $S_{w}^{p}$.

Let $A={}^{t}A$ be a non-singular real matrix. Let $\Phi =\Phi _{A}$ be the
real non-singular quadratic form in $\mathbb{R}^{n}$ defined by $\Phi
_{A}\left( x\right) =-\left\langle Ax,x\right\rangle /2$, $x\in \mathbb{R}%
^{n}$. Then 
\begin{equation*}
\widehat{\mathtt{e}^{\mathtt{i}\Phi _{A}}}\left( \xi \right) =\left( 2\pi
\right) ^{n/2}\left\vert \det A\right\vert ^{-\frac{1}{2}}\mathtt{e}^{\frac{%
\pi \mathtt{isgn}A}{4}}\mathtt{e}^{-\mathtt{i}\left\langle A^{-1}\xi ,\xi
\right\rangle /2}
\end{equation*}%
This formula suggests the introduction of the operator $T_{A}:\mathcal{S}%
\left( \mathbb{R}^{n}\right) \rightarrow \mathcal{S}^{\prime }\left( \mathbb{%
R}^{n}\right) $ defined by%
\begin{equation*}
T_{A}u=\left( 2\pi \right) ^{-n/2}\left\vert \det A\right\vert ^{\frac{1}{2}}%
\mathtt{e}^{-\frac{\pi \mathtt{isgn}A}{4}}\mathtt{e}^{\mathtt{i}\Phi
_{A}}\ast u.
\end{equation*}%
Since 
\begin{eqnarray*}
\widehat{T_{A}u} &=&\left( 2\pi \right) ^{-n/2}\left\vert \det A\right\vert
^{\frac{1}{2}}\mathtt{e}^{-\frac{\pi \mathtt{isgn}A}{4}}\widehat{\mathtt{e}^{%
\mathtt{i}\Phi _{A}}\ast u} \\
&=&\left( 2\pi \right) ^{-n/2}\left\vert \det A\right\vert ^{\frac{1}{2}}%
\mathtt{e}^{-\frac{\pi \mathtt{isgn}A}{4}}\widehat{\mathtt{e}^{\mathtt{i}%
\Phi _{A}}}\cdot \widehat{u} \\
&=&\mathtt{e}^{\mathtt{i}\Phi _{A^{-1}}}\cdot \widehat{u}
\end{eqnarray*}%
it follows that $T_{A}:\mathcal{S}\left( \mathbb{R}^{n}\right) \rightarrow 
\mathcal{S}\left( \mathbb{R}^{n}\right) $ is invertible and $%
T_{A}^{-1}=T_{-A}$.

\begin{remark}
$\mathtt{supp}\left( \widehat{T_{A}u}\right) =\mathtt{supp}\left( \widehat{u}%
\right) .$
\end{remark}

\begin{lemma}
\label{st23}Let $1\leq p\leq \infty $ and $u\in S_{w}^{p}\left( \mathbb{R}%
^{n}\right) $. Then $T_{A}u\in S_{w}^{p}\left( \mathbb{R}^{n}\right) $ and
the operator $T_{A}:S_{w}^{p}\left( \mathbb{R}^{n}\right) \rightarrow
S_{w}^{p}\left( \mathbb{R}^{n}\right) $ is bounded. In addition, the map%
\begin{equation*}
\mathtt{GL}\left( n,%
\mathbb{R}
\right) \cap \mathtt{S}\left( n\right) \ni A\rightarrow T_{A}u\in
S_{w}^{p}\left( \mathbb{R}^{n}\right)
\end{equation*}%
is continuous. Here $\mathtt{S}\left( n\right) $ is the vector space o f $%
n\times n$ real symmetric matrices.
\end{lemma}

\begin{proof}
Let $u\in S_{w}^{p}\left( \mathbb{R}^{n}\right) $. There are $Q\subset 
\mathbb{R}^{n}$ a compact subset and $\left( u_{k}\right) _{k\in 
\mathbb{Z}
^{n}}\subset L^{p}\left( \mathbb{R}^{n}\right) $ such that%
\begin{gather*}
\mathtt{supp}\left( \widehat{u}_{k}\right) \subset k+Q,\quad k\in 
\mathbb{Z}
^{n}, \\
\sum_{k\in 
\mathbb{Z}
^{n}}\left\Vert u_{k}\right\Vert _{L^{p}}<\infty , \\
u=\sum_{k\in 
\mathbb{Z}
^{n}}u_{k}.
\end{gather*}%
Let $\chi \in \mathcal{C}_{0}^{\infty }\left( \mathbb{R}^{n}\right) $ be
such that $\chi =1$ on a neighborhood of $Q$. Then for any $k\in 
\mathbb{Z}
^{n}$, $\widehat{u}_{k}=\left( \tau _{k}\chi \right) \cdot \widehat{u}_{k}$
so 
\begin{equation*}
u_{k}=\left( \mathcal{F}^{-1}\left( \tau _{k}\chi \right) \right) \ast
u_{k}=\left( \mathtt{e}^{\mathtt{i}\left\langle \cdot ,k\right\rangle }%
\mathcal{F}^{-1}\left( \chi \right) \right) \ast u_{k}=\left( \mathtt{e}^{%
\mathtt{i}\left\langle \cdot ,k\right\rangle }\psi \right) \ast u_{k},
\end{equation*}%
where $\psi =\mathcal{F}^{-1}\left( \chi \right) $. If we set 
\begin{equation*}
C_{A,n}=\left( 2\pi \right) ^{-n/2}\left\vert \det A\right\vert ^{\frac{1}{2}%
}\mathtt{e}^{-\frac{\pi \mathtt{isgn}A}{4}},
\end{equation*}%
then 
\begin{equation*}
T_{A}u_{k}=C_{A,n}\mathtt{e}^{\mathtt{i}\Phi _{A}}\ast u_{k}=C_{A,n}\mathtt{e%
}^{\mathtt{i}\Phi _{A}}\ast \left( \mathtt{e}^{\mathtt{i}\left\langle \cdot
,k\right\rangle }\psi \right) \ast u_{k}.
\end{equation*}%
Since 
\begin{eqnarray*}
\mathtt{e}^{\mathtt{i}\Phi _{A}}\ast \left( \mathtt{e}^{\mathtt{i}%
\left\langle \cdot ,k\right\rangle }\psi \right) \left( x\right) &=&\dint 
\mathtt{e}^{\mathtt{i}\left[ -\left\langle A\left( x-y\right)
,x-y\right\rangle /2+\left\langle y,k\right\rangle \right] }\psi \left(
y\right) \mathtt{d}y \\
&=&\mathtt{e}^{\mathtt{i}\Phi _{A}\left( x\right) }\dint \mathtt{e}^{\mathtt{%
i}\left\langle y,Ax+k\right\rangle }\mathtt{e}^{\mathtt{i}\Phi _{A}\left(
y\right) }\psi \left( y\right) \mathtt{d}y \\
&=&\mathtt{e}^{\mathtt{i}\Phi _{A}\left( x\right) }\Psi _{A}\left(
Ax+k\right)
\end{eqnarray*}%
where $\Psi _{A}=\mathcal{F}\left( \mathtt{e}^{\mathtt{i}\Phi _{A}}\check{%
\psi}\right) $ it follows that%
\begin{eqnarray*}
\left\Vert T_{A}u_{k}\right\Vert _{L^{p}} &\leq &\left( 2\pi \right)
^{-n/2}\left\vert \det A\right\vert ^{\frac{1}{2}}\left\Vert \Psi _{A}\left(
A\cdot +k\right) \right\Vert _{L^{1}}\left\Vert u_{k}\right\Vert _{L^{p}} \\
&=&\left( 2\pi \right) ^{-n/2}\left\vert \det A\right\vert ^{-\frac{1}{2}%
}\left\Vert \Psi _{A}\right\Vert _{L^{1}}\left\Vert u_{k}\right\Vert
_{L^{p}},\quad k\in 
\mathbb{Z}
^{n}.
\end{eqnarray*}%
By summing with respect to $k\in 
\mathbb{Z}
^{n}$ we get%
\begin{eqnarray*}
\left\Vert T_{A}u\right\Vert _{S_{w}^{p}} &\approx &\sum_{k\in 
\mathbb{Z}
^{n}}\left\Vert T_{A}u_{k}\right\Vert _{L^{p}} \\
&\leq &\left( 2\pi \right) ^{-n/2}\left\vert \det A\right\vert ^{-\frac{1}{2}%
}\left\Vert \Psi _{A}\right\Vert _{L^{1}}\sum_{k\in 
\mathbb{Z}
^{n}}\left\Vert u_{k}\right\Vert _{L^{p}} \\
&\approx &\left( 2\pi \right) ^{-n/2}\left\vert \det A\right\vert ^{-\frac{1%
}{2}}\left\Vert \Psi _{A}\right\Vert _{L^{1}}\left\Vert u\right\Vert
_{S_{w}^{p}}
\end{eqnarray*}

Let $V_{0}$ be a compact neighborhood of $A_{0}\in \mathtt{GL}\left( n,%
\mathbb{R}
\right) \cap \mathtt{S}\left( n\right) $. Let $S=\left\{ A_{l}\right\}
_{l\geq 1}\subset V_{0}$ be a sequence such that $A_{l}\rightarrow A_{0}$.
Then%
\begin{equation*}
\mathtt{e}^{\mathtt{i}\Phi _{A_{l}}}\check{\psi}\rightarrow \mathtt{e}^{%
\mathtt{i}\Phi _{A_{0}}}\check{\psi},\quad \mathcal{F}\left( \mathtt{e}^{%
\mathtt{i}\Phi _{A_{l}}}\check{\psi}\right) \rightarrow \mathcal{F}\left( 
\mathtt{e}^{\mathtt{i}\Phi _{A_{0}}}\check{\psi}\right)
\end{equation*}%
in $\mathcal{S}^{\prime }\left( \mathbb{R}^{n}\right) $. Let $S^{\prime
}=\left\{ A_{l_{m}}\right\} _{m\geq 1}\equiv \left\{ A_{m}^{\prime }\right\}
_{m\geq 1}$ be a subsequence of $S$. Since $V_{0}$ is a compact set and $%
\mathcal{S}\left( \mathbb{R}^{n}\right) $ is a Montel space, it follows that
there is a subsequence $S^{\prime \prime }=\left\{ A_{m_{j}}^{\prime
}\right\} _{j\geq 1}\equiv \left\{ A_{j}^{\prime \prime }\right\} _{j\geq 1}$
of $S^{\prime }$ and $f,g\in \mathcal{S}\left( \mathbb{R}^{n}\right) $ such
that 
\begin{equation*}
\mathtt{e}^{\mathtt{i}\Phi _{A_{j}^{\prime \prime }}}\check{\psi}\rightarrow
f,\quad \mathcal{F}\left( \mathtt{e}^{\mathtt{i}\Phi _{A_{j}^{\prime \prime
}}}\check{\psi}\right) \rightarrow g
\end{equation*}%
in $\mathcal{S}\left( \mathbb{R}^{n}\right) $. Then $f=\mathtt{e}^{\mathtt{i}%
\Phi _{A_{0}}}\check{\psi}$, $g=\mathcal{F}\left( \mathtt{e}^{\mathtt{i}\Phi
_{A_{0}}}\check{\psi}\right) $. So any subsequence of $\left\{ \mathtt{e}^{%
\mathtt{i}\Phi _{A_{l}}}\check{\psi}\right\} _{l\geq 1}$ (respectively of $%
\left\{ \mathcal{F}\left( \mathtt{e}^{\mathtt{i}\Phi _{A_{l}}}\check{\psi}%
\right) \right\} _{l\geq 1}$) contains a subsequence convergent to $\mathtt{e%
}^{\mathtt{i}\Phi _{A_{0}}}\check{\psi}$ (respectively to $\mathcal{F}\left( 
\mathtt{e}^{\mathtt{i}\Phi _{A_{0}}}\check{\psi}\right) $). This shows that%
\begin{equation*}
\mathtt{e}^{\mathtt{i}\Phi _{A_{l}}}\check{\psi}\rightarrow \mathtt{e}^{%
\mathtt{i}\Phi _{A_{0}}}\check{\psi},\quad \mathcal{F}\left( \mathtt{e}^{%
\mathtt{i}\Phi _{A_{l}}}\check{\psi}\right) \rightarrow \mathcal{F}\left( 
\mathtt{e}^{\mathtt{i}\Phi _{A_{0}}}\check{\psi}\right)
\end{equation*}%
in $\mathcal{S}\left( \mathbb{R}^{n}\right) $.

Since $V_{0}$ is a compact neighborhood of $A_{0}$, the map 
\begin{equation*}
\mathtt{GL}\left( n,%
\mathbb{R}
\right) \cap \mathtt{S}\left( n\right) \ni A\rightarrow \Psi _{A}=\mathcal{F}%
\left( \mathtt{e}^{\mathtt{i}\Phi _{A}}\check{\psi}\right) \in \mathcal{S}%
\left( \mathbb{R}^{n}\right)
\end{equation*}
is continuous and 
\begin{eqnarray*}
\left\Vert T_{A}u_{k}\right\Vert _{L^{p}} &\leq &\left( 2\pi \right)
^{-n/2}\left\vert \det A\right\vert ^{\frac{1}{2}}\left\Vert \Psi _{A}\left(
A\cdot +k\right) \right\Vert _{L^{1}}\left\Vert u_{k}\right\Vert _{L^{p}} \\
&=&\left( 2\pi \right) ^{-n/2}\left\vert \det A\right\vert ^{-\frac{1}{2}%
}\left\Vert \Psi _{A}\right\Vert _{L^{1}}\left\Vert u_{k}\right\Vert
_{L^{p}},\quad k\in 
\mathbb{Z}
^{n},A\in V_{0},
\end{eqnarray*}%
there corresponds a constant $C=C\left( V_{0},n\right) $ such that we have 
\begin{equation*}
\sum_{k\in 
\mathbb{Z}
^{n}}\left\Vert T_{A}u_{k}\right\Vert _{L^{p}}\leq C\sum_{k\in 
\mathbb{Z}
^{n}}\left\Vert u_{k}\right\Vert _{L^{p}},\quad A\in V_{0}.
\end{equation*}%
Arguing as above, we can show that for any $k\in 
\mathbb{Z}
^{n}$ the map 
\begin{equation*}
A\rightarrow \mathtt{e}^{\mathtt{i}\Phi _{A}}\ast \left( \mathtt{e}^{\mathtt{%
i}\left\langle \cdot ,k\right\rangle }\psi \right) =\mathtt{e}^{\mathtt{i}%
\Phi _{A}\left( \cdot \right) }\Psi _{A}\left( A\cdot +k\right) \in \mathcal{%
S}\left( \mathbb{R}^{n}\right)
\end{equation*}%
is continuous. Otherwise, we put $\check{\phi}=\widehat{\mathtt{e}^{\mathtt{i%
}\left\langle \cdot ,k\right\rangle }\psi }$ and use Fourier transformation
to obtain the equality%
\begin{equation*}
\mathcal{F}\left( \mathtt{e}^{\mathtt{i}\Phi _{A}}\ast \left( \mathtt{e}^{%
\mathtt{i}\left\langle \cdot ,k\right\rangle }\psi \right) \right) =\left(
2\pi \right) ^{n/2}\left\vert \det A\right\vert ^{-\frac{1}{2}}\mathtt{e}^{%
\frac{\pi \mathtt{isgn}A}{4}}\mathtt{e}^{\mathtt{i}\Phi _{A^{-1}}}\check{\phi%
}.
\end{equation*}%
Now the results already proved show the continuity of the maps%
\begin{equation*}
A\rightarrow \mathcal{F}\left( \mathtt{e}^{\mathtt{i}\Phi _{A}}\ast \left( 
\mathtt{e}^{\mathtt{i}\left\langle \cdot ,k\right\rangle }\psi \right)
\right) ,\quad A\rightarrow \mathtt{e}^{\mathtt{i}\Phi _{A}}\ast \left( 
\mathtt{e}^{\mathtt{i}\left\langle \cdot ,k\right\rangle }\psi \right) .
\end{equation*}%
It follows that the map%
\begin{multline*}
A\rightarrow T_{A}u_{k}=C_{A,n}\mathtt{e}^{\mathtt{i}\Phi _{A}}\ast u_{k} \\
=C_{A,n}\mathtt{e}^{\mathtt{i}\Phi _{A}}\ast \left( \mathtt{e}^{\mathtt{i}%
\left\langle \cdot ,k\right\rangle }\psi \right) \ast u_{k} \\
=C_{A,n}\left( \mathtt{e}^{\mathtt{i}\Phi _{A}\left( \cdot \right) }\Psi
_{A}\left( A\cdot +k\right) \right) \ast u_{k}\in L^{p}.
\end{multline*}%
is continuous.

Finally by an $\varepsilon /3$-argument we obtain the desired continuity.
More specifically, for $V_{0}$ and $\varepsilon >0$ there is $\delta _{0}>0$
and finite set $F_{\varepsilon }\subset 
\mathbb{Z}
^{n}$ so that $\left\Vert A-A_{0}\right\Vert <\delta _{0}$ implies $A\in
V_{0}$ and 
\begin{equation*}
\sum_{k\notin F_{\varepsilon }}\left\Vert T_{A}u_{k}\right\Vert
_{L^{p}}<\varepsilon /3,\quad A\in V_{0}.
\end{equation*}%
Let $0<\delta _{\varepsilon }\leq \delta _{0}$ be such that $\left\Vert
A-A_{0}\right\Vert <\delta _{\varepsilon }$ implies 
\begin{equation*}
\sum_{k\in F_{\varepsilon }}\left\Vert T_{A}u_{k}-T_{A_{0}}u_{k}\right\Vert
_{L^{p}}<\varepsilon /3.
\end{equation*}%
Then for $\left\Vert A-A_{0}\right\Vert <\delta _{\varepsilon }$ we have 
\begin{eqnarray*}
\left\Vert T_{A}u-T_{A_{0}}u\right\Vert _{S_{w}^{p}} &\approx &\sum_{k\in 
\mathbb{Z}
^{n}}\left\Vert T_{A}u_{k}-T_{A_{0}}u_{k}\right\Vert _{L^{p}} \\
&\leq &\sum_{k\in F_{\varepsilon }}\left\Vert
T_{A}u_{k}-T_{A_{0}}u_{k}\right\Vert _{L^{p}}+\sum_{k\notin F_{\varepsilon
}}\left\Vert T_{A}u_{k}\right\Vert _{L^{p}}+\sum_{k\notin F_{\varepsilon
}}\left\Vert T_{A_{0}}u_{k}\right\Vert _{L^{p}} \\
&<&\varepsilon /3+\varepsilon /3+\varepsilon /3=\varepsilon .
\end{eqnarray*}
\end{proof}

We shall now use the spectral characterization theorem to obtain a new proof
of lemma \ref{st12} $(\mathtt{c})$. We start with an auxiliary result.

\begin{lemma}
\label{st13}$(\mathtt{a})$ Let $u\in \mathcal{S}^{\prime }\left( \mathbb{R}%
^{n^{\prime }}\times \mathbb{R}^{n^{\prime \prime }}\right) $, $f,g\in 
\mathcal{S}\left( \mathbb{R}^{n^{\prime }}\right) $ and $h\in \mathcal{S}%
\left( \mathbb{R}^{n^{\prime \prime }}\right) $. Then 
\begin{equation*}
\left\langle u,\left( f\ast g\right) \otimes h\right\rangle =\dint f\left(
x^{\prime }\right) \left\langle u,\left( \tau _{x^{\prime }}g\right) \otimes
h\right\rangle \mathtt{d}x^{\prime }.
\end{equation*}

$(\mathtt{b})$ Let $f\in L^{1}\left( \mathbb{R}^{n^{\prime }}\right) $, $%
g\in L^{q}\left( \mathbb{R}^{n^{\prime \prime }}\right) $ and $u\in
L^{p}\left( \mathbb{R}^{n^{\prime }}\times \mathbb{R}^{n^{\prime \prime
}}\right) $, where $p,q\in \left[ 1,\infty \right] $, $\frac{1}{p}+\frac{1}{q%
}=1$. Then 
\begin{equation*}
v\left( x^{\prime }\right) =\iint f\left( y^{\prime }\right) g\left(
y^{\prime \prime }\right) u\left( x^{\prime }-y^{\prime },y^{\prime \prime
}\right) \mathtt{d}y^{\prime }\mathtt{d}y^{\prime \prime },
\end{equation*}%
defined almost everywhere, is in $L^{p}\left( \mathbb{R}^{n^{\prime
}}\right) $ and 
\begin{equation*}
\left\Vert v\right\Vert _{L^{p}}\leq \left\Vert f\right\Vert
_{L^{1}}\left\Vert g\right\Vert _{L^{q}}\left\Vert u\right\Vert _{L^{p}}.
\end{equation*}
\end{lemma}

\begin{proof}
$(\mathtt{a})$ Let $F\in \mathcal{S}^{\prime }\left( \mathbb{R}^{n^{\prime
}}\times \mathbb{R}^{n^{\prime }}\times \mathbb{R}^{n^{\prime \prime
}}\right) $, defined by%
\begin{equation*}
F\left( x^{\prime },y^{\prime },y^{\prime \prime }\right) =f\left( x^{\prime
}\right) g\left( y^{\prime }-x^{\prime }\right) h\left( y^{\prime \prime
}\right) =\left( f\otimes g\otimes h\right) \circ S\left( x^{\prime
},y^{\prime },y^{\prime \prime }\right) ,
\end{equation*}%
where 
\begin{gather*}
S:\mathbb{R}^{n^{\prime }}\times \mathbb{R}^{n^{\prime }}\times \mathbb{R}%
^{n^{\prime \prime }}\rightarrow \mathbb{R}^{n^{\prime }}\times \mathbb{R}%
^{n^{\prime }}\times \mathbb{R}^{n^{\prime \prime }}, \\
S\left( x^{\prime },y^{\prime },y^{\prime \prime }\right) =\left( x^{\prime
},y^{\prime }-x^{\prime },y^{\prime \prime }\right) ,\quad S\equiv \left( 
\begin{array}{rcc}
\mathtt{I} & 0 & 0 \\ 
-\mathtt{I} & \mathtt{I} & 0 \\ 
0 & 0 & \mathtt{I}%
\end{array}%
\right) .
\end{gather*}

We consider $1\otimes u\in \mathcal{S}^{\prime }\left( \mathbb{R}^{n^{\prime
}}\times \mathbb{R}^{n^{\prime }}\times \mathbb{R}^{n^{\prime \prime
}}\right) $ and calculate $\left\langle 1\otimes u,F\right\rangle $. We have 
\begin{eqnarray*}
\left\langle 1\otimes u,F\right\rangle &=&\left\langle 1\left( x^{\prime
}\right) ,\left\langle u\left( y^{\prime },y^{\prime \prime }\right)
,f\left( x^{\prime }\right) g\left( y^{\prime }-x^{\prime }\right) h\left(
y^{\prime \prime }\right) \right\rangle \right\rangle \\
&=&\dint f\left( x^{\prime }\right) \left\langle u,\left( \tau _{x^{\prime
}}g\right) \otimes h\right\rangle \mathtt{d}x^{\prime }
\end{eqnarray*}%
and%
\begin{eqnarray*}
\left\langle 1\otimes u,F\right\rangle &=&\left\langle u\left( y^{\prime
},y^{\prime \prime }\right) ,\left\langle 1\left( x^{\prime }\right)
,f\left( x^{\prime }\right) g\left( y^{\prime }-x^{\prime }\right) h\left(
y^{\prime \prime }\right) \right\rangle \right\rangle \\
&=&\left\langle u\left( y^{\prime },y^{\prime \prime }\right) ,\left( \left(
f\ast g\right) \otimes h\right) \left( y^{\prime },y^{\prime \prime }\right)
\right\rangle \\
&=&\left\langle u,\left( f\ast g\right) \otimes h\right\rangle
\end{eqnarray*}%
hence 
\begin{equation*}
\left\langle u,\left( f\ast g\right) \otimes h\right\rangle =\dint f\left(
x^{\prime }\right) \left\langle u,\left( \tau _{x^{\prime }}g\right) \otimes
h\right\rangle \mathtt{d}x^{\prime }.
\end{equation*}

$(\mathtt{b})$ We evaluate $\iint \left\vert f\left( y^{\prime }\right)
g\left( y^{\prime \prime }\right) u\left( x^{\prime }-y^{\prime },y^{\prime
\prime }\right) \right\vert \mathtt{d}y^{\prime }\mathtt{d}y^{\prime \prime
} $. Using Holder's inequality we obtain: 
\begin{multline*}
\left\vert v\left( x^{\prime }\right) \right\vert \leq \iint \left\vert
f\left( y^{\prime }\right) g\left( y^{\prime \prime }\right) u\left(
x^{\prime }-y^{\prime },y^{\prime \prime }\right) \right\vert \mathtt{d}%
y^{\prime }\mathtt{d}y^{\prime \prime } \\
=\int \left\vert f\left( y^{\prime }\right) \right\vert \left( \int
\left\vert g\left( y^{\prime \prime }\right) \right\vert \left\vert u\left(
x^{\prime }-y^{\prime },y^{\prime \prime }\right) \right\vert \mathtt{d}%
y^{\prime \prime }\right) \mathtt{d}y^{\prime } \\
\leq \int \left\vert f\left( y^{\prime }\right) \right\vert \left( \int
\left\vert g\left( y^{\prime \prime }\right) \right\vert ^{q}\mathtt{d}%
y^{\prime \prime }\right) ^{1/q}\left( \int \left\vert u\left( x^{\prime
}-y^{\prime },y^{\prime \prime }\right) \right\vert ^{p}\mathtt{d}y^{\prime
\prime }\right) ^{1/p}\mathtt{d}y^{\prime } \\
=\left\Vert g\right\Vert _{L^{q}}\int \left\vert f\left( y^{\prime }\right)
\right\vert \left( \int \left\vert u\left( x^{\prime }-y^{\prime },y^{\prime
\prime }\right) \right\vert ^{p}\mathtt{d}y^{\prime \prime }\right) ^{1/p}%
\mathtt{d}y^{\prime }
\end{multline*}%
Next we apply the integral version of the Minkowski's inequality:%
\begin{eqnarray*}
\left\Vert v\right\Vert _{L^{p}} &\leq &\left\Vert g\right\Vert _{L^{q}}\int
\left\vert f\left( y^{\prime }\right) \right\vert \left\Vert \left( \int
\left\vert u\left( \cdot -y^{\prime },y^{\prime \prime }\right) \right\vert
^{p}\mathtt{d}y^{\prime \prime }\right) ^{1/p}\right\Vert _{L^{p}}\mathtt{d}%
y^{\prime } \\
&=&\left\Vert f\right\Vert _{L^{1}}\left\Vert g\right\Vert
_{L^{q}}\left\Vert u\right\Vert _{L^{p}}.
\end{eqnarray*}
\end{proof}

Suppose that $\mathbb{R}^{n}=\mathbb{R}^{n^{\prime }}\times \mathbb{R}%
^{n^{\prime \prime }}$. Let $u\in \mathcal{S}^{\prime }\left( \mathbb{R}%
^{n^{\prime }}\times \mathbb{R}^{n^{\prime \prime }}\right) $ be such that $%
\widehat{u}\in \mathcal{E}^{\prime }\left( \mathbb{R}^{n^{\prime }}\times 
\mathbb{R}^{n^{\prime \prime }}\right) $. Then $u\left( x\right) =\left(
2\pi \right) ^{-n}\left\langle \widehat{u},\mathtt{e}^{\mathtt{i}%
\left\langle x,\cdot \right\rangle }\right\rangle $ is a smooth function
with slow growth which can be extended to an entire analytic function in $%
\mathbb{%
\mathbb{C}
}^{n}$. Therefore $v=u_{|\mathbb{R}^{n^{\prime }}\times \left\{ 0\right\} }$
is well defined and $v\in \mathcal{S}^{\prime }\left( \mathbb{R}^{n^{\prime
}}\right) $. Let $\chi ^{\prime }\in \mathcal{C}_{0}^{\infty }\left( \mathbb{%
R}^{n^{\prime }}\right) $ and $\chi ^{\prime \prime }\in \mathcal{C}%
_{0}^{\infty }\left( \mathbb{R}^{n^{\prime \prime }}\right) $ such that $%
\chi =\chi ^{\prime }\otimes \chi ^{\prime \prime }=1$ on a neighborhood of
the support of $\widehat{u}$.

\begin{lemma}
Let $\varphi \in \mathcal{S}\left( \mathbb{R}^{n^{\prime }}\right) $. Then 
\begin{equation*}
\left\langle \widehat{v},\varphi \right\rangle =\left( 2\pi \right)
^{-n^{\prime \prime }}\left\langle \widehat{u},\left( \varphi \chi ^{\prime
}\right) \otimes \chi ^{\prime \prime }\right\rangle .
\end{equation*}
\end{lemma}

\begin{proof}
We have 
\begin{eqnarray*}
u\left( x^{\prime },x^{\prime \prime }\right) &=&\left( 2\pi \right)
^{-n}\left\langle \widehat{u},\mathtt{e}^{\mathtt{i}\left\langle x,\cdot
\right\rangle }\right\rangle =\left( 2\pi \right) ^{-n}\left\langle \widehat{%
u},\mathtt{e}^{\mathtt{i}\left\langle x,\cdot \right\rangle }\chi
\right\rangle \\
&=&\left( 2\pi \right) ^{-n}\left\langle \widehat{u},\mathtt{e}^{\mathtt{i}%
\left\langle x^{\prime },\cdot \right\rangle }\chi ^{\prime }\otimes \mathtt{%
e}^{\mathtt{i}\left\langle x^{\prime \prime },\cdot \right\rangle }\chi
^{\prime \prime }\right\rangle \\
&=&\left( 2\pi \right) ^{-n}\left\langle u,\widehat{\mathtt{e}^{\mathtt{i}%
\left\langle x^{\prime },\cdot \right\rangle }\chi ^{\prime }}\otimes 
\widehat{\mathtt{e}^{\mathtt{i}\left\langle x^{\prime \prime },\cdot
\right\rangle }\chi ^{\prime \prime }}\right\rangle \\
&=&\left( 2\pi \right) ^{-n}\left\langle u,\tau _{x^{\prime }}\widehat{\chi
^{\prime }}\otimes \tau _{x^{\prime \prime }}\widehat{\chi ^{\prime \prime }}%
\right\rangle .
\end{eqnarray*}%
Hence $v\left( x^{\prime }\right) =u\left( x^{\prime },0\right) =\left( 2\pi
\right) ^{-n}\left\langle u,\tau _{x^{\prime }}\widehat{\chi ^{\prime }}%
\otimes \widehat{\chi ^{\prime \prime }}\right\rangle $ is a smooth function
with slow growth which can be extended to an entire analytic function in $%
\mathbb{%
\mathbb{C}
}^{n^{\prime }}$. Applying lemma \ref{st13} $(\mathtt{a})$ we obtain that%
\begin{eqnarray*}
\left\langle \widehat{v},\varphi \right\rangle &=&\left\langle v,\widehat{%
\varphi }\right\rangle =\left\langle u\left( \cdot ,0\right) ,\widehat{%
\varphi }\right\rangle =\int \widehat{\varphi }\left( x^{\prime }\right)
u\left( x^{\prime },0\right) \\
&=&\left( 2\pi \right) ^{-n}\int \widehat{\varphi }\left( x^{\prime }\right)
\left\langle u,\tau _{x^{\prime }}\widehat{\chi ^{\prime }}\otimes \widehat{%
\chi ^{\prime \prime }}\right\rangle \mathtt{d}x^{\prime } \\
&=&\left( 2\pi \right) ^{-n}\left\langle u,\left( \widehat{\varphi }\ast 
\widehat{\chi ^{\prime }}\right) \otimes \widehat{\chi ^{\prime \prime }}%
\right\rangle \\
&=&\left( 2\pi \right) ^{-n^{\prime \prime }}\left\langle u,\widehat{\varphi
\chi ^{\prime }}\otimes \widehat{\chi ^{\prime \prime }}\right\rangle \\
&=&\left( 2\pi \right) ^{-n^{\prime \prime }}\left\langle \widehat{u},\left(
\varphi \chi ^{\prime }\right) \otimes \chi ^{\prime \prime }\right\rangle .
\end{eqnarray*}
\end{proof}

\begin{corollary}
$\mathtt{supp}\left( \widehat{v}\right) \subset \mathtt{pr}_{\mathbb{R}%
^{n^{\prime }}}\left( \mathtt{supp}\left( \widehat{u}\right) \right) .$
\end{corollary}

Otherwise, we can use lemma \ref{st37}. Suppose that $\mathbb{R}^{n}=\mathbb{%
R}^{n^{\prime }}\times \mathbb{R}^{n^{\prime \prime }}$. Let $u\in \mathcal{S%
}^{\prime }\left( \mathbb{R}^{n^{\prime }}\times \mathbb{R}^{n^{\prime
\prime }}\right) $ and $\chi \in \mathcal{S}\left( \mathbb{R}^{n^{\prime
}}\times \mathbb{R}^{n^{\prime \prime }}\right) $. Then $\chi \left(
D\right) u\in \mathcal{S}^{\prime }\left( \mathbb{R}^{n}\right) \cap 
\mathcal{C}_{pol}^{\infty }\left( \mathbb{R}^{n}\right) $. In fact we have 
\begin{equation*}
\chi \left( D\right) u\left( x\right) =\left( 2\pi \right) ^{-n}\left\langle 
\widehat{u},\mathtt{e}^{\mathtt{i}\left\langle x,\cdot \right\rangle }\chi
\right\rangle ,\quad x\in \mathbb{R}^{n}.
\end{equation*}%
If $\left( u_{j}\right) _{j\in 
\mathbb{N}
}\subset \mathcal{S}^{\prime }\left( \mathbb{R}^{n^{\prime }}\times \mathbb{R%
}^{n^{\prime \prime }}\right) $, $u_{j}\rightarrow u$ in $\mathcal{S}%
^{\prime }$, then for any $x\in \mathbb{R}^{n}$ we have%
\begin{equation*}
\chi \left( D\right) u_{j}\left( x\right) =\left( 2\pi \right)
^{-n}\left\langle \widehat{u}_{j},\mathtt{e}^{\mathtt{i}\left\langle x,\cdot
\right\rangle }\chi \right\rangle \rightarrow \left( 2\pi \right)
^{-n}\left\langle \widehat{u},\mathtt{e}^{\mathtt{i}\left\langle x,\cdot
\right\rangle }\chi \right\rangle =\chi \left( D\right) u\left( x\right) .
\end{equation*}%
Using Banach-Steinhaus' theorem it follows that there is a seminorm $p$ on $%
\mathcal{S}\left( \mathbb{R}^{n}\right) $ such that 
\begin{equation*}
\left\vert \chi \left( D\right) u_{j}\left( x\right) \right\vert \leq
p\left( \mathtt{e}^{\mathtt{i}\left\langle x,\cdot \right\rangle }\chi
\right) .
\end{equation*}%
Therefore there are $C>0$ and $N\in 
\mathbb{N}
$ such that 
\begin{equation*}
\left\vert \chi \left( D\right) u_{j}\left( x\right) \right\vert \leq
C\left\langle x\right\rangle ^{N},\quad x\in \mathbb{R}^{n}.
\end{equation*}

Since $\chi \left( D\right) u\in \mathcal{S}^{\prime }\left( \mathbb{R}%
^{n}\right) \cap \mathcal{C}_{pol}^{\infty }\left( \mathbb{R}^{n}\right) $, $%
v_{\chi }=\chi \left( D\right) u_{|\mathbb{R}^{n^{\prime }}\times \left\{
0\right\} }$ is well defined and $v_{\chi }\in \mathcal{S}^{\prime }\left( 
\mathbb{R}^{n^{\prime }}\right) \cap \mathcal{C}_{pol}^{\infty }\left( 
\mathbb{R}^{n^{\prime }}\right) $. If $\left( u_{j}\right) _{j\in 
\mathbb{N}
}\subset \mathcal{S}^{\prime }\left( \mathbb{R}^{n^{\prime }}\times \mathbb{R%
}^{n^{\prime \prime }}\right) $ and $u_{j}\rightarrow u$ in $\mathcal{S}%
^{\prime }$, then $v_{j\chi }=\chi \left( D\right) u_{j|\mathbb{R}%
^{n^{\prime }}\times \left\{ 0\right\} }\rightarrow v_{\chi }=\chi \left(
D\right) u_{|\mathbb{R}^{n^{\prime }}\times \left\{ 0\right\} }$ pointwise
and there are $C>0$ and $N\in 
\mathbb{N}
$ such that 
\begin{equation*}
\left\vert v_{j\chi }\left( x^{\prime }\right) \right\vert \leq
C\left\langle x^{\prime }\right\rangle ^{N},\quad x^{\prime }\in \mathbb{R}%
^{n^{\prime }}.
\end{equation*}%
Using the dominated convergence theorem we obtain $v_{j\chi }\rightarrow
v_{\chi }$ in $\mathcal{S}^{\prime }\left( \mathbb{R}^{n^{\prime }}\right) $.

\begin{lemma}
Let $\varphi \in \mathcal{S}\left( \mathbb{R}^{n^{\prime }}\right) $. Then 
\begin{equation*}
\left\langle \widehat{v}_{\chi },\varphi \right\rangle =\left( 2\pi \right)
^{-n^{\prime \prime }}\left\langle \widehat{u},\chi \left( \varphi \otimes
1\right) \right\rangle .
\end{equation*}
\end{lemma}

\begin{proof}
It is sufficient to prove the equality for $u\in \mathcal{S}\left( \mathbb{R}%
^{n^{\prime }}\times \mathbb{R}^{n^{\prime \prime }}\right) $, the general
case follows by continuity. Let $u\in \mathcal{S}\left( \mathbb{R}%
^{n^{\prime }}\times \mathbb{R}^{n^{\prime \prime }}\right) $ and $\varphi
\in \mathcal{S}\left( \mathbb{R}^{n^{\prime }}\right) $. Then 
\begin{eqnarray*}
\left\langle \widehat{v}_{\chi },\varphi \right\rangle &=&\left\langle
v_{\chi },\widehat{\varphi }\right\rangle =\left\langle \widehat{\varphi }%
\otimes \delta ,\chi \left( D\right) u\right\rangle \\
&=&\left( 2\pi \right) ^{-n^{\prime \prime }}\left\langle \widehat{\varphi }%
\otimes \widehat{1},\chi \left( D\right) u\right\rangle \\
&=&\left( 2\pi \right) ^{-n^{\prime \prime }}\left\langle \varphi \otimes
1,\chi \widehat{u}\right\rangle \\
&=&\left( 2\pi \right) ^{-n^{\prime \prime }}\left\langle \widehat{u},\chi
\left( \varphi \otimes 1\right) \right\rangle .
\end{eqnarray*}
\end{proof}

\begin{lemma}
Let $L\subset \mathbb{R}^{n}$ be a linear subspace. Then the map%
\begin{equation*}
S_{w}^{p}\left( \mathbb{R}^{n}\right) \ni u\rightarrow u_{|L}\in
S_{w}^{p}\left( L\right)
\end{equation*}%
is well defined and continuous.
\end{lemma}

\begin{proof}
We can assume $L=\mathbb{R}^{n^{\prime }}\times \left\{ 0\right\} $ where $%
\mathbb{R}^{n}=\mathbb{R}^{n^{\prime }}\times \mathbb{R}^{n^{\prime \prime
}} $. By corollary \ref{st14} there are a compact subset $Q\subset \mathbb{R}%
^{n}$ and a sequence $\left( u_{k}\right) _{k\in 
\mathbb{Z}
^{n}}\subset L^{p}\left( \mathbb{R}^{n}\right) $ such that 
\begin{equation*}
u=\sum_{k\in 
\mathbb{Z}
^{n}}u_{k},\quad \mathtt{supp}\left( \widehat{u}_{k}\right) \subset
k+Q,\quad k\in 
\mathbb{Z}
^{n}
\end{equation*}%
and 
\begin{equation*}
\sum_{k\in 
\mathbb{Z}
^{n}}\left\Vert u_{k}\right\Vert _{L^{p}}<\infty ,\quad \left\Vert
u\right\Vert _{S_{w}^{p}}\approx \sum_{k\in 
\mathbb{Z}
^{n}}\left\Vert u_{k}\right\Vert _{L^{p}}.
\end{equation*}%
The series $\sum_{k\in 
\mathbb{Z}
^{n}}u_{k}$ converges to $u$ in $\mathcal{BC}\left( \mathbb{R}^{n}\right) $.

Let $v=u_{|\mathbb{R}^{n^{\prime }}\times \left\{ 0\right\} }$ and $%
w_{k}=u_{k}{}_{|\mathbb{R}^{n^{\prime }}\times \left\{ 0\right\} }$, $k\in 
\mathbb{Z}
^{n}$. Then by the previous lemma%
\begin{equation*}
\mathtt{supp}\left( \widehat{w}_{k}\right) \subset \mathtt{pr}_{\mathbb{R}%
^{n^{\prime }}}\left( \mathtt{supp}\left( \widehat{u}_{k}\right) \right)
\subset k^{\prime }+\mathtt{pr}_{\mathbb{R}^{n^{\prime }}}\left( Q\right)
,\quad k=\left( k^{\prime },k^{\prime \prime }\right) \in 
\mathbb{Z}
^{n}.
\end{equation*}%
Let $\chi ^{\prime }\in \mathcal{S}\left( \mathbb{R}^{n^{\prime }}\right) $
and $\chi ^{\prime \prime }\in \mathcal{S}\left( \mathbb{R}^{n^{\prime
\prime }}\right) $ be such that $\widehat{\chi ^{\prime }}\otimes \widehat{%
\chi ^{\prime \prime }}\in \mathcal{C}_{0}^{\infty }\left( \mathbb{R}%
^{n^{\prime }}\times \mathbb{R}^{n^{\prime \prime }}\right) $ and $\widehat{%
\chi ^{\prime }}\otimes \widehat{\chi ^{\prime \prime }}=1$ on a
neighborhood of the set $Q$. Then $\tau _{k}\left( \widehat{\chi ^{\prime }}%
\otimes \widehat{\chi ^{\prime \prime }}\right) \widehat{u}_{k}=\widehat{u}%
_{k}$. Hence 
\begin{equation*}
\left( \mathtt{e}^{\mathtt{i}\left\langle \cdot ,k^{\prime }\right\rangle
}\chi ^{\prime }\otimes \mathtt{e}^{\mathtt{i}\left\langle \cdot ,k^{\prime
\prime }\right\rangle }\chi ^{\prime \prime }\right) \ast u_{k}=u_{k},\quad
k\in 
\mathbb{Z}
^{n}.
\end{equation*}%
It follows that 
\begin{eqnarray*}
w_{k}\left( x^{\prime }\right) &=&\left( \mathtt{e}^{\mathtt{i}\left\langle
\cdot ,k^{\prime }\right\rangle }\chi ^{\prime }\otimes \mathtt{e}^{\mathtt{i%
}\left\langle \cdot ,k^{\prime \prime }\right\rangle }\chi ^{\prime \prime
}\right) \ast u_{k}\left( x^{\prime },0\right) \\
&=&\iint \mathtt{e}^{\mathtt{i}\left\langle y^{\prime }.k^{\prime
}\right\rangle }\chi ^{\prime }\left( y^{\prime }\right) \mathtt{e}^{\mathtt{%
i}\left\langle y^{\prime \prime }.k^{\prime \prime }\right\rangle }\chi
^{\prime \prime }\left( y^{\prime \prime }\right) u_{k}\left( x^{\prime
}-y^{\prime },-y^{\prime \prime }\right) \mathtt{d}y^{\prime }\mathtt{d}%
y^{\prime \prime }.
\end{eqnarray*}%
If we use lemma \ref{st13} we obtain that $w_{k}\in L^{p}\left( \mathbb{R}%
^{n^{\prime }}\right) $ and 
\begin{equation*}
\left\Vert w_{k}\right\Vert _{L^{p}}\leq \left\Vert \chi ^{\prime
}\right\Vert _{L^{1}}\left\Vert \chi ^{\prime \prime }\right\Vert
_{L^{q}}\left\Vert u_{k}\right\Vert _{L^{p}},\quad k\in 
\mathbb{Z}
^{n}.
\end{equation*}%
Let $k^{\prime }\in 
\mathbb{Z}
^{n^{\prime }}$. Put 
\begin{equation*}
v_{k^{\prime }}=\sum_{k^{\prime \prime }\in 
\mathbb{Z}
^{n^{\prime \prime }}}w_{\left( k^{\prime },k^{\prime \prime }\right) }.
\end{equation*}%
Then $v_{k^{\prime }}\in L^{p}\left( \mathbb{R}^{n^{\prime }}\right) $, 
\begin{eqnarray*}
\left\Vert v_{k^{\prime }}\right\Vert _{L^{p}} &\leq &\left\Vert \chi
^{\prime }\right\Vert _{L^{1}}\left\Vert \chi ^{\prime \prime }\right\Vert
_{L^{q}}\sum_{k^{\prime \prime }\in 
\mathbb{Z}
^{n^{\prime \prime }}}\left\Vert u_{\left( k^{\prime },k^{\prime \prime
}\right) }\right\Vert _{L^{p}}, \\
\mathtt{supp}\left( \widehat{v}_{k^{\prime }}\right) &\subset &\overline{%
\bigcup_{k^{\prime \prime }\in 
\mathbb{Z}
^{n^{\prime \prime }}}\mathtt{supp}\left( \widehat{w}_{\left( k^{\prime
},k^{\prime \prime }\right) }\right) }\subset k^{\prime }+\mathtt{pr}_{%
\mathbb{R}^{n^{\prime }}}\left( Q\right) ,\quad k^{\prime }\in 
\mathbb{Z}
^{n^{\prime }}
\end{eqnarray*}%
and%
\begin{eqnarray*}
\sum_{k^{\prime }\in 
\mathbb{Z}
^{n^{\prime }}}\left\Vert v_{k^{\prime }}\right\Vert _{L^{p}} &\leq
&\left\Vert \chi ^{\prime }\right\Vert _{L^{1}}\left\Vert \chi ^{\prime
\prime }\right\Vert _{L^{q}}\left( \sum_{k^{\prime }\in 
\mathbb{Z}
^{n^{\prime }}}\sum_{k^{\prime \prime }\in 
\mathbb{Z}
^{n^{\prime \prime }}}\left\Vert u_{\left( k^{\prime },k^{\prime \prime
}\right) }\right\Vert _{L^{p}}\right) \\
&=&\left\Vert \chi ^{\prime }\right\Vert _{L^{1}}\left\Vert \chi ^{\prime
\prime }\right\Vert _{L^{q}}\sum_{k\in 
\mathbb{Z}
^{n}}\left\Vert u_{k}\right\Vert _{L^{p}}<\infty .
\end{eqnarray*}%
Hence $v\in S_{w}^{p}\left( \mathbb{R}^{n^{\prime }}\right) $ and $%
\left\Vert v\right\Vert _{S_{w}^{p}}\leq Cst\left\Vert u\right\Vert
_{S_{w}^{p}}$.
\end{proof}

The spectral characterization of ideals $S_{w}^{p}$ enables us to establish H%
\"{o}lder and Young type inequalities i.e.

\begin{enumerate}
\item $S_{w}^{p}\cdot S_{w}^{q}\subset S_{w}^{r}$ if $1\leq p,q,r\leq \infty 
$ and $\frac{1}{p}+\frac{1}{q}=\frac{1}{r}$.

\item $L^{q}\ast S_{w}^{p}\subset S_{w}^{r}$ if $1\leq p,q,r\leq \infty $
and $\frac{1}{p}+\frac{1}{q}=1+\frac{1}{r}$.
\end{enumerate}

\begin{lemma}
Let $u,v\in \mathcal{S}^{\prime }\left( \mathbb{R}^{n}\right) $.

$\mathtt{(a)}$ If $\widehat{u},\widehat{v}\in \mathcal{E}^{\prime }\left( 
\mathbb{R}^{n}\right) $, then $uv\in \mathcal{S}^{\prime }\left( \mathbb{R}%
^{n}\right) $, $\widehat{uv}=\left( 2\pi \right) ^{-n}\widehat{u}\ast 
\widehat{v}$ and 
\begin{equation*}
\mathtt{supp}\left( \widehat{uv}\right) \subset \mathtt{supp}\left( \widehat{%
u}\right) +\mathtt{supp}\left( \widehat{v}\right) .
\end{equation*}

$\mathtt{(b)}$ If $u\in L^{p}\left( \mathbb{R}^{n}\right) ,v\in L^{q}\left( 
\mathbb{R}^{n}\right) $, where $1\leq p,q,r\leq \infty $ and $\frac{1}{p}+%
\frac{1}{q}=1+\frac{1}{r}$, then $u\ast v\in L^{r}\left( \mathbb{R}%
^{n}\right) \subset \mathcal{S}^{\prime }\left( \mathbb{R}^{n}\right) $ and 
\begin{equation*}
\mathtt{supp}\left( \widehat{u\ast v}\right) \subset \mathtt{supp}\left( 
\widehat{u}\right) \cap \mathtt{supp}\left( \widehat{v}\right) .
\end{equation*}
\end{lemma}

\begin{proof}
$\mathtt{(a)}$ Since $\widehat{u},\widehat{v}\in \mathcal{E}^{\prime }\left( 
\mathbb{R}^{n}\right) $ it follows that $u\left( x\right) =\left( 2\pi
\right) ^{-n}\left\langle \widehat{u},\mathtt{e}^{\mathtt{i}\left\langle
x,\cdot \right\rangle }\right\rangle $, $v\left( x\right) =\left( 2\pi
\right) ^{-n}\left\langle \widehat{v},\mathtt{e}^{\mathtt{i}\left\langle
x,\cdot \right\rangle }\right\rangle $. Therefore$\ u$ and $v$ are smooth
functions with slow growth which can be extended to an entire analytic
functions in $\mathbb{%
\mathbb{C}
}^{n}$. Hence $uv\in \mathcal{S}^{\prime }\left( \mathbb{R}^{n}\right) $. We
have 
\begin{equation*}
\mathcal{F}\left( \left( 2\pi \right) ^{-n}\widehat{u}\ast \widehat{v}%
\right) =\left( 2\pi \right) ^{-n}\widehat{\widehat{u}}\widehat{\widehat{v}}%
=\left( 2\pi \right) ^{n}\check{u}\check{v}=\widehat{\widehat{uv}}=\mathcal{F%
}\widehat{uv}.
\end{equation*}%
Here we used the relation $\widehat{\widehat{w}}=\left( 2\pi \right) ^{n}%
\check{w}$, $w\in \mathcal{S}^{\prime }\left( \mathbb{R}^{n}\right) $, where 
$\check{w}\left( x\right) =w\left( -x\right) $. Since $\mathcal{F}$ is an
isomorphism it follows that $\widehat{uv}=\left( 2\pi \right) ^{-n}\widehat{u%
}\ast \widehat{v}$.

$\mathtt{(b)}$ If $u\in \mathcal{S}\left( \mathbb{R}^{n}\right) $ and $v\in
L^{q}\left( \mathbb{R}^{n}\right) $, then $\widehat{u\ast v}=\widehat{u}%
\widehat{v}$ and $\mathtt{supp}\left( \widehat{u\ast v}\right) \subset 
\mathtt{supp}\left( \widehat{v}\right) $. Let $u\in L^{p}\left( \mathbb{R}%
^{n}\right) $ and $v\in L^{q}\left( \mathbb{R}^{n}\right) $. Let $\left(
u_{j}\right) _{j\in 
\mathbb{N}
}\subset \mathcal{S}\left( \mathbb{R}^{n}\right) $ such that $%
u_{j}\rightarrow u$ in $L^{p}\left( \mathbb{R}^{n}\right) $. Then $u_{j}\ast
v\rightarrow u\ast v$ in $L^{r}\left( \mathbb{R}^{n}\right) \subset \mathcal{%
S}^{\prime }\left( \mathbb{R}^{n}\right) $. Hence $\widehat{u_{j}\ast v}%
\rightarrow \widehat{u\ast v}$ in $\mathcal{S}^{\prime }\left( \mathbb{R}%
^{n}\right) $. Since $\mathtt{supp}\left( \widehat{u_{j}\ast v}\right)
\subset \mathtt{supp}\left( \widehat{v}\right) $ for any $j\in 
\mathbb{N}
$, it follows that $\mathtt{supp}\left( \widehat{u\ast v}\right) \subset 
\mathtt{supp}\left( \widehat{v}\right) $. Similarly, it is shown that $%
\mathtt{supp}\left( \widehat{u\ast v}\right) \subset \mathtt{supp}\left( 
\widehat{u}\right) $.
\end{proof}

\begin{proposition}
Let $1\leq p,q,r\leq \infty $.

$\mathtt{(a)}$ If $\frac{1}{p}+\frac{1}{q}=\frac{1}{r}$, then $%
S_{w}^{p}\cdot S_{w}^{q}\subset S_{w}^{r}$.

$\mathtt{(b)}$ If $\frac{1}{p}+\frac{1}{q}=1+\frac{1}{r}$, then $L^{q}\ast
S_{w}^{p}\subset S_{w}^{r}$.
\end{proposition}

\begin{proof}
$\mathtt{(a)}$ Let $u\in S_{w}^{p}$ and $v\in S_{w}^{q}$. It follows from
corollary \ref{st14} that there are two compact sets $Q_{u},Q_{v}\subset 
\mathbb{R}^{n}$ and two sequences $\left( u_{k}\right) _{k\in 
\mathbb{Z}
^{n}}\subset L^{p}\left( \mathbb{R}^{n}\right) $ and $\left( v_{l}\right)
_{l\in 
\mathbb{Z}
^{n}}\subset L^{q}\left( \mathbb{R}^{n}\right) $ such that%
\begin{eqnarray*}
u &=&\sum_{k\in 
\mathbb{Z}
^{n}}u_{k},\quad \mathtt{supp}\left( \widehat{u}_{k}\right) \subset
k+Q_{u},\quad k\in 
\mathbb{Z}
^{n}, \\
v &=&\sum_{l\in 
\mathbb{Z}
^{n}},\quad \mathtt{supp}\left( \widehat{v}_{l}\right) \subset l+Q_{v},\quad
l\in 
\mathbb{Z}
^{n}
\end{eqnarray*}%
and 
\begin{eqnarray*}
\sum_{k\in 
\mathbb{Z}
^{n}}\left\Vert u_{k}\right\Vert _{L^{p}} &<&\infty ,\quad \left\Vert
u\right\Vert _{S_{w}^{p}}\approx \sum_{k\in 
\mathbb{Z}
^{n}}\left\Vert u_{k}\right\Vert _{L^{p}}, \\
\sum_{l\in 
\mathbb{Z}
^{n}}\left\Vert v_{l}\right\Vert _{L^{q}} &<&\infty ,\quad \left\Vert
v\right\Vert _{S_{w}^{q}}\approx \sum_{k\in 
\mathbb{Z}
^{n}}\left\Vert v_{l}\right\Vert _{L^{q}}.
\end{eqnarray*}%
The usual H\"{o}lder inequality implies $u_{k}v_{l}\in L^{r}\left( \mathbb{R}%
^{n}\right) $ and $\left\Vert u_{k}v_{l}\right\Vert _{L^{r}}\leq \left\Vert
u_{k}\right\Vert _{L^{p}}\left\Vert v_{l}\right\Vert _{L^{q}}$. For each $%
m\in 
\mathbb{Z}
^{n}$ we put 
\begin{equation*}
w_{m}=\sum_{k+l=m}u_{k}v_{l}.
\end{equation*}%
Then $w_{m}\in L^{r}\left( \mathbb{R}^{n}\right) $ and 
\begin{eqnarray*}
\sum_{m\in 
\mathbb{Z}
^{n}}\left\Vert w_{m}\right\Vert _{L^{r}} &=&\sum_{m\in 
\mathbb{Z}
^{n}}\sum_{k+l=m}\left\Vert u_{k}v_{l}\right\Vert _{L^{r}} \\
&\leq &\sum_{m\in 
\mathbb{Z}
^{n}}\sum_{k+l=m}\left\Vert u_{k}\right\Vert _{L^{p}}\left\Vert
v_{l}\right\Vert _{L^{q}} \\
&=&\left( \sum_{k\in 
\mathbb{Z}
^{n}}\left\Vert u_{k}\right\Vert _{L^{p}}\right) \left( \sum_{k\in 
\mathbb{Z}
^{n}}\left\Vert v_{l}\right\Vert _{L^{q}}\right) <\infty .
\end{eqnarray*}%
Also we have 
\begin{eqnarray*}
\mathtt{supp}\left( \widehat{w}_{m}\right) &\subset &\bigcup_{k+l=m}\mathtt{%
supp}\left( \widehat{u_{k}v_{l}}\right) \subset \bigcup_{k+l=m}\left( 
\mathtt{supp}\left( \widehat{u}_{k}\right) +\mathtt{supp}\left( \widehat{v}%
_{l}\right) \right) \\
&\subset &\bigcup_{k+l=m}\left( k+Q_{u}+l+Q_{v}\right) \subset m+Q_{u}+Q_{v}.
\end{eqnarray*}%
Hence $uv\in S_{w}^{r}$ and $\left\Vert uv\right\Vert _{S_{w}^{r}}\leq
Cst\left\Vert u\right\Vert _{S_{w}^{p}}\left\Vert v\right\Vert _{S_{w}^{q}}$.

$\mathtt{(b)}$ Let $v\in L^{q}$ and $u\in S_{w}^{p}$. By corollary \ref{st14}
there are a compact subset $Q_{u}\subset \mathbb{R}^{n}$ and a sequence $%
\left( u_{k}\right) _{k\in 
\mathbb{Z}
^{n}}\subset L^{p}\left( \mathbb{R}^{n}\right) $ such that%
\begin{equation*}
u=\sum_{k\in 
\mathbb{Z}
^{n}}u_{k},\quad \mathtt{supp}\left( \widehat{u}_{k}\right) \subset
k+Q_{u},\quad k\in 
\mathbb{Z}
^{n},
\end{equation*}%
and 
\begin{equation*}
\sum_{k\in 
\mathbb{Z}
^{n}}\left\Vert u_{k}\right\Vert _{L^{p}}<\infty ,\quad \left\Vert
u\right\Vert _{S_{w}^{p}}\approx \sum_{k\in 
\mathbb{Z}
^{n}}\left\Vert u_{k}\right\Vert _{L^{p}}.
\end{equation*}%
Let $k\in 
\mathbb{Z}
^{n}$. Then the usual Young inequality implies $v\ast u_{k}\in L^{r}\left( 
\mathbb{R}^{n}\right) $ and $\left\Vert v\ast u_{k}\right\Vert _{L^{r}}\leq
\left\Vert v\right\Vert _{L^{q}}\left\Vert u_{k}\right\Vert _{L^{p}}$. Also
we have $\mathtt{supp}\left( \widehat{v\ast u_{k}}\right) \subset \mathtt{%
supp}\left( \widehat{u}_{k}\right) \subset k+Q_{u}$. Thus 
\begin{equation*}
v\ast u=\sum_{k\in 
\mathbb{Z}
^{n}}v\ast u_{k},\quad \mathtt{supp}\left( \widehat{v\ast u_{k}}\right)
\subset k+Q_{u},\quad k\in 
\mathbb{Z}
^{n},
\end{equation*}%
and 
\begin{equation*}
\sum_{k\in 
\mathbb{Z}
^{n}}\left\Vert v\ast u_{k}\right\Vert _{L^{r}}\leq \left\Vert v\right\Vert
_{L^{q}}\sum_{k\in 
\mathbb{Z}
^{n}}\left\Vert u_{k}\right\Vert _{L^{p}}<\infty .
\end{equation*}%
Hence $v\ast u\in S_{w}^{r}$ and $\left\Vert v\ast u\right\Vert
_{S_{w}^{r}}\leq Cst\left\Vert v\right\Vert _{L^{q}}\left\Vert u\right\Vert
_{S_{w}^{p}}$.
\end{proof}

\section{Schatten-class properties of pseudo-differential operators ($%
\protect\tau \in \mathtt{End}_{%
\mathbb{R}
}\left( 
\mathbb{R}
^{n}\right) $)\label{st33}}

The spectral characterization of ideals $S_{w}^{p}$ allows us to deal with
Schatten-von Neumann class properties of pseudo-differential operators.

Let $\tau \in \mathtt{End}_{%
\mathbb{R}
}\left( 
\mathbb{R}
^{n}\right) \equiv M_{n\times n}\left( 
\mathbb{R}
\right) $, $a\in \mathcal{S}\left( \mathbb{R}^{n}\times \mathbb{R}%
^{n}\right) $, $v\in \mathcal{S}\left( \mathbb{R}^{n}\times \mathbb{R}%
^{n}\right) $. We define 
\begin{eqnarray*}
\mathtt{Op}_{\tau }\left( a\right) v\left( x\right) &=&a^{\tau }\left(
X,D\right) v\left( x\right) \\
&=&\left( 2\pi \right) ^{-n}\iint \mathtt{e}^{\mathtt{i}\left\langle
x-y,\eta \right\rangle }a\left( \left( 1-\tau \right) x+\tau y,\eta \right)
v\left( y\right) \mathtt{d}y\mathtt{d}\eta .
\end{eqnarray*}%
If $u,v\in \mathcal{S}\left( \mathbb{R}^{n}\right) $, then%
\begin{eqnarray*}
\left\langle \mathtt{Op}_{\tau }\left( a\right) v,u\right\rangle &=&\left(
2\pi \right) ^{-n}\iiint \mathtt{e}^{\mathtt{i}\left\langle x-y,\eta
\right\rangle }a\left( \left( 1-\tau \right) x+\tau y,\eta \right) u\left(
x\right) v\left( y\right) \mathtt{d}x\mathtt{d}y\mathtt{d}\eta \\
&=&\left\langle \left( \left( 1\otimes \mathcal{F}^{-1}\right) a\right)
\circ \mathtt{C}_{\tau },u\otimes v\right\rangle ,
\end{eqnarray*}%
where 
\begin{equation*}
\mathtt{C}_{\tau }:\mathbb{R}^{n}\times \mathbb{R}^{n}\rightarrow \mathbb{R}%
^{n}\times \mathbb{R}^{n},\quad \mathtt{C}_{\tau }\left( x,y\right) =\left(
\left( 1-\tau \right) x+\tau y,x-y\right) .
\end{equation*}%
We can define $\mathtt{Op}_{\tau }\left( a\right) $ as an operator in $%
\mathcal{B}\left( \mathcal{S},\mathcal{S}^{\prime }\right) $ for any $a\in 
\mathcal{S}^{\prime }\left( \mathbb{R}^{n}\times \mathbb{R}^{n}\right) $ by%
\begin{eqnarray*}
\left\langle \mathtt{Op}_{\tau }\left( a\right) v,u\right\rangle _{\mathcal{S%
},\mathcal{S}^{\prime }} &=&\left\langle \mathcal{K}_{\mathtt{Op}_{\tau
}\left( a\right) },u\otimes v\right\rangle , \\
\mathcal{K}_{\mathtt{Op}_{\tau }\left( a\right) } &=&\left( \left( 1\otimes 
\mathcal{F}^{-1}\right) a\right) \circ \mathtt{C}_{\tau }
\end{eqnarray*}

Let $a=\mathtt{e}^{\mathtt{i}\left( \left\langle x,k\right\rangle
+\left\langle l,\eta \right\rangle \right) }b=\left( \mathtt{e}^{\mathtt{i}%
\left\langle \cdot ,k\right\rangle }\otimes \mathtt{e}^{\mathtt{i}%
\left\langle l,\cdot \right\rangle }\right) b$, $b\in \mathcal{S}\left( 
\mathbb{R}^{n}\times \mathbb{R}^{n}\right) $. Then 
\begin{multline*}
\left\langle \mathtt{Op}_{\tau }\left( a\right) v,u\right\rangle \\
=\left( 2\pi \right) ^{-n}\iiint \mathtt{e}^{\mathtt{i}\left\langle x-y,\eta
\right\rangle }\mathtt{e}^{\mathtt{i}\left( \left\langle \left( 1-\tau
\right) x+\tau y,k\right\rangle +\left\langle l,\eta \right\rangle \right)
}b\left( \left( 1-\tau \right) x+\tau y,\eta \right) u\left( x\right)
v\left( y\right) \mathtt{d}x\mathtt{d}y\mathtt{d}\eta \\
=\left( 2\pi \right) ^{-n}\iiint \mathtt{e}^{\mathtt{i}\left\langle x+\tau
l-\left( y-\left( 1-\tau \right) l\right) ,\eta \right\rangle }b\left(
\left( 1-\tau \right) \left( x+\tau l\right) +\tau \left( y-\left( 1-\tau
\right) l\right) ,\eta \right) \\
\cdot u\left( x\right) \mathtt{e}^{\mathtt{i}\left\langle \left( 1-\tau
\right) x,k\right\rangle }v\left( y\right) \mathtt{e}^{\mathtt{i}%
\left\langle \tau y,k\right\rangle }\mathtt{d}x\mathtt{d}y\mathtt{d}\eta .
\end{multline*}%
Using the change of variables $X=x+\tau l$, $Y=y-\left( 1-\tau \right) l$ we
obtain further that%
\begin{multline*}
\left\langle \mathtt{Op}_{\tau }\left( a\right) v,u\right\rangle \\
=\left( 2\pi \right) ^{-n}\iiint \mathtt{e}^{\mathtt{i}\left\langle X-Y,\eta
\right\rangle }b\left( \left( 1-\tau \right) X+\tau Y,\eta \right) u\left(
X-\tau l\right) \mathtt{e}^{\mathtt{i}\left\langle \left( 1-\tau \right)
X,k\right\rangle } \\
\cdot v\left( Y+\left( 1-\tau \right) l\right) \mathtt{e}^{\mathtt{i}%
\left\langle \tau Y,k\right\rangle }\mathtt{e}^{\mathtt{i}\left(
-\left\langle \left( 1-\tau \right) \tau l,k\right\rangle +\left\langle
\left( 1-\tau \right) \tau l,k\right\rangle \right) }\mathtt{d}X\mathtt{d}Y%
\mathtt{d}\eta \\
=\left( 2\pi \right) ^{-n}\iiint \mathtt{e}^{\mathtt{i}\left\langle X-Y,\eta
\right\rangle }b\left( \left( 1-\tau \right) X+\tau Y,\eta \right) u\left(
X-\tau l\right) \mathtt{e}^{\mathtt{i}\left\langle \left( 1-\tau \right)
X,k\right\rangle } \\
\cdot v\left( Y+\left( 1-\tau \right) l\right) \mathtt{e}^{\mathtt{i}%
\left\langle \tau Y,k\right\rangle }\mathtt{d}X\mathtt{d}Y\mathtt{d}\eta \\
=\left\langle \mathtt{Op}_{\tau }\left( b\right) V\left( \tau ,k,l\right)
v,U\left( \tau ,k,l\right) u\right\rangle \\
=\left\langle U\left( \tau ,k,l\right) ^{\ast }\mathtt{Op}_{\tau }\left(
b\right) V\left( \tau ,k,l\right) v,u\right\rangle .
\end{multline*}%
Hence%
\begin{equation*}
\mathtt{Op}_{\tau }\left( a\right) =U\left( \tau ,k,l\right) ^{\ast }\mathtt{%
Op}_{\tau }\left( b\right) V\left( \tau ,k,l\right)
\end{equation*}%
where $V\left( \tau ,k,l\right) $ and $U\left( k,l\right) $ are unitary
operators defined by 
\begin{eqnarray*}
V\left( \tau ,k,l\right) v\left( y\right) &=&v\left( y+\left( 1-\tau \right)
l\right) \mathtt{e}^{\mathtt{i}\left\langle \tau y,k\right\rangle }, \\
U\left( \tau ,k,l\right) u\left( x\right) &=&u\left( x-\tau l\right) \mathtt{%
e}^{\mathtt{i}\left\langle \left( 1-\tau \right) x,k\right\rangle }.
\end{eqnarray*}

\begin{theorem}
\label{st25}Let $1\leq p<\infty ,$ $\tau \in 
\mathbb{R}
$ and $a\in S_{w}^{p}\left( \mathbb{R}^{n}\times \mathbb{R}^{n}\right) $.
Then%
\begin{equation*}
\mathtt{Op}_{\tau }\left( a\right) =a^{\tau }\left( X,D\right) \in \mathcal{B%
}_{p}\left( L^{2}\left( 
\mathbb{R}
^{n}\right) \right)
\end{equation*}%
where $\mathcal{B}_{p}\left( L^{2}\left( 
\mathbb{R}
^{n}\right) \right) $ denote the Schatten ideal of compact operators whose
singular values lie in $l^{p}$. We have%
\begin{equation*}
\left\Vert \mathtt{Op}_{\tau }\left( a\right) \right\Vert _{\mathcal{B}%
_{p}}\leq Cst\left\Vert a\right\Vert _{S_{w}^{p}}.
\end{equation*}
\end{theorem}

\begin{proof}
Let $a\in S_{w}^{p}\left( \mathbb{R}^{n}\times \mathbb{R}^{n}\right) $. By
corollary \ref{st14} and lemma \ref{st11} there is a sequence $\left(
a_{kl}\right) _{k,l\in 
\mathbb{Z}
^{n}}\subset L^{p}\left( \mathbb{R}^{2n}\right) \cap \mathcal{C}^{\infty
}\left( \mathbb{R}^{2n}\right) $ such that 
\begin{equation*}
a\left( x,\eta \right) =\sum_{k,l\in 
\mathbb{Z}
^{n}}\mathtt{e}^{\mathtt{i}\left( \left\langle x,k\right\rangle
+\left\langle l,\eta \right\rangle \right) }a_{kl}\left( x,\eta \right)
\end{equation*}%
with convergence in $\mathcal{S}^{\prime }\left( \mathbb{R}^{n}\times 
\mathbb{R}^{n}\right) $ and 
\begin{gather*}
\left\Vert \partial ^{\alpha }a_{kl}\right\Vert _{L^{p}}\leq C_{\alpha
}\left\Vert a_{kl}\right\Vert _{L^{p}},\quad k,l\in 
\mathbb{Z}
^{n},\alpha \in 
\mathbb{N}
^{2n}, \\
\sum_{k,l\in 
\mathbb{Z}
^{n}}\left\Vert a_{kl}\right\Vert _{L^{p}}\approx \left\Vert a\right\Vert
_{S_{w}^{p}}.
\end{gather*}%
It follows that 
\begin{eqnarray*}
\mathtt{Op}_{\tau }\left( a\right) &=&\sum_{k,l\in 
\mathbb{Z}
^{n}}\mathtt{Op}_{\tau }\left( \left( \mathtt{e}^{\mathtt{i}\left\langle
\cdot ,k\right\rangle }\otimes \mathtt{e}^{\mathtt{i}\left\langle l,\cdot
\right\rangle }\right) a_{kl}\right) \\
&=&\sum_{k,l\in 
\mathbb{Z}
^{n}}U\left( \tau ,k,l\right) ^{\ast }\mathtt{Op}_{\tau }\left(
a_{kl}\right) V\left( \tau ,k,l\right) .
\end{eqnarray*}%
Since for any $N\in 
\mathbb{N}
$ there is $C_{N}>0$ such that 
\begin{equation*}
\sup_{\left\vert \alpha \right\vert ,\left\vert \beta \right\vert \leq
N}\left\Vert \partial _{x}^{\alpha }\partial _{\xi }^{\beta
}a_{kl}\right\Vert _{L^{p}}\leq C_{N}\left\Vert a_{kl}\right\Vert
_{L^{p}},\quad k,l\in 
\mathbb{Z}
^{n},
\end{equation*}%
then using results from \cite{Arsu 1} and \cite{Arsu 2}, it follows that $%
\mathtt{Op}_{\tau }\left( a_{kl}\right) \in \mathcal{B}_{p}\left(
L^{2}\left( 
\mathbb{R}
^{n}\right) \right) $ and there is $N\in 
\mathbb{N}
$ such that%
\begin{eqnarray*}
\left\Vert \mathtt{Op}_{\tau }\left( a_{kl}\right) \right\Vert _{\mathcal{B}%
_{p}} &\leq &Cst\sup_{\left\vert \alpha \right\vert ,\left\vert \beta
\right\vert \leq N}\left\Vert \partial _{x}^{\alpha }\partial _{y}^{\beta
}a_{kl}\right\Vert _{L^{p}} \\
&\leq &Cst\left\Vert a_{kl}\right\Vert _{L^{p}}.
\end{eqnarray*}%
Since $V\left( \tau ,k,l\right) $ and $U\left( \tau ,k,l\right) $ are
unitary operators it follows that the series 
\begin{equation*}
\sum_{k,l\in 
\mathbb{Z}
^{n}}U\left( \tau ,k,l\right) ^{\ast }\mathtt{Op}_{\tau }\left(
a_{kl}\right) V\left( \tau ,k,l\right)
\end{equation*}%
converges in $\mathcal{B}_{p}\left( L^{2}\left( 
\mathbb{R}
^{n}\right) \right) $. Hence $\mathtt{Op}_{\tau }\left( a\right) =a^{\tau
}\left( X,D\right) \in \mathcal{B}_{p}\left( L^{2}\left( 
\mathbb{R}
^{n}\right) \right) $ and%
\begin{equation*}
\left\Vert \mathtt{Op}_{\tau }\left( a\right) \right\Vert _{\mathcal{B}%
_{p}}\leq Cst\sum_{k,l\in 
\mathbb{Z}
^{n}}\left\Vert a_{kl}\right\Vert _{L^{p}}\leq Cst\left\Vert a\right\Vert
_{S_{w}^{p}}.
\end{equation*}
\end{proof}

Let $\tau \in \mathtt{End}_{%
\mathbb{R}
}\left( 
\mathbb{R}
^{n}\right) \equiv M_{n\times n}\left( 
\mathbb{R}
\right) $. We associate the quadratic form%
\begin{eqnarray*}
\theta _{\tau } &:&%
\mathbb{R}
^{n}\times 
\mathbb{R}
^{n}\rightarrow 
\mathbb{R}
, \\
\theta _{\tau }\left( x,\xi \right) &=&\left\langle \tau x,\xi \right\rangle
=\frac{1}{2}\left( \left\langle \tau x,\xi \right\rangle +\left\langle
x,^{t}\tau \xi \right\rangle \right) \\
&=&\frac{1}{2}\left\langle \left( 
\begin{array}{cc}
0 & ^{t}\tau \\ 
\tau & 0%
\end{array}%
\right) \left( 
\begin{array}{c}
x \\ 
\xi%
\end{array}%
\right) ,\left( 
\begin{array}{c}
x \\ 
\xi%
\end{array}%
\right) \right\rangle .
\end{eqnarray*}%
The symmetric operator associated to $\theta _{\tau }$ is given by the
formula 
\begin{eqnarray*}
A_{\theta _{\tau }} &\equiv &A_{\tau } \\
&=&-\left( 
\begin{array}{cc}
0 & ^{t}\tau \\ 
\tau & 0%
\end{array}%
\right) \\
&=&-\left( 
\begin{array}{cc}
^{t}\tau & 0 \\ 
0 & \mathtt{I}%
\end{array}%
\right) \left( 
\begin{array}{cc}
\mathtt{I} & 0 \\ 
0 & \tau%
\end{array}%
\right) \left( 
\begin{array}{cc}
0 & \mathtt{I} \\ 
\mathtt{I} & 0%
\end{array}%
\right) .
\end{eqnarray*}%
We have 
\begin{equation*}
\det A_{\tau }=\left( -1\right) ^{n}\left( \det \tau \right) ^{2}.
\end{equation*}

Let $\tau \in \mathtt{End}_{%
\mathbb{R}
}\left( 
\mathbb{R}
^{n}\right) \equiv M_{n\times n}\left( 
\mathbb{R}
\right) $ and $a\in S_{w}^{p}\left( \mathbb{R}^{n}\times \mathbb{R}%
^{n}\right) $. The Weyl symbol $a_{w}^{\tau }$ of the operator $a^{\tau
}\left( X,D\right) $ is given by 
\begin{equation*}
a_{w}^{\tau }=\mathtt{e}^{\mathtt{i}\theta _{\frac{1}{2}-\tau }\left(
D_{x},D_{\xi }\right) }a=\mathtt{e}^{\mathtt{i}\left\langle \left( \frac{1}{2%
}-\tau \right) D_{x},D_{\xi }\right\rangle }a
\end{equation*}%
where 
\begin{equation*}
\left\langle \left( \frac{1}{2}-\tau \right) D_{x},D_{\xi }\right\rangle
=\sum_{k,l=1}^{n}\left( \frac{1}{2}\delta _{kl}-\tau _{kl}\right)
D_{x_{l}}D_{\xi _{k}}.
\end{equation*}

Let $\tau \in \mathtt{End}_{%
\mathbb{R}
}\left( 
\mathbb{R}
^{n}\right) \equiv M_{n\times n}\left( 
\mathbb{R}
\right) $ and $\lambda \in 
\mathbb{R}
\smallsetminus \left\{ \frac{1}{2}\right\} $ be such that $\det \left(
\lambda -\tau \right) \neq 0$. If $a\in S_{w}^{p}\left( \mathbb{R}^{n}\times 
\mathbb{R}^{n}\right) $, then using lemma \ref{st23} twice we obtain that%
\begin{gather*}
a_{w}^{\lambda }=\mathtt{e}^{\mathtt{i}\theta _{\frac{1}{2}-\lambda }\left(
D_{x},D_{\xi }\right) }a=\mathtt{e}^{\mathtt{i}\left\langle \left( \frac{1}{2%
}-\lambda \right) D_{x},D_{\xi }\right\rangle }a\in S_{w}^{p}\left( \mathbb{R%
}^{n}\times \mathbb{R}^{n}\right) , \\
\left\Vert a_{w}^{\lambda }\right\Vert _{S_{w}^{p}}\leq Cst\left\Vert
a\right\Vert _{S_{w}^{p}}
\end{gather*}%
and 
\begin{gather*}
a_{w}^{\tau }=\mathtt{e}^{\mathtt{i}\theta _{\lambda -\tau }\left(
D_{x},D_{\xi }\right) }a_{w}^{\lambda }=\mathtt{e}^{\mathtt{i}\left\langle
\left( \lambda -\tau \right) D_{x},D_{\xi }\right\rangle }a_{w}^{\lambda
}\in S_{w}^{p}\left( \mathbb{R}^{n}\times \mathbb{R}^{n}\right) , \\
\left\Vert a_{w}^{\tau }\right\Vert _{S_{w}^{p}}\leq Cst\left\Vert
a_{w}^{\lambda }\right\Vert _{S_{w}^{p}}\leq Cst\left\Vert a\right\Vert
_{S_{w}^{p}}.
\end{gather*}

We remark that $a^{\tau }\left( X,D\right) =\mathtt{Op}_{\tau }\left(
a\right) =\mathtt{Op}_{\frac{1}{2}}\left( a_{w}^{\tau }\right) =\left(
a_{w}^{\tau }\right) ^{\frac{1}{2}}\left( X,D\right) $ and $a_{w}^{\tau }\in
S_{w}^{p}\left( \mathbb{R}^{n}\times \mathbb{R}^{n}\right) $. From the
previous theorem it follows that for any $\tau \in \mathtt{End}_{%
\mathbb{R}
}\left( 
\mathbb{R}
^{n}\right) \equiv M_{n\times n}\left( 
\mathbb{R}
\right) $ and for any $a\in S_{w}^{p}\left( \mathbb{R}^{n}\times \mathbb{R}%
^{n}\right) $ 
\begin{equation*}
\mathtt{Op}_{\tau }\left( a\right) =a^{\tau }\left( X,D\right) \in \mathcal{B%
}_{p}\left( L^{2}\left( 
\mathbb{R}
^{n}\right) \right)
\end{equation*}
and 
\begin{equation*}
\left\Vert \mathtt{Op}_{\tau }\left( a\right) \right\Vert _{\mathcal{B}%
_{p}}\leq Cst\left\Vert a\right\Vert _{S_{w}^{p}}.
\end{equation*}%
Thus we obtain an extension of the previous theorem to the case $\tau \in 
\mathtt{End}_{%
\mathbb{R}
}\left( 
\mathbb{R}
^{n}\right) \equiv M_{n\times n}\left( 
\mathbb{R}
\right) $.

\begin{theorem}
\label{st21}Let $1\leq p<\infty ,$ $\tau \in \mathtt{End}_{%
\mathbb{R}
}\left( 
\mathbb{R}
^{n}\right) \equiv M_{n\times n}\left( 
\mathbb{R}
\right) $ and $a\in S_{w}^{p}\left( \mathbb{R}^{n}\times \mathbb{R}%
^{n}\right) $. Then 
\begin{equation*}
\mathtt{Op}_{\tau }\left( a\right) =a^{\tau }\left( X,D\right) \in \mathcal{B%
}_{p}\left( L^{2}\left( 
\mathbb{R}
^{n}\right) \right)
\end{equation*}%
where $\mathcal{B}_{p}\left( L^{2}\left( 
\mathbb{R}
^{n}\right) \right) $ denote the Schatten ideal of compact operators whose
singular values lie in $l^{p}$. We have%
\begin{equation*}
\left\Vert \mathtt{Op}_{\tau }\left( a\right) \right\Vert _{\mathcal{B}%
_{p}}\leq Cst\left\Vert a\right\Vert _{S_{w}^{p}}.
\end{equation*}
\end{theorem}

Let $\tau _{0}\in \mathtt{End}_{%
\mathbb{R}
}\left( 
\mathbb{R}
^{n}\right) \equiv M_{n\times n}\left( 
\mathbb{R}
\right) $ and $\lambda _{0}\in 
\mathbb{R}
\smallsetminus \left\{ \frac{1}{2}\right\} $ be such that $\det \left(
\lambda _{0}-\tau _{0}\right) \neq 0$. Put%
\begin{equation*}
V_{0}=\left\{ \tau \in \mathtt{End}_{%
\mathbb{R}
}\left( 
\mathbb{R}
^{n}\right) \equiv M_{n\times n}\left( 
\mathbb{R}
\right) :\det \left( \lambda _{0}-\tau \right) \neq 0\right\}
\end{equation*}%
Then $V_{0}$ is a neighborhood of $\tau _{0}$.

If $\tau \in V_{0}$ and $a\in S_{w}^{p}\left( \mathbb{R}^{n}\times \mathbb{R}%
^{n}\right) $ then 
\begin{gather*}
a_{w}^{\tau }=\mathtt{e}^{\mathtt{i}\theta _{\lambda _{0}-\tau }\left(
D_{x},D_{\xi }\right) }a_{w}^{\lambda _{0}}=\mathtt{e}^{\mathtt{i}%
\left\langle \left( \lambda _{0}-\tau \right) D_{x},D_{\xi }\right\rangle
}a_{w}^{\lambda _{0}}, \\
a_{w}^{\lambda _{0}}=\mathtt{e}^{\mathtt{i}\theta _{\frac{1}{2}-\lambda
_{0}}\left( D_{x},D_{\xi }\right) }a=\mathtt{e}^{\mathtt{i}\left\langle
\left( \frac{1}{2}-\lambda _{0}\right) D_{x},D_{\xi }\right\rangle }a\in
S_{w}^{p}\left( \mathbb{R}^{n}\times \mathbb{R}^{n}\right) ,
\end{gather*}%
hence%
\begin{eqnarray*}
\mathtt{Op}_{\tau }\left( a\right) -\mathtt{Op}_{\tau _{0}}\left( a\right)
&=&\mathtt{Op}_{\frac{1}{2}}\left( \mathtt{e}^{\mathtt{i}\left\langle \left(
\lambda _{0}-\tau \right) D_{x},D_{\xi }\right\rangle }a_{w}^{\lambda _{0}}-%
\mathtt{e}^{\mathtt{i}\left\langle \left( \lambda _{0}-\tau _{0}\right)
D_{x},D_{\xi }\right\rangle }a_{w}^{\lambda _{0}}\right) , \\
a^{\tau }\left( X,D\right) -a^{\tau _{0}}\left( X,D\right) &=&\left( \mathtt{%
e}^{\mathtt{i}\left\langle \left( \lambda _{0}-\tau \right) D_{x},D_{\xi
}\right\rangle }a_{w}^{\lambda _{0}}-\mathtt{e}^{\mathtt{i}\left\langle
\left( \lambda _{0}-\tau _{0}\right) D_{x},D_{\xi }\right\rangle
}a_{w}^{\lambda _{0}}\right) ^{\frac{1}{2}}\left( X,D\right) .
\end{eqnarray*}%
From the previous theorem and lemma \ref{st23} it follows that 
\begin{equation*}
\left\Vert \mathtt{Op}_{\tau }\left( a\right) -\mathtt{Op}_{\tau _{0}}\left(
a\right) \right\Vert _{\mathcal{B}_{p}}\leq Cst\left\Vert \mathtt{e}^{%
\mathtt{i}\left\langle \left( \lambda _{0}-\tau \right) D_{x},D_{\xi
}\right\rangle }a_{w}^{\lambda _{0}}-\mathtt{e}^{\mathtt{i}\left\langle
\left( \lambda _{0}-\tau _{0}\right) D_{x},D_{\xi }\right\rangle
}a_{w}^{\lambda _{0}}\right\Vert _{S_{w}^{p}}\rightarrow 0
\end{equation*}%
when $\tau \rightarrow \tau _{0}$. Thus we proved the following result.

\begin{theorem}
\label{st30}Let $1\leq p<\infty $ and $a\in S_{w}^{p}\left( \mathbb{R}%
^{n}\times \mathbb{R}^{n}\right) $. Then the map 
\begin{equation*}
\mathtt{End}_{%
\mathbb{R}
}\left( 
\mathbb{R}
^{n}\right) \ni \tau \rightarrow \mathtt{Op}_{\tau }\left( a\right) =a^{\tau
}\left( X,D\right) \in \mathcal{B}_{p}\left( L^{2}\left( 
\mathbb{R}
^{n}\right) \right)
\end{equation*}%
is continuous.
\end{theorem}

\section{Schatten-class properties of pseudo-differential operators
(continue)\label{st34}}

In this section we consider pseudo-differential operators defined by 
\begin{equation}
\mathtt{Op}\left( a\right) v\left( x\right) =\iint \mathtt{e}^{\mathtt{i}%
\left\langle x-y,\theta \right\rangle }a\left( x,y,\theta \right) v\left(
y\right) \mathtt{d}y\mathtt{d}\theta ,\quad v\in \mathcal{S}\left( \mathbb{R}%
^{n}\right) ,  \label{st20}
\end{equation}%
with $a\in S_{w}^{p}\left( 
\mathbb{R}
^{3n}\right) $, $1\leq p\leq \infty $.

For $a\in S_{w}^{\infty }\left( 
\mathbb{R}
^{3n}\right) $ such operators and Fourier integral operators were studied by
A. Boulkhemair in \cite{Boulkhemair 2}. In \cite{Boulkhemair 2}, the author
give a meaning to the above integral and proves $L^{2}$-boundedness of
global non-degenerate Fourier integral operators related to the Sj\"{o}%
strand class $S_{w}=S_{w}^{\infty }$. The results that we need are proved in
the appendix \ref{st38}.

Assume for the moment that $a\in \mathcal{S}\left( \mathbb{R}^{3n}\right) $.
Then there is $a_{0}\in \mathcal{S}^{\prime }\left( \mathbb{R}^{n}\times 
\mathbb{R}^{n}\right) $ so that $\mathtt{Op}\left( a\right) =\mathtt{Op}%
_{0}\left( a_{0}\right) $. We shall now derive an expression which will
connect the symbol $a_{0}$ with the amplitude $a$. One can express $a_{0}$
in terms of $a$ by means of Fourier's inversion formula and Fubini's
theorem. Let $u$, $v\in \mathcal{S}\left( \mathbb{R}^{n}\right) $. Then 
\begin{multline*}
\left\langle \mathtt{Op}\left( a\right) v,u\right\rangle =\iiint \mathtt{e}^{%
\mathtt{i}\left\langle x-y,\theta \right\rangle }a\left( x,y,\theta \right)
u\left( x\right) v\left( y\right) \mathtt{d}x\mathtt{d}y\mathtt{d}\theta \\
=\left( 2\pi \right) ^{-n}\iiint \mathtt{e}^{\mathtt{i}\left\langle
x-y,\theta \right\rangle }a\left( x,y,\theta \right) u\left( x\right) \left(
\int \mathtt{e}^{\mathtt{i}\left\langle y,\xi \right\rangle }\widehat{v}%
\left( \xi \right) \mathtt{d}\xi \right) \mathtt{d}x\mathtt{d}y\mathtt{d}%
\theta \\
=\left( 2\pi \right) ^{-n}\iint \mathtt{e}^{\mathtt{i}\left\langle x,\xi
\right\rangle }\left( \iint \mathtt{e}^{\mathtt{i}\left( \left\langle
x-y,\theta \right\rangle -\left\langle x,\xi \right\rangle +\left\langle
y,\xi \right\rangle \right) }a\left( x,y,\theta \right) \mathtt{d}y\mathtt{d}%
\theta \right) u\left( x\right) \widehat{v}\left( \xi \right) \mathtt{d}x%
\mathtt{d}\xi \\
=\left( 2\pi \right) ^{-n}\iint \mathtt{e}^{\mathtt{i}\left\langle x,\xi
\right\rangle }\left( \iint \mathtt{e}^{-\mathtt{i}\left\langle x-y,\xi
-\theta \right\rangle }a\left( x,y,\theta \right) \mathtt{d}y\mathtt{d}%
\theta \right) u\left( x\right) \widehat{v}\left( \xi \right) \mathtt{d}x%
\mathtt{d}\xi \\
=\left\langle \mathtt{Op}_{0}\left( a_{0}\right) v,u\right\rangle
\end{multline*}%
with 
\begin{eqnarray*}
a_{0}\left( x,\xi \right) &=&\iint \mathtt{e}^{-\mathtt{i}\left\langle
x-y,\xi -\theta \right\rangle }a\left( x,y,\theta \right) \mathtt{d}y\mathtt{%
d}\theta \\
&=&\left( \delta \otimes \mathtt{e}^{\mathtt{i}\Phi }\right) \ast a\left(
x,x,\xi \right)
\end{eqnarray*}%
where $\Phi :\mathbb{R}^{2n}\rightarrow \mathbb{R}$ is the non-degenerate
quadratic form defined by $\Phi \left( y,\theta \right) =-\left\langle
y,\theta \right\rangle $. Hence 
\begin{equation*}
\mathtt{Op}\left( a\right) =\mathtt{Op}_{0}\left( a_{0}\right) ,\quad
a_{0}=\left( \delta \otimes \mathtt{e}^{\mathtt{i}\Phi }\right) \ast
a_{|L},\quad L=\left\{ \left( x,x,\xi \right) :\left( x,\xi \right) \in 
\mathbb{R}^{2n}\right\} .
\end{equation*}

Next we show that $a\in S_{w}^{p}\left( 
\mathbb{R}
^{3n}\right) \Rightarrow a_{0}\in S_{w}^{p}\left( 
\mathbb{R}
^{2n}\right) $.

Let $A$ be a real non-singular symmetric $n\times n$ matrix. We denote by $%
\Phi =\Phi _{A}$ the real non-singular quadratic form defined by $\Phi
_{A}\left( x\right) =-\left\langle Ax,x\right\rangle /2$, $x\in \mathbb{R}%
^{n}$. We first recall that 
\begin{equation*}
\widehat{\mathtt{e}^{\mathtt{i}\Phi _{A}}}\left( \xi \right) =\left( 2\pi
\right) ^{n/2}\left\vert \det A\right\vert ^{-\frac{1}{2}}\mathtt{e}^{\frac{%
\pi \mathtt{isgn}A}{4}}\mathtt{e}^{-\mathtt{i}\left\langle A^{-1}\xi ,\xi
\right\rangle /2}
\end{equation*}%
We shall consider the operator $S_{A}=I\otimes T_{A}:\mathcal{S}\left( 
\mathbb{R}^{m}\times \mathbb{R}^{n}\right) \rightarrow \mathcal{S}^{\prime
}\left( \mathbb{R}^{m}\times \mathbb{R}^{n}\right) $ defined by 
\begin{equation*}
S_{A}u=\left( 2\pi \right) ^{-n/2}\left\vert \det A\right\vert ^{\frac{1}{2}}%
\mathtt{e}^{-\frac{\pi \mathtt{isgn}A}{4}}\left( \delta \otimes \mathtt{e}^{%
\mathtt{i}\Phi _{A}}\right) \ast u.
\end{equation*}%
Then%
\begin{eqnarray*}
\widehat{S_{A}u} &=&\left( 2\pi \right) ^{-n/2}\left\vert \det A\right\vert
^{\frac{1}{2}}\mathtt{e}^{-\frac{\pi \mathtt{isgn}A}{4}}\left( 1\otimes 
\widehat{\mathtt{e}^{\mathtt{i}\Phi _{A}}}\right) \cdot \widehat{u}=\left(
1\otimes \mathtt{e}^{\mathtt{i}\Phi _{A^{-1}}}\right) \cdot \widehat{u}, \\
\widehat{S_{-A}u} &=&\left( 1\otimes \mathtt{e}^{\mathtt{i}\Phi
_{-A^{-1}}}\right) \cdot \widehat{u},
\end{eqnarray*}%
so $S_{A}:\mathcal{S}\left( \mathbb{R}^{m}\times \mathbb{R}^{n}\right)
\rightarrow \mathcal{S}\left( \mathbb{R}^{m}\times \mathbb{R}^{n}\right) $
is invertible and $S_{A}^{-1}=S_{-A}$.

\begin{remark}
$\mathtt{supp}\left( \widehat{S_{A}u}\right) =\mathtt{supp}\left( \widehat{u}%
\right) .$
\end{remark}

\begin{lemma}
Let $1\leq p\leq \infty $ and $u\in S_{w}^{p}\left( \mathbb{R}^{m}\times 
\mathbb{R}^{n}\right) $. Then $S_{A}u\in S_{w}^{p}\left( \mathbb{R}%
^{m}\times \mathbb{R}^{n}\right) $. The map $S_{A}:S_{w}^{p}\left( \mathbb{R}%
^{m}\times \mathbb{R}^{n}\right) \rightarrow S_{w}^{p}\left( \mathbb{R}%
^{m}\times \mathbb{R}^{n}\right) $ is continuous.
\end{lemma}

\begin{proof}
There is a compact set $Q=Q_{\mathbb{%
\mathbb{Z}
}^{m}\times \mathbb{%
\mathbb{Z}
}^{n}}\subset \mathbb{R}^{m}\times \mathbb{R}^{n}$ such that each $u\in
S_{w}^{p}\left( \mathbb{R}^{m}\times \mathbb{R}^{n}\right) $ can be written
as a series 
\begin{equation*}
u=\sum_{\left( k,l\right) \in \mathbb{%
\mathbb{Z}
}^{m}\times \mathbb{%
\mathbb{Z}
}^{n}}u_{kl}
\end{equation*}%
convergent in $\mathcal{S}^{\prime }\left( \mathbb{R}^{n}\times \mathbb{R}%
^{n}\right) $ with $\left( u_{kl}\right) _{k,l\in 
\mathbb{Z}
^{n}}\subset L^{p}\left( \mathbb{R}^{m}\times \mathbb{R}^{n}\right) $
satisfying%
\begin{gather*}
\mathtt{supp}\left( \widehat{u}_{kl}\right) \subset \left( k,l\right)
+Q,\quad \left( k,l\right) \in \mathbb{%
\mathbb{Z}
}^{m}\times \mathbb{%
\mathbb{Z}
}^{n}, \\
\sum_{\left( k,l\right) \in \mathbb{%
\mathbb{Z}
}^{m}\times \mathbb{%
\mathbb{Z}
}^{n}}\left\Vert u_{kl}\right\Vert _{L^{p}}<\infty ,\sum_{\left( k,l\right)
\in \mathbb{%
\mathbb{Z}
}^{m}\times \mathbb{%
\mathbb{Z}
}^{n}}\left\Vert u_{kl}\right\Vert _{L^{p}}\approx \left\Vert u\right\Vert
_{S_{w}^{p}}.
\end{gather*}%
Let us note that $\mathtt{supp}\left( \widehat{S_{A}u_{kl}}\right) \subset 
\mathtt{supp}\left( \widehat{u}_{kl}\right) \subset \left( k,l\right)
+Q,\quad \left( k,l\right) \in \mathbb{%
\mathbb{Z}
}^{m}\times \mathbb{%
\mathbb{Z}
}^{n}$.

Let $\chi \in \mathcal{C}_{0}^{\infty }\left( \mathbb{R}^{n}\right) $ and $%
h\in \mathcal{C}_{0}^{\infty }\left( \mathbb{R}^{n}\right) $ be such that $%
\chi \otimes h=1$ on a neighborhood of $Q$. Then for any $\left( k,l\right)
\in \mathbb{%
\mathbb{Z}
}^{m}\times \mathbb{%
\mathbb{Z}
}^{n}$, $\widehat{u}_{kl}=\left( \tau _{k}\chi \otimes \tau _{l}h\right)
\cdot \widehat{u}_{kl}$ so 
\begin{equation*}
u_{kl}=\left( \mathcal{F}^{-1}\left( \tau _{k}\chi \right) \otimes \mathcal{F%
}^{-1}\left( \tau _{l}h\right) \right) \ast u_{kl}=\left( \left( \mathtt{e}^{%
\mathtt{i}\left\langle \cdot ,k\right\rangle }\psi \right) \otimes \left( 
\mathtt{e}^{\mathtt{i}\left\langle \cdot ,l\right\rangle }H\right) \right)
\ast u_{kl},
\end{equation*}%
where $\psi =\mathcal{F}^{-1}\left( \chi \right) $, $H=\mathcal{F}%
^{-1}\left( h\right) $.

If we set 
\begin{equation*}
C_{A,n}=\left( 2\pi \right) ^{-n/2}\left\vert \det A\right\vert ^{\frac{1}{2}%
}\mathtt{e}^{-\frac{\pi \mathtt{isgn}A}{4}},
\end{equation*}%
then 
\begin{equation*}
S_{A}u_{kl}=C_{A,n}\left( \delta \otimes \mathtt{e}^{\mathtt{i}\Phi
_{A}}\right) \ast u_{kl}=C_{A,n}\left( \delta \otimes \mathtt{e}^{\mathtt{i}%
\Phi _{A}}\right) \ast \left( \left( \mathtt{e}^{\mathtt{i}\left\langle
\cdot ,k\right\rangle }\psi \right) \otimes \left( \mathtt{e}^{\mathtt{i}%
\left\langle \cdot ,l\right\rangle }H\right) \right) \ast u_{kl}.
\end{equation*}%
We have%
\begin{eqnarray*}
\left( \delta \otimes \mathtt{e}^{\mathtt{i}\Phi _{A}}\right) \ast \left(
\left( \mathtt{e}^{\mathtt{i}\left\langle \cdot ,k\right\rangle }\psi
\right) \otimes \left( \mathtt{e}^{\mathtt{i}\left\langle \cdot
,l\right\rangle }H\right) \right) &=&\left( \mathtt{e}^{\mathtt{i}%
\left\langle \cdot ,k\right\rangle }\psi \right) \otimes \left( \mathtt{e}^{%
\mathtt{i}\Phi _{A}}\ast \left( \mathtt{e}^{\mathtt{i}\left\langle \cdot
,l\right\rangle }H\right) \right) \\
&=&\left( \mathtt{e}^{\mathtt{i}\left\langle \cdot ,k\right\rangle }\psi
\right) \otimes \left( \mathtt{e}^{\mathtt{i}\Phi _{A}}H_{A}\left( A\cdot
+l\right) \right)
\end{eqnarray*}%
since%
\begin{eqnarray*}
\mathtt{e}^{\mathtt{i}\Phi _{A}}\ast \left( \mathtt{e}^{\mathtt{i}%
\left\langle \cdot ,l\right\rangle }H\right) \left( x\right) &=&\dint 
\mathtt{e}^{\mathtt{i}\left[ -\left\langle A\left( x-y\right)
,x-y\right\rangle /2+\left\langle y,l\right\rangle \right] }H\left( y\right) 
\mathtt{d}y \\
&=&\mathtt{e}^{\mathtt{i}\Phi _{A}\left( x\right) }\dint \mathtt{e}^{\mathtt{%
i}\left\langle y,Ax+l\right\rangle }\mathtt{e}^{\mathtt{i}\Phi _{A}\left(
y\right) }H\left( y\right) \mathtt{d}y \\
&=&\mathtt{e}^{\mathtt{i}\Phi _{A}\left( x\right) }H_{A}\left( Ax+l\right) ,
\end{eqnarray*}%
where $H_{A}=\mathcal{F}\left( \mathtt{e}^{\mathtt{i}\Phi _{A}}\check{H}%
\right) $. It follows that 
\begin{eqnarray*}
\left\Vert S_{A}u_{kl}\right\Vert _{L^{p}} &\leq &\left( 2\pi \right)
^{-n/2}\left\vert \det A\right\vert ^{\frac{1}{2}}\left\Vert \mathtt{e}^{%
\mathtt{i}\left\langle \cdot ,k\right\rangle }\psi \right\Vert
_{L^{1}}\left\Vert H_{A}\left( A\cdot +l\right) \right\Vert
_{L^{1}}\left\Vert u_{kl}\right\Vert _{L^{p}} \\
&=&\left( 2\pi \right) ^{-n/2}\left\vert \det A\right\vert ^{-\frac{1}{2}%
}\left\Vert \psi \right\Vert _{L^{1}}\left\Vert H_{A}\right\Vert
_{L^{1}}\left\Vert u_{kl}\right\Vert _{L^{p}},\quad \left( k,l\right) \in 
\mathbb{%
\mathbb{Z}
}^{m}\times \mathbb{%
\mathbb{Z}
}^{n}.
\end{eqnarray*}%
By summing with respect to $\left( k,l\right) \in \mathbb{%
\mathbb{Z}
}^{m}\times \mathbb{%
\mathbb{Z}
}^{n}$ we get%
\begin{eqnarray*}
\left\Vert S_{A}u\right\Vert _{S_{w}^{p}} &\approx &\sum_{\left( k,l\right)
\in \mathbb{%
\mathbb{Z}
}^{m}\times \mathbb{%
\mathbb{Z}
}^{n}}\left\Vert S_{A}u_{kl}\right\Vert _{L^{p}} \\
&\leq &\left( 2\pi \right) ^{-n/2}\left\vert \det A\right\vert ^{-\frac{1}{2}%
}\left\Vert \psi \right\Vert _{L^{1}}\left\Vert H_{A}\right\Vert
_{L^{1}}\sum_{\left( k,l\right) \in \mathbb{%
\mathbb{Z}
}^{m}\times \mathbb{%
\mathbb{Z}
}^{n}}\left\Vert u_{kl}\right\Vert _{L^{p}} \\
&\approx &\left( 2\pi \right) ^{-n/2}\left\vert \det A\right\vert ^{-\frac{1%
}{2}}\left\Vert \psi \right\Vert _{L^{1}}\left\Vert H_{A}\right\Vert
_{L^{1}}\left\Vert u\right\Vert _{S_{w}^{p}}
\end{eqnarray*}
\end{proof}

Using the previous lemma and lemma \ref{st12} we obtain the desired result.

\begin{corollary}
$a\in S_{w}^{p}\left( 
\mathbb{R}
^{3n}\right) \Rightarrow a_{0}=\left( \delta \otimes \mathtt{e}^{\mathtt{i}%
\Phi }\right) \ast a_{|L}\in S_{w}^{p}\left( 
\mathbb{R}
^{2n}\right) $ and 
\begin{equation*}
\left\Vert a_{0}\right\Vert _{S_{w}^{p}}\leq Cst\left\Vert a\right\Vert
_{S_{w}^{p}}.
\end{equation*}
\end{corollary}

In \cite{Boulkhemair 2}, Boulkhemair extends the mapping $\mathtt{Op}\left(
\cdot \right) $ to elements of Sj\"{o}strand algebra $S_{w}=S_{w}^{\infty }$
and obtains a linear bounded operator $\mathtt{BOp:}S_{w}\left( 
\mathbb{R}
^{3n}\right) \rightarrow \mathtt{BOp}\left( a\right) \in \mathcal{B}\left(
L^{2}\left( \mathbb{R}^{n}\right) \right) $ which satisfies 
\begin{equation*}
\left\Vert \mathtt{BOp}\left( a\right) \right\Vert _{\mathcal{B}\left(
L^{2}\left( \mathbb{R}^{n}\right) \right) }\leq Cst\left\Vert a\right\Vert
_{S_{w}}.
\end{equation*}

Let $1\leq p<\infty $. By theorem \ref{st21} and the above corollary it
follows that for $a\in \mathcal{S}\left( \mathbb{R}^{3n}\right) $ we have%
\begin{equation*}
\mathtt{Op}\left( a\right) =\mathtt{Op}_{0}\left( a_{0}\right) \in \mathcal{B%
}_{p}\left( L^{2}\left( 
\mathbb{R}
^{n}\right) \right)
\end{equation*}%
and 
\begin{equation*}
\left\Vert \mathtt{Op}\left( a\right) \right\Vert _{\mathcal{B}%
_{p}}=\left\Vert \mathtt{Op}_{0}\left( a_{0}\right) \right\Vert _{\mathcal{B}%
_{p}}\leq Cst\left\Vert a_{0}\right\Vert _{S_{w}^{p}}\leq Cst\left\Vert
a\right\Vert _{S_{w}^{p}}.
\end{equation*}%
Since $\mathcal{S}\left( \mathbb{R}^{3n}\right) $ is dense in $%
S_{w}^{p}\left( \mathbb{R}^{3n}\right) $ (see proposition \ref{st22}), it
follows that $\mathtt{Op}$ extends to a unique operator 
\begin{equation*}
\mathtt{Op:}S_{w}^{p}\left( \mathbb{R}^{3n}\right) \rightarrow \mathcal{B}%
_{p}\left( L^{2}\left( 
\mathbb{R}
^{n}\right) \right)
\end{equation*}%
which is the restriction of Boulkhemair operator to $S_{w}^{p}\left( \mathbb{%
R}^{3n}\right) $. This is seen from the commutative diagram%
\begin{equation*}
\begin{array}{ccccc}
\mathcal{S}\left( \mathbb{R}^{3n}\right) & \overset{\mathtt{j}}{%
\hookrightarrow } & S_{w}^{p}\left( \mathbb{R}^{3n}\right) & \overset{%
\mathtt{j}_{p}}{\hookrightarrow } & S_{w}\left( \mathbb{R}^{3n}\right) \\ 
\downarrow & \quad \swarrow \mathtt{Op} &  & \searrow & \downarrow \mathtt{%
BOp} \\ 
\mathcal{B}_{p}\left( L^{2}\left( 
\mathbb{R}
^{n}\right) \right) &  & \underset{\mathtt{J}_{p}}{\hookrightarrow } &  & 
\mathcal{B}\left( L^{2}\left( 
\mathbb{R}
^{n}\right) \right)%
\end{array}%
\end{equation*}%
We have the bounded operators $\mathtt{BOp}_{|S_{w}^{p}\left( \mathbb{R}%
^{3n}\right) }=\mathtt{BOp}\circ \mathtt{j}_{p}:S_{w}^{p}\left( \mathbb{R}%
^{3n}\right) \rightarrow \mathcal{B}\left( L^{2}\left( 
\mathbb{R}
^{n}\right) \right) $ and $\mathtt{J}_{p}\circ \mathtt{Op}:S_{w}^{p}\left( 
\mathbb{R}^{3n}\right) \rightarrow \mathcal{B}_{p}\left( L^{2}\left( 
\mathbb{R}
^{n}\right) \right) $ which coincide on $\mathcal{S}\left( \mathbb{R}%
^{3n}\right) $. It follows that 
\begin{equation*}
\mathtt{BOp}_{|S_{w}^{p}\left( \mathbb{R}^{3n}\right) }=\mathtt{J}_{p}\circ 
\mathtt{Op}.
\end{equation*}

\begin{proposition}
Let $a\in S_{w}^{p}\left( \mathbb{R}^{3n}\right) $. Let $\mathtt{Op}\left(
a\right) $ be the operator defined by $\left( \ref{st16}\right) $, $\left( %
\ref{st17}\right) $. Then $\mathtt{Op}\left( a\right) \in \mathcal{B}%
_{p}\left( L^{2}\left( 
\mathbb{R}
^{n}\right) \right) $ and 
\begin{equation*}
\left\Vert \mathtt{Op}\left( a\right) \right\Vert _{\mathcal{B}_{p}\left(
L^{2}\left( \mathbb{R}^{n}\right) \right) }\leq Cst\left\Vert a\right\Vert
_{S_{w}^{p}}.
\end{equation*}
\end{proposition}

\section{An embedding theorem\label{st35}}

Another proof of the main results on pseudo-differential operators of \cite%
{Arsu 1} and \cite{Arsu 2} can be obtained using results from previous
sections and an embedding theorem. To formulate the result we define some
spaces of Sobolev type. Let $k\in \left\{ 1,...,n\right\} $. Suppose that $%
\mathbb{R}^{n}=\mathbb{R}^{n_{1}}\times ...\times \mathbb{R}^{n_{k}}$. Let $%
\mathbf{t}=\left( t_{1},...,t_{k}\right) \in 
\mathbb{R}
^{k}$ and $1\leq p\leq \infty $. We find it convenient to introduce the
following space 
\begin{gather*}
\mathcal{H}_{p}^{\mathbf{t}}\left( 
\mathbb{R}
^{n}\right) =\left\{ u\in \mathcal{S}^{\prime }\left( 
\mathbb{R}
^{n}\right) :\left( 1-\triangle _{\mathbb{R}^{n_{1}}}\right)
^{t_{1}/2}\otimes ...\otimes \left( 1-\triangle _{\mathbb{R}^{n_{k}}}\right)
^{t_{k}/2}u\in L^{p}\left( \mathbb{R}^{n}\right) \right\} , \\
\left\Vert u\right\Vert _{\mathcal{H}_{p}^{\mathbf{t}}}=\left\Vert \left(
1-\triangle _{\mathbb{R}^{n_{1}}}\right) ^{t_{1}/2}\otimes ...\otimes \left(
1-\triangle _{\mathbb{R}^{n_{k}}}\right) ^{t_{k}/2}u\right\Vert
_{L^{p}},\quad u\in \mathcal{H}_{p}^{\mathbf{t}}.
\end{gather*}%
For $\mathbf{s}=\left( s_{1},...,s_{k}\right) \in 
\mathbb{R}
^{k}$ we define the function%
\begin{gather*}
\left\langle \left\langle \cdot \right\rangle \right\rangle ^{\mathbf{s}}:%
\mathbb{R}^{n}=\mathbb{R}^{n_{1}}\times ...\times \mathbb{R}%
^{n_{k}}\rightarrow \mathbb{R}, \\
\left\langle \left\langle \cdot \right\rangle \right\rangle ^{\mathbf{s}%
}=\left\langle \cdot \right\rangle _{\mathbb{R}^{n_{1}}}^{s_{1}}\otimes
...\otimes \left\langle \cdot \right\rangle _{\mathbb{R}^{n_{k}}}^{s_{k}}.
\end{gather*}%
Then 
\begin{equation*}
\left\langle \left\langle D\right\rangle \right\rangle ^{\mathbf{s}}=\left(
1-\triangle _{\mathbb{R}^{n_{1}}}\right) ^{s_{1}/2}\otimes ...\otimes \left(
1-\triangle _{\mathbb{R}^{n_{k}}}\right) ^{s_{k}/2},
\end{equation*}%
and 
\begin{gather*}
\mathcal{H}_{p}^{\mathbf{t}}=\left\{ u\in \mathcal{S}^{\prime }\left( 
\mathbb{R}
^{n}\right) :\left\langle \left\langle D\right\rangle \right\rangle ^{%
\mathbf{t}}u\in L^{p}\left( \mathbb{R}^{n}\right) \right\} , \\
\left\Vert u\right\Vert _{\mathcal{H}_{p}^{\mathbf{t}}}=\left\Vert
\left\langle \left\langle D\right\rangle \right\rangle ^{\mathbf{t}%
}u\right\Vert _{L^{p}},\quad u\in \mathcal{H}_{p}^{\mathbf{t}}.
\end{gather*}%
Let us note an immediate consequence of Peetre's inequality:%
\begin{equation*}
\left\langle \left\langle \xi +\eta \right\rangle \right\rangle ^{\mathbf{s}%
}\leq 2^{\left\vert \mathbf{s}\right\vert _{1}/2}\left\langle \left\langle
\xi \right\rangle \right\rangle ^{\mathbf{s}}\left\langle \left\langle \eta
\right\rangle \right\rangle ^{\left\vert \mathbf{s}\right\vert },\quad \xi
,\eta \in \mathbb{R}^{n}
\end{equation*}%
where $\left\vert \mathbf{s}\right\vert _{1}=\left\vert s_{1}\right\vert
+...+\left\vert s_{k}\right\vert $ and $\left\vert \mathbf{s}\right\vert
=\left( \left\vert s_{1}\right\vert ,...,\left\vert s_{k}\right\vert \right)
\in 
\mathbb{R}
^{k}$. We put $\mathbf{n=}\left( n_{1},...,n_{k}\right) $.

\begin{theorem}
\label{st28}Suppose that $\mathbb{R}^{n}=\mathbb{R}^{n_{1}}\times ...\times 
\mathbb{R}^{n_{k}}$. If $t_{1}>n_{1},...,t_{k}>n_{k}$ and $1\leq p\leq
\infty $, then $\mathcal{H}_{p}^{\mathbf{t}}(\mathbb{R}^{n})\hookrightarrow
S_{w}^{p}\left( \mathbb{R}^{n}\right) $.
\end{theorem}

\begin{proof}
Let $\chi \in \mathcal{C}_{0}^{\infty }\left( \mathbb{R}^{n}\right) $ be
such that $\sum_{\gamma \in 
\mathbb{Z}
^{n}}\tau _{\gamma }\chi =1$. Let $u\in \mathcal{H}_{p}^{\mathbf{t}}$. For $%
\gamma \in 
\mathbb{Z}
^{n}$ we put $u_{\gamma }=\chi \left( D-\gamma \right) u$. Then 
\begin{equation*}
u=\sum_{\gamma \in \Gamma }u_{\gamma },\quad \mathtt{supp}\left( \widehat{u}%
_{\gamma }\right) \subset \gamma +\mathtt{supp}\chi ,\quad \gamma \in 
\mathbb{Z}
^{n}.
\end{equation*}

Next we shall show that $\left( u_{\gamma }\right) _{\gamma \in 
\mathbb{Z}
^{n}}\subset L^{p}\left( \mathbb{R}^{n}\right) $ and $\sum_{\gamma \in 
\mathbb{Z}
^{n}}\left\Vert u_{\gamma }\right\Vert _{L^{p}}<\infty $. Let $\gamma \in 
\mathbb{Z}
^{n}$. Using lemma \ref{st24} we get%
\begin{eqnarray*}
\left\Vert u_{\gamma }\right\Vert _{L^{p}} &=&\left\Vert \chi \left(
D-\gamma \right) u\right\Vert _{L^{p}} \\
&=&\left\Vert \left\langle \left\langle D\right\rangle \right\rangle ^{-%
\mathbf{t}}\chi \left( D-\gamma \right) \left\langle \left\langle
D\right\rangle \right\rangle ^{\mathbf{t}}u\right\Vert _{L^{p}} \\
&\leq &\left\Vert \mathcal{F}^{-1}\left( \left\langle \left\langle \cdot
\right\rangle \right\rangle ^{-\mathbf{t}}\chi \left( \cdot -\gamma \right)
\right) \right\Vert _{L^{1}}\left\Vert \left\langle \left\langle
D\right\rangle \right\rangle ^{\mathbf{t}}u\right\Vert _{L^{p}} \\
&=&\left\Vert \mathcal{F}^{-1}\left( \left\langle \left\langle \cdot
\right\rangle \right\rangle ^{-\mathbf{t}}\chi \left( \cdot -\gamma \right)
\right) \right\Vert _{L^{1}}\left\Vert u\right\Vert _{\mathcal{H}_{p}^{%
\mathbf{t}}}.
\end{eqnarray*}%
It remains to evaluate $\left\Vert \mathcal{F}^{-1}\left( \left\langle
\left\langle \cdot \right\rangle \right\rangle ^{-\mathbf{t}}\chi \left(
\cdot -\gamma \right) \right) \right\Vert _{L^{1}}$. We have%
\begin{multline*}
\mathcal{F}^{-1}\left( \left\langle \left\langle \cdot \right\rangle
\right\rangle ^{-\mathbf{t}}\chi \left( \cdot -\gamma \right) \right) \left(
x\right) =\left( 2\pi \right) ^{-n}\int \mathtt{e}^{\mathtt{i}\left\langle
x,\xi \right\rangle }\left\langle \left\langle \xi \right\rangle
\right\rangle ^{-\mathbf{t}}\chi \left( \xi -\gamma \right) \mathtt{d}\xi \\
=\left( 2\pi \right) ^{-n}\left\langle \left\langle x\right\rangle
\right\rangle ^{-2\mathbf{n}}\int \left\langle \left\langle D\right\rangle
\right\rangle ^{2\mathbf{n}}\left( \mathtt{e}^{\mathtt{i}\left\langle
x,\cdot \right\rangle }\right) \left( \xi \right) \left\langle \left\langle
\xi \right\rangle \right\rangle ^{-\mathbf{t}}\chi \left( \xi -\gamma
\right) \mathtt{d}\xi \\
=\left( 2\pi \right) ^{-n}\left\langle \left\langle x\right\rangle
\right\rangle ^{-2\mathbf{n}}\int \mathtt{e}^{\mathtt{i}\left\langle x,\xi
\right\rangle }\left\langle \left\langle D\right\rangle \right\rangle ^{2%
\mathbf{n}}\left( \left\langle \left\langle \cdot \right\rangle
\right\rangle ^{-\mathbf{t}}\chi \left( \cdot -\gamma \right) \right) \left(
\xi \right) \mathtt{d}\xi .
\end{multline*}%
We obtain%
\begin{equation*}
\left\Vert \mathcal{F}^{-1}\left( \left\langle \left\langle \cdot
\right\rangle \right\rangle ^{-\mathbf{t}}\chi \left( \cdot -\gamma \right)
\right) \right\Vert _{L^{1}}\leq \left( 2\pi \right) ^{-n}\left\Vert
\left\langle \left\langle \cdot \right\rangle \right\rangle ^{-2\mathbf{n}%
}\right\Vert _{L^{1}}\left\Vert \left\langle \left\langle D\right\rangle
\right\rangle ^{2\mathbf{n}}\left( \left\langle \left\langle \cdot
\right\rangle \right\rangle ^{-\mathbf{t}}\chi \left( \cdot -\gamma \right)
\right) \right\Vert _{L^{1}}
\end{equation*}%
with%
\begin{equation*}
\left\Vert \left\langle \left\langle D\right\rangle \right\rangle ^{2\mathbf{%
n}}\left( \left\langle \left\langle \cdot \right\rangle \right\rangle ^{-%
\mathbf{t}}\chi \left( \cdot -\gamma \right) \right) \right\Vert
_{L^{1}}\leq \sum_{\left\vert \alpha \right\vert \leq 2n}C_{\alpha
}\left\Vert \left\langle \left\langle \cdot \right\rangle \right\rangle ^{-%
\mathbf{t}}\partial ^{\alpha }\chi \left( \cdot -\gamma \right) \right\Vert
_{L^{1}}.
\end{equation*}%
Further we use Peetre's inequality to obtain 
\begin{eqnarray*}
\left\langle \left\langle \gamma \right\rangle \right\rangle ^{\mathbf{t}%
}\left\Vert \left\langle \left\langle \cdot \right\rangle \right\rangle ^{-%
\mathbf{t}}\partial ^{\alpha }\chi \left( \cdot -\gamma \right) \right\Vert
_{L^{1}} &=&\left\langle \left\langle \gamma \right\rangle \right\rangle ^{%
\mathbf{t}}\int \left\langle \left\langle \xi \right\rangle \right\rangle ^{-%
\mathbf{t}}\left\vert \partial ^{\alpha }\chi \left( \xi -\gamma \right)
\right\vert \mathtt{d}\xi \\
&\leq &2^{\left\vert \mathbf{t}\right\vert _{1}/2}\int \left\langle
\left\langle \xi -\gamma \right\rangle \right\rangle ^{\mathbf{t}}\left\vert
\partial ^{\alpha }\chi \left( \xi -\gamma \right) \right\vert \mathtt{d}\xi
\\
&=&2^{\left\vert \mathbf{t}\right\vert _{1}/2}\left\Vert \left\langle
\left\langle \cdot \right\rangle \right\rangle ^{\mathbf{t}}\partial
^{\alpha }\chi \left( \cdot \right) \right\Vert _{L^{1}}
\end{eqnarray*}%
where $\left\vert \mathbf{t}\right\vert _{1}=\left\vert t_{1}\right\vert
+...+\left\vert t_{k}\right\vert $. Hence 
\begin{multline*}
\left\Vert \left\langle \left\langle D\right\rangle \right\rangle ^{2\mathbf{%
n}}\left( \left\langle \left\langle \cdot \right\rangle \right\rangle ^{-%
\mathbf{t}}\chi \left( \cdot -\gamma \right) \right) \right\Vert _{L^{1}} \\
\leq 2^{\left\vert \mathbf{t}\right\vert _{1}/2}\left( \sum_{\left\vert
\alpha \right\vert \leq 2n}C_{\alpha }\left\Vert \left\langle \left\langle
\cdot \right\rangle \right\rangle ^{\mathbf{t}}\partial ^{\alpha }\chi
\left( \cdot \right) \right\Vert _{L^{1}}\right) \left\langle \left\langle
\gamma \right\rangle \right\rangle ^{-\mathbf{t}}
\end{multline*}%
and%
\begin{multline*}
\left\Vert u_{\gamma }\right\Vert _{L^{p}} \\
\leq 2^{\left\vert \mathbf{t}\right\vert _{1}/2}\left( 2\pi \right)
^{-n}\left\Vert \left\langle \left\langle \cdot \right\rangle \right\rangle
^{-2\mathbf{n}}\right\Vert _{L^{1}}\left( \sum_{\left\vert \alpha
\right\vert \leq 2n}C_{\alpha }\left\Vert \left\langle \left\langle \cdot
\right\rangle \right\rangle ^{\mathbf{t}}\partial ^{\alpha }\chi \left(
\cdot \right) \right\Vert _{L^{1}}\right) \left\Vert u\right\Vert _{\mathcal{%
H}_{p}^{\mathbf{t}}}\left\langle \left\langle \gamma \right\rangle
\right\rangle ^{-\mathbf{t}}.
\end{multline*}%
Using the spectral characterization theorem (theorem \ref{st27}), it follows
that $u\in S_{w}^{p}\left( \mathbb{R}^{n}\right) $ and there is $C=C_{%
\mathbf{t},\mathbf{n},\chi }$ such that 
\begin{equation*}
\left\Vert u\right\Vert _{S_{w}^{p}}\approx \sum_{\gamma \in \Gamma
}\left\Vert u_{\gamma }\right\Vert _{L^{p}}\leq C\left\Vert u\right\Vert _{%
\mathcal{H}_{p}^{\mathbf{t}}}.
\end{equation*}
\end{proof}

In fact, the result proved in the particular orthogonal decomposition $%
\mathbb{R}^{n}=\mathbb{R}^{n_{1}}\times ...\times \mathbb{R}^{n_{k}}$ is
true for arbitrary orthogonal decomposition of $\mathbb{R}^{n}$.

First we shall associate to an orthogonal decomposition a family of Banach
spaces such as those introduced at the beginning of the section. Let $k\in
\left\{ 1,...,n\right\} $, $1\leq p\leq \infty $ and $\mathbf{t}=\left(
t_{1},...,t_{k}\right) \in 
\mathbb{R}
^{k}$. Let $\mathbb{V}=\left( V_{1},...,V_{k}\right) $ denote an orthogonal
decomposition, i.e.%
\begin{equation*}
\mathbb{R}^{n}=V_{1}\oplus ...\oplus V_{k}.
\end{equation*}%
We introduce the Banach space $\mathcal{H}_{p,\mathbb{V}}^{\mathbf{t}}\left( 
\mathbb{R}
^{n}\right) =\mathcal{H}_{p,V_{1},...,V_{k}}^{t_{1},...,t_{k}}\left( 
\mathbb{R}
^{n}\right) $ defined by%
\begin{eqnarray*}
\mathcal{H}_{p,\mathbb{V}}^{\mathbf{t}}\left( 
\mathbb{R}
^{n}\right) &=&\left\{ u\in \mathcal{S}^{\prime }\left( 
\mathbb{R}
^{n}\right) :\left( 1-\triangle _{V_{1}}\right) ^{t_{1}/2}\otimes ...\otimes
\left( 1-\triangle _{V_{k}}\right) ^{t_{k}/2}u\in L^{p}\left( \mathbb{R}%
^{n}\right) \right\} , \\
\left\Vert u\right\Vert _{\mathcal{H}_{p,\mathbb{V}}^{\mathbf{t}}}
&=&\left\Vert \left( 1-\triangle _{V_{1}}\right) ^{t_{1}/2}\otimes
...\otimes \left( 1-\triangle _{V_{k}}\right) ^{t_{k}/2}u\right\Vert
_{L^{p}},\quad u\in \mathcal{H}_{p,\mathbb{V}}^{\mathbf{t}}.
\end{eqnarray*}%
We recall that if $\lambda :%
\mathbb{R}
^{n}\rightarrow V$ is an isometric linear isomorphism, where $V$ is an
euclidean space, then $\left( 1-\triangle _{V}\right) ^{\gamma /2}u=\left(
1-\triangle \right) ^{\gamma /2}\left( u\circ \lambda \right) \circ \lambda
^{-1}$ (see appendix \ref{st36}).

If $\lambda :%
\mathbb{R}
^{n}\rightarrow \mathbb{R}^{n}$ is a orthogonal transformation satisfying $%
\lambda \left( \mathbb{R}^{n_{1}}\right) =V_{1}$, ..., $\lambda \left( 
\mathbb{R}^{n_{k}}\right) =V_{k}$, then the operator%
\begin{equation*}
\mathcal{H}_{p,V_{1},...,V_{k}}^{t_{1},...,t_{k}}\left( 
\mathbb{R}
^{n}\right) \ni u\rightarrow u\circ \lambda \in \mathcal{H}_{p,\mathbb{R}%
^{n_{1}},...,_{\mathbb{R}^{n_{k}}}}^{\mathbf{t}}\left( 
\mathbb{R}
^{n}\right) =\mathcal{H}_{p}^{\mathbf{t}}\left( 
\mathbb{R}
^{n}\right)
\end{equation*}%
is a linear isometric isomorphism (see appendix \ref{st36}). Previous result
and the invariance of the ideals $S_{w}^{p}$ to the composition with $%
\lambda \in \mathtt{GL}\left( n,%
\mathbb{R}
\right) $ imply the following corollary.

\begin{corollary}
Suppose that $\mathbb{R}^{n}=V_{1}\oplus ...\oplus V_{k}$. If $t_{1}>\dim
V_{1},...,t_{k}>\dim V_{k}$ and $1\leq p\leq \infty $, then $\mathcal{H}%
_{p,V_{1},...,V_{k}}^{t_{1},...,t_{k}}\left( 
\mathbb{R}
^{n}\right) \hookrightarrow S_{w}^{p}\left( \mathbb{R}^{n}\right) $.
\end{corollary}

Let $\mathbb{V}=\left( V_{1},...,V_{k}\right) $ be an orthogonal
decomposition of $\mathbb{R}^{n}$. Let $1\leq p\leq \infty $ and $\mathbf{s}%
=\left( s_{1},...,s_{k}\right) \in 
\mathbb{N}
^{k}$. We define the Banach space $\mathcal{C}_{\mathbf{s,}\mathbb{V}}^{p}(%
\mathbb{R}^{n})=\mathcal{C}_{\mathbf{s,}V_{1},...,V_{k}}^{p}(\mathbb{R}^{n})$
by%
\begin{gather*}
\mathcal{C}_{\mathbf{s,}V_{1},...,V_{k}}^{p}(\mathbb{R}^{n})=\left\{ u\in 
\mathcal{S}^{\prime }(\mathbb{R}^{n}):\partial _{V_{1}}^{\alpha
_{1}}...\partial _{V_{k}}^{\alpha _{k}}u\in L^{p}\left( \mathbb{R}%
^{n}\right) ,\left\vert \alpha _{1}\right\vert \leq s_{1},...,\left\vert
\alpha _{k}\right\vert \leq s_{k}\right\} , \\
\left\Vert u\right\Vert _{p,\mathbf{s,}\mathbb{V}}=\max \left\{ \left\Vert
\partial _{V_{1}}^{\alpha _{1}}...\partial _{V_{k}}^{\alpha
_{k}}u\right\Vert _{L^{p}}:\left\vert \alpha _{1}\right\vert \leq
s_{1},...,\left\vert \alpha _{k}\right\vert \leq s_{k}\right\} .
\end{gather*}%
A consequence of Mihlin's theorem is the equality%
\begin{equation*}
\mathcal{H}_{p,\mathbb{V}}^{\mathbf{s}}\left( 
\mathbb{R}
^{n}\right) =\mathcal{C}_{\mathbf{s,}\mathbb{V}}^{p}(\mathbb{R}^{n}),\quad
1<p<\infty .
\end{equation*}%
Put $m_{1}=\left[ \frac{\dim V_{1}}{2}\right] +1,...,m_{k}=\left[ \frac{\dim
V_{k}}{2}\right] +1$ and $\mathbf{m=}\left( m_{1},...,m_{k}\right) $. Then
lemma 5.3 in \cite{Arsu 1} can be reformulated in this way.

\begin{lemma}
Suppose that $\mathbb{R}^{n}=V_{1}\oplus ...\oplus V_{k}$. Then there are $%
t_{1}>\frac{\dim V_{1}}{2}$,..., $t_{k}>\frac{\dim V_{k}}{2}$ such that $%
\mathcal{C}_{\mathbf{m,}V_{1},...,V_{k}}^{p}(\mathbb{R}^{n})\hookrightarrow 
\mathcal{H}_{p,V_{1},...,V_{k}}^{t_{1},...,t_{k}}\left( 
\mathbb{R}
^{n}\right) $ for any $1\leq p\leq \infty $. In particular there is $\gamma
>0$ such that 
\begin{eqnarray*}
\left\Vert u\right\Vert _{\mathcal{H}_{p,\mathbb{V}}^{\mathbf{t}}}
&=&\left\Vert \left( 1-\triangle _{V_{1}}\right) ^{t_{1}/2}\otimes
...\otimes \left( 1-\triangle _{V_{k}}\right) ^{t_{k}/2}u\right\Vert _{L^{p}}
\\
&\leq &\gamma \left\Vert u\right\Vert _{p,\mathbf{m,}\mathbb{V}}=\gamma \max
\left\{ \left\Vert \partial _{V_{1}}^{\alpha _{1}}...\partial
_{V_{k}}^{\alpha _{k}}u\right\Vert _{L^{p}}:\left\vert \alpha
_{1}\right\vert \leq m_{1},...,\left\vert \alpha _{k}\right\vert \leq
m_{k}\right\} .
\end{eqnarray*}%
Here $\mathbf{t}=\left( t_{1},...,t_{k}\right) \in 
\mathbb{R}
^{k}$.
\end{lemma}

By recurence we obtain the following consequence.

\begin{corollary}
Suppose that $\mathbb{R}^{n}=V_{1}\oplus ...\oplus V_{k}$. Then there are $%
t_{1}>\dim V_{1}$,..., $t_{k}>\dim V_{k}$ such that $\mathcal{C}_{2\mathbf{m,%
}V_{1},...,V_{k}}^{p}(\mathbb{R}^{n})\hookrightarrow \mathcal{H}%
_{p,V_{1},...,V_{k}}^{t_{1},...,t_{k}}\left( 
\mathbb{R}
^{n}\right) $ for any $1\leq p\leq \infty $. In particular there is $\gamma
>0$ such that 
\begin{eqnarray*}
\left\Vert u\right\Vert _{\mathcal{H}_{p,\mathbb{V}}^{\mathbf{t}}}
&=&\left\Vert \left( 1-\triangle _{V_{1}}\right) ^{t_{1}/2}\otimes
...\otimes \left( 1-\triangle _{V_{k}}\right) ^{t_{k}/2}u\right\Vert _{L^{p}}
\\
&\leq &\gamma \left\Vert u\right\Vert _{p,2\mathbf{m,}\mathbb{V}}=\gamma
\max \left\{ \left\Vert \partial _{V_{1}}^{\alpha _{1}}...\partial
_{V_{k}}^{\alpha _{k}}u\right\Vert _{L^{p}}:\left\vert \alpha
_{1}\right\vert \leq 2m_{1},...,\left\vert \alpha _{k}\right\vert \leq
2m_{k}\right\} ,
\end{eqnarray*}%
where $\mathbf{t}=\left( t_{1},...,t_{k}\right) \in 
\mathbb{R}
^{k}$.
\end{corollary}

\begin{remark}
If $1<p<\infty $, then using Mihlin's theorem we can replace $2\mathbf{m}$
with $\left( \dim V_{1}+1,...,\dim V_{k}+1\right) $.
\end{remark}

Next we combine previous results with theorem \ref{st21} and theorem \ref%
{st30} to obtain a theorem which extends the main results on
pseudo-differential operators of \cite{Arsu 1} and \cite{Arsu 2}.

Let $1\leq p\leq \infty $, $k\in \left\{ 1,...,2n\right\} $, $\mathbf{t}%
=\left( t_{1},...,t_{k}\right) \in 
\mathbb{R}
^{k}$ and $\mathbf{s}=\left( s_{1},...,s_{k}\right) \in 
\mathbb{N}
^{k}$. Let $\mathbb{V}=\left( V_{1},...,V_{k}\right) $ be an orthogonal
decomposition of $\mathbb{R}^{n}\times \mathbb{R}^{n}$, i.e. $\mathbb{R}%
^{n}\times \mathbb{R}^{n}=V_{1}\oplus ...\oplus V_{k}$. This decomposition
determine the Banach spaces $\mathcal{H}%
_{p,V_{1},...,V_{k}}^{t_{1},...,t_{k}}\left( \mathbb{R}^{n}\times \mathbb{R}%
^{n}\right) $ and $\mathcal{C}_{\mathbf{s,}V_{1},...,V_{k}}^{p}(\mathbb{R}%
^{n}\times \mathbb{R}^{n})$. Put $m_{1}=\left[ \frac{\dim V_{1}}{2}\right]
+1,...,m_{k}=\left[ \frac{\dim V_{k}}{2}\right] +1$ and $\mathbf{m=}%
(m_{1},...,m_{k})$.

\begin{theorem}
Let $1\leq p<\infty $ and $\tau \in \mathtt{End}_{%
\mathbb{R}
}\left( 
\mathbb{R}
^{n}\right) \equiv M_{n\times n}\left( 
\mathbb{R}
\right) $. Suppose that $\mathbb{R}^{n}\times \mathbb{R}^{n}=V_{1}\oplus
...\oplus V_{k}$.

$\left( \mathtt{i}\right) $ Let $\mathbf{t}=\left( t_{1},...,t_{k}\right)
\in 
\mathbb{R}
^{k}$ be such that $t_{1}>\dim V_{1},...,t_{k}>\dim V_{k}$ and let $a\in 
\mathcal{H}_{p,V_{1},...,V_{k}}^{t_{1},...,t_{k}}\left( \mathbb{R}^{n}\times 
\mathbb{R}^{n}\right) $. Then 
\begin{equation*}
\mathtt{Op}_{\tau }\left( a\right) =a^{\tau }\left( X,D\right) \in \mathcal{B%
}_{p}\left( L^{2}\left( 
\mathbb{R}
^{n}\right) \right)
\end{equation*}%
where $\mathcal{B}_{p}\left( L^{2}\left( 
\mathbb{R}
^{n}\right) \right) $ denote the Schatten ideal of compact operators whose
singular values lie in $l^{p}$. We have%
\begin{equation*}
\left\Vert \mathtt{Op}_{\tau }\left( a\right) \right\Vert _{\mathcal{B}%
_{p}}\leq Cst\left\Vert a\right\Vert _{S_{w}^{p}}.
\end{equation*}

$\left( \mathtt{ii}\right) $ Let $a\in \mathcal{C}_{2\mathbf{m,}%
V_{1},...,V_{k}}^{p}(\mathbb{R}^{n}\times \mathbb{R}^{n})$. Then%
\begin{equation*}
\mathtt{Op}_{\tau }\left( a\right) =a^{\tau }\left( X,D\right) \in \mathcal{B%
}_{p}\left( L^{2}\left( 
\mathbb{R}
^{n}\right) \right) .
\end{equation*}%
We have%
\begin{equation*}
\left\Vert \mathtt{Op}_{\tau }\left( a\right) \right\Vert _{\mathcal{B}%
_{p}}\leq Cst\left\Vert a\right\Vert _{p,2\mathbf{m,}\mathbb{V}}.
\end{equation*}

$\left( \mathtt{iii}\right) $ If $1<p<\infty $, then we can replace $2%
\mathbf{m}$ with $\left( \dim V_{1}+1,...,\dim V_{k}+1\right) $.

In all cases the mapping%
\begin{equation*}
\mathtt{End}_{%
\mathbb{R}
}\left( 
\mathbb{R}
^{n}\right) \ni \tau \rightarrow \mathtt{Op}_{\tau }\left( a\right) =a^{\tau
}\left( X,D\right) \in \mathcal{B}_{p}\left( L^{2}\left( 
\mathbb{R}
^{n}\right) \right)
\end{equation*}%
is continuous.
\end{theorem}

Obviously we have similar results for the case $p=\infty $, but we do not
insist on them. If we note that $a^{\tau }\left( X,D\right) \in \mathcal{B}%
_{2}\left( L^{2}\left( 
\mathbb{R}
^{n}\right) \right) $ whenever $a\in L^{2}\left( \mathbb{R}^{n}\times 
\mathbb{R}^{n}\right) =H_{2}^{0}\left( \mathbb{R}^{n}\times \mathbb{R}%
^{n}\right) $, then the last theorem and standard interpolation results in
Sobolev spaces (see \cite[Theorem 6.4.5]{Bergh}) give us the following
result.

\begin{theorem}
Let $\mu >1$ and $1\leq p<\infty $. If $\tau \in \mathtt{End}_{%
\mathbb{R}
}\left( 
\mathbb{R}
^{n}\right) \equiv M_{n\times n}\left( 
\mathbb{R}
\right) $ and $a\in H_{p}^{2\mu n\left\vert 1-2/p\right\vert }\left( \mathbb{%
R}^{n}\times \mathbb{R}^{n}\right) $ then 
\begin{equation*}
\mathtt{Op}_{\tau }\left( a\right) =a^{\tau }\left( X,D\right) \in \mathcal{B%
}_{p}\left( L^{2}\left( 
\mathbb{R}
^{n}\right) \right) .
\end{equation*}%
The mapping 
\begin{equation*}
\mathtt{End}_{%
\mathbb{R}
}\left( 
\mathbb{R}
^{n}\right) \ni \tau \rightarrow \mathtt{Op}_{\tau }\left( a\right) =a^{\tau
}\left( X,D\right) \in \mathcal{B}_{p}\left( L^{2}\left( 
\mathbb{R}
^{n}\right) \right)
\end{equation*}%
is continuous and for any $K$ a compact subset of $\mathtt{End}_{%
\mathbb{R}
}\left( 
\mathbb{R}
^{n}\right) \equiv M_{n\times n}\left( 
\mathbb{R}
\right) $, there is $C_{K}>0$ such that 
\begin{equation*}
\left\Vert a_{X}^{\tau }\left( Q,P\right) \right\Vert _{p}\leq
C_{K}\left\Vert a\right\Vert _{H_{p}^{2\mu n\left\vert 1-2/p\right\vert
}\left( \mathfrak{S}\right) },
\end{equation*}%
for any $\tau \in K$.
\end{theorem}

\section{Cordes' lemma}

In \cite{Arsu 2}, using properties of some Sobolev spaces proved by means of
a dyadic decomposition of the unity, we obtained an extension of Cordes'
lemma. This lemma is a basic tool in the method proposed in \cite{Arsu 1}
and \cite{Arsu 2} to study the Schatten-von Neumann class properties of
pseudo-differential operators. For technical reasons, mainly due to the
complicate structure of the Sobolev norm of non-integer order, the parameter 
$\tau \in \mathtt{End}_{%
\mathbb{R}
}\left( 
\mathbb{R}
^{n}\right) $ used in $\tau $-quantization is subject to some restrictions.
More specifically, the parameter $\tau $ belongs to the set 
\begin{eqnarray*}
U &=&\left\{ \tau \in \mathtt{End}_{%
\mathbb{R}
}\left( 
\mathbb{R}
^{n}\right) :\left\vert \det \left( 1-\tau \right) \right\vert +\left\vert
\det \left( \tau \right) \right\vert \neq 0\right\} \\
&=&\left\{ \tau \in \mathtt{End}_{%
\mathbb{R}
}\left( 
\mathbb{R}
^{n}\right) :\left\{ 0,1\right\} \nsubseteq \sigma \left( \tau \right)
\right\} .
\end{eqnarray*}

\begin{remark}
These restrictions appear only if we want that the regularity of the symbols
should be minimal. If we relax the regularity condition on the symbols, then
the set $U$ can be taken the whole space $\mathtt{End}_{%
\mathbb{R}
}\left( 
\mathbb{R}
^{n}\right) $. As consequence, the proof of theorem \ref{st25} $\left( \tau
\in 
\mathbb{R}
\right) $ works without problems for $\tau \in \mathtt{End}_{%
\mathbb{R}
}\left( 
\mathbb{R}
^{n}\right) $.
\end{remark}

Using the results proved in sections \ref{st33} and \ref{st35} we can remove
the restrictions on the parameter $\tau $. Thus we obtain a more general
version of Cordes' lemma true for all $\tau \in \mathtt{End}_{%
\mathbb{R}
}\left( 
\mathbb{R}
^{n}\right) $. Accordingly, the restrictions on the parameter $\tau $ will
be removed from the results concerning the Schatten-von Neumann class
properties of pseudo-differential operators.

In addition to some results from \cite{Arsu 2}, in the proof of Cordes'
lemma, we shall use a particular case of theorem \ref{st28}.

Recall the definition of generalized Sobolev spaces. Let $s\in 
\mathbb{R}
$, $1\leq p\leq \infty $. Define 
\begin{gather*}
H_{p}^{s}\left( 
\mathbb{R}
^{n}\right) =\left\{ u\in \mathcal{S}^{\prime }\left( 
\mathbb{R}
^{n}\right) :\left\langle D\right\rangle ^{s}u\in L^{p}\left( \mathbb{R}%
^{n}\right) \right\} , \\
\left\Vert u\right\Vert _{H_{p}^{s}}=\left\Vert \left\langle D\right\rangle
^{s}u\right\Vert _{L^{p}},\quad u\in H_{p}^{s}.
\end{gather*}

\begin{lemma}
If $s>n$ and $1\leq p\leq \infty $, then $H_{p}^{s}\left( 
\mathbb{R}
^{n}\right) \hookrightarrow S_{w}^{p}\left( \mathbb{R}^{n}\right) $.
\end{lemma}

To state and prove Cordes' lemma we shall work with a very restricted class
of symbols. We shall say that $a:\mathbb{R}^{n}\rightarrow 
\mathbb{C}
$ is a symbol of order $m$ ($m$ any real number) if $a\in \mathcal{C}%
^{\infty }\left( \mathbb{R}^{n}\right) $ and for any $\alpha \in 
\mathbb{N}
^{n}$, there is $C_{\alpha }>0$ such that 
\begin{equation*}
\left\vert \partial ^{\alpha }a\left( x\right) \right\vert \leq C_{\alpha
}\left\langle x\right\rangle ^{m-\left\vert \alpha \right\vert },\quad x\in 
\mathbb{R}^{n}.
\end{equation*}%
We denote by $\mathcal{S}^{m}\left( \mathbb{R}^{n}\right) $ the vector space
of all symbols of order $m$ and observe that 
\begin{equation*}
m_{1}\leq m_{2}\Rightarrow \mathcal{S}^{m_{1}}\left( \mathbb{R}^{n}\right)
\subset \mathcal{S}^{m_{2}}\left( \mathbb{R}^{n}\right) ,\quad \mathcal{S}%
^{m_{1}}\left( \mathbb{R}^{n}\right) \cdot \mathcal{S}^{m_{2}}\left( \mathbb{%
R}^{n}\right) \subset \mathcal{S}^{m_{1}+m_{2}}\left( \mathbb{R}^{n}\right) .
\end{equation*}%
Observe also that $a\in \mathcal{S}^{m}\left( \mathbb{R}^{n}\right)
\Rightarrow \partial ^{\alpha }a\in \mathcal{S}^{m-\left\vert \alpha
\right\vert }\left( \mathbb{R}^{n}\right) $ for each $\alpha \in 
\mathbb{N}
^{n}$. The function $\left\langle x\right\rangle ^{m}$ clearly belongs to $%
\mathcal{S}^{m}\left( \mathbb{R}^{n}\right) $ for any $m\in 
\mathbb{R}
$. We denote by $\mathcal{S}^{\infty }\left( \mathbb{R}^{n}\right) $ the
union of all the spaces $\mathcal{S}^{m}\left( \mathbb{R}^{n}\right) $ and
we note that $\mathcal{S}\left( X\right) =\bigcap_{m\in 
\mathbb{R}
}\mathcal{S}^{m}\left( X\right) $ the space of tempered test functions. It
is clear that $\mathcal{S}^{m}\left( \mathbb{R}^{n}\right) $ is a Fr\'{e}%
chet space with the seni-norms given by 
\begin{equation*}
\left\vert a\right\vert _{m,\alpha }=\sup_{x\in \mathbb{R}^{n}}\left\langle
x\right\rangle ^{-m+\left\vert \alpha \right\vert }\left\vert \partial
^{\alpha }a\left( x\right) \right\vert ,\quad a\in \mathcal{S}^{m}\left( 
\mathbb{R}^{n}\right) ,\alpha \in 
\mathbb{N}
^{n}.
\end{equation*}

Besides the earlier lemma, we use part $\left( \mathtt{iii}\right) $ of
proposition 2.4 from \cite{Arsu 2} which state that%
\begin{equation}
\mathcal{F}^{-1}\left( \bigcap_{m<0}\mathcal{S}^{m}\left( \mathbb{R}%
^{n}\right) \right) \subset L^{1}\left( \mathbb{R}^{n}\right) .  \label{st29}
\end{equation}

\begin{corollary}
$\mathcal{F}^{-1}\left( \bigcap_{m>n}\mathcal{S}^{-m}\left( \mathbb{R}%
^{n}\right) \right) \cup \mathcal{F}\left( \bigcap_{m>n}\mathcal{S}%
^{-m}\left( \mathbb{R}^{n}\right) \right) \subset S_{w}^{1}\left( \mathbb{R}%
^{n}\right) $.
\end{corollary}

\begin{proof}
It is sufficient to prove the inclusion 
\begin{equation*}
\mathcal{F}^{-1}\left( \bigcap_{m>n}\mathcal{S}^{-m}\left( \mathbb{R}%
^{n}\right) \right) \subset S_{w}^{1}\left( \mathbb{R}^{n}\right) .
\end{equation*}

Let $m>n$ and $s\in \left( n,m\right) $. We shall show that if $a\in 
\mathcal{S}^{-m}\left( \mathbb{R}^{n}\right) $, then $\mathcal{F}^{-1}a\in
H_{p}^{s}$. Indeed, if $u=\mathcal{F}^{-1}a$, then%
\begin{equation*}
\left\langle D\right\rangle ^{s}u=\mathcal{F}^{-1}\left( \left\langle \cdot
\right\rangle ^{s}a\right) \in L^{1}\left( \mathbb{R}^{n}\right)
\end{equation*}%
since $\left\langle \cdot \right\rangle ^{s}a\in \mathcal{S}^{s-m}\left( 
\mathbb{R}^{n}\right) $ and $s-m<0$. Using previous lemma we obtain that $%
\mathcal{F}^{-1}a$ is in the Feichtinger algebra $S_{w}^{1}\left( \mathbb{R}%
^{n}\right) $.
\end{proof}

Using the above lemma, lemma \ref{st12} $\left( \mathtt{b}\right) $, theorem %
\ref{st21} and theorem \ref{st30} we get the following extension of Cordes'
lemma.

\begin{corollary}
Suppose that $\mathbb{R}^{n}\times \mathbb{R}^{n}=\mathbb{R}^{n_{1}}\times
...\times \mathbb{R}^{n_{k}}$. Let $t_{1}>n_{1},...,$ $t_{k}>n_{k}$. Let $%
a_{1}\in \mathcal{S}^{-t_{1}}\left( \mathbb{R}^{n_{1}}\right) ,...,a_{k}\in 
\mathcal{S}^{-t_{k}}\left( \mathbb{R}^{n_{k}}\right) $ and $g:\mathbb{R}%
^{n}\times \mathbb{R}^{n}\rightarrow 
\mathbb{C}
$,%
\begin{equation*}
g=\mathcal{F}^{\dag }\left( a_{1}\right) \otimes ...\otimes \mathcal{F}%
^{\dag }\left( a_{k}\right) ,
\end{equation*}%
where $\mathcal{F}^{\dag }$ is either $\mathcal{F}$ or $\mathcal{F}^{-1}$.
If $\tau \in \mathtt{End}_{%
\mathbb{R}
}\left( 
\mathbb{R}
^{n}\right) \equiv M_{n\times n}\left( 
\mathbb{R}
\right) $, then $g^{\tau }\left( X,D\right) $ has an extension in $\mathcal{B%
}_{1}\left( L^{2}\left( 
\mathbb{R}
^{n}\right) \right) $ denoted also by $g^{\tau }\left( X,D\right) $. The
mapping%
\begin{equation*}
\mathtt{End}_{%
\mathbb{R}
}\left( 
\mathbb{R}
^{n}\right) \ni \tau \rightarrow \mathtt{Op}_{\tau }\left( g\right) =g^{\tau
}\left( X,D\right) \in \mathcal{B}_{1}\left( L^{2}\left( 
\mathbb{R}
^{n}\right) \right)
\end{equation*}%
is continuous.
\end{corollary}

If we use this corollary insted of corollary 5.3 in \cite{Arsu 1}, we
recover the results from section \ref{st35} by using the method proposed in 
\cite{Arsu 1}.

\appendix

\section{$L^{2}$-boundedness of global non-degenerate Fourier integral
operators. Boulkhemair's method.\label{st38}}

The spectral characterization of Sj\"{o}strand algebra $S_{w}=S_{w}^{\infty
} $ can be used in the study of $L^{2}$-boundedness of global non-degenerate
Fourier integral operators, i.e. global Fourier integral operators defined
by 
\begin{equation}
Av\left( x\right) =\int_{%
\mathbb{R}
^{n+\nu }}\mathtt{e}^{\mathtt{i}\varphi \left( x,y,\theta \right) }a\left(
x,y,\theta \right) v\left( y\right) \mathtt{d}y\mathtt{d}\theta ,
\label{st15}
\end{equation}%
where $v\in \mathcal{S}\left( \mathbb{R}^{n}\right) $, $a\in S_{w}^{\infty
}\left( 
\mathbb{R}
^{n}\times 
\mathbb{R}
^{n}\times 
\mathbb{R}
^{\nu }\right) $ and $\varphi :%
\mathbb{R}
^{n}\times 
\mathbb{R}
^{n}\times 
\mathbb{R}
^{\nu }\rightarrow 
\mathbb{R}
$ is a $\mathcal{C}^{\infty }$ function which verifies the conditions of 
\cite{Asada} i.e. there is $\delta _{0}>0$ so that 
\begin{gather}
\left\vert \det \mathtt{D}\left( \varphi \right) \left( x,y,\theta \right)
\right\vert \geq \delta _{0},  \label{st18} \\
\mathtt{D}\left( \varphi \right) \left( x,y,\theta \right) =\left( 
\begin{array}{cc}
\partial _{x}\partial _{y}\varphi & \partial _{\theta }\partial _{y}\varphi
\\ 
\partial _{x}\partial _{\theta }\varphi & \partial _{\theta }\partial
_{\theta }\varphi%
\end{array}%
\right)  \label{st19}
\end{gather}%
and $\mathtt{d}\left( \varphi \right) \in \mathcal{BC}^{\infty }\left( 
\mathbb{R}
^{n}\times 
\mathbb{R}
^{n}\times 
\mathbb{R}
^{\nu }\right) $ where $\mathtt{d}\left( \varphi \right) $ is an arbitrary
element of the matrix $\mathtt{D}\left( \varphi \right) $.

\textsc{the definition of oscillatory integral (\ref{st15}).}

Let $a\in S_{w}^{\infty }\left( 
\mathbb{R}
^{n}\times 
\mathbb{R}
^{n}\times 
\mathbb{R}
^{\nu }\right) $. According to corollary \ref{st10} and lemma \ref{st11} we
have the representation 
\begin{equation}
a\left( x,y,\theta \right) =\sum_{\left( k,l,m\right) \in 
\mathbb{Z}
^{2n+\upsilon }}\mathtt{e}^{\mathtt{i}\left( \left\langle x,k\right\rangle
+\left\langle y,l\right\rangle +\left\langle \theta ,m\right\rangle \right)
}a_{klm}\left( x,y,\theta \right)  \label{st39}
\end{equation}%
with $\left( a_{klm}\right) _{\left( k,l,m\right) \in 
\mathbb{Z}
^{2n+\upsilon }}\subset \mathcal{BC}^{\infty }\left( 
\mathbb{R}
^{n}\times 
\mathbb{R}
^{n}\times 
\mathbb{R}
^{\nu }\right) $ a bounded family satisfying%
\begin{equation*}
\sum_{\left( k,l,m\right) \in 
\mathbb{Z}
^{2n+\upsilon }}\left\Vert a_{klm}\right\Vert _{L^{\infty }}<\infty
,\sum_{\left( k,l,m\right) \in 
\mathbb{Z}
^{2n+\upsilon }}\left\Vert a_{klm}\right\Vert _{L^{\infty }}\approx
\left\Vert a\right\Vert _{S_{w}^{\infty }}.
\end{equation*}

We define the operator $A$ as follows. For $u$, $v\in \mathcal{S}\left( 
\mathbb{R}^{n}\right) $, the integral 
\begin{equation*}
\int_{%
\mathbb{R}
^{2n+\nu }}\mathtt{e}^{\mathtt{i}\left( \varphi \left( x,y,\theta \right)
+\left\langle \theta ,m\right\rangle \right) }a_{klm}\left( x,y,\theta
\right) \mathtt{e}^{\mathtt{i}\left\langle x,k\right\rangle }u\left(
x\right) \mathtt{e}^{\mathtt{i}\left\langle y,l\right\rangle }v\left(
y\right) \mathtt{d}x\mathtt{d}y\mathtt{d}\theta
\end{equation*}%
makes sense as an oscillatory integral because $a_{klm}\in \mathcal{BC}%
^{\infty }\left( 
\mathbb{R}
^{n}\times 
\mathbb{R}
^{n}\times 
\mathbb{R}
^{\nu }\right) $. Put 
\begin{eqnarray*}
u_{k}\left( x\right) &=&\mathtt{e}^{\mathtt{i}\left\langle x,k\right\rangle
}u\left( x\right) ,\quad \left\Vert u_{k}\right\Vert _{L^{2}}=\left\Vert
u\right\Vert _{L^{2}}, \\
v_{l}\left( y\right) &=&\mathtt{e}^{\mathtt{i}\left\langle y,l\right\rangle
}v\left( y\right) ,\quad \left\Vert v_{l}\right\Vert _{L^{2}}=\left\Vert
v\right\Vert _{L^{2}}.
\end{eqnarray*}%
Now, since the phase function $\varphi \left( x,y,\theta \right)
+\left\langle \theta ,m\right\rangle $ satisfies the same conditions as $%
\varphi $, we can use the results of Asada and Fujiwara on $L^{2}$%
-boundedness to obtain that there is $N\in 
\mathbb{N}
$ independent of $k,l,m$ such that 
\begin{multline*}
\left\vert \int_{%
\mathbb{R}
^{2n+\nu }}\mathtt{e}^{\mathtt{i}\left( \varphi \left( x,y,\theta \right)
+\left\langle \theta ,m\right\rangle \right) }a_{klm}\left( x,y,\theta
\right) \mathtt{e}^{\mathtt{i}\left\langle x,k\right\rangle }u\left(
x\right) \mathtt{e}^{\mathtt{i}\left\langle y,l\right\rangle }v\left(
y\right) \mathtt{d}x\mathtt{d}y\mathtt{d}\theta \right\vert \\
\leq Cst\sup_{\left\vert \alpha \right\vert \leq N}\left\Vert \partial
^{\alpha }a_{klm}\right\Vert _{L^{\infty }}\left\Vert u_{k}\right\Vert
_{L^{2}}\left\Vert v_{l}\right\Vert _{L^{2}} \\
\leq Cst\left\Vert a_{klm}\right\Vert _{L^{\infty }}\left\Vert u\right\Vert
_{L^{2}}\left\Vert v\right\Vert _{L^{2}}.
\end{multline*}%
Here we used lemma \ref{st11}. Since%
\begin{equation*}
\sum_{\left( k,l,m\right) \in 
\mathbb{Z}
^{2n+\upsilon }}\left\Vert a_{klm}\right\Vert _{L^{\infty }}<\infty
,\sum_{\left( k,l,m\right) \in 
\mathbb{Z}
^{2n+\upsilon }}\left\Vert a_{klm}\right\Vert _{L^{\infty }}\approx
\left\Vert a\right\Vert _{S_{w}^{\infty }}
\end{equation*}%
we obtain that the series 
\begin{equation}
\sum_{k,l,m}\int_{%
\mathbb{R}
^{2n+\nu }}\mathtt{e}^{\mathtt{i}\left( \varphi \left( x,y,\theta \right)
+\left\langle \theta ,m\right\rangle \right) }a_{klm}\left( x,y,\theta
\right) \mathtt{e}^{\mathtt{i}\left\langle x,k\right\rangle }u\left(
x\right) \mathtt{e}^{\mathtt{i}\left\langle y,l\right\rangle }v\left(
y\right) \mathtt{d}x\mathtt{d}y\mathtt{d}\theta =\left\langle
Av,u\right\rangle  \label{st16}
\end{equation}%
is absolutely convergent and that the operator $A$ defined by this series
satisfies 
\begin{equation*}
\left\vert \left\langle Av,u\right\rangle \right\vert \leq Cst\left\Vert
a\right\Vert _{S_{w}^{\infty }}\left\Vert u\right\Vert _{L^{2}}\left\Vert
v\right\Vert _{L^{2}}.
\end{equation*}

It remains to show that the definition is independent of the representation $%
\left( \ref{st39}\right) $. Let $\psi \in \mathcal{S}\left( \mathbb{R}^{\nu
}\right) $ be such that $\psi \left( 0\right) =1$. For $\varepsilon \in
\left( 0,1\right] $, we put $\psi ^{\varepsilon }=\psi \left( \varepsilon
\cdot \right) $ and 
\begin{equation*}
A_{\varepsilon }v\left( x\right) =\int_{%
\mathbb{R}
^{n+\nu }}\mathtt{e}^{\mathtt{i}\varphi \left( x,y,\theta \right) }a\left(
x,y,\theta \right) \psi \left( \varepsilon \theta \right) v\left( y\right) 
\mathtt{d}y\mathtt{d}\theta ,\quad v\in \mathcal{S}\left( \mathbb{R}^{\nu
}\right) .
\end{equation*}%
This is an usual integral. Let $u$, $v\in \mathcal{S}\left( \mathbb{R}%
^{n}\right) $. Then 
\begin{equation*}
\int_{%
\mathbb{R}
^{n+\nu }}\mathtt{e}^{\mathtt{i}\varphi \left( x,y,\theta \right) }a\left(
x,y,\theta \right) \psi \left( \varepsilon \theta \right) u\left( x\right)
v\left( y\right) \mathtt{d}x\mathtt{d}y\mathtt{d}\theta =\left\langle a,%
\mathtt{e}^{\mathtt{i}\varphi }u\otimes v\otimes \psi ^{\varepsilon
}\right\rangle =\left\langle A_{\varepsilon }v,u\right\rangle 
\end{equation*}%
and again the integral is an usual one. If 
\begin{equation*}
a\left( x,y,\theta \right) =\sum_{\left( k,l,m\right) \in 
\mathbb{Z}
^{2n+\upsilon }}\mathtt{e}^{\mathtt{i}\left( \left\langle x,k\right\rangle
+\left\langle y,l\right\rangle +\left\langle \theta ,m\right\rangle \right)
}a_{klm}\left( x,y,\theta \right) 
\end{equation*}%
is a decomposition of $a\in S_{w}^{\infty }\left( 
\mathbb{R}
^{n}\times 
\mathbb{R}
^{n}\times 
\mathbb{R}
^{\nu }\right) $ given by the corollary \ref{st10}, then 
\begin{multline*}
\left\langle A_{\varepsilon }v,u\right\rangle =\left\langle a,\mathtt{e}^{%
\mathtt{i}\varphi }u\otimes v\otimes \psi ^{\varepsilon }\right\rangle  \\
=\sum_{k,l,m}\left\langle \mathtt{e}^{\mathtt{i}\left( \left\langle
x,k\right\rangle +\left\langle y,l\right\rangle +\left\langle \theta
,m\right\rangle \right) }a_{klm},\mathtt{e}^{\mathtt{i}\varphi }u\otimes
v\otimes \psi ^{\varepsilon }\right\rangle  \\
=\sum_{k,l,m}\int_{%
\mathbb{R}
^{2n+\nu }}\mathtt{e}^{\mathtt{i}\left( \varphi \left( x,y,\theta \right)
+\left\langle \theta ,m\right\rangle \right) }a_{klm}\left( x,y,\theta
\right) \psi \left( \varepsilon \theta \right) \mathtt{e}^{\mathtt{i}%
\left\langle x,k\right\rangle }u\left( x\right) \mathtt{e}^{\mathtt{i}%
\left\langle y,l\right\rangle }v\left( y\right) \mathtt{d}x\mathtt{d}y%
\mathtt{d}\theta .
\end{multline*}%
We obtain that 
\begin{equation*}
\left\langle Av,u\right\rangle -\left\langle A_{\varepsilon
}v,u\right\rangle =\sum_{k,l,m}t_{klm}\left( \varepsilon \right) 
\end{equation*}%
with 
\begin{multline*}
t_{klm}\left( \varepsilon \right)  \\
=\int_{%
\mathbb{R}
^{2n+\nu }}\mathtt{e}^{\mathtt{i}\left( \varphi \left( x,y,\theta \right)
+\left\langle \theta ,m\right\rangle \right) }a_{klm}\left( x,y,\theta
\right) \left( 1-\psi \left( \varepsilon \theta \right) \right) \mathtt{e}^{%
\mathtt{i}\left\langle x,k\right\rangle }u\left( x\right) \mathtt{e}^{%
\mathtt{i}\left\langle y,l\right\rangle }v\left( y\right) \mathtt{d}x\mathtt{%
d}y\mathtt{d}\theta 
\end{multline*}%
Now lemma \ref{st11} implies the estimate%
\begin{eqnarray*}
\sup_{\left\vert \alpha \right\vert \leq N}\left\Vert \partial ^{\alpha
}\left( \left( 1-\psi ^{\varepsilon }\right) a_{klm}\right) \right\Vert
_{L^{\infty }} &\leq &C\left( \psi \right) \sup_{\left\vert \alpha
\right\vert \leq N}\left\Vert \partial ^{\alpha }a_{klm}\right\Vert
_{L^{\infty }} \\
&\leq &C^{\prime }\left( \psi \right) \left\Vert a_{klm}\right\Vert
_{L^{\infty }},\quad \left( k,l,m\right) \in 
\mathbb{Z}
^{2n+\upsilon }.
\end{eqnarray*}%
It follows that the series $\sum_{k,l,m}t_{klm}\left( \varepsilon \right) $
is uniformly convergent. As for any $\left( k,l,m\right) $ $\in 
\mathbb{Z}
^{2n+\upsilon }$ we have $\lim_{\varepsilon \rightarrow 0}t_{klm}\left(
\varepsilon \right) =0$ it follows that 
\begin{equation*}
\lim_{\varepsilon \rightarrow 0}\left\vert \left\langle Av,u\right\rangle
-\left\langle A_{\varepsilon }v,u\right\rangle \right\vert =0.
\end{equation*}%
Hence 
\begin{equation}
\left\langle Av,u\right\rangle =\lim_{\varepsilon \rightarrow 0}\left\langle
A_{\varepsilon }v,u\right\rangle =\lim_{\varepsilon \rightarrow
0}\left\langle a,\mathtt{e}^{\mathtt{i}\varphi }u\otimes v\otimes \psi
^{\varepsilon }\right\rangle ,\quad u,v\in \mathcal{S}\left( \mathbb{R}%
^{n}\right) .  \label{st17}
\end{equation}

\begin{proposition}
Let $a\in S_{w}^{\infty }\left( 
\mathbb{R}
^{n}\times 
\mathbb{R}
^{n}\times 
\mathbb{R}
^{\nu }\right) $ and let $\varphi :%
\mathbb{R}
^{n}\times 
\mathbb{R}
^{n}\times 
\mathbb{R}
^{\nu }\rightarrow 
\mathbb{R}
$ be a smooth phase function such that the Asada-Fujiwara conditions $(\ref%
{st18})$ and $(\ref{st19})$ are fulfilled. Let $A$ be the operator defined
by $\left( \ref{st16}\right) $, $\left( \ref{st17}\right) $. Then $A\in 
\mathcal{B}\left( L^{2}\left( \mathbb{R}^{n}\right) \right) $ and 
\begin{equation*}
\left\Vert A\right\Vert _{\mathcal{B}\left( L^{2}\left( \mathbb{R}%
^{n}\right) \right) }\leq Cst\left\Vert a\right\Vert _{S_{w}^{\infty }}.
\end{equation*}
\end{proposition}

\section{Convolution operators\label{st36}}

Let $V$ be an $n$ dimensional real vector, $V^{\prime }$ its dual and $%
\left\langle \cdot ,\cdot \right\rangle _{V,V^{\prime }}:V\times V^{\prime
}\rightarrow 
\mathbb{R}
$ the duality form. We define the (Fourier) transforms 
\begin{eqnarray*}
\mathcal{F}_{V},\ \overline{\mathcal{F}}_{V} &:&\mathcal{S}^{\prime }\left(
V\right) \rightarrow \mathcal{S}^{\prime }\left( V^{\prime }\right) , \\
\mathcal{F}_{V^{\prime }},\ \overline{\mathcal{F}}_{V^{\prime }} &:&\mathcal{%
S}^{\prime }\left( V^{\prime }\right) \rightarrow \mathcal{S}^{\prime
}\left( V\right) ,
\end{eqnarray*}%
by 
\begin{eqnarray*}
(\mathcal{F}_{V}u)(\xi ) &=&\int_{V}\mathtt{e}^{-\mathtt{i}\left\langle
x,\xi \right\rangle _{V,V^{\prime }}}u(x)\mathtt{d}\mu \left( x\right) , \\
(\overline{\mathcal{F}}_{V}u)(\xi ) &=&\int_{V}\mathtt{e}^{+\mathtt{i}%
\left\langle x,\xi \right\rangle _{V,V^{\prime }}}u(x)\mathtt{d}\mu \left(
x\right) , \\
(\mathcal{F}_{V^{\prime }}v)(x) &=&\int_{V^{\prime }}\mathtt{e}^{-\mathtt{i}%
\left\langle x,\xi \right\rangle _{V,V^{\prime }}}v(\xi )\mathtt{d}\mu
^{\prime }\left( \xi \right) , \\
(\overline{\mathcal{F}}_{V^{\prime }}v)(x) &=&\int_{V^{\prime }}\mathtt{e}^{+%
\mathtt{i}\left\langle x,\xi \right\rangle _{V,V^{\prime }}}v(p\xi )\mathtt{d%
}\mu ^{\prime }\left( \xi \right) .
\end{eqnarray*}%
Here $\mu $ is a Lebesgue (Haar) measure in $V$ and $\mu ^{\prime }$ is the
dual one in $V^{\prime }$ such that Fourier's inversion formulas hold: $%
\overline{\mathcal{F}}_{V^{\prime }}\circ \mathcal{F}_{V}=\mathtt{1}_{%
\mathcal{S}^{\prime }\left( V\right) }$, $\mathcal{F}_{V}\circ \overline{%
\mathcal{F}}_{V^{\prime }}=\mathtt{1}_{\mathcal{S}^{\prime }\left( V^{\prime
}\right) }$. Replacing $\mu $ by $c\mu $ one must change $\mu ^{\prime }$ to 
$c^{-1}\mu ^{\prime }$. These (Fourier) transforms defined above are unitary
operators $\mathcal{F}_{V},\overline{\mathcal{F}}_{V}:L^{2}\left( V\right)
\rightarrow L^{2}\left( V^{\prime }\right) $ and $\mathcal{F}_{V^{\prime }},%
\overline{\mathcal{F}}_{V^{\prime }}:L^{2}\left( V^{\prime }\right)
\rightarrow L^{2}\left( V\right) $.

If $V=%
\mathbb{R}
^{n}$ we use $\mu =m$ -Lebesgue measure in $V$ with the dual $\mu ^{\prime
}=\left( 2\pi \right) ^{-n}m$ in $V^{\prime }$, so that $\mathcal{F}_{V}=%
\mathcal{F}$ and $\overline{\mathcal{F}}_{V^{\prime }}=\mathcal{F}^{-1}$,
where $\mathcal{F}$ is the usual Fourier transform.

Let $\lambda :%
\mathbb{R}
^{n}\rightarrow V$ be a linear isomorphism. Then there is $c_{\lambda }\in 
\mathbb{C}
^{\ast }$ such that for $u\in \mathcal{S}\left( V\right) $ we have 
\begin{equation*}
\mathcal{F}\left( u\circ \lambda \right) =c_{\lambda }\left( \mathcal{F}%
_{V}u\right) \circ {}^{t}\lambda ^{-1}.
\end{equation*}%
Indeed, the measure $m\circ \lambda ^{-1}$ is a Haar measure in $V$ and
therefore there is $c_{\lambda }\in 
\mathbb{C}
^{\ast }$ such that $m\circ \lambda ^{-1}=c_{\lambda }\mu $. Then 
\begin{eqnarray*}
\mathcal{F}\left( u\circ \lambda \right) &=&\int_{%
\mathbb{R}
^{n}}\mathtt{e}^{-\mathtt{i}\left\langle y,\cdot \right\rangle }\left(
u\circ \lambda \right) \left( y\right) \mathtt{d}y=\int_{V}\mathtt{e}^{-%
\mathtt{i}\left\langle \lambda ^{-1}x,\cdot \right\rangle }u\left( x\right) 
\mathtt{d}\left( m\circ \lambda ^{-1}\right) \left( x\right) \\
&=&c_{\lambda }\int_{V}\mathtt{e}^{-\mathtt{i}\left\langle x,^{t}\lambda
^{-1}\cdot \right\rangle _{V,V^{\prime }}}u\left( x\right) \mathtt{d}\mu
\left( x\right) =c_{\lambda }\left( \mathcal{F}_{V}u\right) \circ
{}^{t}\lambda ^{-1}.
\end{eqnarray*}

Similarly it is shown that there is $c_{\lambda }^{\prime }\in 
\mathbb{C}
^{\ast }$ such that for $w\in \mathcal{S}\left( V^{\prime }\right) $ we have%
\begin{equation*}
\mathcal{F}^{-1}\left( w\circ {}^{t}\lambda ^{-1}\right) =c_{\lambda
}^{\prime }\left( \overline{\mathcal{F}}_{V^{\prime }}w\right) \circ \lambda
.
\end{equation*}

Let us note that the above formulas hold for elements in $L^{2}$, $\mathcal{S%
}^{\prime }$, ....

Since the measures $\mu $ and $\mu ^{\prime }$ are dual, we obtain that $%
c_{\lambda }c_{\lambda }^{\prime }=1$.%
\begin{eqnarray*}
u\circ \lambda &=&\mathcal{F}^{-1}\mathcal{F}\left( u\circ \lambda \right)
=c_{\lambda }\mathcal{F}^{-1}\left( \left( \mathcal{F}_{V}u\right) \circ
{}^{t}\lambda ^{-1}\right) \\
&=&c_{\lambda }c_{\lambda }^{\prime }\left( \overline{\mathcal{F}}%
_{V^{\prime }}\mathcal{F}_{V}u\right) \circ \lambda =c_{\lambda }c_{\lambda
}^{\prime }\left( u\circ \lambda \right) \\
&\Rightarrow &c_{\lambda }c_{\lambda }^{\prime }=1.
\end{eqnarray*}

Let $a:V^{\prime }\rightarrow 
\mathbb{C}
$ be a measurable function. We introduce the densely defined operator $%
a\left( D_{V}\right) $ defined by 
\begin{equation*}
a\left( D_{V}\right) =\overline{\mathcal{F}}_{V^{\prime }}M_{a}\mathcal{F}%
_{V}.
\end{equation*}

\begin{lemma}
Let $\lambda :%
\mathbb{R}
^{n}\rightarrow V$ be a linear isomorphism. Then

$\left( \mathtt{a}\right) $ $u\in \mathcal{D}\left( a\left( D_{V}\right)
\right) $ if and only if $u\circ \lambda \in \mathcal{D}\left( \left( a\circ
{}^{t}\lambda ^{-1}\right) \left( D\right) \right) $.

$\left( \mathtt{b}\right) $ For any $u\in \mathcal{D}\left( a\left(
D_{V}\right) \right) $ 
\begin{equation*}
\left( a\circ {}^{t}\lambda ^{-1}\right) \left( D\right) \left( u\circ
\lambda \right) =a\left( D_{V}\right) u\circ \lambda .
\end{equation*}
\end{lemma}

\begin{proof}
$\left( \mathtt{a}\right) $ The following statements are equivalent:

$\left( \mathtt{i}\right) $ $u\in \mathcal{D}\left( a\left( D_{V}\right)
\right) .$

$\left( \mathtt{ii}\right) $ $a\mathcal{F}_{V}u\in L^{2}\left( V\right) .$

$\left( \mathtt{iii}\right) $ $\left( a\mathcal{F}_{V}u\right) \circ
{}^{t}\lambda ^{-1}\in L^{2}\left( 
\mathbb{R}
^{n}\right) .$

$\left( \mathtt{iv}\right) $ $\left( a\circ {}^{t}\lambda ^{-1}\right)
\left( \mathcal{F}_{V}u\circ {}^{t}\lambda ^{-1}\right) \in L^{2}\left( 
\mathbb{R}
^{n}\right) .$

$\left( \mathtt{v}\right) $ $\left( a\circ {}^{t}\lambda ^{-1}\right) 
\mathcal{F}\left( u\circ \lambda \right) \in L^{2}\left( 
\mathbb{R}
^{n}\right) .$

$\left( \mathtt{vi}\right) $ $u\circ \lambda \in \mathcal{D}\left( \left(
a\circ {}^{t}\lambda ^{-1}\right) \left( D\right) \right) .$

$\left( \mathtt{b}\right) $ Let $u\in \mathcal{D}\left( a\left( D_{V}\right)
\right) $. Then 
\begin{eqnarray*}
\left( a\circ {}^{t}\lambda ^{-1}\right) \left( D\right) \left( u\circ
\lambda \right) &=&\mathcal{F}^{-1}M_{a\circ {}^{t}\lambda ^{-1}}\mathcal{F}%
\left( u\circ \lambda \right) =c_{\lambda }\mathcal{F}^{-1}\left[ \left(
M_{a}\mathcal{F}_{V}u\right) \circ {}^{t}\lambda ^{-1}\right] \\
&=&c_{\lambda }c_{\lambda }^{\prime }\left( \overline{\mathcal{F}}%
_{V^{\prime }}M_{a}\mathcal{F}_{V}u\right) \circ \lambda =a\left(
D_{V}\right) u\circ \lambda .
\end{eqnarray*}
\end{proof}

If $a\in \mathcal{C}_{\mathtt{pol}}^{\infty }\left( V^{\prime }\right) $
then the operator $a\left( D_{V}\right) =\overline{\mathcal{F}}_{V^{\prime
}}M_{a}\mathcal{F}_{V}$ belongs to $\mathcal{B}\left( \mathcal{S}^{\prime
}\left( V\right) \right) $. In this case we have the following result.

\begin{lemma}
Let $\lambda :%
\mathbb{R}
^{n}\rightarrow V$ be a linear isomorphism. If $a\in \mathcal{C}_{\mathtt{pol%
}}^{\infty }\left( V^{\prime }\right) $ then for any $u\in \mathcal{S}%
^{\prime }\left( V\right) $ 
\begin{equation*}
\left( a\circ {}^{t}\lambda ^{-1}\right) \left( D\right) \left( u\circ
\lambda \right) =a\left( D_{V}\right) u\circ \lambda .
\end{equation*}
\end{lemma}

\begin{proof}
By continuity and density, it suffice to prove the equality for $u\in 
\mathcal{S}\left( V\right) $. But this is done in previous lemma.
\end{proof}

Suppose that $\left( V,\left\vert \cdot \right\vert _{V}\right) $ is an
euclidian space and that $a$ is symbol of the form 
\begin{equation*}
a=b\left( \left\vert \cdot \right\vert _{V^{\prime }}\right) ,
\end{equation*}%
where $b\in \mathcal{C}_{\mathtt{pol}}^{\infty }\left( 
\mathbb{R}
\right) $ and $\left\vert \cdot \right\vert _{V^{\prime }}$ is the dual norm
of $\left\vert \cdot \right\vert _{V}$.

Let $\lambda :%
\mathbb{R}
^{n}\rightarrow V$ be a linear isometric isomorphism. Then $^{t}\lambda
^{-1}:%
\mathbb{R}
^{n}\rightarrow V^{\prime }$ is a linear isometric isomorphism so that $%
\left\vert \cdot \right\vert _{V^{\prime }}\circ {}^{t}\lambda
^{-1}=\left\vert \cdot \right\vert $, where $\left\vert \cdot \right\vert $
is the euclidian in $%
\mathbb{R}
^{n}$. It follows that 
\begin{equation*}
b\left( \left\vert \cdot \right\vert _{V^{\prime }}\right) \circ
{}^{t}\lambda ^{-1}=b\left( \left\vert \cdot \right\vert \right) .
\end{equation*}%
Using the previous lemma we obtain the following result.

\begin{corollary}
Let $\lambda :%
\mathbb{R}
^{n}\rightarrow V$ be a linear isometric isomorphism. If $b\in \mathcal{C}_{%
\mathtt{pol}}^{\infty }\left( 
\mathbb{R}
\right) $ then for any $u\in \mathcal{S}^{\prime }\left( V\right) $ 
\begin{equation*}
b\left( \left\vert D\right\vert \right) \left( u\circ \lambda \right)
=b\left( \left\vert D_{V}\right\vert \right) u\circ \lambda .
\end{equation*}
\end{corollary}

\end{document}